\documentclass[11pt]{amsart}

\DeclareFontFamily{U}{matha}{\hyphenchar\font45}
\DeclareFontShape{U}{matha}{m}{n}{
  <5> <6> <7> <8> <9> <10> gen * matha
  <10.95> matha10 <12> <14.4> <17.28> <20.74> <24.88> matha12
  }{}
\DeclareSymbolFont{matha}{U}{matha}{m}{n}
\DeclareFontFamily{U}{mathx}{\hyphenchar\font45}
\DeclareFontShape{U}{mathx}{m}{n}{
  <5> <6> <7> <8> <9> <10>
  <10.95> <12> <14.4> <17.28> <20.74> <24.88>
  mathx10
  }{}
\DeclareSymbolFont{mathx}{U}{mathx}{m}{n}

\DeclareMathSymbol{\obot}         {2}{matha}{"6B}
\DeclareMathSymbol{\bigobot}       {1}{mathx}{"CB}

\usepackage[pdfauthor={Congling Qiu}, 
    pdftitle={???}%
  dvips,colorlinks=true]{hyperref}

\hypersetup{
    colorlinks,
    citecolor=black,
    filecolor=black,
    linkcolor=black,
    urlcolor=black
}

\usepackage[english]{babel}
\usepackage[toc,page]{appendix}

  \usepackage{multirow}

\usepackage[utf8]{inputenc}

\usepackage{xtab}
\usepackage{amsmath, amssymb%, cancel,verbatim
}
\usepackage{mathrsfs}
\usepackage[all]{xy}
\usepackage{extarrows}

\usepackage{enumerate}
\usepackage{mathtools,booktabs}
\usepackage{color}
\usepackage[nameinlink]{cleveref}
\usepackage{epstopdf} 
\usepackage{booktabs}

\numberwithin{equation}{section}

\theoremstyle{plain}
\newtheorem{proposition}{Proposition}[subsection]
\newtheorem{conj}[proposition]{Conjecture}
\newtheorem{cor}[proposition]{Corollary}
\newtheorem{lem}[proposition]{Lemma}
\newtheorem{thm}[proposition]{Theorem}
\newtheorem{prop}[proposition]{Proposition}

\newtheorem{problem}[proposition]{Problem}

\theoremstyle{definition}
\newtheorem{defn}[proposition]{Definition}

\newtheorem{eg}[proposition]{Example}
\newtheorem{asmp}[proposition]{Assumption}

\theoremstyle{remark}
\newtheorem{rmk}[proposition]{Remark}

\numberwithin{equation}{section}

\newcommand{\BA}{{\mathbb {A}}} \newcommand{\BB}{{\mathbb {B}}}
\newcommand{\BC}{{\mathbb {C}}} 
\newcommand{\BE}{{\mathbb {E}}} 
\newcommand{\BG}{{\mathbb {G}}} 
 
 \newcommand{\BL}{{\mathbb {L}}}
\newcommand{\BM}{{\mathbb {M}}} 
 \newcommand{\BP}{{\mathbb {P}}}
\newcommand{\BQ}{{\mathbb {Q}}} \newcommand{\BR}{{\mathbb {R}}}
 
 \newcommand{\BV}{{\mathbb {V}}}
\newcommand{\BW}{{\mathbb {W}}} \newcommand{\BX}{{\mathbb {X}}}
 \newcommand{\BZ}{{\mathbb {Z}}}

\newcommand{\cA}{{\mathcal {A}}} 
 \newcommand{\cD}{{\mathcal {D}}}
\newcommand{\cE}{{\mathcal {E}}} 
\newcommand{\cG}{{\mathcal {G}}} \newcommand{\cH}{{\mathcal {H}}}
 
 \newcommand{\cL}{{\mathcal {L}}}
\newcommand{\cM}{{\mathcal {M}}} \newcommand{\cN}{{\mathcal {N}}}
\newcommand{\cO}{{\mathcal {O}}} \newcommand{\cP}{{\mathcal {P}}}
 
\newcommand{\cS}{{\mathcal {S}}} 
 
 \newcommand{\cX}{{\mathcal {X}}}
\newcommand{\cY}{{\mathcal {Y}}} \newcommand{\cZ}{{\mathcal {Z}}}

     \newcommand{\aut}{{\mathrm {aut}}} 
\newcommand{\fa}{{\mathfrak{a}}} \newcommand{\fb}{{\mathfrak{b}}}
\newcommand{\fc}{{\mathfrak{c}}} 
\newcommand{\fe}{{\mathfrak{e}}}

\newcommand{\fk}{{\mathfrak{k}}}

\newcommand{\fs}{{\mathfrak{s}}} \newcommand{\ft}{{\mathfrak{t}}}
 
\newcommand{\fw}{{\mathfrak{w}}}

\newcommand{\wt}{\widetilde}\newcommand{\ol}{\overline}
\newcommand{\wh}{\widehat}

\newcommand{\pair}[1]{\langle {#1} \rangle}

\newcommand{\incl}{\hookrightarrow}

\newcommand{\bsl}{\backslash}

 \newcommand{\ep}{\epsilon}

\newcommand{\vil}{\varinjlim}  
\newcommand{\lb}{\left(} \newcommand{\rb}{\right)}

\newcommand{\etale}{\'{e}tale~}

\newcommand{\ad}{{\mathrm{ad}}}

\newcommand{\Aut}{{\mathrm{Aut}}}

\newcommand{\Char}{{\mathrm{Char}}}
\newcommand{\Ch}{{\mathrm{Ch}}}

\newcommand{\curv}{{\mathrm{curv}}}

\newcommand{\diff}{{\mathrm{Diff}}}
\newcommand{\disc}{{\mathrm{Disc}}} 
\newcommand{\der}{{\mathrm{der}}}
 \renewcommand{\div}{{\mathrm{div}}}
\newcommand{\End}{{\mathrm{End}}}  \newcommand{\Tor}{{\mathrm{Tor}}}
\newcommand{\Ram}{{\mathfrak{Ram}}}

\newcommand{\Gal}{{\mathrm{Gal}}} \newcommand{\GL}{{\mathrm{GL}}}

\newcommand{\GSp}{{\mathrm{GSp}}}

\newcommand{\G}{{\mathrm{G}}}\newcommand{\Hom}{{\mathrm{Hom}}}

\newcommand{\wKL}{{\wt K_\Lambda}}

\newcommand{\Lie}{{\mathrm{Lie}}}\newcommand{\LC}{{\mathrm{LC}}}

\newcommand{\diag}{{\mathrm{diag}}}
\newcommand{\Nm}{{\mathrm{Nm}}}\newcommand{\rf}{{\mathrm{f}}}

 \renewcommand{\Re}{{\mathrm{Re}}}

            \newcommand{\Res}{{\mathrm{Res}}}

\newcommand{\Sh}{{\mathrm{Sh}}}

\newcommand{\SL}{{\mathrm{SL}}}
 
\DeclareMathOperator{\Spec}{Spec}\DeclareMathOperator{\Spf}{Spf}\DeclareMathOperator{\MOD}{mod}

  \newcommand{\Ei}{{\mathrm{Ei}}}

\newcommand{\SU}{{\mathrm{SU}}}

\newcommand{\ur}{{\mathrm{ur}}}  
\newcommand{\Vol}{{\mathrm{Vol}}}

\newcommand{\zar}{{\mathrm{zar}}}

\newcommand{\Tr}{{\mathrm{Tr}}}

\newcommand{\qclE}{{}}%\newcommand{\qclE}{{(E)}}
    \newcommand{\qhol}{{\mathrm{qhol}}}  \newcommand{\chol}{{\mathrm{chol}}}

\newcommand{\hol}{{\mathrm{hol}}}

\newcommand{\adm}{{\ol\cL}}
\newcommand{\nadm}{{\ol\cL}}

    \newcommand{\te}{{\theta E}}  
            \newcommand{\tw}{{W\theta}}

\newcommand{\whz}{{\widehat Z}}

\newcommand\supervisor[1]{\def\@supervisor{#1}}

\newcounter{elno}

\renewcommand{\cong}{\simeq}

\usepackage{amsthm}
\usepackage{thmtools}


\author{Congling Qiu} % Your name
\address{Department of Mathematics, Yale University, New Haven, CT 06520, United States}

\email{qiucongling@gmail.edu}

\address{Department of Mathematics, 
Columbia University, 
New York, NY, USA}
\email{xu.yujie@columbia.edu}
\begin{document} 

\title{Modularity of arithmetic special divisors for unitary Shimura varieties (with an appendix by Yujie Xu)}

\begin{abstract} We construct explicit generating series of arithmetic extensions of Kudla's special divisors on integral models of unitary  Shimura varieties
over CM fields with arbitrary split levels and prove that they are modular forms valued in the arithmetic Chow groups.
This provides a partial solution  to  Kudla's   modularity problem.   
 The main ingredient in   our  construction is S.~Zhang's theory of admissible  arithmetic divisors.
 The main ingredient in the proof   is an arithmetic mixed Siegel-Weil formula.
   % The arithmetic mixed Siegel-Weil formula
% the numeric modularity of  arithmetic intersection numbers with   CM points 
%The   innovation in proving the arithmetic mixed Siegel-Weil formula is 
%to compute    arithmetic improper intersections using admissibility and modularity on the generic fiber.   
\end{abstract}
\maketitle 
\tableofcontents

\section{Introduction}
 
 Let $E $ be a CM   field, $V$  a hermitian space over $E$ of signature $(n,1), (n+1,0),. . . ,(n+1,0)$,
 and $X$ a  Shimura variety for $U(V)$. 
 Let  $F$ be the maximal totally real subfield and $F_{>0}$ 
the set of totally positive elements of $F$. 
  For $t \in F_{>0}$,   we have a  special divisor  $Z_t$ on $X$, 
following Kudla's work \cite{Kud97}  for orthogonal Shimura varieties.
Let $[ Z_t]$ be the class of $Z_t$ in the Chow group  $\Ch^1(X)_\BC$  of divisors on $X$ with $\BC$-coefficients.
By Liu \cite{Liu}, 
the generating series 
\begin{equation}
 \label{series20} \text{constant term}+\sum_{t \in F_{>0}}[ Z_t]q^t,
\end{equation}
with 
 a suitable constant term,
is a $\Ch^1(X)_\BC$-valued holomorphic modular form.  %on $\cH^{[F:\BQ]}$ of parallel weight $n+1$ 
  Here   $q=\prod_{k=1}^{[F:\BQ]} e^{2\pi i\tau_k}$ with $\tau=(\tau_k )_{k=1}^{[F:\BQ]}\in\cH^{[F:\BQ]}$ where $\cH$    the usual upper half plane. This is an analog of the theorem of  Borcherds  \cite{Bor}, Yuan, S.~Zhang and W.~Zhang \cite{YZZ1} 
for orthogonal Shimura varieties, which was originally conjectured by Kudla  \cite{Kud97}.
  In \cite{MR1886765,Kud02,Kud03},   Kudla also raised the problem of  finding  (canonical)  arithmetic extensions of  special divisors on 
integral models of  Shimura varieties  to obtain a modular generating series, which is crucial for Kudla's program on arithmetic theta lifting.

The main result of this paper provides a solution to Kudla's modularity problem in the case that $X$ is proper with
arbitrary level structures at   split  places
and 
certain lattice level structures 
   at nonsplit  places.  The  arithmetic extensions are defined using S.~Zhang's theory of admissible  arithmetic divisors.
   %The    normalized admissible arithmetic  extension of a divisor on the generic fiber  is simply the Zariski closure at  almost all places.
Slightly more explicitly, we construct  a   regular integral model $\cX$ of  $X$  proper flat over  $\cO_E$.  An admissible  arithmetic divisor on $\cX$
is an analog of 
an admissible Green function, i.e., one with  harmonic  curvature.
  Consider the    normalized admissible  extension $ Z_t^\nadm$ of  $Z_t$,  which  is   the Zariski closure  at every finite place of $E$ where the model is smooth. Let  $[ Z_t^\nadm]$ be its    class in the arithmetic Chow group.  Then  the generating series 
\begin{equation}
 \label{series2} \text{constant term}+\sum_{t \in F_{>0}}([ Z_t^\nadm]+\fe_t) q^t,
\end{equation}
with 
 a suitable constant term, is a holomorphic modular form. %on $\cH^{[F:\BQ]}$ of parallel weight $n+1$ 
Here   $  \fe_t $ is formed using coefficients of  an explicit   Eisenstein series and its derivative.

Previous to our work, solutions to Kudla's modularity problem  
   were obtained  using different methods by   Kudla,  Rapoport and  Yang \cite{Kud02} \cite{KRY2} for  quaternionic Shimura curves,  Bruinier,  Burgos Gil, and  K\"uhn \cite{BBGK}  for  Hilbert modular surfaces, over $\BQ$ 
with minimal level structures, by 
    Howard  and Madapusi Pera \cite{HMP} for orthogonal Shimura varieties
     over $\BQ$, and 
 by Bruinier,  Howard,  Kudla,  Rapoport and Yang     \cite{Bet} 
   for unitary Shimura varieties over   imaginary quadratic fields, with self-dual lattice level structures.   
 Compared to these results, we expect that the greater generality of the level structures  in our result
  could be  more useful for some purposes. 
For example,  to approach    modularity  in higher codimensions  following the inductive process  in     \cite{YZZ1} for the generic fibers.

In the other direction, S.~Zhang  \cite{Zha20}  introduced the notion of $\mathsf{L}$-liftings of divisor classes (on general polarized
arithmetic varieties), and then deduced  a solution to Kudla's modularity problem directly from        the  modularity results for the generic fibers in the first paragraph, regardless of level structures. The    $\mathsf{L}$-lifting of a divisor class 
 is also admissible but ``normalized" in the level of arithmetic divisor classes using the  Faltings heights.
 Our approach is an explicit alternative of S.~Zhang's.
In some applications, an explicit   modular generating series as our \eqref{series2}  is necessary. 
For example,   W.~Zhang's proof of the arithmetic fundamental lemma \cite{Zha19} used  the explicit result of   \cite{Bet}.

 The main ingredient in the proof  of our main result  is an  arithmetic mixed Siegel-Weil formula, which identifies 
 the arithmetic intersection between the generating series  \eqref{series2}   with a CM 1-cycle on  $\cX$ (associated to an 1-dimensional hermitian subspace of $V$) and an explicit modular form constructed from theta series and (derivatives of) Eisenstein series.

Arithmetic mixed Siegel-Weil formulas  appeared in the literature  in different contexts.
The  first  one  appeared   in the work of Gross and Zagier \cite[p 233, (9.3)]{GZ}
 for   generating series of Hecke operators on the square of a modular curve, and implies their celebrated  formula relating heights of heegner points and derivatives of $L$-functions.  This arithmetic mixed Siegel-Weil formula was
partially generalized to quaternionic  Shimura curves over totally real fields in the work of   Yuan, S.~Zhang and W.~Zhang \cite[1.5.6]{YZZ} on the general Gross-Zagier formula. 
 For certain orthogonal    Shimura varieties over $\BQ$,  an arithmetic mixed Siegel-Weil formula was  conjectured by Bruinier and Yang  \cite[Conjecture 1.3]{BY}. Its analog  
for  unitary Shimura varieties  over   imaginary quadratic fields with certain self-dual lattice level structures  was
 proved by Bruinier,  Howard and Yang    \cite[Theorem C]{BHY}.

In the rest of this introduction,  we first state our   main result  in more detail. Then
we  discuss  its proof. % in the title of the paper.
Finally, we mention   two non-holomorphic modular variants of  \eqref{series2}.

 \subsection{Main result}\label{maind}
  
   To state our main result, we need some preliminaries.
    % In Kudla's original conjecture and   Liu's result  for unitary Shimura varieties, ``weighted" sum of special divisors were considered.
\subsubsection{Admissible divisors}

   Let  $E$ be a number field, $\cX$  a regular scheme  (or more generally Deligne-Mumford stack)   proper flat over $\Spec \cO_E$ and   $\ol\cL=(\cL,\|\cdot\|)$ an ample  hermitian line bundle  on $\cX$.  At  each infinite place $v$ of $E$,
equip  the complex manifold $ \cX_{E_v}  $  with the K\"ahler form that is the curvature form  $\curv(\ol \cL_{E_v})$. 
First, a Green function  is  admissible  (introduced by  Gillet and Soul\'e \cite[5.1]{GS} following Arakelov \cite{Ara}) if its curvature  form $\alpha$ is harmonic, equivalently, on each connected component of $\cX_{E_v}$, $\curv(\ol \cL_{E_v})^{n-1}\wedge \alpha$ is proportional to  $\curv(\ol \cL_{E_v})^{n}$, where $n=\dim \cX_{E_v}$.
It is further    normalized  if  on each connected component of $\cX_{E_v}$, its pairing with (i.e., integration against)  $\curv(\ol \cL_{E_v})^{n}$ is 0.
%(this notion  appeared in \cite[2.3.2]{BGS}). 
Second,  
at each finite place $v$,  a divisor  $Y$ on  $\cX_{ \cO_{E_v}}$ is admissible if it  has ``harmonic curvature"  with respect to  $\ol \cL_{ \cO_{E_v}}$, in the sense that  on each connected component of $\cX_{\cO_{E_v}}$, the linear form  on the space of vertical divisors defined by intersecting with 
$Y\cdot c_1( \cL_{ \cO_{E_v}})^{n-1}$    is proportional to   the linear form defined by intersecting with  $ c_1(  \cL_{ \cO_{E_v}})^{n}$.
%An admissible  Green  function is  normalized if its harmonic projection is 0. %For  a finite place $v$ of $E$,  we analogously have    normalized admissible extension of a divisor  (with respect to  the ample line bundle $\ol \cL$). 
We further call $Y$  normalized  if  its vertical part has   intersection pairing 0 with 
$c_1(\cL)^{n}$.  
Finally, 
an arithmetic divisor on $\cX$ is (normalized)  admissible if it is   (normalized) admissible  at every finite place and its Green function is  (normalized) admissible.  
  For a divisor $Z$ on $\cX_E $,
we have the unique  normalized admissible extension  $Z^\nadm$ on $\cX$ %, unique  modulo  $\im \partial+\im \bar{\partial}$ (in particular, its class in $  \wh\Ch^1(\cX)_{\BC}$ is uniquely  determined).  
%If $\cX_{ \cO_{E_v}}$ is smooth over $\Spec \cO_{E_v}$, the restriction of  $Z^\nadm$ to $\cX_{ \cO_{E_v}}$  is simply the Zariski closure  of $Z$. 
(called the Arakelov lifting of $Z$ in \cite{Zha20}).

Let $\wh\Ch_{\adm,\BC} ^1(\cX) $ be the space of    admissible    arithmetic divisors with $\BC$-coefficients, modulo the $\BC$-span of the principal ones. If $\cX$ is connected, then    the natural map
\begin{equation}\label{natp}\wh\Ch_{\adm,\BC}^1(\cX)\to \Ch^1(\cX_E)_\BC\end{equation} is  surjective  and has an 1-dimensional kernel.  It is the pullback of $   \wh \Ch_\BC^1\lb \Spec \cO_{E} \rb\cong \BC$, where the isomorphism is by  taking degrees.
Then the  $\fe_t\in \BC\subset \wh\Ch_{\adm,\BC}^1(\cX)$ in \eqref{series2} is understood in this way.

\subsubsection{Shimura varieties and integral models}\label{Shimura varieties and integral models}

 Let $V(\BA_E^\infty)$ be the space of  finite adelic points of $V$.
For  an open compact subgroup $K\subset U(V(\BA_E^\infty))$, we have a    $U(V)$-Shimura variety $\Sh(V)_K$  (which could be stacky) of level $K$ defined   over $E$. 
We assume that  $\Sh(V)_K$ is proper, equivalently, $F\neq \BQ$  or $F=\BQ,n=1$ and  $V $ is nonsplit at some finite place.  
 
 Let   $\Lambda\subset V(\BA_E^\infty)$  be  a hermitian  lattice with stabilizer $K_\Lambda\subset U(V(\BA_E^\infty))$. Let
 $K\subset K_\Lambda$  such that $K_v=K_{\Lambda,v}$ for $v$ nonsplit in $E$. 
 We construct  a regular integral model $\cX_E$ of  $\Sh(V)_K$  proper flat over  $\cO_E$ 
 under some conditions  on $E,F,\Lambda$ %, i.e.,a regular Deligne-Mumford stack $\cX_K$  proper flat over $\Spec \cO_E$ with a fixed isomorphism $\cX_{K,E}\cong \Sh(V)_K$,  
(Theorem \ref{hypothm}).
  Our construction   is largely suggested by    Liu.

 %, using the work of Kisin \cite{Kisin-integral-model}, Kisin and Pappas \cite{Kisin-Pappas}, Rapoport, Smithling and W.~Zhang  \cite{RSZ-arithmetic-diagonal-cycles}. 

 We have two constructions  according to different conditions on $E,F$.
First,
 assume that $E/\BQ$ is tamely ramified. We have the normalization in $\Sh(V)_K$ of 
 the  flat model   of $\Sh(V)_{K_\Lambda}$ of Kisin \cite{Kisin-integral-model}, Kisin and Pappas \cite{Kisin-Pappas} over $\cO_{E,(v)}$, for every finite place $v$.
We want to show their regularity and glue them  to obtain a 
regular integral model 
   $\cX_K$ over $\cO_E$.  
 % (without the tamely ramified condition in \cite{Kisin-integral-model}  \cite{Kisin-Pappas} qcl by luck? mention KP need E'=E? Hodge type... natual Hodge type same reflex field).   
 %In particular, if $K_v$ is hyperspecial, $\cX_K$ is the integral canonical model in the sense of Milne \cite{Milne-canonical} and Kisin \cite{Kisin-integral-model}.
For this purpose, we use a certain regular  PEL  moduli space   for a group closely related to $U(V)$ over the ring of integers of  a reflex field $E'/E$, constructed by     Rapoport, Smithling and W.~Zhang 
 \cite{RSZ-arithmetic-diagonal-cycles}. Expectably,  the moduli space and 
 our integral models  are closely related, as shown by Xu in Appendix B (the proof for the general level at split places was  suggested by Liu). 
  Second,
replacing the tameness assumption by  the assumption that $E/\BQ$ is Galois or $E $ is the composition of $F$ with some imaginary quadratic field, which implies that  $E'=E$,
  we can  construct a regular integral model over $\cO_{E,(v)}$ from the above moduli space, following \cite{Let}.
      Moreover, if both the tameness assumption and the replacement     hold, the two constructions give the same  model. 

 We remark that by our choice of $\Lambda$, $\cX_{K_\Lambda}$ is smooth over $\cO_E$ so that the finite part of the normalized  admissible extension of a divisor  on $\Sh(V)_{K_\Lambda}$ is   the Zariski closure   on $\cX_{K_\Lambda}$.

 \subsubsection{Hodge bundles and CM cycles}

Let $\cL_{K_\Lambda}$ be an  arbitrarily line bundle on $\cX_{K_\Lambda}$ extending the  Hodge (line) bundle on $\Sh(V)_{K_\Lambda}$.  
Let $\cL_K$, denoted by $\cL$  if $K$ is clear from the context, be the pullback of $\cL_{K_\Lambda}$ to $\cX_K$.  
Let  $\ol\cL=(\cL,\|\cdot\|)$ where  $\|\cdot\|$ is the descent of  the natural hermitian metric on the hermitian symmetric domain   uniformizing   $\Sh(V)_K$.  It  is   compatible under pullbacks as $K$ shrinks. 
Changing $\ol\cL$, $c_1(\ol\cL)\in  \wh \Ch_\BC^1\lb \cX \rb$ changes by an element in the pullback of $   \wh \Ch_\BC^1\lb \Spec \cO_{E}\rb $.  (It is a special feature due to the smoothness of $\cX_{K_\Lambda}$ over $\cO_E$.)
  In particular, changing  $\ol\cL$ 
will not change the generating series  \eqref{series2}.
 However, this fact does not play a role in our proof.

For an   1-dimensional hermitian subspace  $  W\subset V$, we have 
an associated  0-dimensional Shimura subvariety of $\Sh(V)$. 
On $\cX_{K_\Lambda}$, let the 1-cycle  $\cP_{K_\Lambda}$  be  its Zariski closure, divided by the degree of its genetic fiber  (so that $\deg\cP_{K_\Lambda,E}=1$).  Then  $\cP_{K_\Lambda}$   is independent of the choice of this subspace (\Cref{PKLprop}). We do not need CM cycles at other levels.

   \subsubsection{Generating series}  
 We start with the non-constant terms of the generating series  \eqref{series2}. 
   
For $x\in \BV^{\infty}$ with norm   in $    F_{>0}$,    the orthogonal complement of $\BA_E ^\infty x  $ in $\BV^\infty$ 
defines a  (shifted) unitary Shimura subvariety $Z(x)_K$ of $\Sh(V)_K$ of codimension 1.  
 For   $t\in    F_{>0}$ and a Schwartz function $\phi$ on $V(\BA_E^\infty)$ invariant by $K$,   the weighted  special divisor  is   \begin{equation*}  Z_t=Z_t(\phi)_K=\sum_{x\in K\bsl \BV^{ \infty}, \ q(x)=t}\phi (   x) Z(x)_K,   \end{equation*}
It  is   compatible under pullbacks as $K$ shrinks.

Now we define $\fe_t$.
Let $E(s,\tau)$ be the Siegel-Eisenstein series   on $\cH^{[F:\BQ]}$ associated to $\phi $.
Its $t$-th Whittaker function $E_t(s,\tau)$ has a decomposition  $$E_t(s,\tau)= W_{\infty,t} (s, \tau ) W^\infty_{t} (s ) $$  
into the  infinite component and the finite component.
Here we choose the  $s$-variable so that   $E(0,\tau)$  is holomorphic  of weight $n+1$, equivalently, the infinite component 
$W_{\infty,t} (0, \tau )$ is a multiple of  $q^t$. (Note that $s=0$ is the critical point for the Siegel-Weil formula, but not the center for the functional equation.) The $t$-th Fourier coefficient    of  $E(0,\tau)$ is $\frac{E_t(0,\tau)} {q^t}$.
Define  \begin{equation}\label{intro1.3}\fe_t=\frac{ W_{\infty,t} (0, \tau )   } {q^t}\left.\frac{d}{ds}\right|_{s=0}W^\infty_{t} (s )+\frac{E_t(0,\tau) } {q^t}\log \Nm_{F/\BQ}t.\end{equation}
Then    $\fe_t$  is independent of $\tau$.

We   introduce a number that will appear in the constant term of the generating series. 
 % (\Cref{PKLprop} and \Cref{PKL}).
 Let                                 
\begin{align*}\fa =c_1(\ol \cL_{K_\Lambda})\cdot \cP_{ K_\Lambda}
+2\frac{L'(0,\eta)}{L(0,\eta)}+ \log|\disc_{E}|- {\fb}[F:\BQ]- \frac{[F:\BQ]}{n}\end{align*}
 where $\disc_E\in \BZ$ is the discriminant of $E/\BQ$,  and
  $\fb =-({1+\log 4})$ when $ n=1$ and  more complicated in general. See
\eqref {Finalw} and the remarking following it. 
    We hope  to compute  the Faltings height  $c_1(\ol \cL_{K_\Lambda})\cdot \cP_{ K_\Lambda}$  based on  \cite{YZ} in a future work. And we expect cancellation among  the  terms defining $\fa$ so that the definition of $\fa$ will be elementary  and transparent. 

\begin{thm}[Theorem \ref{main}, \eqref{main1}] \label{mainint}    If $\phi_v=1_{\Lambda_v}$ at ramified places,  the generating series  \eqref{series2} with the constant term being $\phi(0)\lb c_1(\ol \cL^\vee)+\fa\rb$   is a holomorphic modular form  on $\cH^{[F:\BQ]}$ of parallel  weight $n+1$ valued in   $\wh\Ch^{1}_{\adm,\BC}( \cX_K)$.   Here we understand    $\fa,\fe_t\in \BC\subset \wh\Ch_{\adm,\BC}^1(\cX_K)$ as discussed below \eqref{natp}.
\end{thm}  
Since the formation of normalized admissible extension  is compatible under  flat pullbacks, the generating series  \eqref{series2} 
 is compatible under   pullbacks as $K$ shrinks.

 We note that the sum of the normalized admissible Green function for $Z_t$   and $\fe_t$ recovers the    Bruinier-Borcherds   Green function used by Bruinier,  Howard,  Kudla,  Rapoport and Yang     \cite{Bet} for    $F=\BQ$, so that Theorem \ref {mainint} 
 is an analog of  \cite[Theorem B]{Bet}. 
Though
the  Bruinier-Borcherds  construction can not be directly extended to a general  $F$ (as explained to us by Bruinier), our construction could be considered as an alternative generalization.

%Though  in this introduction, we use the language of classical modular forms, we in fact expect an adelic (and stronger) version of \Cref{mainint}, analogous to the adelic modularity theorems in  \cite{Liu}\cite{Liu0}\cite{YZZ1} on the generic fiber. 
%See Conjecture \ref{conj1}.  Theorem \ref{main} is  ``partially adelic".
 
%Theorem \ref{mainint} could be regarded as the analog of the work of Bruinier,  Howard,  Kudla,  Rapoport and Yang \cite[Theorem B]{Bet}   over $\BQ$ 
%(which is for a base level  and $\phi=1_{\Lambda}$, while we allow arbitrary levels and  very general test functions).
%Indeed, when $F=\BQ$, 
%the number $\fe_t$ is essentially the difference between our Green function    for $Z_t$ and the  automorphic Green function with Borcherds' normalization
%used in \cite[Theorem B]{Bet}.  
%However, over a general totally real field $F$, the analog of Borcherds' automorphic  Green function is missing.
%See \ref{Green functions01} for more information.  

% If $K=K_\Lambda$, due to the smoothness of  $\cX_{K_\Lambda}$, the generating series \eqref{suitable1}    does not depend on the choice of $\cL$. For, example the constant term $\lb c_1(\ol \cL_{K_\Lambda}^\vee)+\fa\rb$  has the following simple description: it is the    unique $\BC$-coefficient  arithmetic extension of $c_1(L_{K_\Lambda}^\vee) $   with harmonic curvature at infinite places  whose intersection with   $\cP_{ K_\Lambda}$ is   $2\frac{L'(0,\eta)}{L(0,\eta)}+ \log|\disc_E|- {\fb}[F:\BQ]+ \frac{[F:\BQ]}{n}$.

\subsection{Sketch of the proof}\label{disproof}

Now we discuss the proof of  Theorem \ref {mainint}.
By the
1-dimensionality of  the kernel of  \eqref{natp} and  the    modularity of the generic fiber of the generating series  \eqref{series2}   (i.e. the generating series \eqref{series20}), the    modularity of \eqref{series2} 
is equivalent to the
  modularity of  the generating series of  arithmetic intersection numbers between     $[Z_t^\nadm]+\fe_t $'s and an 1-cycle on $\cX_K$ whose generic fiber has nonzero degree. 
 (A similar strategy  was   used in      \cite{Kud02,KRY2}.)

  Assume that $K=K_\Lambda$ for simplicity and let us take the 1-cycle to be $\cP_{K_\Lambda}$.
Then this generating series of  arithmetic intersection numbers $   \lb  [Z_t^\nadm]+\fe_t\rb\cdot \cP_{K_\Lambda}$
 is the arithmetic  analog of the integration of
 a theta  series of $U(1,1)\times  U(V)  $  along $U(W)\bsl U(W(\BA_E))$, where 
 $W\subset  V$ is the  1-dimensional hermitian subspace defining $\cP_{K_\Lambda}$. 
 By the  Siegel-Weil formula  for $U(1,1)\times  U(W)  $, 
 this integration is a   theta-Eisenstein series, i.e., a linear combination of products   of 
 theta  series and
Eisenstein series\footnote{The name actually comes from ``mixed Eisenstein-theta series" in \cite{YZZ}. We use a slightly different notation to cope with the later notation  ``$\te(s,\tau)$", which we will take derivative on the Eisenstein series and  so that we make the ``$E$" closer to the $s$-variable.  }.

Let $\te(s,\tau)$  be the  theta-Eisenstein series  associated to $\phi$ and let  $\te_{t}(s,\tau)$ be its  $t$-th Whittaker coefficient.
 We study  the 
holomorphic projections 
of  $\te'(0,\tau)$   in order to match the above generating series of  arithmetic intersection numbers, which is supposed to be holomorphic in view of our goal (Theorem \ref {mainint}). 
A priori, there are  two  holomorphic projections. One is  the projection of $\te'(0,\tau)$ to the space of holomorphic cusp forms, which  we call ``cuspidal holomorphic projection" and denote by $\te'_{\chol}(0,\tau)$. Let $\te'_{\chol,t}(0,\tau)$ be its $t$-th Whittaker coefficient.
The other is
for $\te_{t}'(0,\tau)$  and purely at infinite places, which we call
``quasi-holomorphic projection" following \cite{YZZ} and denote by $\te_{t,\qhol}'(0,\tau)$.  
However, neither of them could be the desired match, since $\te'_{\chol}(0,\tau)$ has no constant term and $\te_{t,\qhol}'(0,\tau)$
is (in general) not the $t$-th Whittaker coefficient  of a modular form. 
Thus, we compute their difference, and find that
$\te_{t,\qhol}'(0,\tau)-\te'_{\chol,t}(0,\tau)$ is the  sum  of $2\fe_t q^t$ and  the $t$-th Whittaker coefficient  of a holomorphic
Eisenstein series. See \eqref{Final}.
 In the case $n=1$, the computations of holomorphic projections have  their roots in \cite{GZ},     the strategy we follow is outlined  in \cite[6.4.3]{YZZ}, and    the explicit computation  was done by Yuan \cite{Yuan}.
We largely follow \cite{Yuan}.

 Let $f$ be the sum of $-\frac{1}{2}\te'_{\chol}(0,\tau)$ and the negative of the holomorphic
Eisenstein series in the last sentence. Then $f$ is a holomorphic modular form on $\cH^{[F:\BQ]}$  of parallel weight $n+1$.
The $t$-th Whittaker coefficient of $f$ is 
$ -\frac{1}{2}\frac{\te_{t,\qhol}'(0,\tau) }{q^t}+\fe_t  $.

%Both $   \lb  [Z_t^\nadm]+\fe_t\rb\cdot \cP_{K_\Lambda}$
% and $ -\frac{1}{2}\frac{\te_{t,\qhol}'(0,\tau) }{q^t}+\fe_t  $
% have   decompositions into local components (in fact, only for some ``regular" $\phi$).
In \ref{ecompositio}, we define some explicit Schwartz functions $\phi_v'$, for every ramified place $v$, which are ``error functions"
due to ramification. Let  $g$ be sum of the theta-Eisenstein series  associated to $\phi^v\phi_v'$'s.
 The following  is   our arithmetic mixed Siegel-Weil formula (Theorem \ref{nct}) in the case $K=K_\Lambda$. 
   \begin{thm}\label{masw} 
     Assume that   $\phi_v=1_{\Lambda_v}$ for $v$ nonsplit in $E$.   The arithmetic intersection number  $   \lb  [Z_t^\nadm]+\fe_t\rb\cdot \cP_{K_\Lambda}$ is the $t$-th Fourier coefficient of $f-g-\frac{1}{n}E(0,\tau)$.
   %    Equivalently, $$   [Z_t^\nadm]  \cdot \cP_{K_\Lambda}  =-\frac{1}{2}\frac{\te_{t,\qhol}'(0,\tau) }{q^t}-\frac{1}{n}\frac{E_t(0,\tau)}{q^t}.$$
\end{thm}
In Theorem \ref{nct} for a general $K$, we use the pullback of $\cP_{K_\Lambda}$
to  $\cX_K$, instead of  natural CM cycles on $\cX_K$, to simplify certain local computations. See \Cref{uselat}.

We remind the reader that in Theorem \ref{nct}, we actually use 
the automorphic Green function for $Z_t$  constructed by   Bruinier \cite{Bru0,Bru},   Oda and Tsuzuki \cite{OT} (for $n=1$ and $F=\BQ$, it was well-known, and used by Gross and Zagier   \cite{GZ}). Its
  difference with the
normalized admissible Green function is  $\frac{1}{n}E_t(0,\tau)$  by Lemma \ref{gnm11} and the remark following it.    
 %In the case $F=\BQ$,  by (the analog for unitary Shimura varieties of) \cite[Proposition 2.11]{Bru0}, $\fe_t$ is   the difference between the automorphic Green function  for $Z_t$   and the aforementioned  Bruinier-Borcherds   Green function used by Bruinier,  Howard,  Kudla,  Rapoport and Yang     \cite{Bet}. 

%L2 holo auto form is  cuspidal

 Let us remark the   innovation in proving the arithmetic mixed Siegel-Weil formula.
We  consider the difference of two CM  cycles. The generic  fiber of  the difference has degree 0. 
Then the generating series of  arithmetic intersection numbers is modular by the admissibility.
(A similar observation   was   used in \cite{MZ}   to generalize the arithmetic fundamental lemma. See also \cite{ZZY}.) %. See also 
  This modularity enables us to   ``switch  CM cycles"   and thus avoid   computing improper intersections directly. 
%We refer  the reader to \ref{Strategy}  for a more detailed account  of this strategy. 
 This idea is inspired by \cite{YZZ} and \cite{Zha19}.
%In \cite{BHY},  Bruinier,  Howard and Yang were able to   compute  improper intersections.

\subsection{Non-holomorphic variants}\label{Non-holomorphic variant}
We   obtain 
a non-holomorphic modular variant of the generating series \eqref{series2}, where the sum of the normalized admissible Green function for $Z_t$      and $\fe_t$
is replaced by 
  Kudla's Green function  \cite{Kud1}. See \Cref{maincor}.   %Kudla's Green function   is  also defined for $t \in F^\times\bsl F_{>0}$ with   $Z_t$  understood as $0$.  
  This is an analog of \cite[Theorem A]{KRY2} \cite[Theorem 7.4.1]{Bet}.
Theorem \ref{maincor} follows from Theorem \ref{mainint}, and 
the modularity of the differences   between the generating series of   two kinds of Green functions (\Cref{diff1}).
The latter (\Cref{diff1}) is an analog of the main result of
Ehlen and Sankaran \cite{ES18} 
for $F=\BQ$.

Note that Kudla's Green function is not admissible.     
 We also obtain a non-holomorphic modular generating series with admissible Green functions (Theorem \ref{main}, \eqref{main2}).
This has not appeared in  the literature yet, as far as we know.

 \subsection*{Acknowledgements}
The author   thanks  Yifeng~Liu for his many  useful comments and suggestions. 
He  also thanks  Xinyi~Yuan for sharing his manuscript and answering the author's questions.   He is   grateful to Shouwu~Zhang for proposing the arithmetic mixed Siegel-Weil formula  and for his help on the revision of the paper.
He   thanks   Jan~Hendrik~Bruinier, Stephan~Ehlen,   Ziqi~Guo,
 Chao~Li, and Siddarth~Sankaran,    Wei~Zhang,  Yihang~Zhu, Jialiang~Zou for  their help.
 He also thanks   Yujie~Xu for  providing Appendix B.
 He also thanks the anonymous referees for their advice to improve the presentation of this paper.
 The research   is partially supported by the NSF grant DMS-2000533.

   \section{Some notations and conventions}\label{nots}

\subsection{}
For a number field $F$, let  $\BA_F$ be the  ring of adeles of $F$ and $\BA_F^\infty$   the ring of finite adeles of $F$.   For a finite place $v$ of a number field $F$, let $\varpi_{F_v}$  be a uniformizer of $F_v$. Let $q_{F_v}$ be the cardinality of
$\cO_{F_v}/\varpi_{F_v}$.  The discrete valuation is $v(\varpi_{F_v})=1$ and 
the absolute value $|  \cdot |_{F_v}$ is  $|\varpi_v|_{F_v}=q_{F_v}^{-1}$.
For an infinite place, $v$ is understood as a pair of complex embeddings. If $v$ is real, the absolute value $|  \cdot |_{F_v}$ is the usual one; if $v$ is complex, the absolute value $|  \cdot |_{F_v}$ is the square of the usual one.
Their product is $|\cdot|_{\BA_F}$.
The symbol $|\cdot|$ without a subscript means the usual real or complex absolute value.

Below in this paper,   $E/F$     is  always   a    CM extension. Let $\infty$ be the set of infinite places of $F$.   Let $F_{>0}\subset F$ be the subset of totally positive elements.  %Let $F_{<>0}\subset F$ be the subset of   elements negative at 1 place and positive at other places.
For  a place $v$ of $F$, $E_v$ is understood as $E\otimes _F F_v$.
The nontrivial Galois action will be denoted by $x\mapsto \ol x$ and the norm map $\Nm_{E/F}$ or its local version is abbreviated  as $\Nm$. 
Let $\eta$ be the  associated  quadratic Hecke character  of   $ F^\times\bsl \BA_F^\times$ 
via the class field theory.   
%For a place $v$ of $F$, let  $\eta_v$ be  the  quadratic Hecke character  of  $F_v^\times$ associated  to the quadratic extension $E_v/F_v$. Then $\eta$ is the pure tensor of $\eta_v$'s over all places of $F$.

For a set of place  $S$ of $F$  and a decomposable adelic object $X$  
over $\BA_F$, we use $X_S$ (resp. $X^S$) to denote the $S$-component (resp. component away from $S$) of $X$ if the decomposition of $X$ into the product of  $X_S$  and $X^S$ is clear from the context. %, unless otherwise specified. 
  For example, $\BA_F=\BA_{F,S}\BA_{F}^S$ and  $\BA_E=\BA_{E,S}\BA_{E}^S$ by regarding $\BA_E$ as over $\BA_F$.  Here is another example which is ubiquitous in the paper:  a function $\phi$ on the space of  $\BA_{F}$-points of an algebraic group over $F$ that can be decomposed as 
  $\phi=\phi_S\otimes \phi^S$
  where $\phi_S$ (resp. $\phi^S$) is a function on the set of $\BA_{F,S}$-points  (resp. $\BA_{F}^{S}$-points) of the group.
  Note that such a decomposition of   $\phi$ is not unique.  By using these notations, we understand that we have fixed such a decomposition. See \ref{Functions} for an example.  
      If $S=\{v\}$, we write $X_v$ (resp. $X^v$) for $X_S$ (resp. $X^S$).

\subsection{}\label{hermitian spaces}All hermitian spaces are assumed to be nondegenerate.
We always use $\pair{\cdot,\cdot}$ to denote  a hermitian pairing    and $q(x)=\pair{x,x}$ the hermitian  norm if the underlying hermitian space (over $E$, $E_v$ or $\BA_E$) is indicated in the context. 
For a hermitian space $V$ over $E$, we use $V$ to denote $V(E)$ to lighten the notation if there  is no confusion. 
For $t\in F$, let $V^t=\{v\in V:q(v)=t\}.$ The same notation applies to a local or adelic hermitian space.
We use  $U(V)$ for both the algebraic group $U(V)$ and its group of $F$-points.
Define $$[U(V)]=U(V)\bsl U(V(\BA_E)).$$

A   hermitian space  $\BV/\BA_E$
is called coherent (resp.  incoherent) if its determinant belongs (resp. does not belong) to $  F^\times\Nm ( \BA_E^\times)$, equivalently $\BV\cong V(\BA_E)$ for some (resp. no)  hermitian space $V/E$.
If $\BV$ is incoherent of
dimension $1$,  for  a place $v$ of $F$ nonsplit in $E$, there is a unique hermitian space $V/E$ such $V(\BA_E^v)\cong \BV^v$. We call $V$ the $v$-nearby hermitian space of $\BW$.

%If $F$ is a local field, let  $|\cdot|_F$ be  the normalized absolute value, i.e., if $F$ is nonarchimedean, the value of $|\cdot|_F$ at the uniformizer is the cardinality of  the residue field of $\cO_F$;   $|\cdot|_\BR$ is the usual absolute value;   $|\cdot| $ is the square of the usual absolute value. Let $|x|_E=|  \Nm x|_F$. 
%of $ F\bsl \BA_F$  Adelic absolute values are product of local ones.

\subsection{}\label{Functions}
Let $\cS(\BV)$ be the space of $\BC$-valued Schwartz functions.
For $v\in \infty$ such that $\BV(E_v)$ is positive definite, the standard Gaussian function on $\BV_\infty$ is $e^{-2\pi{q(x)}}\in \cS(\BV(E_v)) $. 
If a hermitian space  $\BV/\BA_E$ is totally positive-definite,  let $$ \ol\cS(\BV)\subset \cS(\BV)$$ be 
the  subspace of functions of the form  
$\phi= \phi_\infty\otimes  \phi^\infty$  where $\phi_\infty$ is the pure tensor of standard Gaussian functions over all infinite places and  $  \phi^\infty \in  \cS(\BV ^\infty)  $    taking values in $\BC$.     
For $\phi\in \cS(\BV ^\infty)  $, we always fix such a decomposition. 

\subsection{}\label{Additive characters, Fourier transform, and measures}
Fix the  additive character   of $F\bsl \BA_F$ to be   $\psi :=\psi_\BQ\circ \Tr_{F/ \BQ} $  where $\psi_\BQ$ is the unique 
additive character  of $\BQ\bsl \BA_\BQ$ such that $\psi_{\BQ,\infty}(x)=e^{2\pi ix}$.
The  additive character  of $\BA_E$ is  $\psi_{E}:=\psi  \circ \Tr_{E/ F} $.
  For  $ t\in \BA_F $ (we in fact only use $t\in F$),
%the twisted character   $\psi_{t}$   of $   \BA_F$
%is  
let $\psi_{t}(b)=\psi  (tb)  .$
For a place $v$ of $F$ and $t\in F_v$,  %the twisted character  $\psi_{v,t}$ of $  F_v$is
Let  $\psi_{v,t}(b)=\psi_v  (tb)  .$
Then $\psi_{v,t}=\psi_{v,t_v}$ for  $ t\in \BA_F $.

Fix the  self-dual Haar measures for $F_v$ and  $E_v$.  
Then $$d_{F_v^\times}x:=\zeta_{F_v}(1)|x|_{F_v}^{-1} d_{F_v}x,\ d_{E_v^\times}x:=\zeta_{E_v}(1)|x|_{E_v}^{-1} d_{E_v}x$$ are the induced Haar measures  on $F_v^\times$ and $E_v^\times$.  The subscripts  will be omitted later in the paper. They induce  the quotient measure on $E_v^\times/F_v^\times \cong U(1)({F_v})$. 
%Then for every $U(1)$ over $F$ (with respect to  $E/F$), we have the same volume $$\Vol( [U(1)])=2L(1,\eta).$$

For $\phi\in \cS(\BV(E_v))$, the Fourier transform of $\phi$ (with respect to  $\psi$ and a Haar measure) is $$\wh \phi(x)=\int_{\BV(E_v)}\phi(y)\psi_{E,v}(\pair{x,y})  dy.$$ 
We fix  the self-dual Haar measure on $\BV(E_v)$.

\subsection{} \label{Group}

Let $G=U(1,1) $ be 
the   unitary group over $F$  of  the  standard skew-hermitian space over $E$ of dimension $2$, i.e., the skew-hermitian  form is given by the matrix \begin{equation*}w= \begin{bmatrix}0&1\\
-1&0\end{bmatrix}.\label{www}\end{equation*}  Then $w\in G(F)$.
We use $w_v$ to denote the same matrix in $G(F_v)$ 

For  $b\in \BG_{a,F}$, let  $$n(b) = \begin{bmatrix}1&b\\
0&1\end{bmatrix}.$$ 
For $a\in \Res_{E/F}\BG_{m,E} $, let $$m(a) = \begin{bmatrix}a&0\\
0&\ol a ^{-1} \end{bmatrix}.$$
Let $N=\{n(b):b \in  \BG_{a,F}\}\subset G,$ $M=\{m(a):a \in \Res_{E/F}\BG_{m,E}  \}\subset G,$
and $P=MN$ the subgroup of upper triangular   matrices. Then $G$ is generated  $P$ and $w$. The isomorphism  $N\cong \BG_{a,F}$ induces an additive character and a Haar measure on $N(\BA_F)$ which we fix in this paper. 

Let $K^{\max}_v$  be the intersection of $G(F_v)  $ with the standard 
maximal compact subgroup  $\GL_{2}(E_v)$. 
Then $K^{\max}_v$ is a maximal compact subgroup  of $G(F_v) $.
For $v\in \infty$,
$K^{\max}_v$
is the group of matrices
$$ [k_1,k_2]:=\frac{1}{2}\begin{bmatrix}k_1+k_2&-ik_1+ik_2\\
ik_1-ik_2&k_1+k_2\end{bmatrix},$$  
where $k_1,k_2\in E_v$  are of norm 1.
We have the   Iwasawa decomposition $$G(F_v)=N(F_v)M(F_v) K^{\max}_v.$$

\subsection{}\label{modulus character}
For  a place   $v$   of $F$, the local modulus character of   $G(F_v) $ is given by
$$\delta_v (g)=|  a|_{E_v}$$
if $g=n(b)m(a)k$  with $k\in K^{\max}_v$ under the Iwasawa decomposition.
The global modulus character $\delta$ of   $G(\BA_F)  $ is the product of  the local ones. 
Since we will use results in \cite{YZ,YZZ}, where the subgroup $\SL_2 \subset G $ is used, we remind the reader that our 
modulus character, when restricted to $\SL_2(F_v)\subset G(F_v)$ is the square of the one in loc. cit..

\subsection{} \label{Holomorphic  automorphic  forms} 
For $\fw=(\fw_v)_{v\in \infty}$, where  $\fw_v$ is a pair of  integers,
let $ \cA (G,\fw) $ be the space of  smooth automorphic forms  for $G$ of weight $\fw$. 
Let $\cA_{\hol}(G,\fw)$ be the subspace of holomorphic automorphic  forms. A characterization is  as follow.
%Let  $f:N(F)\bsl G(\BA_F)\to \BC$ be  a continuous function. 
%For  $t\in F$, define its $\psi_t$-Whittaker function  to be
%$$f_t(g):=\int_{N(F)\bsl N(\BA)}f(n g)\psi_t(-n) dn.$$
For $v\in \infty$, $t\in F_{v, \geq 0}$   and a pair of  integers $(w_1,w_2) $,
the standard holomorphic $\psi_{v,t}$-Whittaker function   on $G(F_v)$ of weight $(w_1,w_2)$  
is $$W^{(w_1,w_2)}_{v,t}(g)=e^{2\pi i t(b+ i  |a|_{E_v})}   |a|_{E_v} ^{(w_1+w_2)/2}k_1^{ w_1} k_2^{ w_2} , $$
for $g=n(b)m(a)[k_1,k_2]$  under the Iwasawa decomposition.  
For $t\in \BA_{F,\infty}$,
let
$$W^{\fw}_{\infty, t}=\prod_{v\in \infty}W^{\fw_v}_{v, t}.$$
An automorphic form $f$  on $G(\BA_F)$ is holomorphic of weight $\fw$ if  for   $t\in F_{>0}\cup \{0\}$, its  $\psi_t$-Whittaker function is a tensor of the finite and infinite  component: $f_t=   f_t^\infty\otimes W^{\fw}_{\infty, t}$, and for   other $t\in F$,  its  $\psi_t$-Whittaker function is 0.
In this case,  we call the locally constant function 
$f_t^\infty$ on $G(\BA_F^\infty)$ the  $t$-th Fourier coefficient of $f$. (Then $f_t^\infty(1)$     is the $t$-th Fourier coefficient in the sense of classical modular forms.)
% to be the following function on $G(\BA_F^\infty)$: $$ a_{t,f}%(g):=W_t^\infty(g) .$$
% In other words,  the $\psi_t$-Fourier coefficient is the finite part of the  $\psi_t$-Whittaker function.

%Let $\cA_{\hol}(G,\fw)$ be the space of automorphic forms  holomorphic of weight $\fw$.
For a subfield $C\subset \BC$, let
$$\cA_{\hol}(G,\fw)_{C}\subset \cA_{\hol}(G,\fw)$$ be the $C$-subspace of automorphic forms  $f$ whose Fourier coefficients take values  in $C$. (In  the sense of classical modular forms, it means that the  coefficients of the $q$-expansion of $f$ along all cusps are in $C$.) 
Taking Fourier coefficients defines an embedding   of $C$-vector spaces
$$\cA_{\hol}(G,\fw)_{C}\to \prod_{t\in {F_{>0}}\cup\{0\}} \LC\lb G(\BA_F^\infty), C\rb, $$ 
where $\LC\lb G(\BA_F^\infty), C\rb$ means locally constant functions on $G(\BA_F^\infty)$ valued in $ C$. 
For a $C$-vector space $X$, we have  the induced embedding
$$\cA_{\hol}(G,\fw)_{C}\otimes_C X%\to\prod_{t\in {F_{>0}}\cup\{0\}} \LC\lb G(\BA_F^\infty), C\rb\otimes_C V
\to \prod_{t\in {F_{>0}}\cup\{0\}} \LC\lb G(\BA_F^\infty), X\rb. $$  Define  the $t$-th Fourier coefficient of an element in $\cA_{\hol}(G,\fw)_{C}\otimes_C X$  to be  the $t$-th component of its image.

\subsection{}\label {Weil representation}
Let $\BV/\BA_E$ be a   hermitian space. 
For  a character  $\chi_{_\BW}$  of $E^\times\bsl \BA_E^\times$ such that $\chi_{_\BV,v}|_{ F_v^\times}=\eta_v^{\dim \BV}$ for   every  place $v$ of $F$,   the Weil representation $\omega=\omega_{_\BV}$ on $\cS( {\BV})$  is the restricted tensor product of  local Weil representations  of $G(F_v)\times U({\BV(E_v)})$ on   $\cS({\BV(E_v)})$. 
The  
local Weil representation   (which we still denote by $ \omega  $ instead of $ \omega_v$ if the meaning is clear from the context) of $G(F_v)$ is defined as follows: for $\phi\in \cS( {\BV(E_v)}) $, 
\begin{align*} \omega (m(a))\phi(x)&=\chi_{_\BV,v}(  a)|  a|_{E_v}^{m/2}\phi(xa),\ a\in E_v^\times;\\
\omega (n(b))\phi(x)&=\psi_v\lb  bq(x)\rb \phi(x),\ b\in F_v;\\
\omega (w_v)\phi&=\gamma_{\BV(E_v)}\wh \phi ;\\
\omega (h)\phi(x)&=  \phi(h^{-1}x),\ h\in U({\BV(E_v)}) .
\end{align*}
Here $\gamma_{\BV(E_v)}$ is the Weil index associated to $\psi_v$ and  ${\BV(E_v)}$. 
%The local Weil indices satisfy
%\begin{equation}\label{gammapm}\prod_{v\in \Sigma_F}\gamma_{\BV(E_v)}=\left\{ \begin{array} {ll} 1& \mbox{if }  \BV\mbox{ is coherent};\\      -1 &\mbox{if }  \BV \mbox{ is incoherent}.\end{array} \right. 
%\end{equation}

\subsection{}\label{Weight}
For $v\in \infty$, define   $\fk^{\chi_{_\BV,v}}$ to be the unique integer such that
\begin{equation*}\chi_{_\BV,v}(z)=z^{\fk^\chi}\label{tx}\end{equation*}  
for $z\in E_v$   of norm 1.
Define $\fw_{\chi_{_\BV}}=(\fw_{\chi_{_\BV},v})_{v\in \infty}$ where $$\fw_{\chi_{_\BV},v}:=\lb \frac{\dim \BV+\fk^{\chi_{_\BV,v}}}{2}, \frac{\dim \BV-\fk^{\chi_{_\BV,v}}}{2}\rb.$$   

% If $\chi_{_\BV}$ is indicated in the context, we write $\fw$ for $\fw_{\chi_{_\BV}}$.

\addtocontents{toc}{\protect\setcounter{tocdepth}{2}}

\section{Theta-Eisenstein series}
First, we recall basic knowledge about  Eisenstein series and theta series. Then we set up basic properties of 
theta-Eisenstein series, i.e.,  linear combinations of  products of  theta series and  Eisenstein series.   
Finally, we study  two kinds of  holomorphic projections   of theta-Eisenstein series. The origin of theta-Eisenstein series is in the work of Gross and Zagier\cite{GZ}.
We largely follow the works of Yuan, S.~Zhang and W.~Zhang \cite{Yuan,YZ,YZZ}. 

\subsection{Eisenstein series and theta series}

Let  $\BW$ be  a   hermitian space over  $ \BA_E$  (with respect to  the extension $E/F$).   Let $\chi_{_\BW}$ be   a character   of $E^\times\bsl \BA_E^\times$ such that $\chi_{_\BV}|_{\BA_F^\times}=\eta $.  
We have  the Weil representation  $\omega_{_\BW}$, which we simply denote by $\omega$.
\subsubsection{Local Whittaker integrals}\label{Local Whittaker integrals}

Let $v$ be a place of $F$.
For $t\in F_v$,  $\phi\in   \cS(  \BW_v )$ and $g\in G(F_v)$,
define the Whittaker integral 
\begin{equation} W_{v,t} (s, g, \phi) =\int_{N(F_v) } \delta_v(w_vng)^{s }\omega (w_vng) \phi(0) \psi_{-t}(n) dn. \label{expandWT0}\end{equation} 
We immediately have the following equations  \begin{equation} \label{ktrivial} W_{v,t} (s, gk, \phi)=W_{v,t} (s, g, \omega(k)\phi),\  k\in K^{\max}_v ,  \end{equation} \begin{equation} \label{translaw1} 
W_{v,t}(s,  n(b) g , \phi )=\psi _v\lb bt\rb W_{v,t}(s,   g , \phi ),\ b\in  F_v.  \end{equation}
Since $wn(b)m(a)=m(\ol a^{-1})wn(b\Nm(a)^{-1})$, 
$ m(\ol a^{-1})n(b)=n(b')m(\ol a^{-1})$ for some $b'$,
 and $\chi_{_\BW,v}(\Nm(a))=1$,  a direct computation gives
%    $wn(b)m(a)=m(\ol a^{t,-1})wn(a^{-1}b\ol a^{t,-1})$
 \begin{equation} \label{translaw} 
 W_{v,t}(s, m(a)g, \phi)=|a|_{E_v}^{1-\dim \BW/2-s}\chi_{_\BW,v}(a)W_{v,\Nm(a)t}(s,  g, \phi), \ a\in E_v^\times    ,    \end{equation}

\begin{lem} [{\cite[Proposition 6.2]{Ich}\cite[Proposition 2.7 (2)]{YZZ}}]\label{lSWlem}

Let $t\neq 0$.

(1)  The set
$  \BW_v^t$ is either empty or consists of one orbit of $U( \BW_v)$.

(2) If  $  \BW_v^t$ is   empty, then $W_{v,t}(0,g,\phi)=0$.
Otherwise,  for $x\in  \BW_v^t$, we have
\begin{equation*}W_{v,t}(0,g,\phi)= \kappa \int_{  U( \BW_v)}\omega(g)\phi(h^{-1} x)dh ,\end{equation*}  
for a nonzero constant $\kappa$.

(3) If $\dim \BW=1$, with the measure fixed in \ref{Additive characters, Fourier transform, and measures}, $\kappa=\frac{\gamma_{\BW_v}}{L(1,\eta_v)}$.    \end{lem}

Assume $\dim \BW=1$.  \begin{lem}\label{mylem1} 
  Assume that $v$ is a finite place and $\phi(0)=0$. Then  for $t$ small enough
  $W_{v,t} (s, g, \phi)=W_{v,0} (s, g, \phi)$ and is a holomorphic function. 
  \end{lem}
  \begin{proof}The proof is by the reasoning as the proof of \cite[Lemma 4.2.4 (2)]{Qiu21}
  \end{proof}

We define the following normalization (following \cite[6.1.1]{YZZ}): 
\begin{equation}
W_{v,t}^\circ=\gamma_{\BW_v}^{-1}W_{v,t},\ t\neq 0.
\label{Wood0}\end{equation}
For $t=0$, $W_{v,t} (s, g, \phi) $ has a possible pole at $s=0$. 
And we take a different normalization (following \cite[6.1.1]{YZZ}  and taking care of the difference  between the modulus characters mentioned in \ref{modulus character}):
\begin{equation}W_{v,0}^\circ (s, g, \phi):=\gamma_{\BW_v}^{-1}|\diff_v\disc_v|_{F_v}^{-1/2}\frac{L(2s+1,\eta_v)}{L(2s,\eta_v)}W_{v,0} (s, g, \phi) .\label{Wood}\end{equation}
Here   if $v$ is a finite place, $\diff_v$ is  the different of $F_v/\BQ_v$
and   $\disc_v$ is the discriminant of $E_v/F_v$,  and  if $v\in \infty$, $\diff_{v}=\disc_v=1$.
% Note that $\frac{L(2s+1,\eta_v)}{L(2s,\eta_v)}$ has a zero at $s=0$, so that $ W^\circ_0(s,g,\phi)$ has a holomorphic continuation to  $s=0$ which we still denote by $ W^\circ_0(s,g,\phi)$.
% In fact, $\frac{L(2s+1,\eta_v)}{L(2s,\eta_v)}$ is the normalizing factor of the normalized intertwining operator \cite[Proposition 2.1]{Tan}.

 \begin{lem}\label{lSWlem1}

(1)  There is an analytic continuation of $ W^\circ_0(s,g,\phi)$ to $\BC$ such that
$ W^\circ_0(0,g,\phi)=\omega(g)\phi(0)$.

(2) If $E_v/\BQ_{v}$ is unramified where $\BW_v=E_v$ with  $q=\Nm$,
and $\phi=1_{\cO_{E_v}}$, then $ W^\circ_0(s,g,\phi)=\delta_v(g)^{-s}\omega(g)\phi(0).$

(3) If $v\in \infty$ and $\phi$ is the standard Gaussian function, then $ W^\circ_0(s,g,\phi)=\delta_v(g)^{-s}\omega(g)\phi(0).$
\end{lem} 
\begin{proof} (1)   follows from  \cite[Proposition 6.1]{YZZ}.   (2)   follows from\cite[Proposition 2.1]{Tan}.
(3) follows from   \cite[Lemma 7.6 (1)]{YZ} (or its proof).
\end{proof}

 \subsubsection{Siegel-Eisenstein series}  \label{Siegel-Eisenstein series} 
 Now we come back to a general $\BW$.
For $\phi\in \cS(\BW)$, we have a Siegel-Eisenstein series of $G$: 
\begin{equation}E(s,g,\phi)=\sum_{\gamma\in P(F)\bsl G(F)}\delta (\gamma  g)^{s }  \omega (\gamma g)\phi (0)\label{Edef}
\end{equation}
which is absolutely convergent if $\Re s>1-\dim \BW/2$ and   has a meromorphic continuation to the entire  complex plane  \cite{Tan}. 
%and has a functional equation with center at $(1-m)/2$   \cite{Tan}.  (We do not use the functional equation.)
Moreover, it is
holomorphic at $s=0$  \cite[Proposition 4.1]{Tan}.

Let $E_t(s,g,\phi)$ be the $\psi_t$-Whittaker  function of $E(s,g,\phi)$. 
Let $W_{t} (s, g, \phi)$ be the global counterpart  of the Whittaker integral  \eqref{expandWT0} so that 
if $\phi$ is a pure tensor then $W_{t} (s, g, \phi)$ is the product of $W_{v,t} (s, g_v, \phi_v)$ over all places of $F$.
By the non-vanishing of $L(1,\eta)$ and Lemma \ref{lSWlem1}, $W_{0} (s, g, \phi)$ has a meromorphic continuation to the entire  complex plane, which is holomorphic at $s=0$. 
Then  we have 
\begin{equation}\label{E0t}E_{t} (s, g, \phi)=W_{t} (s, g, \phi),  \ t\neq 0 ;
\end{equation}
\begin{equation}\label{E0}E_0(s,g,\phi)=\delta (   g)^{s }  \omega ( g)\phi (0)+W_{0} (s, g, \phi).
\end{equation}
By  Lemma \ref {lSWlem} (2) and \eqref{E0t}, for $h\in U(\BW)$, we have 
\begin{equation}\label{E0tk}
E_t(0,g,\omega(h)\phi)=E_t(0,g,\phi).
\end{equation}

For $t\neq 0$ and a pure tensor  $\phi$, define  \begin{equation}\label{etv}  E_t'(0,g,\phi)(v)=  W_{v,t} '(0, g_v, \phi_v) \prod_{u\neq v} W_{u,t} (0, g_u, \phi_u).\end{equation} 
Extend the definition to all Schwartz functions by linearity.
Then
\begin{equation} \label{2.9}E_t'(0,g,\phi)=\sum_{v }   E_t'(0,g,\phi)(v). \end{equation}

\subsubsection{Coherent case:   Siegel-Weil formula} \label{TSW}

%Extend the Weil representation  $\omega_{V^\sharp}=\omega_{\psi,\chi_{V^\sharp}} $ of $G$   on  the Schwartz space $\cS(V^\sharp (\BA_E))$ 
% to $G(\BA_F ) \times  U(V^\sharp)(\BA_F  )$ 
% by  $$\omega(g,h) \phi(x)=\omega(g) \phi(h^{-1}x).$$

%We still call $\omega $  the Weil representation of $G(\BA_F ) \times  U(\BW  )$ .

Assume that $\BW=W(\BA_E)$ is coherent. 
For  $\phi \in \cS(W (\BA_E))$ and $(g,h)\in   G(\BA_F ) \times  U(\BW  ) $, we define a  theta series, which is absolutely convergent:
$$ 
\theta(g,h,\phi)=\sum_{x\in W{\qclE}} \omega(g,h)
\phi ( x).$$
Then $ \theta(g,h,\phi)$ is  smooth, slowly increasing  and     $  G(F)\times U(W)  $-invariant.
%Let $$   \theta(g,\phi) =   \theta(g,1,\phi) .$$  

Assume that   $W$ is anisotropic. 
Then $E(s,g,\phi)$ is holomorphic at $s=0$, and the following equation is a special case of the regularized Siegel-Weil formula {\cite[Theorem 4.2]{Ich}}: \begin{equation}\label{Ichthm} E(0,g,\phi)=\frac{\kappa}{\Vol([U(W)])}  \int _{[U(W)]} \theta(g,h,\phi) dh ,\end{equation} 
where $\kappa=2$ if $ \dim W=1$ and $\kappa=1$ if $ \dim W>1$.
%We in fact will only need the case $r=m=1$, and we will not need the constant $\kappa$ explicitly. The absolute convergence of the right hand side of \eqref{gSW} is shown in \cite{Wei}.

\subsubsection{Incoherent case: derivative}
Assume that $\BW$ is incoherent. 
By  Lemma \ref {lSWlem} (1), for $t\neq 0$, the summand in \eqref{2.9} corresponding to $v$ is nonzero  only if $t$ is represented by $\BW^v$. 
 By the incoherence, $\BW_v$ does not represent $t$. In particular, we have the following lemma.

\begin{lem}\label{lSWlemlate}
For $v$ split in 
$E$, the summand  in \eqref{2.9} corresponding $v$  is $0$.
\end{lem}

Assume that $\dim \BW=1$.  
Assume that   $\phi$ is a pure tensor. 

First, assume that $t\neq 0$.
By   the product formula for Hasse invariant and the Hasse principle,  if $t$ is represented by $\BW^v$, then
$v$ is nonsplit in $E$ and  $t$ is represented by the $v$-nearby hermitian space $W$ of $\BW$. See \ref{hermitian spaces}. 
%Thus  by  \eqref{gammapm} and  Lemma \ref {lSWlem} (2), we have \begin{align}\label{ET'dec} E_t'(0,g,\phi)(v)=   & W_{v,t} '(0, g_v, \phi_v) \prod_{u\neq v} W_{u,t} (0, g_u, \phi_u)\\
%=& \frac{-1}{  L(1,\eta^v)} \sum_{x\in U(W)\bsl  W{\qclE}^t}    {W_{v,q(x)}^\circ} '(0, g_v, \phi_v)     \int_{  U (\BW^v)}\omega(g^v)\phi^v(h^{-1} x)dh \nonumber\\
%=& \frac{-1}{  L(1,\eta^v)\Vol(U(W(E_v))} \sum_{x\in U(W)\bsl  W{\qclE}^t}   \int_{  U ( W)(\BA_F)}  {W_{v,q(h_v^{-1}x)}^\circ} '(0, g_v, \phi_v)  \omega(g^v)\phi^v(h^{v,-1} x)dh \nonumber\\
%=& \frac{-1}{  L(1,\eta^v)\Vol(U(W(E_v))}  \int_{ [ U(W)]}\sum_{x\in  W{\qclE}^t} {W_{v,q(h_v^{-1}x)}^\circ} '(0, g_v, \phi_v)   \omega(g^v)\phi^v(h^{-1} x)dh. \nonumber\end{align} 
%Here $ W{\qclE}^t$ is the subset of  elements with norm $T$, and the  second  ``=" is from Lemma \ref {lSWlem} and  \eqref{gammapm}. 
%By Lemma \ref{lSWlemlate} and \eqref{ET'dec},
By  Lemma \ref {lSWlem} (2) and
Lemma \ref{lSWlemlate}, we have 
\begin{equation}\label{E'0tk}
E'_t(0,g,\omega(h)\phi)=E'_t(0,g, \phi)
\end{equation}
for $h\in U(\BW_v)$ where $v$ is split in $E$.

Now we consider the constant term. %By \cite[Proposition 2.8]{YZZ} or 
Let 
\begin{equation}\label{fcdef}{\fc}=4\frac{L'(0,\eta)}{L(0,\eta)}+2\log |\disc_E/\disc_F|,\end{equation}
where $\disc_E,\disc_F\in \BZ$  are the  discriminants of $E$ and $F$ over $\BQ$.
By  \cite[p 586]{YZ} (note   the difference by $2$ between the $s$-variables in the $L$-factors in loc. cit. and  \eqref{Wood}),\begin{equation}\label{W0'}W_0'(s,g,\phi )=-{\fc}\omega(g)\phi(0)-\sum_v \omega(g^v)\phi^v(0){W_{v,0}^\circ} '(0, g_v, \phi_{v}) .  \end{equation}
%Let $S$ be a set of   places of $F$ of cardinality $2$. By Lemma \ref{lSWlem1} (1), \eqref{E0} and \eqref{W0'}, we    have the following lemma,see also \cite[Proposition 6.7]{YZZ} \cite[Proposition 2.4]{Liu0}.
%\begin{lem}\label{Tsing0} Assume that $\phi_v(0)=0$ for $v\in S$. For   $g\in P_1(\BA_{F,{S}})G_1(\BA_F^{S})$, we have $$ E_0'(0,g,\phi)=0.$$\end{lem}

%which is absolutely convergent. %by definition. It is defined by analytic continutation of the right hand side
%By the definition of the Weil representation $\omega$, we  have the following lemma.
%\begin{lem}\label{gphigood} For $g\in P(\BA_{F,{S}})G(\BA_F^{S})$, then\begin{equation*}\omega(g)\phi \in \cS(\BW_{{S},\reg})\cS(\BW^{S,r}) .\end{equation*}\end{lem}
%For $v\in \Sigma_\nspl$, let  $ E'_v (0,g,\phi)$ 
%be  the sum of the summands in \eqref{EsTdec} of $t$'s represented by $W(v)$.
%Then we have  \begin{align}\label{Evdec}& E'_v (0,g,\phi)
%\\
%=&\frac{-1}{  L(1,\eta^v)\Vol(U(\BW_v))}  \int_{  U(W)\bsl  U (\BW )}\sum_{x\in  W{\qclE}-\{0\}} {W_{v,q(h_vx)}^\circ} '(0, g_v, \phi_v)   \omega(g^v)\phi^v(h^{-1} x)dh.\nonumber \end{align} 
%For $g\in P(\BA_{F,{S}})G(\BA_F^{S})$, we have \begin{align}\label{Esvdec}E'(0,g,\phi)=\sum_{v\in \Sigma_\nspl} E_v'(0,g,\phi) . \end{align} 
%which is  absolutely convergent.

\subsection{Theta-Eisenstein series}\label{Theta-Eisenstein seriesdef}
From now on,    always assume  $\dim \BW=1$. %Moreover, except in \ref{Theta-Eisenstein seriesdef}, $\BW$ is always assumed to be incoherent.
\subsubsection{Definition}\label{Theta-Eisenstein seriesdefdef}

Let $V^\sharp/E$ be  a    hermitian space of dimension $n>0$. Let $\chi_{V^\sharp}$ be a character of $E^\times\bsl \BA_E^\times$ such that $\chi_{V^\sharp}|_{\BA_F^\times}=\eta^{n}$.
Let $\BV=    \BW\oplus V^\sharp(\BA_E)$ be the orthogonal direct sum and $\chi _\BV=\chi_{_\BW}\chi_{V^\sharp}$.  We have the  corresponding Weil representations.
Below we shall use $\omega$ to denote a Weil representation if the hermitian space is indicated in the context, e.g by the function that it acts on.

For $ \phi\in \cS\lb \BV\rb$, we  define a  theta-Eisenstein series   $ \te(s,g,\phi)$ on $G$ associated the to the  orthogonal
decomposition $\BV=    \BW\oplus V^\sharp(\BA_E)$:
\begin{equation} \te (s,g,\phi)= \sum_{\gamma\in P(F)\bsl G(F)}\delta (\gamma  g)^{s }  \sum_{x\in V^\sharp{\qclE}}  \omega (\gamma g)\phi ((0,x)).\label{thetae}
\end{equation}
If $\phi=\phi_1\otimes \phi_2$ with $\phi_1\in  \cS\lb \BW\rb$  and  $\phi_2\in  \cS\lb V^\sharp(\BA_E)\rb$,
then  \begin{equation} \te(s,g,\phi)= E(s,g,\phi_1) \theta(g, \phi_2)  \label{thetae1}.
\end{equation}
%In particular,  %from the absolute  convergence of Eisenstein series. See \ref{Siegel-Eisenstein series}),
%  \eqref{thetae}
% is absolutely convergent if $\Re s>1/2$, and has a meromorphic continuation to the entire  complex plane, holomorphic at $s=0$.
For    $t\in F$, let  $\te _t (s,g,\phi)$ be the $\psi_t$-Whittaker function  of $\te  (s,g,\phi)$.

\subsubsection{Coherent case}
Assume that $\BW=W(\BA_E)$  is coherent.   The regularized Siegel-Weil formula \eqref{Ichthm}    immediately  implies the following  ``mixed   Siegel-Weil formula":
  \begin{equation}\label{Ichthm1}  \te(0,g,\phi)=\frac{2}{\Vol([U(W)])}  \int _{[U(W)]} \theta(g,h,\phi) dh.\end{equation}
(We will prove an arithmetic analog of \eqref{Ichthm1} for $\BW$ being incoherent in \ref{Arithmetic mixed Siegel-Weil formula}.) Then
for $t\in F$, 
\begin{equation}\label{Evdec1100} \te_{t} (0,g,\phi) = \frac{2}{ \Vol([U(W)]) }  \int_{  [U(W)]}\sum_{x\in   V  ^t   }  \omega(g)\phi(h^{-1} x)dh.
\end{equation}   
For $\phi$   invariant by $U(W(E_v)), v\in \infty$,   the integration in \eqref{Evdec1100}  is  a finite sum.

\begin{lem}\label{OOO0} Let   $t\in F_{>0}$ and  $\phi  $   a pure tensor.
Let $u$ be a finite place of $F$, $O\subset \BV_u$   an open compact neighborhood of $0$ and $\phi^O= \phi^u\otimes(\phi_u 1_{\BV_u-O})$.
Given  $g\in    G(\BA_F^{u })P(F_u)$, for $O$ small enough, we have  $ \te _t (0,g,\phi) = \te _t (0,g,\phi^O) .$
 \end{lem}
\begin{proof} Write $g_u=m(a)n(b)$  where $a\in E_u^\times$. See \ref{Group}.
Then $\{a h_u^{-1}x:x\in V^t,h\in U(\BV)\}\subset \BV_u^{a^2t} $. The latter is closed in $\BV_u$ and does not contain 0. Thus  $O\cap  \{ah_u^{-1}x:x\in V^t,h\in U(\BV)\}=\emptyset$ if $O$ is small enough. Then the lemma  follows from    \eqref{Evdec1100}, the remark below it, and the definition of the {Weil representation} in \ref{Weil representation}. 
\end{proof}

\subsubsection{Incoherent case}
Assume that $\BW$ is incoherent.

For a place $v$ of $F$ nonsplit in $E$, let $W$ be the $v$-nearby hermitian space of $\BW$.
For  $\phi=\phi_{1}\otimes \phi_{2}$ with $\phi_{1}\in  \cS\lb  \BW_v \rb$,  $\phi_2\in  \cS\lb V^\sharp(E_v)\rb$ and  
$x=(x_1,x_2)\in V:= W(E_v)\oplus V^\sharp(E_v)$ with $ x_1\neq 0,$  let 
\begin{equation}   \tw_{v,x}(s,g,\phi)=  \frac{L(1,\eta_v)}{  \Vol(U(W(E_v))}  
W^\circ_{v,q(x_1)}(s,g,\phi_1)  \omega (g)\phi_2(x_2)   \label{thetaw1}\end{equation}
This is a local analog  of  \eqref{thetae1}.
%Extend this definition to general Schwartz functions by linearity.
Extend this definition to $ \cS\lb  \BW_v \rb\otimes \cS\lb V^\sharp(E_v)\rb\subset  \cS\lb  \BV(E_v)  \rb$ by linearity.
The inclusion is an equality unless $v\in \infty$. However, this subspace is enough for our purpose.
(Besides, there is another definition of $\tw_{v,x}$ for  the whole  $  \cS\lb  \BV(E_v)  \rb$.
We will not   need it.)
%Besides, later we will only use such decomposable functions.
   For $v$ nonsplit in $E$ and $t\neq 0$,   define 
 \begin{equation}\label{Evdec11} \te'_{t} (0,g,\phi) (v) = \frac{2}{ \Vol([U(W)]) }  \int_{  [U(W)]}\sum_{x\in   V  ^t -V^\sharp }  \tw_{v,h_v^{-1}x} '(0, g_v, \phi_v)   \omega(g^v)\phi^v(h^{v,-1} x)dh.
\end{equation} 
Note that the  analog of  \eqref{2.9} does not hold.
%where $W$ is the $v$-nearby hermitian space of $\BW$ and $V=W\oplus V^\sharp$.

%By    \eqref{ET'dec} and that $ \Vol([U(W)])=2L(1,\eta)$, we have   the following lemma.
%\begin{lem} \label{E'dec11} Assume that $\phi=\phi_1\otimes \phi_2$ with $\phi_1\in  \cS\lb \BW\rb$  and  $\phi_2\in  \cS\lb V^\sharp(\BA_E)\rb$. Then
%$$ \te '_t (0,g,\phi) =-\sum_{v \ \text{nonsplit} }\te'_{t}(0,g,\phi)(v)+E_0'(0,g,\phi_1)
%\sum_{x\in \lb V^\sharp\rb^t}  \omega (  g)\phi _2(x).$$\end{lem} 

%   \subsubsection{Properties of $\tw_{v,x} '(0,g,\phi)$}\label{PPPsg}

We study $\tw_{v,x} '(0,g,\phi)$   following  \cite{YZZ}. Indeed,   the computation is only on the Eisenstein (i.e., Whittaker)  part. 
We remind the reader of the difference  between the modulus characters mentioned in \ref{modulus character}.
By \eqref{translaw1}, \eqref{translaw}  and Lemma \ref{lSWlem}, we have the following lemma, which says that under the action of $P(F_v)$,  $\tw_{v,x} '(s_0,   g , \phi )$ behaves in the same was as the Weil representation.
\begin{lem}[{\cite[Lemma 6.6]{YZZ}}] \label {YZZ6.6}
The following relations hold:
$$\tw_{v,x} '(0, m(a) g , \phi )= \chi_{_\BV,v} (a)|\det a|_{E_v }^{\dim \BV/2} \tw_{v,ax} '(0,  g , \phi ),\ a\in E_v^\times;$$
$$\tw_{v,x} '(0, n(b) g , \phi )=\psi _v\lb bq(x)\rb  \tw_{v,x} '(0,  g , \phi ),\ b\in  F_v. $$

\end{lem}

%\begin{rmk}For the first equation, we need the first part of Lemma \ref{lSWlem} (2).\end{rmk}

\begin{cor} \label{mainv}For $a\in E^\times$, $\te'_{t} (0,m(a)g,\phi) (v)=\te'_{a^2t} (0,g,\phi) (v)$.
\end{cor}

%We also have the following lemma by the same proof as in \cite{YZZ}.
%\begin{lem}[{\cite[Proposition 6.11]{YZZ}}] \label {YZZ6.11}
%Assume that  $v\in \infty$, $\BW_v$ is  positive definite (so that $W(E_v)$ is negative definite) and  $\phi$ is the standard Gaussian function.  For  $x=(x_1,x_2)\in  W(E_v)\oplus V^\sharp(E_v)$ with $ x_1\neq 0,$   we have $$ \tw_{v,x} '(0, 1, \phi)=-2\pi \Ei(4\pi{q(x_1)})  \phi(x).$$ \end{lem}
%Here the exponential integral $\Ei$ is defined as \begin{equation}\label{Ei}\Ei(z)=\int_{-\infty}^z\frac{e^t}{t}dt,\  z<0.\end{equation}

\subsection{Holomorphic projections}\label{Holomorphic projections} 
 We define    quasi-holomorphic projection and cuspidal holomorphic projection,   compare them  for   theta-Eisenstein series (Lemma \ref{becom}). After imposing Gaussian condition at infinite places in
 \ref{Gaussian functions and holomorphy}, we make the comparison more explicit in  \eqref{Final}. 
 Finally, after imposing the incoherence condition, we explicitly compute the   quasi-holomorphic projection (\Cref{omitted}).

\subsubsection{Definitions}\label{holsection} 

For  $v\in \infty$, let     $\fw_v$ be a pair of integers whose sum $|\fw_v|$ is $\geq 2$.
For  $t\in F_{v,>0}$, let $W^{\fw_v}_{v, t}$ be the   standard holomorphic  Whittaker function  of weight $\fw_v$ as in \ref {Holomorphic  automorphic  forms}. 
Then $$\int_{Z(F_v)N(F_v)\bsl G(F_v)} |W^{\fw_v}_{v, t}(h)|^2 dh=(4\pi)^{-|\fw_v|+1}\Gamma(|\fw_v|-1)$$

For  $t\in F_{v,>0}$, 
a $\psi_{v,t}$-Whittaker function $W$ on $G(F_v)$, and $g\in G(F_v)$, 
define $$W_s(g)=\frac{(4\pi)^{|\fw_v|-1}}{\Gamma(|\fw_v|-1)} W^{\fw_v}_{v, t}(g) \int_{Z(F_v)N(F_v)\bsl G(F_v)}\delta(h)^sW(h) \ol{W^{\fw_v}_{v, t}(h)} dh.$$
If $W_{s}$ has a meromorphic continuation to  $s=0$,  define
the  quasi-holomorphic projection     $$W_{\qhol}:=\wt{\lim\limits _{s\to 0}}W_{s}$$
 of  $W$ of  weight $\fw_v$.
Here $\wt{\lim\limits _{s\to 0}}$ denotes the constant term at $s=0$.

Let $\fw=(\fw_v)_{v\in \infty}$  where  $\fw_v$ is a pair of integers.
% For $t\in F$, let $$W^{\fw}_{\infty, t}=\prod_{v\in \infty}W^{\fw_v}_{v, t}$$ be the product of    the standard holomorphic  Whittaker function as in \ref {Holomorphic  automorphic  forms}.
For   a continuous function $f:N(F)\bsl G(\BA_F)\to \BC$ with
$\psi_t$-Whittaker function  $f_t$ for $t\in F_{>0}$, %. See  \ref  {Holomorphic  automorphic  forms}).
let $f_{t,\qhol}$  be the quasi-holomorphic projection  of   $f_t$ of  weight $\fw$ at all infinite places (if it is well-defined).

%(1) Let $f$ be as  above.The quasi-holomorphic projection  of  weight $\fw$ of     $f$  
%is  $$f_{\qhol}:=\sum_{t\in F^\times}f_{t,\qhol} .$$

For  an automorphic form  $f$ on $G(\BA_F)$, the cuspidal 
holomorphic projection  $f_{\chol}$ of  weight $\fw$ of     $f$  is the $L^2$-orthogonal projection of $f$ to the subspace  $\cA_{\hol}(G,\fw)$ of cusp forms.  I.e., for every   cusp form $\phi\in \cA_{\hol}(G,\fw)$, the Petersson inner product    $\pair{f,\phi}$ equals   $\pair{f_{\chol},\phi}$.

%We will omit the weight if it is clear from the context.

\begin{lem}[{\cite[Proposition 6.2]{Liu0}\cite[Proposition 6.12]{YZZ}}] \label{growth}
  Assume that there exists $\ep>0$ such that
for $v\in \infty$ and    $a \in E_v^\times$ with  $|a|_{E_v}\to \infty$,  we have   
$$f\lb  m(a)g \rb=O_{g}\lb |a | _{E_v}^{|\fw_v|/2-\ep}\rb,$$
where $m(a)$ is as in \ref{Group}. 
Then for $t\in F_{>0}$, 
$f_{t,\qhol}$ is well-defined and $f_{t,\qhol}=f_{\chol,t}$. 
\end{lem}

\subsubsection{Holomorphic projections of $\te '  (0,g,\phi)$}
Let $\fw=\fw_{\chi_{_\BV}}$ be as in  \ref{Weight}. Then $|\fw_v|=n+1$.
Holomorphic projections   below are of weight $\fw$. 
Retrieve the notations in \ref{Theta-Eisenstein seriesdefdef}.
Let $\te'_{\chol}(0,g,\phi)$ be the cuspidal holomorphic projection  of  the derivative $\te'(0,g,\phi)$.
For $t\in F_{>0}$,  let $\te'_{\chol,t}(0,g,\phi)$ be  its $\psi_t$-Whittaker function.
Let $\te_t(s,g,\phi)$ be  the  $\psi_t$-Whittaker function of $\te(s,g,\phi)$.  
 Let $\te_{t,\qhol}'(0,g,\phi)$ be  the quasi-holomorphic projection of    $ \te_t'(0,g,\phi)$ if it is well-defined.
The difference between $\te'_{\chol,t}(0,g,\phi)$ and   $\te_{t,\qhol}'(0,g,\phi)$ is given as follows.

For $\phi=\phi_1\otimes \phi_2\in \ol \cS\lb \BV\rb$ with $\phi_1\in  \cS\lb \BW\rb$  and  $\phi_2\in \ol  \cS\lb V^\sharp(\BA_E)\rb$, define $$\te_{00}(s,g,\phi)=  \delta(g)^s   \omega(g)\phi(0)+ W_0(s,g,\phi_1) \omega(g)\phi_2(0) ,$$
which is the product of $E_0(s,g,\phi_1)$ and the constant term of $\theta(g,\phi_2)$. See \eqref{E0}.
The definition extends to  general  $\phi\in \cS\lb \BV\rb$ by linearity.
By \eqref{translaw1} and  \eqref{translaw}, we can define   an 
Eisenstein series
$$
J(s,g,\phi)=\sum_{\gamma\in P(F)\bsl G(F)}\te_{00}(s,\gamma g,\phi). 
$$
%Extend the definition to $\cS(\BV)$ by linearity.
For $t\in F_{>0}$, let $J_t(s,g,\phi)$ be its $\psi_t$-Whittaker function.
Let $J_{t,\qhol}'(0,g,\phi)$ be the quasi-holomorphic projection  of  the derivative $J_t'(0,g,\phi)$  if it is well-defined.   
\begin{rmk}
In the  notations of \cite{Tan}, $\delta(g)^s   \omega(g)\phi(0)$ is in the   degenerate  principal series   $I(n/2+s,\chi_{_\BV})$, while $W_0(s,g,\phi_1) \omega(g)\phi_2(0)$ is in
$I(n/2-s,\chi_{_\BV})$ by \eqref{translaw1} and  \eqref{translaw}.
%Thus $J(s,g,\phi)$ is a sum of two Eisenstein series.
\end{rmk}

\begin{lem} \label{becom}
If one of $ \te_{t,\qhol}'(0,g,\phi)$ and $J_{t,\qhol}'(0,g,\phi)$ is well-defined, then so is the other one. In this case, $\te'_{\chol,t}(0,g,\phi)=\te_{t,\qhol}'(0,g,\phi)-J_{t,\qhol}'(0,g,\phi)$.
\end{lem}
\begin{proof} By the same proof of \cite[Lemma 6.13]{YZZ},   $\te' (0,m(a)g,\phi)-\te_{00}'(0,m(a)g,\phi)$ is exponentially decay, and  $J' (0,m(a)g,\phi)-\te_{00}'(0,m(a)g,\phi)$ is exponentially decay up to the derivative at $s=0$ of the intertwining part of the constant term. The intertwining part of the constant term lies in $I(-n/2\pm s,\chi_{_\BV})$ so that its derivative at $s=0$ has growth rate $O_g\lb|a | _{E_v}^{-n/2+1/2+\ep} \rb$ (in the notations in Lemma \ref{growth}).
In particular, both  differences
satisfy the growth condition in Lemma \ref{growth}.
Thus $\te' (0,g,\phi)- J' (0,g,\phi) $ satisfies the growth condition in Lemma \ref{growth}.
Since the cuspidal holomorphic projection  of  the Eisenstein series $J(s,g,\phi)$ is $0$, the lemma follows.
\end{proof}

\subsubsection{A new Eisenstein series}\label{Derivative of Eisenstein series}

We  introduce a new  Eisenstein series in order to compute $J'(0,g,\phi)$.
For $\phi_v=\phi_{v,1}\otimes \phi_{2,v}\in  \cS\lb \BV(E_v)\rb$ with $\phi_{v,1}\in   \cS\lb \BW_v\rb$  and  $\phi_{2,v}\in   \cS\lb V^\sharp(E_v)\rb$,   
 define a function on $G(F_v)\times V^\sharp(E_v)$: 
   \begin{equation} c(g,x,\phi_v)={W_{v,0}^\circ} '(0, g, \phi_{1,v}) \omega(g)\phi_{2,v}(x) +\log\delta_v(g)\omega(g)\phi_{v}(x).\label{cphiv}\end{equation}
 For the moment, we only need $$c(g,\phi_v):=c(g,0,\phi_v).$$
Extend this definition to the whole $ \cS\lb \BV(E_v)\rb$ linearly. 
%For the moment, we only need $$c(g,\phi_v):=c(g,0,\phi_v).$$

By \eqref{translaw1}, \eqref{translaw} and  Lemma \ref{lSWlem1} (1), a direct computation shows
the following lemma.
\begin{lem}\label{sameprin} The function
$c(g,\phi_v)$ on $G(F_v)$  is in the same  principal series  as $   \omega(g)\phi_v(0)$, i.e.

(1) $c(m(a) g , \phi_v )= \chi_{_\BV,v} (a)|\det a|_{E_v }^{\dim \BV/2} c(g,\phi_v)$ for $ a\in E_v^\times$;

(2) $c(n(b) g , \phi_v )=   c(g,\phi_v)$ for $ b\in F_v$.

\end{lem}
Thus we can define the following Eisenstein series in the case that $\phi $ is a pure tensor
$$C(s,g,\phi)(v)=\sum_{\gamma\in P(F)\bsl G(F)}c(\gamma g_v,\phi_v)  \omega(\gamma  g^v)\phi^v(0).  $$  
\begin{lem}
For all but finitely many  \textit{finite} places, $c(g,\phi_v)=0$ for all $g$.
\end{lem}
\begin{proof}

   If $v\in \infty$,  by Lemma \ref{lSWlem1} (3) which says $ W^\circ_0(s,g,\phi)=\delta_v(g)^{-s}\omega(g)\phi(0),$
   we clearly have  $c(g,\phi_v) = 0$. The same is true if 
 $E_v/\BQ_{v}$ is unramified  and $\phi=1_{\cO_{E_v}}$  by Lemma \ref{lSWlem1} (2).  
 These cases cover all but finitely many  \textit{finite} places.

\end{proof}
Let 
$$C(s,g,\phi) =\sum_v C(s,g,\phi)(v)$$
where the sum is over these  \textit{finite} places of $F$.
   Let $C_t(s,g,\phi)$  be the $\psi_t$-Whittaker function of  $C(0,g,\phi)$.
The   definitions   can be obviously extended to the whole $  \cS\lb \BV\rb$  by linearity.

By \eqref{W0'}, a direct computation shows that 
\begin{equation}  \label{2.89}
J'(0,g,\phi)=2E'(0,g,\phi)-{\fc}E(0,g,\phi)-C(0,g,\phi).
\end{equation}

\subsubsection{Gaussian functions and holomorphy}\label{Gaussian functions and holomorphy}
Below in this section,
 assume that $\BV $ is totally  positive definite and  $\phi= \phi_\infty\otimes  \phi^\infty \in\ol   \cS\lb \BV\rb$. (So  $ \phi_\infty$ is Gaussian. See \ref{Functions}.)
 Let $v\in \infty.$  
 Then
\begin{equation}\label{wellknown1}\omega ([k_1,k_2])\phi_v =k_1^{ w_1} k_2^{ w_2}\phi_v 
 \end{equation} 
for $[k_1,k_2]\in K_v^{\max}$  as in \ref
 {Group}
 if $\fw_v=(w_1,w_2)$.
(Indeed,   first  check \eqref{wellknown} for $g=w_v$ and $g\in K_v^{\max}$ being diagonal, then for $g\in K_v^{\max}$ being anti-diagonal, finally for general $g\in K_v^{\max}$.)
Then by the   Iwasawa decomposition, it is easy to check that for  $g\in G(F_v), x\in \BV(E_v)$,
\begin{equation}\label{wellknown}\omega (g)\phi_v(x)=W^{\fw_v}_{v,q(x)}(g).\end{equation} 

By  \eqref{ktrivial} combined with \eqref{wellknown1},
   \eqref{translaw1} and  \eqref{translaw},  
   $W_{v,t} (0, \cdot, \phi_v)$ is  a multiple of  $ W^{\fw_v}_{v, t} $.
 Then   by 
    \cite[Proposition 2.11 (2) (4)]{YZZ},  we have  
    \begin{align} \label{Wtt} W_{v,t} (0, g, \phi_v) &=\gamma_{\BV(E_v)} \frac{(2\pi)^{n+1}}{\Gamma(n+1)} t^n W^{\fw_v}_{v, t}(g), \ t>0,\\
W_{v,t} (0, g, \phi_v) &=0,\ t\leq 0. \label{W00}
\end{align}

 \begin{lem}\label{holomorphy}
 Both $E(0,g,\phi)$ and  $C(0,g,\phi)$ are holomorphic of weight $\fw$.
\end{lem}
 \begin{proof}For $E(0,g,\phi)$, use
  \eqref{E0t},  \eqref{E0},     \eqref{wellknown}, \eqref{Wtt} and  \eqref{W00}. 
  For $C(0,g,\phi)$, by   Lemma \ref{lSWlem1} (3),    $c(g,\phi_v)=0$ for $v\in \infty$. Thus
  $C(s,g,\phi)(v)=0$ for $v\in \infty$. The rest of the proof is the same as  the proof for $E(0,g,\phi)$.
   \end{proof}
For $t\in F^\times$, let $ E_t'(0,g,\phi)(v)$ be as in \eqref{etv} so that we have the decomposition \eqref{2.9}. 
Let   
\begin{equation}  \label{2.90} E'_{t,\rf}(0,g,\phi)=\sum_{v\not\in \infty}E_t'(0,g,\phi)(v).\end{equation} By  \eqref{Wtt} and  \eqref{W00},   if $t\in F_{>0}$, then $E'_{t,\rf}(0,g,\phi)$ is a multiple of 
$W^{\fw}_{\infty, t}(g_\infty)$; otherwise   $E'_{t,\rf}(0,g,\phi)=0$.  We call $E'_{t,\rf}(0,g,\phi)$  the holomorphic part of the  Whittaker function $E_t'(0,g,\phi)$.

\subsubsection{Properties of Eisenstein series} 
We list some properties of the above Eisenstein series for later use. The reader may skip these properties for the moment.

By   \cite[Lemma 7.6 (2)]{YZ}    (or its proof)  and taking care of the difference  between the modulus characters mentioned in \ref{modulus character}, we have the following lemma.
\begin{lem}\label{sameprin1} 

Let  $E_v/F_v$ be split, $\BW_v=E_v$ and  $q=\Nm$. 
Assume    $\phi_v=\phi_{v,1}\otimes \phi_{v,2}$
where $\phi_{v,1}=1_{\cO_{E_v}}$ and $\phi_{v,2}\in \cS\lb V^\sharp(E_v)\rb$, then $c(1,\phi_v) =2 \log |\diff_{v}|_v  \phi_v(0),$  where $\diff_{v}$  is  the different of $F_v/\BQ_{v}$.
\end{lem}

We omit the routine proof of the following analog of \eqref{ktrivial}.
\begin{lem}\label{wr} For   a place $v$ of $F$ and   $k\in K^{\max}_v$,  we have $E(s,g,\omega(k)\phi)=E(s,gk,\phi)$. The same relation holds for  $C(s,g,\phi)$,  $\te(s,g,\phi) $,  $\te'_{\chol}(0,g,\phi),$
and their  $t$-th Whittaker/Fourier  coefficients, and 
$E'_{t}(0,g,\phi) (v)$ (thus $E'_{t,\rf}(0,g,\phi)$) for $t\in F^\times$.

\end{lem}

\begin{lem}\label{wr1} 

(1) We have  $E_0(0,g,\phi)= \omega ( g)\phi (0)$.  %(i.e., the contribution from the intertwining part at  is 0).

(2)  If moreover $\BV$ is incoherent, then  for   a finite place $v$,  $C_0(0,g,\phi)(v)= c(g_v,\phi_v)\omega ( g^v)\phi^v (0)$. 

(3) In (2), assume that    $\phi_{{{v}}}$ is supported outside $V^\sharp(E_{v})$ for $v$ in a set $S$ of  two places   of $F$ and   
$g\in   P (\BA_{F,S  })G(\BA_F^{S })$, then   $E_0(0,g,\phi)=0$ and $C_0(0,g,\phi)=0$.\end{lem}
\begin{proof}  (1)
If $F=\BQ,n=1$ and  $\BV^\infty$ is not split at some finite place,
it is proved in \cite[Proposition 2.9 (3)]{YZZ}.
If
$F\neq \BQ$ so that we have at least 2 infinite places, its proof is similar to the one in \cite[Proposition 2.9 (3)]{YZZ}
by using \eqref{E0} and \eqref{W00} (for $t=0$).   

  (2) The proof   is similar, with the fact that $\BV$ is not split at (at least) 2 places outside $v$ by the incoherence. See also \cite[p 65-66]{Yuan}.

 (3) follows from (1)(2) directly.   
\end{proof}
%\begin{lem}\label{wr2}
% Assume $F\neq \BQ$, or $F=\BQ,n=1$ and  $\BV^\infty$ is not split at some finite place.
%Then $E'_0(0,g,\phi)(v)= \omega ( g)\phi (0)\log\delta_v(g) .$ 

%\end{lem}
%\begin{proof}  The proof is similar to the proof of Lemma \ref{wr1} and is omitted.
%\end{proof}

\subsubsection{Compute   $J_{t,\qhol}'(0,g,\phi)$}\label{QhJ}
%Now we compute  the quasi-holomorphic projection $J_{t,\qhol}'(0,g,\phi)$  of $J_t'(0,g,\phi)$ for $t\neq 0$ by the last equation. 

Let $t\in F_{>0}$.
For  $v\in \infty$, let $E_{t,\qhol}'(0,g,\phi)(v)$ be  the quasi-holomorphic projection of $ E_t'(0,g,\phi)(v)$. 
By  \eqref{2.89}, \Cref{holomorphy} and the discussion below it,   to compute   $J_{t,\qhol}'(0,g,\phi)$, we    only need to compute $E_{t,\qhol}'(0,g,\phi)(v)$.

Consider the   quasi-holomorphic projection $W_{v,t,\qhol}' (0, g, \phi_v)$ of $W_{v,t}' (0, g, \phi_v)$. By   definition, it is a multiple of $W^{\fw_v}_{v, t}(g)$.
Then by   \eqref{Wtt}, 
$$b_{v,t}: = \frac{W_{v,t,\qhol}' (0, g, \phi_v)}{ W_{v,t} (0, g, \phi_v)}
$$
is a  well-defined constant $b_{v,t}$. 
We  define 
\begin{equation}\label{Finalw}
{\fb}=b_{v,1}.
\end{equation}

\begin{rmk} 
The constant ${\fb}$ can be explicitly computed using  \cite[Proposition 2.11]{YZZ} and  \cite[Lemma 3.3]{Yuan} in principle. 
For example, if  $ n=1$, then  $\fb =-({1+\log 4})$.  (This is twice of the corresponding number in \cite[Lemma 3.3 (2)]{Yuan} due to the difference  between the modulus characters mentioned in \ref{modulus character}.) It  is more complicated in general.   The full computation could be tedious and the result in a previous version of our paper actually contains a mistake. (Fortunately, we will not need the explict number of $\fb$.)
 Ziqi Guo (student of Yuan, author of  \cite{Yuan}) pointed this out to us and  informed us that he will give full details on this in his upcoming work.
\end{rmk}

\begin{lem} We have  $b_{v,t}=b_{v,1}+\log | t|_v$ and  $b_{v,1}$ is independent of $v$.
\end{lem}
\begin{proof}The lemma follows direct computations  with the following ingredients.
For the equation, use \eqref{translaw} and \eqref{Wtt}.
Note that the dependence on $v$ is on the weight $\fw_v$.
We use   \eqref{ktrivial}  and \eqref{wellknown} for $g\in K_v^{\max}$. 
\end{proof}

Then 
$E_{t,\qhol}'(0,g,\phi)(v)=({\fb} + \log |t|_v)
E_t(0,g,\phi).
$ 
Since both $E(0,g,\phi)$ and $C(0,g,\phi)$ in  \eqref{2.89} are holomorphic of weight $\fw$, we have  \begin{equation*} J_{t,\qhol}'(0,g,\phi)=( 2{\fb}[F:\BQ]-{\fc})E_t(0,g,\phi)-C_t(0,g,\phi)+2\lb E'_{t,\rf}(0,g,\phi)+E_t(0,g,\phi) \log \Nm_{F/\BQ}t\rb .
\end{equation*}
Combined with Lemma \ref{becom}, we have 
\begin{equation}\label{Final} \begin{split}&\te_{t,\qhol}'(0,g,\phi)-2\lb E'_{t,\rf}(0,g,\phi) +E_t(0,g,\phi) \log \Nm_{F/\BQ}t\rb\\
=&\te'_{\chol,t}(0,g,\phi) +( 2{\fb}[F:\BQ]-{\fc})E_t(0,g,\phi)-C_t(0,g,\phi).
\end{split}
\end{equation}
%qcl  This term seems should compensate the $G^B-G^{OT}$ $C_0$ really comes from Hodge bundle.

\subsubsection{Quasi-holomorphic projection of $\te ' _t (0,g,\phi)$}
  Assume that    $\BW/\BE$ is incoherent.
  For $v\in \infty$ and $W$ the $v$-nearby hermitian space of $\BW$, 
let $V=W\oplus V^\sharp$.  
For $x\in  V(E_v)-V^\sharp(E_v)$, define    \begin{equation*}\wt \tw_{s}(x)= \frac{\Gamma(s+n)}{\Gamma(n) (4\pi)^{s}}P_s(-{q(x_1)}) ,\label{wttw}\end{equation*}
where $x_1$  is the projection of $x$ to $W(E_v)$ (so that $x_1\neq0$),
and   \begin{equation} \label{Pst}P_s(t):=\int_{u=1}^\infty\frac{1}{u\lb1+tu\rb^{s+n}}du,\ t>0.\end{equation} 
(For $s \in \BC$ with    $ \Re s>-n$, $P_s(t)$ converges absolutely.)   
Define   \begin{equation}\label{Evdechol} \te'_{t,s} (0,g,\phi)(v)
=  \frac{2 W^\fw_{v,t}(g_v)  }{ \Vol([U(W)]) }  \int_{  [U(W)]}\sum_{x\in  V {\qclE} ^t -V^\sharp}   \wt\tw_{s} (h_v^{-1}x)  \omega^v(g^v)\phi^v \lb h^{v,-1} x\rb dh. 
\end{equation}  
\begin{lem}\label{sconv}   
For $s\in \BC$ with $\Re s>0$, \eqref{Evdechol} converges absolutely and $\te'_{t,s} (0,g,\phi)(v)$ is holomorphic on $s$. Moreover, $ \te'_{t,s} (0,g,\phi)(v)$  admits a meromorphic continuation to $\{s\in \BC, \Re s>-1\}$   with at most a simple pole at 
$s=0$.

\end{lem} 
Then the  constant term $\wt{\lim \limits _{s\to 0}}\te'_{t,s} (0,g,\phi)(v)$ of $ \te'_{t,s} (0,g,\phi)(v)$  at $s=0$ is well-defined (and used in the following proposition). 
We will prove  Lemma \ref{sconv} after Lemma \ref{difflem}, by comparing $\te'_{t,s} (0,g,\phi)(v)$ to a Green function with $s$-variable, which has a meromorphic continuation. 
\begin{prop} \label{omitted}
Let   $t\in F_{>0}$ and let $\phi\in \ol\cS(\BV)$ be a pure tensor. 

(1) We have
\begin{equation} \begin{split}\label{E'dec1hol0} \te_{t,\qhol}'(0,g,\phi) =&-\sum_{v \not\in \infty,\ \text{nonsplit}}\te'_t (0,g,\phi)(v)-\sum_{v\in  \infty}\wt{\lim _{s\to 0}}\te'_{t,s} (0,g,\phi)(v)\\
-&\lb 4\frac{L_\rf'(0,\eta)}{L_\rf(0,\eta)}+2\log |\disc_E/\disc_F|\rb\sum_{x\in (V^\sharp)^t}\omega(g)\phi(x)\\
-&\sum_{v\not\in \infty}\sum_{x\in (V^\sharp)^t}c(g,x,\phi_v)\omega(g^v)\phi^v(x)\\
+&\sum_{x\in (V^\sharp)^t}(2\log \delta(g^\infty)+\log|t^\infty|)\omega(g)\phi(x)\end{split}.   \end{equation}
  Here    
$L_\rf(s,\eta)$ is the finite part of $L(s,\eta)$.

(2)  Assume that    $\phi_{{{v}}}$ is supported outside $V^\sharp(E_{v})$ for $v$ in a set $S$ of  two places   of $F$ and   
$g\in   P (\BA_{F,S  })G(\BA_F^{S })$. Then 
we have
\begin{align}\label{E'dec1hol} \te_{t,\qhol}'(0,g,\phi) =&-\sum_{v \not\in \infty, \ \text{nonsplit in }E}\te'_t (0,g,\phi)(v)-\sum_{v\in  \infty}\wt{\lim _{s\to 0}}\te'_{t,s} (0,g,\phi)(v).
\end{align}
  \end {prop}
\begin{proof}

The proof of (1) is almost identical with \cite[Theorem 7.2]{YZ}   and is omitted. (Note that  \cite[Assumption 7.1]{YZ}  in \cite[Theorem 7.2]{YZ} is only used to identify the quasi-holomorphic projection with the cuspidal holomorphic projection and does not play a role in computing the quasi-holomorphic projection). (2) follows from (1) immediately. % The ingredients  in our setting are 
% \ref{E'dec11}, Lemma \ref{YZZ6.6},
%Lemma \ref{YZZ6.11},  \eqref{wellknown} and Lemma \ref{sconv}.   
\end{proof}

\begin{lem}\label{OOO} Let   $t\in F_{>0}$ and let $\phi$ be  a pure tensor such  that $\phi^\infty$ is ${\ol\BQ}$-valued.
Let $u$ be a finite place of $F$ of residue characteristic $p$. Let $O\subset \BV_u$ be an open compact neighborhood of $0$ and $\phi^O= \phi^u\otimes(\phi_u 1_{\BV_u-O})$.
  Given  $g\in    G(\BA_F^{u })P(F_u)$, for $O$ small enough, we have 
$$ \frac{\te_{t,\qhol}'(0,g,\phi)}{W^{\fw}_{\infty, t}(g_\infty)}= \frac{\te '_{t,\qhol} (0,g,\phi^O)} {W^{\fw}_{\infty, t}(g_\infty)}(\MOD\ol{\BQ}\log p).$$ 
\end{lem}
\begin{proof}   The proof   is similar to the one of Lemma \ref{OOO0}, except that we further need  \eqref{Evdec11}, \eqref{wellknown},   \eqref{Evdechol} and \eqref{E'dec1hol0}. 
\end{proof}

By \eqref{E'dec1hol0},    using  \eqref{Evdec11} and  \eqref{Evdechol}, for $h\in U(\BW_v)$ where $v$ is split in $E$, 
\begin{equation}\label{Eqhol'0tk}
\te_{t,\qhol}'(0,g,\omega(h)\phi)=\te_{t,\qhol}'(0,g,\phi) .  \end{equation}

\section {Special divisors}
This section    is about special divisors on unitary Shimura varieties.  It consists of \ref{Geometric modularity of special divisors}-\ref{Arithmetic modularity1}.
First, we define their generating series  which are modular. Second, we introduce their  Green functions  and show the modularity of 
the differences between the generating series of different kinds of Green functions. 
Third, we   raise two modularity problems for their admissible extensions on   integral models. 
Finally,  we propose a precise conjecture and   state our modularity theorems.
 
\subsection{Generating series}\label{Geometric modularity of special divisors}
Let $\BV $ be a totally positive-definite incoherent hermitian space over $\BA_E$ (with respect to  the extension $E/F$)  of dimension $n+1$ where $n>0$.
Fix an infinite place  $v_0 \in \infty$
 of $F$. 
Let $V_0$ be the unique hermitian space over $E$ such that $V_0(\BA_E^v)\cong \BV^v$ and $V_0(E_v)$ is of signature $(n,1)$. 
For  an open compact subgroup $K$   of $U(\BV^\infty)$, 
let   $\Sh(\BV)_K$ be the   $n$-dimensional smooth unitary Shimura variety  associated to $U(V_0)$ of level $K$ over $E$ (see   \cite[3.2]{Let}),
%\qcl{qcl. See Appendix \ref{SectionB} (1.25) }
which we allow to be a   Deligne-Mumford stack. (It is expected that  $\Sh(\BV)_K$  does not depend on the choice of $v_0$. See \cite[Remark 1.2]{LL}.)  %(see \cite[Remark 3.2]{RSZ-arithmetic-diagonal-cycles})
See \eqref{compl} for  its usual complex uniformization.
%More precisely, the Shimura datum is  specified   in \cite{GGP}. See also \cite{RSZ-arithmetic-diagonal-cycles}. And we refer  \cite{Gro1} for the incoherent convention. 
\begin{itemize}\item
From now on, we always assume  that  $\Sh(\BV)_K$  is proper.
\end{itemize}
Equivalently, $F\neq \BQ$, or $F=\BQ,n=1$ and  $\BV^\infty$ is not split at some finite place.

\subsubsection{Simple special divisors}\label{Simple special divisors}
Let $\BV^{\infty}_{>0}\subset \BV^{\infty}$ be the subset of  
$x$'s such that ${q(x)}\in F_{>0}$.  
For $x\in \BV^{\infty}_{>0}$,   let $ x^\perp $ be  the orthogonal complement of $\BA_E ^\infty x  $ in $\BV^\infty$. Regard $U\lb x^{\perp} \rb$ as  a subgroup of $U(\BV^\infty)$.
Then   we have a  finite morphism  
\begin{equation}\label{finitemor}
\Sh\lb x^{\perp }\rb_{U\lb x^{\perp}  \rb\bigcap K }\to\Sh(\BV)_K,
\end{equation}  
explicated in \cite[(2.4)]{Kud97} and \cite[Definition 4.1]{LL}. The proper pushforward defines a divisor  $Z(x)_K  $   on  $\Sh(\BV)_K$ that is called a simple
special divisor.  %(See \eqref{zhx0}\eqref{zhx} for the   description in terms of the complex uniformization of $\Sh(\BV)_K$.)
%Let     $Z(x)=Z(x)_K $  for simplicity if the meaning is clear from the context.
The following observation is  trivial.

\begin{lem}  \label{trivial10}  We have $Z(x)_K=Z(kxa)_K$ for every $a\in  E^\times, k\in K$.

\end{lem}
%In particular, $Z(x)_K$ is defined for a coset $x\in K\bsl \BV^{ \infty}_{>0}$
Let  $L_K$ be  the Hodge line bundle on $\Sh(\BV)_K$. See \ref{Complex uniformization} for the   description in terms of the complex uniformization of $\Sh(\BV)_K$. Let  $c_1(L_K^\vee)$   be the first   Chern class of the dual of  $ L_K$, and $[Z(x)_K]\in \Ch^{1}(\Sh(\BV)_K)  $   the class of $Z(x)_K$.  
%For $C\subset \BC$, let $\ol \cS\lb \BV  \rb_{C}$ be as in \ref {Functions}, and $  \ol \cS\lb \BV  \rb^{K}_{\ol\BQ}\subset \ol \cS\lb \BV  \rb_{C}$   the subspace of $K$ invariant functions as in \ref   {Weil representation}. 
For  $ \phi\in \ol\cS\lb \BV  \rb^{K}$  the subspace of $K$-invariant functions, define a formal  generating series   of divisor classes:
\begin{equation*}z(\phi)_K=\phi(0)c_1(L_K^\vee)+\sum_{x\in K\bsl \BV^{ \infty}_{>0}}\phi( x_\infty x) [Z(x)_K]  , \end{equation*}
where
$ x_\infty \in   \BV_\infty$ such that $q(x_\infty)={q(x)}\in F_{>0}$.

%Let $ \Ch^{1}(\Sh(\BV)_K)_{\ol\BQ} $ be the Chow group of divisors with $\ol\BQ$-coefficients. 

Let $\fw=\fw_{\chi_{_\BV}}$  which is defined  in  \ref{Weight}.
\begin{thm}   [{\cite[Theorem 3.5]{Liu}}] \label{Modu}

For every $ \phi\in\ol \cS\lb \BV  \rb^{K} $,   we have   $$z\lb\omega(\cdot )\phi\rb_K\in  \cA_{\hol}(G,\fw) \otimes  \Ch^{1}(\Sh(\BV)_K)_{\BC} .$$
\end{thm}
Note that $ \dim \Ch^{1}(\Sh(\BV)_K)_{\BC} <\infty.$ 

\subsubsection{Weighted special divisors}
For  $\phi\in \ol\cS\lb \BV  \rb^{K}$ and  $t\in F_{>0}$,  define the weighted  special divisor   \begin{equation*}  Z_t(\phi)_K=\sum_{x\in K\bsl \BV^{ \infty}, \ q(x)=t}\phi (   x_\infty x) Z(x)_K,   \end{equation*}
which is a finite sum.
Lemma \ref{trivial10} implies the following lemma.
\begin{lem}  \label{trivial1}   For every $a\in  E^\times$, $Z_t(\phi)_K=Z_{a^2t}(\omega(a)\phi)_K.$

\end{lem} 

Let   $\phi= \phi_\infty\otimes  \phi^\infty$ be as in \ref  {Functions}. We define another  weighted  special divisor 
\begin{equation*}
Z_t(\phi^\infty)_K=\sum_{x\in K\bsl \BV^{ \infty}, \ q(x)=t}\phi^\infty(   x) Z(x) _K.  
\end{equation*}
By \eqref{wellknown}, for $g\in G(\BA_F)$,  we have  
\begin{equation}\label{wellknownZ} 
Z_t(\omega(g)\phi)_K=Z_t(\omega(g^\infty)\phi^\infty) _KW^{\fw}_{\infty, t}(g_\infty).\end{equation}
Then %  the $\Ch^{1}(\Sh(\BV)_K)_{\BC}$-valued 
$[Z_t\lb \omega(\cdot )\phi^\infty\rb_K]  $   is 
the $t$-th Fourier coefficient of $z\lb\omega(\cdot)\phi\rb_K$. See \ref{Holomorphic  automorphic  forms}. 

\begin{lem}\label{OOOO}Let $\phi\in \ol \cS\lb \BV  \rb $  be a pure tensor.
Let $u$ be a finite place of $F$, $O\subset \BV_u$  an open compact neighborhood of $0$ and $\phi^O=\phi^u\otimes(\phi_u 1_{\BV_u-O})$.
Given  $g\in    G(\BA_F^{u })P(F_u)$, for $O$ small enough, if $K$  fixes $\phi^O$, then
$ Z_t(\omega(g)\phi) _K= Z_t(\omega(g)\phi^O) _K.$
\end{lem}
\begin{proof}The lemma is an  analog  to    \Cref{OOO0}, and the proof is also similar. We record the proof for the reader's convenience. 
Write $g_u=m(a)n(b)$  where $a\in E_u^\times$. See \ref{Group}.
Then $\{a  x:x\in \BV_u^t\}\subset \BV_u^{a^2t} $. The latter is closed in $\BV_u$ and does not contain 0. Thus  $O\cap  \{ax:x\in \BV_u^t\}=\emptyset$ if $O$ is small enough. Then the lemma  follows from   the definition of $ Z_t(\omega(g)\phi) _K$ and  the {Weil representation} formula  in \ref{Weil representation}. 
 \end{proof}
\subsubsection{Change level}\label{Pullback and pushforward}
For  $K\subset K'$, let $\pi_{_K,_{K'}}:\Sh(\BV)_{K}\to \Sh(\BV)_{K'}$ be the natural projection. 
\begin{lem}[{\cite[PROPOSITION 5.10]{Kud97}\cite[Corollary 3.4]{Liu}}]\label{yespull} We have $\pi_{_K,_{K'}}^*Z_t(\phi)_{K'}=Z_t(\phi)_K$ and $\pi_{_K,_{K'}}^*L_{K
'}=L_K$.
 In particular, $\pi_{_K,_{K'}}^*z(\phi)_{K'}=z(\phi)_K$.
\end{lem}

\begin{rmk}\label{notpush} We have  $\pi_{_K,_{K'},*}Z(x)_K
=d(x) Z(x)_{K'}$ where $d(x)$  is the degree of $Z(x)_K$ over $Z(x)_{K'}$. It is easy to check that   $d(x)$  is  not constant near 0. In particular, it does not
extend to a smooth function on $\BV^\infty$. Thus for a general $\phi$,
 there seems no $\phi'\in \ol \cS\lb \BV  \rb^{K'}$ such that 
$\pi_{_K,_{K'},*}Z_t(\phi)_{K}$ is  of the form $Z_t(\phi')_{K'}$ for every $t$.    

\end{rmk}

Below, let $L=L_K$, $Z(x)=Z(x)_K$, $Z_t(\phi)=Z_t(\phi)_K$ and $z(\phi)=z(\phi)_K$ for simplicity if $K$ is clear from the context.

\subsection{Green functions}\label{Green functions}
We use   complex uniformization to defined automorphic Green functions for special divisors. They are admissible. We compare them with 
the
normalized admissible Green functions, and show the modularity of 
the differences between their generating series. 
Then we recall Kudla's Green functions and prove the modularity of 
the differences between the generating series of normalized admissible Green functions  and Kudla's Green functions.

\subsubsection{Complex uniformization}\label{Complex uniformization}
For nonnegative integers $p,q$,  let $\BC^{p,q}$ be the $p+q$ dimensional hermitian space associated to the hermitian matrix 
$\diag(-1_p,1_q)$. %We will consider $\BC^{1,n}$, i.e., the hermitian  matrix is $\diag(-1,1,1,...,1).$  
Let $U\lb\BC^{1,n}\rb$ be the unitary group of  $\BC^{1,n}$ (so of signature $(n,1)$ in the usual convention). 
Let $\BB_n$ be the complex open unit ball of dimension $n$. 
Embed  $\BB_n$  in $\BP(\BC^{1,n})$ as the set of negative lines as follows:
  $[z_1,...,z_n]\in \BB_{n}$, where $\sum_{i=1}^{n}|z_i|^2<1$,   is   the line containing  the vector $(1,z_1...,z_n) $.  
Then $U\lb\BC^{1,n}\rb$   acts on $\BB_n$  naturally  and transitively. 
Let $\Omega$ be the tautological bundle  of negative lines on $\BB_n$ , and $\ol \Omega$ the hermitian line bundle with the metric induced from the negative of the hermitian form.
The Chern form of $\ol \Omega$ is a 
$U\lb\BC^{1,n}\rb$-invariant 
K\"ahler form
\begin{equation}c_1(\ol\Omega)= \frac{1}{2\pi i } \partial \bar\partial \log (1-\sum_{i=1}^{n}|z_i|^2).\label{mun}\end{equation}  
%Let the volume form on $\BB_n$ be $\Omega^n$ (instead of $\frac{1}{n!}\Omega^n$). 
For an arithmetic subgroup $\Gamma\subset U(\BC^{1,n})$, on (the orbifold) $\Gamma\bsl\BB_n $, we have the descent $\ol\Omega_\Gamma$ of $\ol\Omega$ whose (orbifold) Chern curvature  form  is the descent of $c_1(\ol\Omega)$. 
Define the degree $$\deg(\ol\Omega_\Gamma)=\int_{\Gamma\bsl\BB_n} c_1(\ol\Omega_
\Gamma)^n.$$
If $\Gamma\bsl \BB_n$ is compact,  this degree   is the usual degree via intersection theory.

%Now we consider special divisors.
For $x\in \BC^{1,n}$,
let $\BB_x\subset \BB_n$ be the subset of  negative lines  perpendicular to $x$.
Later, we will use the following function  on $\BB_n $ measuring ``the distance to $\BB_x$":
  $$R_x(z)=-\frac{|\pair{x,\wt z}|^2}{\pair{\wt z,\wt z}}$$
where $\wt z $ is a nonzero vector contained in the  line  $z$.   
In particular, $R_x>0$  outside  $\BB_x$.
If $q(x)<0$, then $x^\perp$ is of signature $(n,0)$ so that $\BB_x=\emptyset$. 
Below, until \ref{Kudla's Green function}, 
assume ${q(x)}>0$. Then $x^\perp$ is of signature $(n-1,1)$. So
$\BB_x$ is a complex unit ball of  dimension $n-1$.  
Assume that $\Gamma\cap U (x^\perp)$ is   an arithmetic subgroup
of $U (x^\perp)$. 
Let   $C(x,\Gamma)$ be    the pushforward of the fundamental cycle by 
$$\lb \Gamma \cap U(x^\perp)\rb \bsl \BB_x \to \Gamma \bsl \BB_n.$$
Define $$\deg_{\ol\Omega_\Gamma}(C(x,\Gamma))=\int_{C(x,\Gamma)} c_1(\ol\Omega_
\Gamma)^{n-1}.$$
%This is also $\deg({\ol\Omega_{\Gamma\cap U(x^\perp)}})$ by definition.

%For $x\in V$ (regarded  as an element in $\BV^\infty$), and  $h\in U(\BV^\infty)$,  we  have $h^{-1}x\in \BV^\infty$.  Assume that $q(x)\neq 0$.

Now we can uniformize Shimura varieties and special divisors. Let $v$ be an infinite place of $F$. 
Let $V$ be the unique hermitian space over $E$ such that $V(\BA_E^v)\cong \BV^v$ and $V(E_v)\cong \BC^{1,n}$.
  By \cite[Lemma 5.5]{LL}, 
we have the complex uniformization: %(see \cite[Remark 3.2]{RSZ-arithmetic-diagonal-cycles})
\begin{equation}\label{compl}\Sh(\BV)_{K,E_v} \cong   U(V)\bsl(   \BB_n \times  U(\BV^\infty)/K). \end{equation}
Then the Hodge bundle $L_{E_v}$ is the descent of $\Omega\times 1_{U(\BV^\infty)/K}$. Let $\ol{L}_{E_v}$ be the descent of $\ol \Omega\times 1_{U(\BV^\infty)/K}$. % Define $\deg(\ol{L}_{E_v})$ as above.
% Equip $\Sh(\BV)_{K,E_v}$ with this K\"ahler form and volume form. 
 Let $h_1,. . .,h_m$ be a set of representatives of  $U(V)\bsl U(\BV^\infty)/K$. 
%By \cite[Lemma 5.5]{LL},  the special divisor $Z(h^{-1}x)_{E_v}$ on    $\Sh(\BV)_{K,E_v}$ is  represented  by 
%\begin{equation}
%\{(z,h_1h): z\in \BB_x, \ h_1\in U(x^\perp) \}\label{zhx0}\end{equation}
%via \eqref{compl}.   
%(Recall that we have defined $Z(h^{-1}x)_{E_v}=\emptyset$ is $q(x)\not\in F_{>0}$.)
Let $\Gamma_{h_j}=U(V)\cap h_jKh_j^{-1}$.  
Then \eqref{compl} decomposes as the disjoint union of $\Gamma_{h_j} \bsl \BB_n$'s.
Then for $t\in F_{>0}$,  we have
(see \cite[PROPOSITION 5.4]{Kud97})
\begin{equation}
Z_t(\phi^\infty)_{E_v}=\sum_{j=1}^m\sum_{x\in \Gamma_{h_j}\bsl V^t} \phi^\infty(h_j^{-1}x)C(x,\Gamma_{h_j}).\label{zhx}\end{equation}
%By \cite[Lemma 3.1]{Liu}, every element of $  \BV^\infty_{>0} $    is of the form $h^{-1}x$ for some $x$ with $q(x)\in F_{>0}$ as above. Thus we have  decomposed    special cycles into geometrically connected components.

\subsubsection{Admissible Green functions}\label{Admissible Green functions}

Admissible Green functions are Green functions with harmonic curvatures. See Appendix \ref{until}. 
Admissible Green functions  for special divisors are constructed  by   Bruinier \cite{Bru0,Bru}, Oda and Tsuzuki \cite{OT}.  For $F=\BQ$ and $n=1$, it  appeared in the work of   Gross and Zagier \cite{GZ} for $n=1$.
We learned the following explicit computation from S.~Zhang.

First, we start by working on $\BB_n$.  
Let $x_0=(0,...,0,1)\in \BC^{1,n}$. Then 
   $\BB_{x_0}\subset \BB_n$ consists of points $  (z_1,...,z_{n-1},0) $'s. 
%The same embedding  $\BC^{1,n-1}\incl \BC^{1,n}$ also fixes an embedding  $U\lb x_0 ^\perp\rb\incl U\lb\BC^{1,n}\rb$.
%The action of $U\lb x_0 ^\perp\rb$ on $\BB_{x_0}$ is compatible with the embeddings. 
 We want a $U\lb x_0 ^\perp\rb$-invariant smooth function $G$ on $\BB_n-\BB_{x_0}$ 
such that $G(z_1,...,z_n)+\log|z_n|^2 $ extends   smoothly to $\BB_n,$
%(so $G$ is a Green current of $\BB_{x_0}$ by the Poincar\'e-Lelong formula \cite[1.3.1.1]{GS},)  
$\lim_{|z|\to 1}G(z) =0,$ 
and   
%\begin{equation}\label{ppg} \frac{i}{2\pi } G\wedge \omega^{n-1}=\ep G\omega^n.\end{equation}
%\begin{equation*} \frac{i}{2\pi } \partial \bar\partial G =\frac{s(s+n) }{2}G \cdot {c_1(\Omega)},\ s\in \BC  .\end{equation*}
$G$ is a solution of   the following Laplacian equation 
\begin{equation}\lb\frac{i }{2\pi } \partial \bar\partial G \rb {c_1(\Omega)}^{n-1}=\frac{s(s+n) }{2}G \cdot {c_1(\Omega)}^{n}  .\label{ppg}\end{equation} 
The quotient of  
$\BB_n-\BB_{x_0}$ by  $U\lb x_0 ^\perp\rb$ is isomorphic to $(1,\infty)$ via 
$$z=[z_1,...,z_n]\mapsto t(z):= 1+ R_{  x_0  }\lb z\rb=\frac{1-\sum_{i=1}^{n-1}|z_i|^2}{1-\sum_{i=1}^{n}|z_i|^2} .$$
Thus we look for  $G=Q(t(z))$ where $Q$ is a smooth function on $(1,\infty)$ such that $ Q(t)+ \log( t-1) $   extends to a smooth function on $\BR$. 
By \eqref{mun} and the $U\lb x_0 ^\perp\rb$-invariance, 
\eqref{ppg} is reduced to the
following   hypergeometric differential equation
\begin{equation*} \lb t-t^2\rb \frac{d^2Q}{dt^2}  +\lb n-(n+1)t\rb\frac{dQ}{dt} +s(s+n) Q=0,\ t\in (1,\infty).\label{ODE1} \end{equation*} 
For $\Re s>-1$, there is a unique solution $Q_s$ such that $Q_s(t)+\log (t-1)$  extends to a smooth function on $\BR$  and $\lim_ {t\to\infty}Q_s(t)=0$:
\begin{equation} \label{Qs} Q_s(t)=\frac{\Gamma(s+n)\Gamma(s+1)}{\Gamma(2s+n+1)t^{s+n}}  F\lb s+n,s+1,2s+n+1;\frac{1}{t},\rb, \ t>1\end{equation} 
where $F$ is the hypergeometric function. (See also \cite[2.5.3]{OT}. 
When 
$n=1$, our $Q_s$ is   the Legendre function of the second kind in \cite[238]{GZ}   up to  shifting $s$ by 1).

For a general $x$ with $q(x)>0$, we have the following Green function  for $\BB_x$
 $$G_{x,s}(z)=Q_s\lb 1+R_x(z)\rb,\ z \in \BB_n- \BB_x .$$ 
We will need the following explicit formula later:  if  $ x=(x_1,x_2)\in \BC^{1,n}= \BC^{1,0}\oplus \BC^{0,n} $, then 
\begin{equation}G_{x,s}([0,...,0]) =Q_s\lb 1-{q(x_1)}\rb.\label{GQs}\end{equation}

Second, we define   Green functions for $ C( x ,\Gamma)$ on arithmetic quotient of $\BB_n.$
Let $\Gamma$ be an arithmetic subgroup of $U\lb\BC^{1,n}\rb$.
 Let $$g_s=\sum_{\gamma\in \Gamma}\gamma^*G_{ x  ,s}.$$
\begin{lem}
[{\cite[Proposition 3.1.1, Remark 3.1.1, Remark 3.2.1]{OT}}]\label{gaodingle} For  $s \in \BC$ with $\Re s>0$,  the sum   $g_s$ converges  absolutely and  defines a smooth function on $ \Gamma\bsl \BB_{n}-C( x ,\Gamma)$. Moreover $g_s$ is holomorphic on  $s$.
\end{lem} 
It is easy to see that    $g_s$  with $\Re s>0$ is a Green function for $ C( x ,\Gamma)$.
\begin{thm}[{\cite[Theorem 7.8.1]{OT}}]\label{OTthm}  
(1)   There is a meromorphic continuation of $g_s$ to $s\in \BC$ with a simple pole at $s=0$ and  residue $-\frac{\deg_{\ol\Omega_\Gamma}(C( x ,\Gamma))}{\deg  (\ol\Omega_
\Gamma)}$. 

(2)  The function $\wt{\lim\limits _{s\to 0}} g_s$ is an admissible Green function for $ C( x ,\Gamma)$.
\end{thm}
Recall that  $\wt{\lim\limits _{s\to 0}}$ denotes  taking the constant term at $s=0$.
\begin{rmk}  (1) We      read the residue
$\kappa$ in  \cite[Theorem 7.8.1 (3)]{OT} as follows.
Beside the obvious differences between the choices of % the $s$-variables   %(up to shift and multiplying  2)
%and 
 Green function  %(up to multiplying  $-2$)
 here and in \cite{OT} (more precisely, $s$-variables and signs),       the K\"ahler form and ``volume form" here and in \cite{OT} are different.  
First, the K\"ahler form on the bottom of \cite[514]{OT} is $\pi c_1(\ol\Omega)$ in our notation.
Second,  \cite[Theorem 7.8.1 (3)]{OT} uses volumes to express the residue, 
while  we use degrees of line bundles to express the residue. 
%    factorial factor when defining the volume form on the  top of \cite[515]{OT}.
The volume form for $  \BB_n$ on the  top of  \cite[515]{OT} is     $\frac{\pi^n}{n!}c_1(\ol\Omega)^n$ 
%and $\frac{\pi^{n-1}}{(n-1)!}c_1(\ol\Omega_
%{\Gamma'})^{n-1}$ 
in our notation.
Third, there is a $\pi$ missing in (the numerator of) the residue $\kappa$ in  \cite[Theorem 7.8.1 (3)]{OT}. It should be easy to spot from \cite[Proposition 3.1.2, Lemma 7.2.2]{OT}.

(2) When $n=1$ and $\Gamma=\SL_2(\BZ)$ via $\SU(\BC^{1,1})\cong \SL_2(\BR)$,  we have $\deg  (\ol\Omega_
\Gamma)=\frac{1}{12}$. See \cite[4.10]{Kuh}. Thus for $\Gamma_0(N)$ a standard congruence subgroup of $\Gamma$, 
Theorem \ref{OTthm} (1) coincides with the residue $\frac{-12}{[\Gamma:\Gamma_0(N)]}$ in \cite[p 239, (2.13)]{GZ}.
Theorem \ref{OTthm} (1)  is not used in any other place of the paper.

\end{rmk}
By  \cite[Proposition 3.1.2]{OT} and taking care of the differences in the remark, we have 
$$\int_{\Gamma\bsl \BB_n} g_s  {c_1(\ol \Omega)}^n=\frac{n\deg_{\ol\Omega_\Gamma}(C( x ,\Gamma))}{s(s+n)}.$$
Thus
\begin{align}\label{gnm1} \int_{\Gamma\bsl \BB_n}(\wt{\lim\limits _{s\to 0}}g_s ) {c_1(\ol \Omega)}^n=-\frac{\deg_{\ol\Omega_\Gamma}(C( x ,\Gamma))}{n}.   \end{align}

Finally, we define   Green functions    for  $Z_t(\phi)$, $t\in F_{>0}$.     Let $v\in \infty$ and $V$ as in \ref{Complex uniformization}.  For   $s \in \BC$ with $\Re s> 0$, consider the  following formal  sum for  $(z,h)\in \BB_n \times  U(\BV^\infty)$:
\begin{align*} \cG_{Z_t(\phi^\infty)_{E_v},s}(z,h):=\sum_{x\in   V^t }   \phi^\infty\lb  h^{-1}x\rb  G_{x,s}(z).
\end{align*}
By \eqref{zhx} and  \Cref{gaodingle},  if $(z,h)$ is not over $Z_t(\phi^\infty)_{E_v}$  via \eqref{compl}, the formal sum is defined and 
absolutely convergent. Thus   $ \cG_{Z_t(\phi^\infty)_{E_v},s}$ 
descends to $\Sh(\BV)_{K,E_v}-Z_t(\phi^\infty)_{E_v} $ via \eqref{compl}, which we still denote by $\cG_{Z_t(\phi^\infty)_{E_v},s}$.
By \eqref{zhx} and Theorem \ref{OTthm},  $\cG_{Z_t(\phi^\infty)_{E_v},s} $ is 
a Green function for  $Z_t(\phi^\infty)_{E_v}$.
%By \eqref{ppg}, $g_s$ satisfies $$\Delta g_s =s(s+n) g_s .$$ 
For $g\in G(\BA_F)$, let
 \begin{equation*}  
\cG_{Z_t(\omega(g)\phi)_{E_v},s}= \cG_{Z_t(\omega(g^\infty)\phi^\infty)_{E_v},s} W^{\fw}_{\infty, t}(g_\infty),\end{equation*}
which is a Green function for  $Z_t(\omega(g)\phi)_{E_v}$ by \eqref{wellknownZ}.
Define the automorphic Green functions for  $Z_t(\phi^\infty)_{E_v}$ and $Z_t(\omega(g)\phi)_{E_v}$ to be
\begin{align}\label{GGZ} \cG^{\aut}_{Z_t(\phi^\infty)_{E_v}}=\wt{\lim_{s\to 0}}  \cG_{Z_t(\phi^\infty)_{E_v},s},\   \cG^{\aut}_{Z_t(\omega(g)\phi)_{E_v}}=\wt{\lim_{s\to 0}}  \cG_{Z_t(\omega(g)\phi)_{E_v},s}  \end{align}
respectively,
which 
are  admissible   by Theorem \ref{OTthm}.

Let $ \cG^{\ol L_{E_v}}_{Z_t(\phi)_{E_v}}$ be the normalized admissible Green function
for $Z_t(\phi)_{E_v}$  with respect to $\ol L_{E_v}$ as in \Cref{admextgreen}. In particular, its integration against $c_1(\ol L_{E_v})^{n}$ is 0.

\begin{lem}\label{gnm11} 

(1)   We have 
$$
\cG^{\aut}_{Z_t(\omega(g)\phi)_{E_v}}- \cG^{\ol L_{E_v}}_{Z_t(\omega(g)\phi)_{E_v}}=-\frac{1}{n} \frac{c_1(L)^{n-1}\cdot  Z_t(\omega(g)\phi) }{c_1(L)^{n} }.$$
Both sides  are independent of $K$.

(2)  We have  $$\frac{1}{n}\omega(g)\phi(0)+\sum_{t\in F_{>0}}  \lb \cG^{\aut}_{Z_t(\omega(g)\phi)_{E_v}}- \cG^{\ol L_{E_v}}_{Z_t(\omega(g)\phi)_{E_v}}\rb\in  \cA_{\hol}(G,\fw).$$

\end{lem}  
\begin{proof} The equation in (1) follows from \eqref{gnm1} and  the independence  of the right hand side follows from the projection formula.  (2)
follows from \Cref {Modu}. 
\end{proof}
\begin{rmk} The automorphic form in \Cref{gnm11} (2) can be made explicit:
\begin{equation}    \label{Proposition 4.1.5}
c_1(L)^{n-1}\cdot z\lb\omega(g)\phi\rb =-c_1(L)^{n}  E(0,g, \phi).  
\end{equation} 
This is  
a geometric version of the Siegel-Weil formula \eqref {Ichthm}. 
It is  stated  in \cite[COROLLARY 10.5]{Kud97} for the orthogonal case, the proof carries  over to the unitary case.

\end{rmk}

\subsubsection{Kudla's Green function}\label{Kudla's Green function}
We recall Kudla's Green functions for special divisors   \cite{Kud1}, following \cite[4C]{Liu} in the unitary case.
We consider simple special divisors $Z(x)$'s, instead of weight special divisors $Z_t(\phi)$'s.
  We extend the definition of special divisors as follows.
 For $x\in   \BV^{\infty} $ such that ${q(x)}\in F ^\times -F_{>0}$,    let $Z(x)= 0$. 

First, we  work on  $\BB_n$.
For    $v\in \infty$, $V$ as in the end of \ref{Complex uniformization} with respect to $v$, 
$g\in G(F_v)$ and $x\in V$, define  
$$G^{\mathrm{Kud}}(x,g)(z)=-\Ei(-2\pi \delta_v(g) R_x(z)),  \ z \in \BB_n\bsl \BB_x ,$$   where   the  exponential integral  $\Ei(t)=\int_{-\infty}^t\frac{e^s}{s}ds$    on $t\in (-\infty,0)$ has a log-singularity at 0. %   \eqref{Ei}. 
 If  $q(x)\neq0$ so that $\BB_x$  is either empty or a complex unit ball of  dimension $n-1$,
  $G^{\mathrm{Kud}}(x,g)$ is a   Green function for $\BB_x$. If  $q(x)<0$,  equivalently $\BB_x$  is empty, then $G^{\mathrm{Kud}}$ is smooth.

Now we work on  $\Sh(\BV)_K$. 
Let $x\in  \BV^\infty  $  with $q(x)\in F^\times$. 
For  $v\in \infty$, if $u(q(x))>0$  for every $u\in \infty-\{v\}$,  
by the Hasse-Minkowski theorem and Witt's theorem,
there exists  $h\in U(\BV^\infty)$
and $x^{(v)}\in V-\{0\}$, where $V$ is as in the last paragraph, such that $ x=h^{-1}x^{(v)}$.
Define 
\begin{equation}\cG^{\mathrm{Kud}}_{Z( x)_{E_v}}(g)=\sum_{j=1}^m\sum_{y\in U(V)x^{(v)}\cap h_jK h^{-1}x^{(v)}} G^{\mathrm{Kud}}(y,g) ,\label{Gzhx}\end{equation}
where   $h_1,. . .,h_m$ is a set of representatives of  $U(V)\bsl U(\BV^\infty)/K$. 
 By  the decomposition analogous to \eqref{zhx} (see the second equation on \cite[p 56]{Kud97}) and \cite[Proposition 4.9]{Liu}), $\cG^{\mathrm{Kud}}_{Z( x)_{E_v}}(g)$ is absolutely convergent and 
descends to $\Sh(\BV)_{K,E_v}$ via \eqref{compl}. And it
  is a Green function for $Z(x)_{E_v}$. 
% In particular, if $q(x)$ is negative at more than 1 infinite places,  then  $\cG^{\mathrm{Kud}}_{Z(x)_{E_v}}(g)=0$ for every $v\in \infty$. 

%For $x=0$, let 
% \begin{equation}G^{\mathrm{Kud}}(0,g)(z)=-\log |\delta(g)|_{E_v}\label{GK0} .\end{equation} %(this term is one of the terms to be added to the metric of Hodge bundle).  

% Taking linear combination, we get  a Green function $ \cG^{\mathrm{Kud}}_{Z_t(\phi)_{E_v}}$ of $Z_t(\phi)_{E_v}$ .

Besides Kudla's Green functions, 
we will need their projections to the constant function $1$ to modify  the  normalized  Green function. See \ref{Conjecture}.

\begin{defn}\label{GK0}  For  $x\in  \BV^\infty  $  with $q(x)\in F^\times$, $g\in G(\BA_{F,\infty})$ and 
$v\in \infty$,   let 
$$\fk(x,g_v)=\frac{1}{\deg (\ol{L}_{E_v})}\int_{\Sh(\BV)_{K,E_v} }\cG^{\mathrm{Kud}}_{Z(x)_{E_v}}(g_v) c_1(\ol{L}_{E_v})^n.$$
% Let $$\fk(x,g)=\sum_{v\in \infty}\fk(x,g_v).$$
%    In particular, $ \fk(x,g_v ) =0$ if $u(q(x)) <0$  for some $u\in \infty-\{v\}$.
\end{defn}
%Note that $b(x,g)$ does not depend on the choice of $\nu$ above $v$. 
Note that If $u(q(x))<0$  for some $u\in \infty-\{v\}$,  then  $\cG^{\mathrm{Kud}}_{Z(x)_{E_v}}(g)=0$. Thus
if $q(x) $ is negative   at more than one infinite places,  then  $\cG^{\mathrm{Kud}}_{Z(x)_{E_v}}(g)=0$ and $  \fk(x,g_v ) =0$ for every $v\in \infty$. 
%Note that for $g^v\in G(\BA_F^v)$, we have   $$ E_t'(0,gg^v,\phi)(v)=E_t'(0,g,\omega(g^v)\phi)(v)$$ by Lemma \ref{lSWlem} (2).
%, equivalently $F\neq \BQ$, or $F=\BQ,n=1$ and  $\BV^\infty$ is not split at some finite place.

The following can be read from  \cite[(1.12), Theorem 1.2, (1.18), (1.19), Proposition 5.9]{GS19}.   

% and let $\fe_t(  \phi^\infty)=\fe_t(1, \phi).$

\begin{thm}  
\label{LSthm}  
Let $ E_t'(0,g,\phi)(v)$ be as in \eqref{etv}.
For $t\in F_{>0}$ and $v\in \infty$, we have
\begin{align*} &-W^{\fw}_{\infty, t}(g_\infty) \sum_{x\in K\bsl \BV^{ \infty}, \ q(x)=t}\omega(g^\infty)\phi^\infty (   x) \fk(x,g_v)\\
=&
E' _t(0,g,\phi)(v) -E_t(0,g,\phi) \lb   \log \pi-(\log\Gamma)'(n+1) + \log v(t)\rb. 
\end{align*}
For $t\in  F^\times$ with $v(t)<0$  for exactly   one infinite place $v$,   we have 
$$ -W^{\fw}_{\infty, t}(g_\infty)  \sum_{x\in K\bsl \BV^{ \infty}, \ q(x)=t}\omega(g^\infty)\phi ^\infty(   x) \fk(x,g_v)=
E' _t(0,g,\phi)(v) .  $$
%For $t=0$, we have  $$ \omega(g)\phi (   0) \fk(0,g_\infty)=E' _0(0,g,\phi)(v) .  $$
And for $t=0$, we have $$  \omega(g)\phi (0 ) \log \delta_\infty(g_\infty)=
E' _0(0,g,\phi).  $$
\end{thm}

%\begin{rmk}One may regard this theorem as a ``semi-local arithmetic Siegel-Weil formula". \end{rmk}
We remind the reader that $\infty$ is the set  of  infinite places of $F$. And our formulas    differ from \cite{GS19} by a factor $1/2$  since in loc. cit., the authors use the set  of  infinite places of $E$.

\subsubsection{Modularity of difference of Green functions}\label{Modularity of difference}
We need a more general notion of modular  forms.

\begin{defn} \label{valmod} Let $V$ be a  topological $\BC$-vector space and $V^*$ the continuous dual. 
Let $  \cA(G,\fw, V) $  be the space of
  smooth $V$-valued function $f$ on $G(\BA)$  such that
  for every $l\in V^*$, we have $l\circ f\in  \cA(G,\fw).$  
  
   \end{defn}
\begin{rmk}  Note that $l\circ f$ is automatically smooth. 

\end{rmk}
Clearly, if $V$ is the topological direct sum of $V_1,   V_2$  and $f_i\in  \cA(G,\fw, V_i) $, then $f_1+f_2\in   \cA(G,\fw, V) $.

Now we  define the formal generating series of Green functions. Recall that $\infty$ is the set  of  infinite places of $F$, of cardinality $[F:\BQ]$.
Let $$E_\infty:=E\otimes_\BQ\BR\cong \prod_{v\in \infty}E_v,$$  which is the product of $[F:\BQ]$ many copies of $\BC$. 
 Then $$\Sh(\BV)_{K,E_\infty}: = \Sh(\BV)_{K} \otimes_E E_\infty  = \Sh(\BV)_{K}\otimes_\BQ\BR $$ is the (disconnected) complex manifold that is the disjoint union of all base changes of  $\Sh(\BV)_{K}$  to $E_v$'s (each base change itself may not be connected either!). 

Let $$ \cG^{\nadm}(g,\phi)=\sum_{t\in F_{>0}}  \sum_{v\in \infty}  \cG^{\ol L_{E_v}}_{Z_t(\omega(g)\phi)_{E_v}}, $$ 
\begin{align*} \cG^{\mathrm{Kud}}(g,\phi)= \sum_{x\in \BV^\infty,\ q(x)\in F^\times}   \omega(g^\infty)\phi^\infty (   x) W^{\fw}_{\infty, q(x)}(g_\infty)  \sum_{v\in \infty} \cG^{\mathrm{Kud}}_{Z( x)_{E_v}}(g) ,
\end{align*} 
which are formal generating series 
of smooth functions  on
$\Sh(\BV)_{K,E_\infty}$ with   logarithmic singularities along  the same formal generating series of special divisors. Then
$ \cG^{\nadm}(g,\phi)- \cG^{\mathrm{Kud}}(g,\phi)$ is a formal generating series   valued in ${C^\infty}(\Sh(\BV)_{K,E_\infty}),$
the space of smooth $\BC$-valued functions on $\Sh(\BV)_{K,E_\infty}$.

Let $E_1:=\cO_{\Sh(\BV)_{K}}\lb \Sh(\BV)_{K}\rb$, which is a finite field extension of $E$ (since $\Sh(\BV)_{K}$ is connected). 
Then we have a morphism $\Sh(\BV)_{K }\to \Spec E_1$.
By Stein factorization, $\Sh(\BV)_{K }$, as a variety over $E_1$, is geometrically connected. 
So 
 the connected components of $\Sh(\BV)_{K,E_\infty} =\Sh(\BV)_{K}\otimes_\BQ\BR $  
 are exactly  indexed  by the underlying set of 
 $\Spec E_1\otimes_\BQ\BR$, equivalently,
 the set of conjugate pairs of  infinite places of $E_1$ (which, as a finite field extension of the CM field  $E$,   has only complex embeddings but no real embeddings).  
Let $  \LC(\Sh(\BV)_{K,E_\infty})$ be the space of locally constant functions on $\Sh(\BV)_{K,E_\infty}$.
 Then 
we have the canonical  isomorphism from $  \LC(\Sh(\BV)_{K,E_\infty})$ to  the product   copies of $\BC$ indexed by  the set of conjugate pairs of  infinite places of $E_1$. Now we can
embed $ \cO_{E_1}^\times$ in $\LC(\Sh(\BV)_{K,E_\infty})$  via this isomorphism and the Dirichlet regulator map, so that the $\BC$-span  of the image of $ \cO_{E_1}^\times$, denoted by $\BC \cO_{E_1}^\times$ is of 
codimension 1 in $  \LC(\Sh(\BV)_{K,E_\infty})$  by Dirichlet's unit theorem.  Let   $\BC \cO_{E_1}^\times$ be this span.
Let  $$\ol{C^\infty}(\Sh(\BV)_{K,E_\infty})=C^\infty(\Sh(\BV)_{K,E_\infty})/\BC \cO_{E_1}^\times.$$
Equip $\ol{C^\infty}(\Sh(\BV)_{K,E_\infty}) $ with the quotient of the $L^\infty$-topology.
Define an embedding \begin{equation}\label{BC1} \BC\cong\LC(\Sh(\BV)_{K,E_\infty})/ \BC \cO_{E_1}^\times \subset \ol{C^\infty}(\Sh(\BV)_{K,E_\infty})
\end{equation}    by
mapping   $a\in \BC$  to     the constant function $a$ on $\Sh(\BV)_{K,E_v}$ for some $v\in \infty$ (rather than $\Sh(\BV)_{K,E_{1,w}}$ for some infinite place $w$ of $E_1$).
Below,  by a complex number in  $ \ol{C^\infty}(\Sh(\BV)_{K,E_\infty})
$, we  understand it as  the image by \eqref{BC1}.

%Let $ \ol{\cG^{\nadm}(g,\phi)- \cG^{\mathrm{Kud}}(g,\phi)}$ denote the generating series   of  elements in  $ \ol{C^\infty}(\Sh(\BV)_{K,E_\infty})$.

%   Consider the natural morphism $\Sh(\BV)_{K,E_\infty}\to \Spec E_\infty$.

%Let  $ E_1=\cO_{\Sh(\BV)_{K}}\lb \Sh(\BV)_{K}\rb$,  which is a finite extension of $E$. %Let  $r_1$ and $r_2$ be the numbers of  real embeddings and pairs of complex embeddings of $E_1$ respectively. 

\begin{thm}\label{diff1}   Let $E'_{t,\rf}(0,g,\phi)$ be as in   \eqref{2.90}.
For $g\in G(\BA)$, the generating series of  
$\ol{C^\infty}(\Sh(\BV)_{K,E_\infty}) $-valued
functions on $G(\BA)$
\begin{align*} \cD(g):=& \sum_{t\in F_{>0}}\lb E'_{t,\rf}(0,g,\phi) +E_t(0,g,\phi ) \log \Nm_{F/\BQ}t\rb
+  \lb \cG^{\nadm}(g,\phi)- \cG^{\mathrm{Kud}}(g,\phi) \rb     \\
-&
 \omega(g)\phi (0 ) \lb -\log \delta_\infty(g_\infty)+ [F:\BQ]\lb \log \pi-(\log\Gamma)'(n+1)\rb\rb
 \end{align*} 
 pointwise converges to an element in 
  $  \cA\lb G,\fw, \ol{C^\infty}(\Sh(\BV)_{K,E_\infty})\rb  .$
 %Replacing the  automorphic Green function  by the normalized admissible Green function, we have an equivalent statement. 
\end{thm}
\begin{rmk}
(1) It might be   interesting   to study whether  Theorem \ref{diff1}   still holds if we put Fr\'echet topology on $\ol{C^\infty}(\Sh(\BV)_{K,E_\infty})$, and if we require stronger convergence.

(2) Theorem \ref{diff1}   is a strengthened   analog of the main   result of Ehlen and Sankaran \cite{ES18}, which is for $F=\BQ$.

\end{rmk}
\begin{proof} 
We follow  \cite {KRY2} and  \cite {MZ}.
Let $C^\infty(\Sh(\BV)_{K,E_\infty})^\circ$ be the $L^2$-orthogonal complement of $\LC(\Sh(\BV)_{K,E_\infty})$ in $C^\infty(\Sh(\BV)_{K,E_\infty})$, endowed with $L^\infty$-topology.  Then $\ol {C^\infty}(\Sh(\BV)_{K,E_\infty})  $  is the topological direct sum of $C^\infty(\Sh(\BV)_{K,E_\infty})^\circ$ and $ \LC(\Sh(\BV)_{K,E_\infty}) /\BC \cO_{E_1}^\times$.

Let the generating series  $\cD^\circ$ be the   projection of  $\cD$ to $C^\infty(\Sh(\BV)_{K,E_\infty})^\circ$. By \cite[Lemma 3.8]{MZ} \footnote{It is  a priori  only at the level of generating series of functions, but will be   at the level of true $L^2$-functions after this paragraph.}   and the same argument  in the proof   \cite[Theorem 4.4.4] {KRY2}\footnote{The convergence in loc. cit. was not stated  explicitly. It could come from the general property of Laplacian spectral decomposition of a smooth function on a compact manifold, applied to the product of $S^1$ and a compact Shimura curve in loc. cit..},
for every $g\in G(\BA)$,    $\cD^\circ(g)$    converges in $L^\infty\lb\Sh(\BV)_{K,E_\infty}\rb $. Particularly, $\cD^\circ(g)$    converges in $L^2\lb\Sh(\BV)_{K,E_\infty}\rb$  since $\Sh(\BV)_{K,E_\infty}$ is compact.  
So we can consider  $\cD^\circ$ (identified as its limit)  as a function on $G(\BA)  \times  \Sh(\BV)_{K,E_\infty}$. Then  by the argument in the proof of \cite[Theorem 3.9]{MZ},   $\cD^\circ \in C^\infty\lb G(\BA)  \times \Sh(\BV)_{K,E_\infty}\rb$.
Also note that for every  $g\in G(\BA)$,    $\cD^\circ(g)\in C^\infty\lb   \Sh(\BV)_{K,E_\infty}\rb^\circ$.
Thus  $\cD^\circ$ is a smooth $C^\infty\lb   \Sh(\BV)_{K,E_\infty}\rb^\circ$-valued function on $G(\BA)$. See \cite[Theorem 40.1 and Corollary]{Tre}. 
Then the argument  in the proof of  either \cite[Theorem 4.4.4] {KRY2} or \cite[Theorem 3.9]{MZ} shows that $\cD^\circ\in   \cA\lb G,\fw, \ol{C^\infty}(\Sh(\BV)_{K,E_\infty})^\circ\rb  .$

 Consider the projection from $\ol{C^\infty}(\Sh(\BV)_{K,E_\infty}) $ to $ \LC(\Sh(\BV)_{K,E_\infty}) /\BC \cO_{E_1}^\times\cong \BC$, i.e., $$
f\mapsto \sum_{v\in \infty}\frac{1}{\deg (\ol{L}_{E_v})}\int_{\Sh(\BV)_{K,E_v} } f\cdot c_1(\ol{L}_{E_v})^n.
$$  
By Theorem \ref{LSthm}  and that $E_0(0,g,\phi)= \omega ( g)\phi (0)$
(see Lemma \ref{wr1} (1)),   
 the projection of 
$\cD$ (defined on each of its terms) pointwise converges to an element in 
  $   \cA\lb G,\fw \rb.$ Since $\cD$ is the sum of this projection and $\cD^\circ$, 
the theorem follows.\end{proof}

\subsection{Modularity problems}\label{Arithmetic modularity}
We  will raise two modularity problems for  admissible extensions of special divisors. Before that, we recall some notions
 and Kudla's modularity problem.
\subsubsection{Preliminaries}
 
A  (regular) integral model  of  $\Sh(\BV)_{K}$ over an integral domain $R$ with fraction field $E$ is  a (regular) Deligne-Mumford stack  proper flat over $\Spec R$    with a fixed isomorphism of its generic fiber to $ \Sh(\BV)_{K }.$ An isomorphism between integral models  is  an isomorphism over $\Spec R$  that  respects the  fixed isomorphisms  to $ \Sh(\BV)_{K }.$
 
Let  $\cX_K$ be a regular integral model  of  $\Sh(\BV)_{K}$ over $\Spec \cO_E$.
Let   $\wh \Ch^1_{\BC}(\cX_K) $ be  the Chow group of 
   arithmetic divisors with $\BC$-coefficients. See Definition \ref{CCH}.
   In particular, we have an isomorphism $$ \deg:  \wh \Ch^1_{\BC}\lb   \Spec \cO_{E}  \rb\cong \BC$$ by taking degrees (see Remark \ref{cmc}), and  an arithmetic  intersection
    pairing 
    \begin{align*}\wh\Ch^1_\BC(\cX_K)\times Z_1(\cX_K)_\BC \to \BC,\ 
(\wh z,Y) \mapsto \wh z\cdot Y. \end{align*}
 See Appendix \ref{arithmetic  intersection
    pairing}.
Here we recall that $Z_1(\cX_K)$ is  the    group of 1-cycles on $\cX_K$.

Let   $   \cL= \cL_K$ be  an     extension  of $L=L_K $ to $\cX_K$,
 which we allow to be a line bundle. 
Let   $  \ol\cL $ be    $\cL$ equipped with a hermitian metric. Let  $c_1(\ol \cL^\vee)\in \wh \Ch^{1}_{\BC}(\cX_K) $ be the first arithmetic Chern class of the dual of  $\ol \cL$. See  Example \ref{admextinf}.

\subsubsection{Kudla's  problem}
 \label{integral modeldef}
We consider the following  modularity problem of Kudla \cite{MR1886765,Kud02,Kud03}: find   
an  arithmetic divisor $\wh \cZ(x)$  on $\cX$  extending $Z(x)$,    explicitly and canonically, such that   
\begin{equation*}   \omega(g)\phi(0)c_1(\ol \cL^\vee)+\sum_{x\in K\bsl \BV^{ \infty}_{>0}}
\omega(g)\phi( x_\infty x)  \wh \cZ(x) , \end{equation*} where
$ x_\infty \in   \BV_\infty$ such that $q(x_\infty)={q(x)}\in F_{>0}$,
lies in   $  \cA_{\hol}(G,\fw) \otimes    \wh \Ch^{1}_{\BC}(\cX_K) $. 
The existence of such $\wh \cZ(x)$  is obvious, by  choosing a section of the natural surjection  $\wh\Ch^1_{\BC}(\cX_K)\to \Ch^1(\Sh(\BV)_K)_\BC$.  However, it is only    defined  at the level of   divisor classes, and not explicit.
%For example, S.~Zhang   defined     $\mathsf{L}$-liftings of divisor classes in  \cite[Corollary 2.5.7]{Zha20}. 
%The computation of  $\mathsf{L}$-liftings is rather  hard,  as it involves   Faltings heights of Shimura varieties. Indeed, this is one potential application of our work. See \ref{Faltings heights of Shimura varieties and Arithmetic Siegel-Weil formula}.

\subsubsection{Admissible extensions} \label{Admissible cycles} 

We consider  the above modularity problem  for   admissible extensions.  
  In particular, we assume Assumption \ref{A11asmp11} for $\cX_K$. (Also recall that $\cX_K$ is connected, as $\Sh(\BV)_K$ is.)
Assume that $  \cL$   is ample. Let    $ \wh \Ch^{1}_{\adm,\BC}(\cX_K)  \subset \wh\Ch^1_\BC(\cX_K)$ be the    subgroup of     arithmetic
divisors   that are admissible with respect to   $\ol \cL$. See Definition \ref{CCH}.
%We use  $ \wh \Ch^{1}_{\adm,\BC}(\cxX_K) $ to denote $ \wh \Ch^{1}_{\adm,\BC}(\cX_K) $ for simplicity. 
By    Lemma \ref{fdime},  the natural  map   \begin{equation}  \label{lebo Lebo Lebo label hey hey hey hey} \wh \Ch^{1}_{\adm,\BC}(\cX_K) \to  \Ch^1(\cX_{K,E })_{\BC}
 \end{equation}
 is  surjective, and  the kernel  is the  image of the pullback 
 \begin{equation}  \label{lebo Lebo Lebo label hey hey hey hey hey} \wh \Ch^1_\BC\lb \Spec \cO_{E} \rb\cong \wh \Ch^1_\BC\lb \Spec \cO_{\cX_K}(\cX_K)\rb_\BC\incl \wh\Ch^1_{\adm,\BC}(\cX_K). 
 \end{equation}
  In particular  $ \wh \Ch^{1}_{\adm,\BC}(\cX_K) $  is finite  dimensional.

\begin{defn}  \label{fdimeq}
 Define    an embedding
$
\BC\incl \wh\Ch^1_{\adm,\BC}(\cX_K)
$  as  the composition of the inverse  $\deg^{-1}:  \BC\cong \wh \Ch^1_{\BC}\lb   \Spec \cO_{E}  \rb$  of taking degree  
and \eqref{lebo Lebo Lebo label hey hey hey hey hey}.  
Below, by a complex number in $\wh\Ch^1_{\adm,\BC}(\cX_K)$, we   understand  it as  the image by this embedding. 

\end{defn} 

\begin{rmk}\label{tochow}
The   intersection $
  \wh\Ch^1_{\adm,\BC}(\cX_K)\cap \ol{C^\infty}(\Sh(\BV)_{K,E_\infty})
$ in $  \wh \Ch_\BC(\cX_K)
$ is $\BC$, where  $\BC$ is in $\ol{C^\infty}(\Sh(\BV)_{K,E_\infty})$ via \eqref{BC1}.  \end{rmk}

Let   $\whz^1_{\adm,\BC}(\cX_K)$  
the group of admissible arithmetic divisors. See                                 
Definition \ref{admext}. For a divisor $Z$ on $\Sh(\BV)_{K}\cong \cX_{K,E}$,   let
\begin{equation}
Z^\nadm\in \whz^1_{\adm,\BC}(\cX_K)
\label{znadm}  
 \end{equation}
 be the  normalized admissible extension   of $Z$
with respect to  $\ol \cL$. See                                 
Definition \ref {admextgreen}. 
Let  $[  Z(x)^\nadm] $ be  its class in  $\wh\Ch^1_{\adm,\BC}(\cX_K)$.
Then a preiamge of $[  Z(x)] $ via \eqref{lebo Lebo Lebo label hey hey hey hey} is of the form 
$[  Z(x)^\nadm] +e(x)$ for some $e(x)\in \BC$. 
%We consider  two arithmetic analogs of $z(\omega(g)\phi)$   in the current general setting (and specify them later  in this subsection \ref{Arithmetic modularity}). %The notations    here  in \ref{General problem and strategy} will not   be used at other part of the paper. 
Note that  $c_1(\ol \cL^\vee)\in \wh \Ch^{1}_{\adm,\BC}(\cX_K) $. See  Example \ref{admextinf}.

   \begin{problem}\label{theq}  Find  $a\in \BC$ and $e=\{e(x) \in \BC\}_{x\in K\bsl \BV^{ \infty}_{>0}}$ explicitly  such that for every $\phi\in \ol\cS\lb \BV  \rb^{K}$,    the generating series
\begin{equation}z(\omega(g)\phi)_{e,a} ^\adm:=\omega(g)\phi(0)\lb c_1(\ol \cL^\vee)+a\rb+\sum_{x\in K\bsl \BV^{ \infty}_{>0}}\omega(g)\phi( x_\infty x) \lb[  Z(x)^\nadm]+e(x)\rb, \label{zpadm-}   \end{equation}
where
$ x_\infty \in   \BV_\infty$ such that $q(x_\infty)={q(x)}\in F_{>0}$,
lies in   $  \cA_{\hol}(G,\fw) \otimes    \wh \Ch^{1}_{\adm,\BC}(\cX_K) $.
Moreover, $e(x)$ should to be naturally decomposed into  a sum of ``local components" such that  %When  $\cX_{\cO_{E_v}}$ is a ``good  model" (not necessary smooth as we will see), the $v$-component should be 0. In particular, 
  the $v$-component should be 0 at all but finitely many places and $\infty$-component should be  independent of the choice of the regular integral model.
%If so, can         \end{ques}
    \end{problem}
In other words, we want a modular generating series by modifying each $[Z(x)^\nadm]$ by an explicit  constant once for all $\phi$.
A weaker statement is to allow the modification to depend on $g,\phi$.

\begin{problem}\label{theq1} Find  $a\in \BC$  and  a smooth function $e_t(g,\phi)$ on $G(\BA_F)$ for  $\phi\in \ol\cS\lb \BV  \rb^{K}$
 explicitly   
such that    the generating series
\begin{equation}\label{zpadm?}\omega(g)\phi(0)\lb c_1(\ol \cL^\vee)+a\rb + 
\sum_{t\in F_{>0}}\lb [Z_t(  \omega(g)\phi )^\nadm]+ e_t(g,\phi)\rb
\end{equation}
lies in   $  \cA_{\hol}(G,\fw) \otimes    \wh \Ch^{1}_{\adm,\BC}(\cX_K) $.
Moreover, $e_t(g,\phi)$ should to be naturally decomposed into  a sum of  ``local components" such that  %When  $\cX_{\cO_{E_v}}$ is a ``good  model" (not necessary smooth as we will see), the $v$-component should be 0. In particular, 
  the $v$-component should be 0 at all but finitely many places and $\infty$-component should be  independent of the choice of the regular integral model.

     \end{problem}
    
 \begin{rmk} 
   In \eqref{zpadm-} and \eqref{zpadm?}, one may    replace 
     $\omega(g)\phi(0)\lb c_1(\ol \cL^\vee)+a\rb$ by   $\omega(g)\phi(0)c_1(\ol \cL^\vee)$     by adding a suitable multiple of the degree of $z\lb\omega(g)\phi\rb$, which is    modular by \Cref {Modu} and can be made explicit by \eqref{Proposition 4.1.5}. However, we keep the freedom to have ``$a$" to get the decomposition of $e_t(g,\phi)$  as we will see in Section \ref{Arithmetic modularity1}. 
     
      \end{rmk}

% By  \Cref{yespull}   and  \Cref{pull}  (the compatibility of the formation of  normalized admissible extension  under flat pullback), we in fact expect  that  $e_t(g,\phi)$ does not depend on $K$.       

%The following numerical criterion will be useful to check the  modularity.\begin{lem}\label{key1}For every $e=\{e(x) \in \BC\}_{x\in K\bsl \BV^{ \infty}_{>0}}$, the   statement $$z(\omega(g)\phi)_{e,a} ^\adm\in  \cA_{\hol}(G,\fw) \otimes    \wh \Ch^{1}_{\adm,\BC}(\cX_K) $$ is equivalent to   that  for   some (equivalently, every) $\cP\in  Z_1(\cX_K )_\BC$ with $\deg \cP_E\neq 0$, we have $$  z\lb\omega(g)\phi\rb_{e,a}^\adm\cdot  \cP \in  \cA_{\hol}(G,\fw) .$$The same is true replacing  $z(\omega(g)\phi) ^\adm$ by the   generating series \eqref{zpadm?}.\end{lem}\begin{proof}  Up to scaling, we may assume that the generic fiber  $\cP_E$ has degree 1.Consider the linear form$ ?\mapsto \wh z\cdot \cP$ on $ \wh\Ch^1_{\adm,\BC}(\cX_K)$. %Here the last term is understood in $ \wh \Ch^{1}_{\adm,\BC}(\cX_K)$ via Definition \ref{fdimeq}.The lemma follows from   Lemma \ref{modcri}, \end{proof}

By   Theorem \ref{Modu} and Lemma \ref{fdime} (1),  
we immediately have the following lemma. 
\begin{lem}\label{fdimeadd}
For $\cP\in  Z_1(\cX_K )_\BC$  such that $\deg \cP_E=0$,  
\begin{equation}\label{fdimeaddeq} 
 z\lb\omega(g)\phi\rb_{e,a}^\adm\cdot  \cP \in  \cA_{\hol}(G,\fw)
\end{equation}

\end{lem}

\begin{rmk}  (1) This  lemma is called almost modularity in \cite[Theorem 4.3]{MZ}. 

(2) In the case that $K=K_{\Lambda}$ with a different $\Lambda$, for Kudla-Rapoport arithmetic divisors and $\cP\in  Z_1(\cX_K )_\BC$ with  $\cP_E=0$, the analogous statement  is   proved  in  \cite[Theorem 14.6]{ZZY}.

  \end{rmk}
 
 When   $\deg \cP_E\neq 0$, the truth of  \eqref{fdimeaddeq}     is in fact equivalent to   the  modularity of $ z\lb\omega(g)\phi\rb_{e,a}^\adm$ by  Lemma \ref{modcri} below.   We will not use exactly this  ``numerical criterion", but use   Lemma \ref{modcri}  in a more sophisticated way  to prove our   modularity results
  in \ref{Proof of Theorem {main}}.

 \begin{lem}\label{modcri}
Let $X$ be a $\BC$-vector space,  $x\in X$ nonzero,  and $l$ a linear form on $X$  such that 
$l(x)=1$.
Let $f$ be a formal generating series of functions on $G(\BA)$ valued in $X$ 
and $\ol f$ the  corresponding formal generating series of functions on $G(\BA)$ valued in $X/\BC x$. 
Assume that $\ol f\in  \cA(G,\fw) \otimes X/\BC x$. Then  
$ f\in \cA(G,\fw) \otimes X $ if and only if $l\circ f\in \cA(G,\fw). $

\end{lem}
\begin{proof} Define  a  section   $\fs$ of the   projection $X\to X/\BC x$ by 
$\fs:z\mapsto s(z)- l\lb s(z)\rb \cdot x,$
where $s$ is any section (and $\fs$ is independent of the choice of $s$). 
Then 
$
f=
\fs(\ol f)+ \lb l\circ f \rb \cdot x
$  and the lemma follows. 
\end{proof}

\subsection{Conjecture and   theorems}\label{Arithmetic modularity1}
  First,  we define specific integral models. Then  we
define explicit admissible extensions. Then we will propose a precise conjecture. Finally, we  state our modularity theorems. 

\subsubsection{Integral models}\label{RSZ-arithmetic-diagonal-cycles}
 Let us  at first  set up some notations and assumptions that are needed to construct our integral models. 
For a finite place $v$ of $F$ and an $\cO_{E_v}$-lattice ${\Lambda_v}$ of $\BV(E_v)$, the dual lattice  is defined as  $\Lambda_v^\vee=\{x\in \BV(E_v):\pair{x,{\Lambda_v}}\subset\cO_{E_v}\}$. % (It also works even if   $v$ is split in $E$.)
Then ${\Lambda_v}$ is called 
\begin{itemize}\item self-dual if     ${\Lambda_v}=\Lambda_v^\vee$;
% \item almost self-dual if     ${\Lambda_v}\subset \Lambda_v^\vee\subset \varpi_E^{-1}{\Lambda_v}$ and the first inclusion is of colength 1;
\item   $\varpi_{E_v}$-modular if $ \Lambda_v^\vee=\varpi_{E_v}^{-1}{\Lambda_v}$;
\item 
almost $\varpi_{E_v}$-modular if $ \Lambda_v^\vee\subset \varpi_E^{-1}{\Lambda_v}$ and the inclusion is of colength 1.  
\end{itemize}

Assume the following assumption in   the rest of the paper.

 \begin{asmp} \label{asmp1}  (1)  At least one of the following three conditions hold: \begin{itemize}\item[(1.a)]   every finite place   of $E$  is at most tamely ramified over $ \BQ$;
 \item[(1.b)] $E/\BQ$ is Galois;
 \item[(1.c)]  $E $ is the composition of $F$ with some imaginary quadratic field.
 \end{itemize}

 (2)
     Every finite place $v$ of $F$  ramified over $ \BQ$ or of  residue characteristic 2 is   unramified   in $E$.

 (3)       At every  finite place  $v$ of $F$ inert    in $E$,   there is a self-dual lattice  $\Lambda_v$  in $\BV(E_v)$.

(4)
At every finite place $v$ of $F$ ramified in $E$, there is a  $\pi_v$-modular (resp. almost $\pi_v$-modular) lattice $\Lambda_v$  in $\BV(E_v)$ if $n$ is odd (resp. $n$ is even).

\end{asmp}

    %    \begin{rmk} \label{infly}   

%
%(1) Under Assumption \ref{asmp1} (2),  we only need the tameness of $F$ rather than $E$ in Assumption \ref{asmp1} (1.a).

%(2)     It is desirable to remove  Assumption \ref{asmp1} (2). See for example \cite[Remark 2.69]{LL2}. In fact, in the analytic side, we do not need  Assumption \ref{asmp1} (2). See \ref{ecompositio}. 

  We will classify $\BV$ containing such $\Lambda$ in   \Cref
{neven}  below. 
%\end{rmk}

At every    place  $v$  split   in $E$, let   $\Lambda_v$    be a self-dual lattice in $\BV(E_v)$. 
Let   $$\Lambda =\prod_{v } \Lambda_v\subset \BV^\infty.$$
Let $K_\Lambda\subset U(\BV^\infty)$ be the stabilizer of $\Lambda$.  

\begin{defn} \label{asmp2} 
Let ${\wKL}$ be the directed poset of compact open subgroups  $K\subset K_\Lambda$, under the inclusion relation, such that for a finite place $v$ of $F$,
\begin{itemize}
\item[(1)] if $v $ is non-split in $E$, then $K_v=K_{\Lambda,v}$;
\item[(2)]  if $v $ is  split in $E$, then $K_v$ is a principal congruence subgroup of $K_{\Lambda,v}$.
\end{itemize}

\end{defn}

Under Assumption \ref{asmp1} (1.a),
for $K\in {\wKL}$ and a finite place $v$ of $E$, we have  an    integral model $\mathscr{S}_v$ of  $\Sh(\BV)_{K}$ over $\Spec \cO_{E,(v)}$, as constructed   
in \eqref{defn-mathscrS-G},  \ref{defn-mathscrS-G2} and \ref{AT-parahoric-level-section}.   (Note that we can always choose a CM type satisfying the matching condition \eqref{matching-condition}.)
 We remind the reader of the  differences between the notations on the fields here and in   Appendix B.
\begin{lem}\label{regmod}(1) The  integral model $\mathscr{S}_v$ is regular.
 
(2) If  $K= {K_{\Lambda}}$, there is an ample $\BQ$-line bundle $\mathscr{P}_v$ on  $\mathscr{S}_v$ extending  $L_{K_\Lambda}$.  
\end{lem} 
\begin{proof} (1) We apply  \ref{Corollary 2.30}. 
We use $M$ to denote the  reflex field  in \ref{Corollary 2.30}, which is defined in \eqref{reflex-field-E} (and denoted by $E$ there), to avoid confusion. We use $N$ to denote the
field extension of the  reflex field in \ref{Corollary 2.30} (denoted by $L$ there).
By the finite \'etaleness of  the moduli space of 
relative dimenison 0 as in  \ref{mathcal-M0-defn}, we may choose  $N/M$  
to be unramified at $\nu$ \cite[Tag 04GL]{stacks-project}. Then 
by  \cite{RSZ-arithmetic-diagonal-cycles} and  \ref{Corollary 2.30},  $\mathscr{S}_{v,\cO_{N,(\nu)}}$ is regular for a place $\nu$ of $N$ over $v$. By the
descent of regularity under faithfully flat morphism \cite[Tag 033D]{stacks-project}, the lemma follows. 

(2)  
Let $N/E$  be a  finite Galois extension 
such that  $v$   is   unramified  in $N$  and every   connected component  of $\mathscr{S}_{v,\cO_{N,(v)}}$ is  geometrically connected.   By its construction \eqref{defn-mathscrS-G}, every connected   component  is a quotient of 
a   connected component of an integral model of Hodge type over  $\cO_{N,(v)}$  by a finite group action.
The  integral model of Hodge type is a closed subscheme of  the integral Siegel moduli space \cite{xu-normalization,PEL-embedding}, on which we have  a well-known ample  Hodge line bundle.  
The restriction is an ample  line bundle on  each  geometrically connected component of the 
  integral model of Hodge type. 
  Taking  norm along the quotient map by the finite group action, we get an
  an ample  line bundle  on every    component  of $\mathscr{S}_{v,\cO_{N,(v)}}$. See \cite[Section 6, Theorem  7  and Example B]{BLR}  and \cite{Vis1} for the stack case.
  Then taking norm map along the  quotient map by the $\Gal(N/E)$, we have an ample   line bundle $\mathscr{P}'$ on $\mathscr{S}_{v}$. 
  %The tensor product of the $\Gal(N/E)$-conjugates of $\mathscr{P}'$ descends to an ample line bundle 
%on  $\mathscr{S}_v$ by \cite[Section 6, Theorem 4]{BLR} (note that the tensor product is not just 
%$\Gal(N/E)$-invariant, but also satisfies the cocycle condition automatically),
%see \cite{Vis1} for the stack case.  
 Dividing $\mathscr{P}'$ by the product of the order of the finite group and $[N:E]$, we get the desired   ample  line bundle  on $\mathscr{S}_{v}$. 
  \end{proof}
 
 Under Assumption \ref{asmp1} (1.b) or (1.c), 
 the  reflex field    \eqref{reflex-field-E}, in our notations, is  $E$. For $K\in {\wKL}$ and a finite place $v$ of $E$,  both sides of 
the morphism  \eqref{map-mathcalM-Gtilde-to-mathcal-M0}
(and the   generalizations  as in \ref{Drinfeld level structure} or
 \ref{AT-parahoric-level-section}) are equipped with a natural
action of the finite group $Z^{\BQ}\lb \BA_\BQ^{\{p\}\cup \infty} \rb/Z^{\BQ}(\BZ_{(p)})K_{Z^{\BQ}}^p $   compatible with the morphism. See 
   \cite[Definition 4.2.2]{Let}. Moreover, taking quotients  by this finite group, we  get a new morphism whose target is isomorphic to $\Spec \cO_{E,(v)}$. Then the source of  the new morphism  is an    integral model of  $\Sh(\BV)_{K}$ over $\Spec \cO_{E,(v)}$, whose regularity is assured by 
   \cite{RSZ-arithmetic-diagonal-cycles} and faithfully flat descent of regularity. We also denote this    integral model by  $\mathscr{S}_v$.
  Moreover, if Assumption \ref{asmp1} (1.a) also holds, then the construction here and in  \Cref{regmod} coincides, by \ref{Corollary 2.30}.

We want to glue these models to obtain a regular integral model  of  $\Sh(\BV)_{K}$ over $\Spec \cO_{E}$.  
\begin{thm} \label{hypothm} 
Assume Assumption \ref{asmp1}.

(1) For $K\in {\wKL}$, there is a regular integral model  $\cX_K$ of  $\Sh(\BV)_{K}$ over $\Spec \cO_{E}$ such that $\cX_{K, \cO_{E,(v)}}\cong \mathscr{S}_v$  as integral models. See   \ref{integral modeldef}. Moreover, Assumption \ref{A11asmp11} holds for  $\cX_K$.

(2) For $K\subset K'$ in ${\wKL}$, there is   a  unique   finite flat morphism  $\pi _{K,K' }:\cX_{K }\to \cX_{K' }$  extending
the natural morphism $\Sh(\BV)_{K }\to \Sh(\BV)_{K' }$.

(3) Regard $\cX_{K }$ as an $ \cX_{K_{\Lambda} }$-scheme and $\Sh(\BV)_{K }$ as an $ \Sh(\BV)_{K_{\Lambda} }$-scheme  via    
$\pi _{K,K' }$. 
There is   a unique  action of $K_\Lambda/K$ (note that 
$K$ is normal in $K_{\Lambda}$) on  the $ \cX_{K_{\Lambda} }$-scheme $\cX_{K }$    extending  the standard action of $K_\Lambda/K$ on the $ \Sh(\BV)_{K_{\Lambda} }$-scheme  $\Sh(\BV)_{K }$  by ``right translation".

\end{thm}
\begin{proof} The construction  of $\cX_K$  is as follow. Continue to use the notation $M$ in the proof of Lemma \ref{regmod}.   
First,  consider the analog of the morphism  \eqref{map-mathcalM-Gtilde-to-mathcal-M0} over $\cO_M[R^{-1}]$ where $R$ is a finite set of finite places of $ \BQ$ such that
$K_v=K_{\Lambda,v}$ for $v$ not over $R$ (in particular, we do not need  the   generalizations  as in \ref{Drinfeld level structure} or
 \ref{AT-parahoric-level-section}).  See for example \cite[5.1]{RSZ-arithmetic-diagonal-cycles} with extra level structures over $R$ (similar to %\eqref{leveloutp0} and
\eqref{leveloutp}). Denote this morphism by $\cM\to \cM_0$.
We use Galois descent to construct a model of $\Sh(\BV)_K$  outside finitely many finite places.  
 Let $N/M$ be a finite   extension, Galois over $E$, such that the base change of  
$\cM_0$ to  $\Spec \cO_{N}[R^{-1}]$ is a finite disjoint   union of  $\Spec \cO_{N}[R^{-1}]$.  Let $S\supset R$ be a finite set of finite places of $\BQ$ such that $\Spec \cO_{N}[S^{-1}]\to \Spec \cO_{E}[S^{-1}] $ is unramified.
Then the fiber of  $\cM_{\cO_{N}[S^{-1}]}\to \cM_{0,\cO_{N}[S^{-1}]}$ over  a chosen  $\Spec \cO_{N}[S^{-1}]$ 
is a  regular Deligne-Mumford stack $\cM^S$  proper over $ \Spec \cO_{N}[S^{-1}]$  with generic fiber  $\Sh(\BV)_{K,N}$. 
By Zariski's main theorem (for stacks which easily follows from the scheme version),  after possibly enlarging $S$, we may assume that   the action of the finite group $\Gal(N/E)$ on  $\Sh(\BV)_{K,N}$ extends to an action on $\cM^S$. 
By  \cite[Section 6, Example]{BLR} (and \cite{Vis1} for the stack case), after  possibly enlarging $S$,  
$\cM^S$  descends to $\Spec \cO_{E}[S^{-1}]$.
 Let $\cY ^S$ be the resulted  Deligne-Mumford stack.  By construction and  \ref{Corollary 2.30}, we have $\cY ^S_{\cO_{E,(v)}}\cong \mathscr{S}_v$ for $v\not\in S$. 
  Now, let $\cX_K$ be the glueing of $\cY^S$ and  $\mathscr{S}_v$ with $v\in S$.   Then $\cX_{K, \cO_{E,(v)}}\cong \mathscr{S}_v$ for every finite place $v$.
  
We check Assumption \ref{A11asmp11}. For $K'\subset K$ small enough with $K'^S=K^S$,   the resulted  Deligne-Mumford stack $\cY'^S$   is representable by \cite[5.2]{RSZ-arithmetic-diagonal-cycles}. Note that $\Spec \cO_{N}[S^{-1}]\to \Spec \cO_{E}[S^{-1}] $ is unramified (so finite \'etale).
Then  Assumption \ref{A11asmp11} (1) holds (with the current $S$) by construction. Similarly, assumption \ref{A11asmp11} (2) holds.

The uniqueness in (2) and (3)  follows from the  separatedness of the  models. The existence 
follows   from the construction and corresponding properties of the PEL type integral models in \cite[Theorem 5.4]{RSZ-arithmetic-diagonal-cycles}. We omit the details.
\end{proof}

 \begin{rmk} \label{asmp2rmk}  If we drop  condition (2) in \Cref{asmp2} on $K$, % (i.e., not requiring $K$ to be a principal congruence subgroup),
\Cref {hypothm}   (1)(2) still holds. Indeed, to construct $\cX_K$,  choose $K_1\in \wt K_\Lambda$ such that $K_1\subset K$. Let $\cX_K$ be the quotient of $\cX_{K_1}$ by $K/K_1$ where is the action is \Cref {hypothm} (3). 

        \end{rmk}
     By \cite[Theorem 5.2]{RSZ-arithmetic-diagonal-cycles} and our construction, we deduce the following lemma.
                  \begin{lem}\label{expectlem}    
If $K_v=K_{\Lambda,v}$, then
$\cX_{K,\cO_{E}, (v)}$ is smooth over $\Spec \cO_{{E}, (v)}$.
                        
                    \end{lem} 
                          
\begin{rmk}   We may relax  Assumption \ref{asmp1} (3) by allowing   $\Lambda_v$ to be almost self-dual. See \ref{AT-parahoric-level-section}. Then   Lemma \ref{regmod}  still holds. 
However, Lemma \ref{expectlem} does not hold any more.

               \end{rmk}
               
                              \begin{cor}  \label{ampleL}  A  $\BQ$-line bundle (in particular a line bundle)
 $\cL_{K_{\Lambda}}$ on  $\cX_{K_{\Lambda}}$ extending  $L_{K_{\Lambda}}$ is  ample.  Moreover, it is unique as a $\BQ$-line bundle.
                    \end{cor} 
                       \begin{proof}       
         By Lemma \ref{expectlem},  the corresponding  divisors of   two different extensions differ by   a  $\BQ$-linear combination   of 
  special fibers of $\cX_{K_{\Lambda}}$, which is  0 in $\Ch^1\lb \cX_{K_{\Lambda}}\rb_\BQ$.  The uniqueness follows.
 Now   the  ampleness follows from Lemma \ref{regmod} (2).       
\end{proof}
%The construction in (1)   for $K=K_\Lambda$ carries over and gives a  line bundle  $\mathscr{P}^S$ over  $\cY^S  $ extending  $L_{K_\Lambda}$ (after possibly enlarging $S$). We use \qcl{qcl: Appendix \ref{SectionB} 2.11} 

% For $K$ small enough,  $\cX_{K  }'$ is  represented by a regular scheme.
 
Below, by a line bundle, we mean a $\BQ$-line bundle.
 
\subsubsection{Admissible extensions} We want to define admissible extensions that are compatible as the level changes.
So we consider a system of integral models.
\begin{defn}\label{Definition 4.4.9}
(1) Let $\wt \cX$ be the system $\{\cX_K\}_{K\in {\wKL}}$ of integral models with transition morphisms $\pi_{_K,_{K'}}$ as in Theorem \ref{hypothm} (1)(2). 

(2) For $h\in K_\Lambda$, define the 
``right translation by $h$"  automorphism  on $\cX_K$ to be the action of $h$ as  in Theorem \ref{hypothm} (3).

(3)   Fix  an  arbitrarily ($\BQ$-)line bundle $\cL_{K_{\Lambda}}$ on $\cX_{K_\Lambda}$ extending $L_{K_{\Lambda}}.$ Let  $ \cL_K=\pi_{_K,{_K}_{\Lambda}}^* \cL_{K_{\Lambda}}$.
\end{defn}
We use $\cL$ to denote the compatible-under-pullback system $\{\cL_K\}_{K\in  {\wKL}}$ of  ample  line bundles on $\wt \cX=\{\cX_K\}_{K\in {\wKL}}$.  The ampleness follows from the ampleness of $\cL_{K_{\Lambda}}$ and finiteness of $\pi_{_K,{_K}_{\Lambda}}$.
While using $\cL$ for $\cX_K$, we mean  $\cL_K$  so that it  has the same meaning as before (we previously used $\cL$  as the abbreviation of $\cL_K$). 
Let    $\ol\cL_K$ be the corresponding  hermitian line bundle  with the hermitian metric as in \ref{Complex uniformization}, and define $\ol \cL$ accordingly.
  
  By    \Cref{pull}, that is, the compatibility of the formation of   admissible  (Chow) cycles  under flat pullback, 
we can define the direct limit group of admissible arithmetic divisors along  the directed poset $\wKL$
\begin{equation}\label{wtz}\wh Z^{1}_{\adm,\BC}(\wt\cX)=\vil_{K\in {\wKL}}\wh Z^{1}_{\adm,\BC}(\cX_K) ,\end{equation} 
\begin{equation}\label{wtc}\wh\Ch^{1}_{\adm,\BC}(\wt\cX)=\vil_{K\in {\wKL}}\wh\Ch^{1}_{\adm,\BC}(\cX_K) \end{equation} 
where   the transition maps are   $\pi_{_K,_{K'}}^*$'s. 
By the projection formula (which gives that the composition of pullback by pushforward is the multiplication by the  degree of the finite flat morphism), 
the natural maps $$
\wh Z^{1}_{\adm,\BC}(\cX_K)\to  \wh Z^{1}_{\adm,\BC}(\wt\cX),\  \Ch^{1}_{\adm,\BC}(\cX_K)\to \Ch^{1}_{\adm,\BC}(\wt\cX)  $$   are  injective and we understand the formers as subspaces of the latters respectively.
The embeddings $ \BC\incl \wh\Ch^1_{\adm,\BC}(\cX_K)$ (Definition \ref{fdimeq})  are then the same embedding \begin{equation}
\BC\incl \wh\Ch^{1}_{\adm,\BC}(\wt\cX)\label{line}.
\end{equation} 
%such that the natural map $\wh\Ch^{1}_{\adm,\BC}(\cX_K)\to \Ch^{1}_{\adm,\BC}(\wt\cX) $induces identity on $\BC$. 

 \begin{rmk} \label{asmp2rmk1} 
  If we drop  condition (2) in \Cref{asmp2} % (i.e., not requiring $K$ to be a principal congruence subgroup) 
  as in 
\Cref{asmp2rmk1},  \eqref{wtz}-\eqref{line} do not change.        \end{rmk}

\begin{defn} \label{asmp22} 
Let  $\ol\cS\lb \BV  \rb ^{{\wKL}}\subset  \ol\cS\lb \BV  \rb $ consist of functions invariant by some $K\in  {{\wKL}}$.
\end{defn}

For  $t\in F_{>0}$,  $\phi\in \ol\cS\lb \BV  \rb ^{{\wKL}}$ and $K\in{\wKL}$ stabilizing $\phi$,      
consider the normalized admissible extension  $  Z_t(  \omega(g)\phi )_K^\nadm $  of $  Z_t(  \omega(g)\phi )_K$ (as in \eqref{znadm}).
   By Lemma \ref{yespull} and \Cref{pull}, 
\begin{equation}\label{lineZt}       Z_t(  \omega(g)\phi )^\nadm:=Z_t(  \omega(g)\phi )_K^\nadm\in \wh Z^{1}_{\adm,\BC}( \cX_K)\subset       \wh Z^{1}_{\adm,\BC}(\wt\cX)
\end{equation}
is  independent of $K$. 
And by definition,  
    $$c_1(\ol\cL^\vee):=c_1(\ol\cL_K^\vee)\in \Ch^{1}_{\adm,\BC}(\cX_K)\subset  \Ch^{1}_{\adm,\BC}(\wt\cX)$$ is independent of    $K$.

For $h\in K_\Lambda$,       the     ``right translation by $h$"  in  $\Aut\lb \cX_{K}/\cX_{K_{\Lambda} }\rb$ 
fixes $\cL_K$ by definition. Thus it
  induces an automorphism  on $\wh Z^{1}_{\adm,\BC}(\wt\cX)$ and  $ \wh\Ch^{1}_{\adm,\BC}(\wt\cX) $.
 Since it sends  $Z_t(g ,\phi )  $ to 
$Z_t(g ,\omega(h)\phi )  $, by  \Cref{pull}, 
we have the following lemma. 
    \begin{lem} \label{rth2} The     ``right translation by $h$"  automorphism
    is  the identity map on the image of $ \BC\incl \wh\Ch^{1}_{\adm,\BC}(\wt\cX)$,  fixes $c_1(\ol\cL^\vee) $, and   
    sends $Z_t(g ,\phi ) ^\nadm$ to 
$Z_t(g ,\omega(h)\phi ) ^\nadm$.

\end{lem}

% \begin{rmk} It is routine to extend the action from $K_{\Lambda,v}$ to $U(\BV(E_v))$ if $v$ is split in $E$. We will not need it.    \end{rmk}

      \subsubsection{Conjecture}\label{Conjecture}
We define generating series of admissible arithmetic divisors in two ways, essentially by choosing different Green functions. (In fact, there will be a third one in \ref{Main result}.)
The first (resp. second) definition is made toward Problem \ref{theq1} (resp. Problem \ref{theq}).

%qcl Let $\ol \cL$ be the hermitian line bundle with the metric as in \ref{Complex uniformization}. 

   Recall    the holomorphic part  $E'_{t,\rf}(0,g,\phi)$   of the derivative  of  $E_t(s,g,\phi)$ at $s=0$. See   \eqref{2.90}.
 \begin{defn} \label{smeway} For $\phi\in \ol\cS\lb \BV  \rb^{{\wKL}}$,  $t\in F_{>0}$ and $g\in G(\BA_F )$,
let $$z_t(g, \phi)_\fe^\nadm=[Z_t(  \omega(g)\phi )^\nadm]+\lb E'_{t,\rf}(0,g,\phi) +E_t(0,g,\phi ) \log \Nm_{F/\BQ}t\rb\in\wh\Ch^1_{\adm,\BC}(\wt\cX) .$$
For $a\in \BC$ and $g\in G(\BA_F)$, let   \begin{equation*} z(g, \phi)_{\fe,a}^{\nadm}=\omega(g )\phi(0)\lb c_1(\ol \cL^\vee)+a\rb+\sum_{t\in {F>0}}z_t(g ,\phi )_\fe^\nadm  . 
\end{equation*}

\end{defn}
%Here  $$\lb E'_{t,\rf}(0,g,\phi) +E_t(0,g,\phi ) \log \Nm_{F/\BQ}t\rb,a\in\BC$$ are understood in $ \wh \Ch^{1}_{\adm,\BC}(\wt\cX)$ via   \eqref{line}. 

\begin{rmk}\label{compwith}
(1)
We  rewrite the definition of  
$z_t(g, \phi)_{\fe}^{\nadm}$ as follows,  which  is closer to the  notation  ``$\lb [ Z_t^\nadm]+\fe_t\rb q^t $" in  \eqref{series2}. 
For $t\in F_{>0}$, we have  the $t$-th Fourier coefficient    $ \frac{E_t(0,g,\phi)} {W^{\fw}_{\infty, t}(1)},$ $g\in G(\BA_F^\infty),$ of $E(0,g,\phi)$. See \ref{Holomorphic  automorphic  forms}.
% and let $\fe_t(  \phi^\infty)=\fe_t(1, \phi).$
We similarly and  formally consider 
$ \frac{E'_{t,\rf}(0,g,\phi) } {W^{\fw}_{\infty, t}(1)},\ g\in G(\BA_F^\infty)$ as the ``$t$-th Fourier coefficient" of $E'_{t,\rf}(0,g,\phi)$.
%Replacing $g$ by $gg_\infty$ and $1$ by $g_\infty$ for any  $g_\infty\in G(F_\infty)$, we still have the same  functions on $G(\BA_F^\infty)$.
Let \begin{equation}\label{fetg} \fe_t(g , \phi^\infty)= \frac{E'_t(0,g,\phi)} {W^{\fw}_{\infty, t}(1)}+\frac{E_{t,\rf}(0,g,\phi) } {W^{\fw}_{\infty, t}(1)}\log \Nm_{F/\BQ}t,\ g\in G(\BA_F^\infty).\end{equation}
Then  $ \fe_t(1 ,\phi^\infty)$ is the constant $\fe_t $ in \eqref{intro1.3}.
And   we can rewrite 
$$z_t(g, \phi)_{\fe}^{\nadm}=\lb[ Z_t\lb  \omega(g^\infty)\phi^\infty\rb^\nadm]+\fe_t(g^\infty,\phi^\infty) \rb W^{\fw}_{\infty, t}(g_\infty) ,$$

(2)   
We also expect that     the (proposed) infinite component of $e_t(g,\phi)$  in Problem \ref{theq1} is $ E'_{t,\rf}(0,g,\phi) +E_t(0,g,\phi ) \log \Nm_{F/\BQ}t$ for any unitary Shimura variety. 
In particular,  in the situation of this section,  the $v$-component is 0 for every finite place $v$. 
It is related to  Lemma \ref{expectlem}.
 (The smoothness is ``exotic" at places where $E/F$ is ramified.)

\end{rmk}

%   and let $\fe'_{t,\rf}(  \phi)=\fe'_{t,\rf}(1, \phi).$

%    \begin{rmk}By   Lemma \ref{lSWlem}, $\fe_t(g, \phi)=\fe_t(\omega(g )\phi)$.  qcl wrong direction of action By \eqref{translaw}, the same relation  is not true for $\fe'_{t,\rf}$.
%   \end{rmk} 
By 
 \eqref{E0tk},
\eqref{E'0tk}  and  \Cref{rth2}, we have the following  lemma. \begin{lem} \label{equiva} For $h\in K_{\Lambda}$,       the     ``right translation by $h$" automorphism on $\cX_{K}$  
sends $z(g,  \phi)_{\fe,a}^{\nadm}$ to $z(g, \omega(h)\phi)_{\fe,a}^{\nadm}$. 
\end{lem}

Now we specify the correct constant  $a$ to be used in $z(g, \phi)_{\fe,a}^{\nadm}$. Let 
\begin{align}\label{fadef1}\fa&= \log|\disc_F|  -\frac{[F:\BQ]}{n}- \lb {\fb}[F:\BQ]-\frac{1}{2}{\fc}\rb -c_1(\ol \cL_{K_\Lambda}^\vee)\cdot \cP_{ K_\Lambda} \\
&=c_1(\ol \cL_{K_\Lambda})\cdot \cP_{ K_\Lambda}
+2\frac{L'(0,\eta)}{L(0,\eta)}+ \log|\disc_E|- {\fb}[F:\BQ]- \frac{[F:\BQ]}{n}\nonumber\end{align}
 where $\fb$ is as in \eqref
{Finalw}, $\fc$ is  as in \eqref{fcdef},  
$\disc_F$  is  the   discriminant of $F/
\BQ$, and $\cP_{K_\Lambda}\in Z^1(\cX_{K_\Lambda})$ is a  CM cycle to be precisely defined in Definition \ref{PKL}. 
\begin{rmk}  (1) The definition of  $\fa$ is rather complicated. The reader may skip it for the moment.

(2)  Looking at \cite[Theorem 1.7]{YZ} and \cite[p 590]{YZ}, one can deduce cancellation between the  terms in the definition of  $\fa$ if one can, expectably,  relate the Shimura varieties here and in loc. cit.. % See \cite{How} for an example in the case $F=\BQ$.

\end{rmk}
      \begin{conj}  \label{conj1}Assume 
    \Cref  {asmp1}.
Let  $ \phi\in\ol \cS\lb \BV  \rb ^{{\wKL}}$. For $g\in G(\BA)$, we have 
          \begin{equation}  \label{conj1eq}z(g, \phi)_{\fe,\fa}^{\nadm}\in  \cA_{\hol}(G,\fw)   \otimes\wh \Ch^{1}_{\adm,\BC}(\wt\cX). \end{equation}
 
    \end{conj}
%    Similar to Lemma \ref{key1}, we have the following numerical criterion    of modularity.   

   %   \begin{lem}\label{key}Assume that $\phi$ is $K$-invariant. 
   %   Then \eqref{conj1}  is equivalent to   that  for   some (equivalently, every) $\cP\in  Z_1(\cX_K )_\BC$ of   nonzero degree when restricted to $\cX_{K,E}$, we have  $$  z(g, \phi)_{\fe,\fa}^{\nadm}\cdot  \cP \in  \cA_{\hol}(G,\fw) . $$

%\end{lem}

%\Cref{conj1}  can be equivalently stated using automorphic Green functions. See \ref {Main result}.

For the second definition, we modify  the normalized admissible extensions of  special divisors $Z(x)$'s instead of weighted special divisors $Z_t(  \omega(g)\phi )^\nadm$.   Unfortunately, the modification does not  only depend on $x$ and the result generating series is not holomorphic.

\begin{defn}\label{Kmod} For $x\in \BV^\infty$ with $q(x)\in F^\times$, $g\in G(\BA_{F,\infty})$ and $K\in {\wKL}$, let $ Z(x)^\cL $ be the divisor on $\cX_K$ that is the normalized admissible extension of 
$Z(x)$ with respect to $\cL$, and $
\cG_{Z(x)_{E_v}}^{\ol L_{E_v}}$ the normalized admissible Green function for $Z(x)_{E_v}$ with respect to $\ol L_{E_v}$.  Let   $$Z(x,g)_{K,\fk}^\nadm = \lb Z(x)^\cL, \lb \cG_{Z(x)_{E_v}}^{\ol L_{E_v}}+\fk(x,g_v)\rb_{v\in \infty}\rb\in  \wh Z^1_{\adm,\BC}(\cX_K)\subset  \wh Z^1_{\adm,\BC}(\wt\cX) .$$
Here $ \fk(x,g_v)$ is defined in Definition  \ref{GK0}. For $\phi\in \ol \cS\lb \BV  \rb^{K}$, $a\in \BC$ and $g\in G(\BA_F)$, let   \begin{align*} z(g, \phi)_{\fk,a}^{\nadm}&=\omega(g )\phi(0)\lb c_1(\ol \cL^\vee)-\log \delta_\infty(g_\infty)+ [F:\BQ]\lb \log \pi-(\log\Gamma)'(n+1)\rb+a\rb\\
&+\sum_{x\in \BV^\infty,\ q(x)\in F^\times}  \omega(g^\infty)\phi^\infty (   x)  [Z(x,g_\infty)_{K,\fk}^\nadm] W^{\fw}_{\infty, q(x)}(g_\infty) . 
\end{align*}

\end{defn}
%\begin{rmk}Recall that if $q(x)\not\in F_{>0}$ then $Z(x) =0$; if $q(x) $ is negative at  exactly  1 infinite place $v$, then $Z(x,g)_\ft^\nadm$  is supported at $v$;  if $q(x) $ is negative at more than 1 infinite places, then $Z(x,g)_\ft^\nadm=0$. \end{rmk}

Though $Z(x,g)_{K,\fk}^\nadm$, as an element in $\wh Z^1_{\adm,\BC}(\wt\cX) $, depends on $K$,
  $z(g, \phi)_{\fk,a}^{\nadm}$  does not.   Indeed,
by   Lemma \ref{wr1} (1), $E_0(0,g,\phi)= \omega ( g)\phi (0)$. Then by  Theorem \ref{LSthm}, we have
\begin{equation}\label{zzE}
z(g, \phi)_{\fe,a }^{\nadm}-
z(g, \phi)_{\fk,a }^{\nadm}=  E'(0,g,\phi)-E(0,g,\phi) [F:\BQ]\lb \log \pi-(\log\Gamma)'(n+1)\rb.
\end{equation}
  Then since $z(g, \phi)_{\fe,a }^{\nadm}$ does not  depend  on $K$, neither does $z(g, \phi)_{\fk,a }^{\nadm}$.

The above reasoning also shows that %\begin{lem}  \label{diffag1} 
\eqref  {conj1eq} is equivalent  to
\begin{equation*} 
% \label{diffag1eq} 
z(g, \phi)_{\fk,\fa }^{\nadm}\in  \cA (G,\fw)   \otimes\wh \Ch^{1}_{\adm,\BC}(\wt\cX) .
\end{equation*}
 
%\end{lem}
\subsubsection{Modularity  theorems}\label{Theorem}
We need   some notations to state  our theorems.
 Let  $\Ram$ be the set of  finite places of $F$ nonsplit in $E$, that are  ramified in $E$ or over $\BQ$. % In other words,  $\Ram$, as the corresponding set of places of $E$, is exactly the ones ramified  over $\BQ$.   
 Let $$\BG=P(\BA_{F,\Ram})G(\BA_F^{\Ram}).$$
    Let       $$
\ol \cS\lb \BV  \rb ^{{\wKL}}_\Ram\subset  \ol \cS\lb \BV  \rb ^{{\wKL}}
$$
be the span of pure tensors $\phi$ such that for every finite place $v$ of $F$ nonsplit in $E$, $\phi_v=\omega(g)1_{\Lambda_v}$   for some $g\in \BG_v$.     
   In particular, we have no condition for $\phi\in     \ol \cS\lb \BV  \rb ^{{\wKL}}_\Ram$ over split places.
The following proposition and remark further  show that we have no condition  for $\phi$ outside $\Ram$. 
    \begin{prop}\label{Howeext}

For a   finite place $v$ of $F$ inert in $E$ (so that $\Lambda_v$ is self-dual by \Cref{asmp1}) such that $\chi_{_\BV,v}$  is unramified, 
the span    of   $ \{\omega(g)1_{\Lambda_v}, g\in G(F_v)\}$ is $\cS(\BV(E_v))^{K_{\Lambda,v}}$. 
\end{prop}
\begin{proof}If $v$  has residue characteristic  $\neq2$, this is a special case of  \cite[Theorem 10.2]{Ho}. In general, let $K^{\max}_v\subset G(F_v)$  be as in \eqref{Group}.
Embedding   $\cS(\BV(E_v))^{K_{\Lambda,v}\times K^{\max}_v}$ in an induced representation as \cite[(3.1)]{Ral}. It is routine to show that  $\cS(\BV(E_v))^{K_{\Lambda,v}\times K^{\max}_v}$ is generated by $1_{\Lambda_v}$ as a module over the Hecke algebra of bi-$K^{\max}_v$-invariant Schwartz functions on $G(F_v)$. Then the proposition follows from
Kudla's supercuspidal support theorem for big theta lift. See \cite[Proposition 5.2]{GI}
\end{proof}
   
    \begin{rmk}        We may choose $\chi_{_\BV}$ such $\chi_{_\BV,v}$  is unramified if $v$ is inert in $E$.        \end{rmk}
    
       Let  
$ \cA_{\hol}\lb \BG,\fw\rb $ and $ \cA \lb \BG,\fw\rb $ 
be the restrictions of  $ \cA_{\hol}(G,\fw)$  and $ \cA (G,\fw)$  to $\BG$ respectively.

\begin{thm}  \label{main}
Assume 
    \Cref  {asmp1}.
Let $\phi\in   \ol \cS\lb \BV  \rb ^{{\wKL}}_\Ram $. For $g\in \BG$,  we have 
 \begin{equation}\label{main1}
 z(g, \phi)_{\fe,\fa}^{\nadm} \in  \cA_{\hol}\lb \BG,\fw\rb\otimes\wh \Ch^{1}_{\adm,\BC}(\wt\cX) ,\end{equation}
and 
 \begin{equation}\label{main2}z(g, \phi)_{\fk,\fa }^{\nadm}\in  \cA \lb \BG,\fw\rb   \otimes\wh \Ch^{1}_{\adm,\BC}(\wt\cX).\end{equation}

 \end{thm}    In fact, 
 by \eqref{zzE}, \eqref{main1} and \eqref{main2}  are equivalent. 
%Here the equivalence follows from \eqref{zzE}.
 Theorem \ref{main} will be proved in \ref{Proof of Theorem {main}}.
\begin{rmk}
(1)  The restriction of \eqref{main1} to $g\in G(\BA_\infty)\subset \BG$ gives \Cref{mainint}, with a less strict condition on $\phi_v$, $v\in \Ram$.  \Cref{asmp2rmk1} allows   $K_v$  to be  arbitrary  for $v$ split in $E$.

 (2) By the density of $\BG\subset  G(F)\bsl G(\BA_F)$, the restriction gives 
\begin{equation*}\cA_{\hol}(G,\fw)\cong \cA_{\hol}\lb \BG,\fw\rb,\  \cA (G,\fw)\cong  \cA \lb \BG,\fw\rb 
\label{hahaha}.
\end{equation*}
In particular, $ z(g, \phi)_{\fe,\fa}^{\nadm}$ with $ g\in \BG$  extends uniquely to an element in 
$ \cA_{\hol}\lb G,\fw\rb\otimes\wh \Ch^{1}_{\adm,\BC}(\wt\cX)$. 
\Cref{conj1} predicts that this extension is $ z(g, \phi)_{\fe,\fa}^{\nadm}$ with $ g\in  G(\BA)$. However, it seems not trivial to check this prediction.
The same discussion applies to  
$z(g, \phi)_{\fk,\fa }^{\nadm}$. 
 \end{rmk}
%Finally, we introduce the generating series   with Kudla's Green functions, which are not admissible.
\begin{defn} For $x\in \BV^\infty$ with $q(x)\in F^\times$ and $g\in G(\BA_{F,\infty})$, let  
$$Z(x,g)^{\cL,\mathrm{Kud}} = \lb Z(x)_K^\cL, \lb \cG_{Z(x)_{E_v}}^{\mathrm{Kud}} (g)\rb_{v\in \infty}\rb\in \wh \Ch^{1}_{\BC}(\cX_K) .$$
Here Kudla's Green function  $ \cG_{Z(x)_{E_v}}^{\mathrm{Kud}}$ is defined in  \eqref{Gzhx}.

For $\phi\in \ol \cS\lb \BV  \rb^{K}$, $a\in \BC$ and $g\in G(\BA_F)$, let   \begin{align*} z(g, \phi)_{ a}^{\cL,\mathrm{Kud}}&=\omega(g )\phi(0)\lb c_1(\ol \cL^\vee)-\log \delta_\infty(g_\infty)+ [F:\BQ]\lb \log \pi-(\log\Gamma)'(n+1)\rb+a\rb\\
&+\sum_{x\in \BV^\infty,\ q(x)\in F^\times}  \omega(g^\infty)\phi^\infty (   x)  [Z(x,g_\infty) ^{\cL,\mathrm{Kud}}] W^{\fw}_{\infty, q(x)}(g_\infty) . 
\end{align*}

\end{defn} 
It is directly to check that $z(g, \phi)_{a}^{\cL,\mathrm{Kud}}$ is compatible  under pullbacks by $\pi_{_K,_{K'}}$'s. 

Define the topology on $\wh \Ch^{1}_{\BC}(\cX_K)$   as follows. 
Let $C^\infty(\Sh(\BV)_{K,E_\infty})^\circ$ be the $L^2$-orthogonal complement of $\LC(\Sh(\BV)_{K,E_\infty})$ in $C^\infty(\Sh(\BV)_{K,E_\infty})$, endowed with $L^\infty$-topology.  Then  
 $\wh \Ch^{1}_{\BC}(\cX_K)$ is the direct sum of $C^\infty(\Sh(\BV)_{K,E_\infty})^\circ$ and the finite dimensional subspace of cycles with harmonic curvatures at $\infty$. Endow   $\wh \Ch^{1}_{\BC}(\cX_K)$ with the direct sum topology.
       Let  
  $  \cA \lb \BG,\fw,\wh \Ch^{1}_{\BC}(\cX_K)\rb$
be the restrictions of   $  \cA \lb G,\fw,\wh \Ch^{1}_{\BC}(\cX_K)\rb$
 to $\BG$.

\Cref{diff1} and \eqref {main1}  of  \Cref{main} imply the following theorem.
\begin{thm}  \label{maincor} Let $\phi\in   \ol \cS\lb \BV  \rb ^{K}_\Ram $ where $K\in {\wKL}$.
The generating series $z(g, \phi)_{ \fa}^{\cL,\mathrm{Kud}}$  of  
$\wh \Ch^{1}_{\BC}(\cX_K)$-valued
functions on $\BG$
 pointwise converges to an element in 
  $  \cA \lb \BG,\fw,\wh \Ch^{1}_{\BC}(\cX_K)\rb.$
\end{thm}

\section {Arithmetic  mixed Siegel-Weil formula} In this section, we prove our   modularity theorem (\Cref{main})
 above using an arithmetic analog of the mixed Siegel-Weil formula    \eqref{Ichthm1}. 
First, we define CM  cycles. Then, we state the formula and  use   this formula to prove \Cref{main}.
  We end this section by discussing  some possible future  generalizations and applications of our results. % both the arithmetic modularity theorem and arithmetic  mixed Siegel-Weil formula.
We continue to use notations and assumptions in \ref {Arithmetic modularity1}.

\subsection{CM cycles}\label{CM cycles}      
In this subsection, we define an orthogonal  decomposition     
\begin{equation}\label{VWV}\BV=\BW\oplus V^\sharp(\BA_E)\end{equation}
where $\BW$ is an incoherent hermitian space  over $\BA_E$ of dimension 1, and $V^\sharp$ is a hermitian space  over $E$ of dimension $n$.
Then we  define a CM cycle $$\cP_{\BW,K}\in Z_1(\cX_K)_\BQ$$ associated to   the 0-dimensional Shimura variety for $\BW$, normalized to be of generic degree 1.  
  \subsubsection{Lattices at unramified places}\label{Lattices at unramified places}

For a finite place $v$ of $F$ unramified in $E$,  by  \cite[Section 7]{Jaco}, there exists $e_v^{(0)},...,e_v^{(n)}\in \Lambda_v$ of unit norms, such that $\Lambda_v$ is their orthogonal direct sum.

\subsubsection{Lattices at ramified places}\label{Lattices at ramified places}
%For the discussion in this subsubsection, we do not assume that $\BV$ is incoherent (but still totally positive definite).

Let $\Ram_{E/F}$ be the set of finite places of $F$ ramified in $E$.        For $v\in \Ram_{E/F}$, let $\BM_v$ be the $
\cO_{E_v}$-lattice of  rank 2 with an isotropic  basis $\{X,Y\}$ such that $\pair{X,Y}=\varpi_{E_v}$. % (Here $\pair{\cdot,\cdot}$ is the hermitian pairing.)
Then $\BM_v$ is a $\varpi_{E_v}$-modular lattice in  $\BM_v\otimes{E_v}$  and its determinant  with respect to  this basis is $-\Nm(\varpi_{E_v})$.  
The hermitian space $\BM_v\otimes{E_v}$ has determinant  $-1\in F_v^\times/\Nm( {E_v^\times})$.  
In the other direction,  starting with a 2-dimensional hermitian space $H$ over $E_v$ of determinant  $-1\in F_v^\times/\Nm( {E_v^\times})$,  let $e_v^{(0)},e_v^{(1)}\in H$ be orthogonal such that $q\lb e_v^{(0)}\rb, q\lb e_v^{(1)}\rb\in \cO_{F_v}^\times.$
Then one can choose $\BM_v\subset  \cO_{E_v}e_v^{(0)}\oplus \cO_{E_v}e_v^{(1)}$ to be the preimage of one of  the two isotropic lines in  $(  \cO_{E_v}e_v^{(0)}\oplus \cO_{E_v}e_v^{(1)})/\varpi_{E_v}.$ In particular, %writing down the coordinate of a basis of the isotropic line, 
one easily sees that 
\begin{equation}  \BM_v\cap \lb \cO_{E_v}e_v^{(0)}\oplus\varpi_{E_v} \cO_{E_v}  e_v^{(1)}\rb=\varpi_{E_v} \cO_{E_v}e_v^{(0)}\oplus\varpi_{E_v} \cO_{E_v}  e_v^{(1)}  \label{Schwartz1}
\end{equation}
and 
\begin{equation}  \BM_v\cap \lb \cO_{E_v}e_v^{(0)}\oplus \cO_{E_v}^\times e_v^{(1)}\rb\subset \cO_{E_v}^\times e_v^{(0)}\oplus \cO_{E_v}^\times e_v^{(1)}.  \label{Schwartz}
\end{equation}
These two relations will only be used in the proof of \Cref{YZ7.41'}.  

%We recall the classification of a general    (almost) $\varpi_{E_v}$-modular lattice.
Recall that for $v\in \Ram_{E/F}$,   $v\nmid 2$  by Assumption \ref{asmp1} (1), and 
$\Lambda_v$ is (almost) $\varpi_{E_v}$-modular as in Assumption \ref{asmp1} (4). 
Recall that the rank of $\Lambda_v$ is $n+1$.
\begin{lem}[{\cite[Section 8]{Jaco}}] \label{jac}       (1) If $n$ is odd, then  $\Lambda_v\cong \BM_v^{\oplus (n+1)/2}$.   

(2) If $n$ is even, then $\Lambda_v$ is the orthogonal  direct sum of $n/2$-copies of  $\BM_v$ and a rank-1  hermitian $\cO_{E_v}$-module with determinant in $\cO_{F_v}^\times$.

\end{lem}

\begin{rmk} \label{neven}Using Lemma \ref{jac} and computing discriminants,
we can classify incoherent $\BV$ containing a  lattice $\Lambda $  as in  Assumption \ref
{asmp1}:
\begin{itemize} 
\item[(1)]
If $n$ is odd, then  there exist a $\Lambda$     as in Assumption \ref
{asmp1}   if and only if $(n+1)/2$ is odd and  $[F:\BQ]$ is odd.
\item[(2)] 
If  $n$ is even, then there exist   a $\Lambda$     as in Assumption \ref
{asmp1}.
\end{itemize} \end{rmk}

 \subsubsection{CM cycles}\label{The case n is odd*}
 
    For a finite place $v\not\in \Ram_{E/F}$, let  $e_v^{(0)}$ be as in  \ref{Lattices at unramified places}.
For $v \in \Ram_{E/F}$, let $\Lambda_{v,1}\subset \Lambda_v$ be a copy of   $\BM_v$. See Lemma \ref{jac}.  Let $e_v^{(0)}\in E_v \Lambda_{v,1}$   such that $$ q\lb e_v^{(0)}\rb\in \cO_{F_v}^\times,\ q\lb e_v^{(0)}\rb =\det (\BV(E_v))\in F_v^\times/\Nm( {E_v^\times}).$$
Let $\BW$ be the restricted tensor product of $E_v e_{v}^{(0)}$, for every $v\not\in \infty$, and  an $1$-dimensional subspace of $\BV(E_v)$, for every  $v\in \infty$. Note that since $$\det (\BW_v)=\det (\BV(E_v))\in F_v^\times/\Nm( {E_v^\times}),$$  $\BW$ is incoherent.
The orthogonal complement of $\BW$ in $\BV$ is coherent of dimension $n$. We denote the corresponding 
hermitian space over $E$ by $V^\sharp$.  This gives \eqref{VWV}.
% The   lattices are summarized as follows:
% $$ \ \ \  \ \ \  \BW \subset  \   \ \ \ \ \ \ \ \ \ \ \  \ \ \ \ \ \ \ \ \ \ \ \ \ \ \ E_v\Lambda_{v,1}\ \ \    \ \  \ \ \ \ \ \ \ \ \ \ \  \ \ \ \      \subset  \BV\  ,$$
%$$\cO_{E_v}e_v^{(0)} \ \  ? \   \ \ \ \ \ \ \ \ \ \Lambda_{v,1}\subset \cO_{E_v}e_v^{(0)}\oplus \cO_{E_v}e_v^{(1)} \  \ \ \ \  \ \ \ \ \ \ \ \ \ \subset  \Lambda_v.$$
%Here $?$ is $\not\subset $ if  $v\in \Ram_{E/F}$, and is  $\subset$ otherwise.

Let $K_{\BW}= K\cap U(\BW^\infty)$. The morphism $\Sh(\BW)_{K_\BW}\to \Sh(\BV)_{K}$, 
analogous to \eqref{finitemor}, 
defines a zero cycle on $\Sh(\BV)_{K}$.  (Indeed, this is a ``simple special zero cycle" compared with \ref{Simple special divisors}.)
%  (Here we in fact should specify $U(1)$ as $U(W)$ for a $v$-nearby hermitian space $W$ of $\BW$ with some $v\in \infty$. See \ref{hermitian spaces}.
%We use $U(1)$ to save the notation $W$ for other nearby hermitian spaces. Note that different $U(1)$'s have the same volume for $[U(1)]$. See \ref{Additive characters, Fourier transform, and measures}.) 
\begin{defn}\label{dKPK}
Let  
$d_{\BW,K} %=\frac{\Vol ([U(1)]) }{\Vol(K_{\BW}) }
$ 
be the degree of this zero cycle. Let $ \cP_{\BW,K}\in Z_1(\cX_K)_\BQ$ be  $1/ d_{\BW,K}$ times the Zariski closure of this cycle.  %Then $\cP$ does not depend on the choice of $K_\BW$, 
\end{defn}
%Then the generic fiber of $\cP$ has degree 1. 
 Recall that 
$K$ is normal in $K_{\Lambda}$ (\Cref{asmp2}), and  for $h\in K_\Lambda$,  there is the ``right translation by $h$"  automorphism on $\cX_{K}$ (\Cref{Definition 4.4.9} (3)).
\begin{lem} \label{Kud2.2}
The ``right translation by $h$"   automorphism on $\cX_{K}$
sends $ \cP_{\BW,K}$ to $ \cP_{h\BW,K}$.
\end{lem}
\begin{proof}When restricted to the generic fiber, the proof goes in the same way as \cite[LEMMA 2.2 (iv)]{Kud97}.
Taking Zariski closure, we get the lemma. 
\end{proof}
In particular, for $k\in K$,  $d_{\BW,K} =d_{k \BW,K},$ and $ %=\frac{\Vol ([U(1)]) }{\Vol(K_{\BW}) }
\cP_{\BW,K} =\cP_{k \BW,K}. $ 
Then by the flatness of $\pi_{_K,_{K'}}$ (and the commutativity of taking Zariski closure and flat pullback) and \cite[Proposition 3.2]{Liu},  \begin{equation} 
\label{PWK} 
\pi_{_K,_{K'}}^*\cP_{\BW,K'}=\sum _{k\in (U(\BW)\cap K')\bsl K'/K} \frac{d_{k^{-1}\BW,K}}{ d_{\BW,K'}}\cP_{k^{-1}\BW,K},
 \end{equation}
where each summand is  independent of the choice of   the representatives $k$.
 %  \begin{rmk}\label{PWKrmk}
   We shall later abbreviate $(U(\BW)\cap K')\bsl K'/K$ as $U(\BW)\bsl K'/K$.       %  One may note that to define $\BW$ and $\cP_{\BW,K}$, it is enough to pick up any  element in $E_v^\times e_{v}^{(0)}$.  In particular,  $ e_v^{(1)}$ does not play a role in defining $\BW$.      However, we specify $e_{v}^{(0)},e_{v}^{(1)}$ for   later use. See  \ref{A useful description} and  \ref{ramplace}.

% It is easy to see that   \begin{equation}   \label{pushp}    \pi_{_K,_{K'},*}\cP_{\BW,K}=\cP_{\BW,K'}.    \end{equation}
\subsubsection{Another description}\label{A useful description}
We will give another description of $\cP_{\BW,K}$  that  shows the  independence of $\cP_{\BW,K_\Lambda}$ on $\BW$.
(It
 will also be used in \ref{ramplace} to compute intersection numbers.)
Before that, we introduce new lattices, open compact subgroups and Shimura varieties.

For $v\in \Ram_{E/F}$, let $e_v^{(1)}\in \Lambda_{v,1}$ be orthogonal to $e_v^{(0)}$ such that $ q\lb e_v^{(1)}\rb\in \cO_{F_v}^\times,$
and let $\Lambda_{v,1}^\perp\subset \Lambda_v$ be the orthogonal complement of $\Lambda_{v,1}$. Then
we have  $$\varpi_{E_v}\lb \cO_{E_v}e_v^{(0)}\oplus \cO_{E_v}e_v^{(1)}\rb \subset \Lambda_{v,1}\subset \cO_{E_v}e_v^{(0)}\oplus \cO_{E_v}e_v^{(1)},$$
where each inclusion is of colength 1, and 
$\Lambda_v=\Lambda_{v,1}\oplus \Lambda_{v,1}^\perp.$

Let $K_v^\dag\subset U(\BV_{v})$ be the    stabilizer of $\cO_{E_v}e_v^{(0)}\oplus \cO_{E_v}e_v^{(1)}\oplus \Lambda_{v,1}^\perp$, and $K^\dag=K_v^\dag\prod_{u\neq v}K_u$. 
%The following lemma is inspired by the discussion  on \cite[p 1148]{RSZ-AT}. See also \cite[Lemma 4.11]{RSZ-arithmetic-diagonal-cycles}).
\begin{lem}\label{index2} We have   $[K_v^\dag:K_v^\dag\cap K_v]=2$.
\end{lem}
\begin{proof}  The index is the cardinality of the isotropic lines in 
$\lb\cO_{E_v}e_v^{(0)}\oplus \cO_{E_v}e_v^{(1)}\rb /\varpi_{E_v}$, which is 2. %consider action on isotropic lines, orbit size =2
\end{proof}

Let  $K_{\BW,v}^{(0)}\subset U(\BW_v)$ be  the stabilizer of $\cO_{E_v}e_v^{(0)}$, i.e., $K_{\BW,v}^{(0)}=U(\BW_v)$, and $K_\BW^{(0)}=K_{\BW,v}^{(0)} \prod_{u\neq v }K_{\BW,u}$.
%$$\varpi_{E_v}\lb \cO_{E_v}e_v^{(0)}\oplus \cO_{E_v}e_v^{(1)}\rb\oplus \Lambda_v^\perp\ \text{ and }\ \cO_{E_v}e_v^{(0)}\oplus \cO_{E_v}e_v^{(1)}\oplus \Lambda_v^\perp,$$ 
Then we have a diagram of morphisms of  Shimura varieties 
\begin{equation}
\label{5terms}
\Sh(\BW)_{K_\BW^{(0)}}\to \Sh(\BV)_{K^\dag}\leftarrow\Sh(\BV )_{K^\dag\cap K } \to  \Sh(\BV)_{K}.
\end{equation}
Applying pushfoward, pullback and  pushfoward along the diagram \eqref{5terms} to the fundamental cycle of $\Sh(\BW)_{K_\BW^{(0)}}$, 
we obtain a  zero  cycle on
$\Sh(\BV)_{K}$. Divide it by its degree to obtain a zero cycle of  degree 1.  
\begin{lem}
\label{cpcp}The Zariski closure of this degree 1 zero cycle is $\cP_{\BW,K}$.
\end{lem}

\begin{proof}  
By Lemma \ref{index2},  $[K^\dag:K^\dag\cap K]=2$.
Since $K_{\BW,v}^{(0)}\not\subset K$, 
$K_{\BW,v}^{(0)} (K^\dag\cap K) =K^\dag$. Then the    fiber product of  the first two morphisms in   \eqref{5terms}
is $\Sh(\BW)_{  (K^\dag\cap K)\cap U(\BW^\infty)  }$ (an analog of \eqref  {PWK}). The  natural morphism 
$\Sh(\BW)_{  (K^\dag\cap K) \cap U(\BW^\infty) }
\to \Sh(\BV)_{ K }$ factors through  $\Sh(\BW)_{  K_\BW}$. The lemma follows. \end{proof}

\begin{rmk}\label{multichain}One may define $\cP_{\BW,K}$    via a  diagram of integral models similar to \eqref{5terms}, as in  \cite[(4.30)]{RSZ-arithmetic-diagonal-cycles}.
 \end{rmk}

\begin{lem}\label{hKdag}  Let  $h\in U(E_v \Lambda_{v,1})\times \{1_{E_v\Lambda_{v,1}^\perp}\}\subset U(\BV(E_v))$ such that 
$\Lambda_{v,1}\subset  h\lb  \cO_{E_v} e_v^{(0)}\oplus \cO_{E_v} e_v^{(1)}\rb$  and is the preimage of one of  the two isotropic lines in  the reduction modulo  $\varpi_{E_v}$. Then $h\in K_v^\dag$.
\end{lem}
\begin{proof} Let $K^c=K^\dag\bsl K$.   % Then by \Cref{index2}, $(K^c)^{-1}=K^c$ and $K^c=K^cK=K^c$.
The two preimages of    the two isotropic lines in  the reduction  of $h\lb  \cO_{E_v} e_v^{(0)}\oplus \cO_{E_v} e_v^{(1)}\rb$
are $h \Lambda_{v,1}$ and $hK^c  \Lambda_{v,1}$.
Then either $h \Lambda_{v,1}=\Lambda_{v,1}$ or 
$hK^c  \Lambda_{v,1}= \Lambda_{v,1}$. Each implies $h\in K_v^\dag$. \end{proof}
\begin{prop}\label{PKLprop} The CM cycle $\cP_{\BW,K_\Lambda}$ does not depend on the choice of $\BW$.
\end{prop}
 \begin{proof} To define $\cP_{\BW,K_\Lambda}$, we specify 
 $e_v^{(0)}\in \BW$ as  in \ref{The case n is odd*}.
For $v\not\in \Ram_{E/F}$, the choices of      $e_v^{(0)}$  differ by $K_{\Lambda,v}$ actions, which do not change $\cP_{\BW,K_\Lambda}$  by \Cref{Kud2.2}.
For $v \in \Ram_{E/F}$,  the choices of   $\Lambda_{v,1}$  differ by $K_{\Lambda,v}$ actions. See Lemma \ref{jac}.
 By \Cref{hKdag}, we only need to show that  $\cP_{\BW,K_\Lambda}= \cP_{h\BW,K_\Lambda}$ for $h\in K_v^\dag$, where $v\in \Ram_{E/F}$.
We use \Cref{cpcp}.     The pushforward of the fundamental cycle  of $\Sh(\BW)_{K_\BW^{(0)}}$ by the first map in \eqref{5terms} is   the same as the one obtained by replacing $\BW$ by $h\BW$, by the analog of \Cref{Kud2.2} (on the 
generic fiber).
 \end{proof}
\begin{defn}\label{PKL} We denote $\cP_{\BW,K_\Lambda}$ by $\cP_{K_\Lambda}$. 
\end{defn}
This definition is only used in \eqref{fadef1}. Later we will still use $\cP_{\BW,K_\Lambda}$ for the uniformity of the notation as  the level changes. (For $K\neq K_\Lambda$,  $\cP_{\BW,K}$ depends on the choice of $\BW$.)

  \subsection{Formula}\label{Arithmetic mixed Siegel-Weil formula}
The arithmetic mixed Siegel-Weil formula compares the generating series of arithmetic intersection numbers between  arithmetic special divisors  and    CM 1-cycles  on the integral models     with  an  explicit automorphic form. 
We use   this formula to prove  our  main  theorem \Cref{main}.
\subsubsection{Error functions}\label{ecompositio} 
%Recall the set $\Ram$ of nonsplit places  that are either ramified in $E$ or over $\BQ$, and $\Ram_{E/F}\subset\Ram$ is the subset of places  that are  ramified in $E$. 
 %We  define the candidate for the  function $\phi_v'$,  $v\in \Ram$, in \Cref{tech1} as follows:  we define $\phi_v'$ for every finite place $v$ (not just for $\Ram$), and show that  $\phi_v'=0$ for $v\not \in \Ram$.  Then we prove  \Cref{tech1} (1). %Here we in fact do not need  Assumption \ref{asmp1} (1) concerning $\Ram$. 

Both sides of the arithmetic   mixed Siegel-Weil formula   will have   decompositions into local components (we will see in the proof in Section \ref{Intersections}).
We define some functions measuring the difference between  these local components, and they will appear in the explicit automorphic form. 

 For a place $v$ of $F$ nonsplit in $E$, let $W$ be   the $v$-nearby hermitian space
of $\BW$. See \ref{hermitian spaces}. Define the orthogonal direct sum $$V=W\oplus V^\sharp$$  
Then   we   have isomorpisms \begin{equation}\label{VVUU} V (\BA_F^{v})\cong  \BV^{\infty,v},\  U(V(\BA_E^{v}))\cong U(\BV^{v}), \end{equation}
and similar isomorphisms for $W$ and $\BW$.
Consider   $$V(E_v)-V^\sharp(E_v)=\{  (x_1,x_2)\in  W(E_v)  \oplus V^\sharp(E_v)  :x_1\neq 0\}.$$  %  By the 

 Let $\Lambda_v^\sharp\subset  \BW_v$ be 
the orthogonal complement of $\cO_{E_v}e_v^{(0)}$ in $\Lambda_v$. For $x=(x_1,x_2)\in V(E_v)-V^\sharp(E_v) ,$ let 
\begin{equation}
\label{Error functions}
\phi_v'(x):= \tw_{v,x} '(0, 1, \phi_v)-    (v({q(x_1)})+1) 1_{\cO_{F_v}}({q(x_1)}) 1_{\Lambda _v^\sharp} (x_2)
\log q _{F_v},
\end{equation}
where the smooth function 
$\tw_{v,x} '(0, 1, \phi_v)$ on $V(E_v)-V^\sharp(E_v)$ is as in \eqref{thetaw1}. %Now we have $\phi_v=1_{\Lambda_v}$.

Note that    $\Lambda_v=\cO_{E_v}e_v^{(0)}\oplus \Lambda_v^\sharp$. So the computation for $\phi_v'(x)$ is   only on the  component $\cO_{E_v}e_v^{(0)}$, and  we can apply  computations in \cite{YZ,YZZ}. 

\begin{lem}[{\cite[Proposition 6.8]{YZZ}}] \label {YZZ6.8}
Assume that $v\not \in \Ram$.  Then $\phi_v'=0$.
\end{lem}
Let  $c(g,\phi_v)$  be as  below \eqref{cphiv}.
Recall that $\diff_{v}$ is the different  of $F_v$ over $\BQ_{v}$.

\begin{lem} [{\cite[Lemma 9.2]{YZ}}] \label {YZ7.41}
Assume that $v\in \Ram-\Ram_{E/F}$.
Then $\phi_v'$     extends to a Schwartz function on  $ V(E_v)$  such that 
$c(1,\phi_v)-2\phi_v'(0)=2\log|\diff_{v}|_{v}$.   \end{lem}

For $v \in \Ram_{E/F}$, as in  \ref{A useful description}, we have 
$$\varpi_{E_v}\lb \cO_{E_v}e_v^{(0)}\oplus \cO_{E_v}e_v^{(1)}\rb \subset \Lambda_{v,1}\subset \cO_{E_v}e_v^{(0)}\oplus \cO_{E_v}e_v^{(1)},$$
and  $\Lambda_v=\Lambda_{v,1}\oplus \Lambda_{v,1}^\perp$.
For $x=(x_1,x_2)\in V(E_v)-V^\sharp(E_v) ,$ let $$\phi_v'(x):= \tw_{v,x} '(0, 1, \phi_v)-    (v({q(x_1)})+1) 1_{\cO_{F_v}}({q(x_1)}) 1_{{\varpi_{E_v}
\cO_{E_v}e_v^{(1)}\oplus \Lambda_v^\perp} } (x_2)
\log q _{F_v}.$$

\begin{lem}  \label {YZ7.41'}
Assume that $v\in \Ram_{E/F}$.
Then $\phi_v'$     extends to a Schwartz function on  $ V(E_v)$  such that 
$c(1,\phi_v)-2\phi_v'(0)=2 \log|\diff_v|_{F_v} $.  

\end{lem}
\begin{proof}The computation for $\phi_v'(x)$ is indeed only on the  component $\Lambda_{v,1}$ of $\Lambda_v$. For simplicity, we assume that $n=2$ so that we do not have the component $ \Lambda_{v,1}^\perp$.

Consider a larger lattice $A=\cO_{E_v}e_v^{(0)}\oplus \cO_{E_v}e_v^{(1)}$.  
For $x=(x_1,x_2)\in V(E_v)-V^\sharp(E_v) ,$ let  
$$ \varphi(x)= \tw_{v,x} '(0, 1, 1_{A})-    (v({q(x_1)})+1) 1_{\cO_{F_v}}({q(x_1)}) 1_{{
\cO_{E_v}e_v^{(1)} } } (x_2)
\log q _{F_v}.$$
By 
\cite[Lemma 9.2]{YZ}, $ \varphi$ extends to a Schwartz function on  $ V(E_v)$  such that 
$c(1,1_{A})-2 \varphi(0)=2\log|\diff_v|_{F_v}$.
It is enough to extend $\varphi(x) 1_{{
\varpi_{E_v}\cO_{E_v}  e_v^{(1)} } } (x_2)-\phi_v'(x)$ to a Schwartz function on  $ V(E_v)$,    and show that the twice of its value at $0$ is    $c(1,1_{A})-c(1,\phi_v) $.
First,     \begin{align*} &\varphi(x) 1_{{
\varpi_{E_v}\cO_{E_v}  e_v^{(1)} } } (x_2)-\phi_v'(x)\\
=
& \tw_{v,x} '(0, 1, 1_{A}) 1_{{
\varpi_{E_v}\cO_{E_v}  e_v^{(1)} } } (x_2)- \tw_{v,x} '(0, 1, \phi_v)\\
=&\tw_{v,x} '(0, 1, 1_{A}) 1_{{
\varpi_{E_v}\cO_{E_v}  e_v^{(1)} } } (x_2)- \tw_{v,x} '(0, 1, \phi_v)1_{{
\varpi_{E_v}\cO_{E_v}  e_v^{(1)} } } (x_2)  \\
-&\tw_{v,x} '(0, 1, \phi_v)1_{{
\cO_{E_v}^\times  e_v^{(1)} } } (x_2) . 
\end{align*}
By    Lemma \ref{mylem1} and \eqref{Schwartz}, 
$\tw_{v,x} '(0, 1, \phi_v)1_{{
\cO_{E_v}^\times  e_v^{(1)} } } (x_2)$ extends to a Schwartz function on  $ V(E_v)$.
It  is supported on $\{x_2\in  \cO_{E_v}^\times  e_v^{(1)}\} $ so that its value at 0 is 0. 
By \eqref{Wood0},\eqref{thetaw1} and \eqref{Schwartz1}, 
\begin{align} \label{p23}& \tw_{v,x} '(0, 1, 1_{A}) 1_{{
\varpi_{E_v}\cO_{E_v}  e_v^{(1)} } } (x_2)- \tw_{v,x} '(0, 1, \phi_v)1_{{
\varpi_{E_v}\cO_{E_v}  e_v^{(1)} } } (x_2)\\
=& \gamma_{\BW_v}^{-1}\frac{1}{  \Vol(U(W(E_v))}  
{ W_{v,q(x_1)}}'(0,1,1_{ \cO_{E_v}^\times e_v^{(0)}})   1_{{
\varpi_{E_v}\cO_{E_v}  e_v^{(1)} } } (x_2).\nonumber
\end{align}  
Here we used that $L(s,\eta_v)=1$ due to the ramification of $v$ in $E$.
By Lemma \ref{mylem1}, \eqref{p23}  extends to a Schwartz function on  $ V(E_v)$.

Second, by a direct computation   using \eqref{Wood}, Lemma \ref{mylem1} and \eqref{Schwartz1}, we have   $$c(1,1_{A})-c(1,\phi_v) =\gamma_{\BW_v}^{-1}|\diff_v\disc_v|_{F_v}^{-1/2} 
{ W _{v,0}}'(0,1,1_{ \cO_{E_v}^\times e_v^{(0)}}).$$ 
By \cite[p 23]{YZZ},   which says $\Vol(U(W(E_v))=2|\diff_v\disc_v|_{F_v}^{1/2} 
$, the lemma follows. 
\end{proof}

\subsubsection{Generating series with automorphic Green functions}\label{Main result}

For    $t\in F_{>0}$, %$\phi\in \ol\cS\lb \BV  \rb ^{{\wKL}}$. See \Cref {asmp22}) and $K\in{\wKL}$ stabilizing $\phi$,
let $Z_t(\phi)^{\cL,\aut}\in   \wh Z^{1}_{\adm,\BC}(  \cX_K)$ be   the admissible extension of 
$Z_t(\phi) $   that is normalized at all finite places with respect to $\cL$ and equals the   automorphic Green function  \eqref{GGZ}    at all infinite places.  
Comparing with    \eqref{lineZt},  by \Cref{gnm11} (1),  
$$Z_t(  \phi )^{\cL,\aut}\in \wh Z^{1}_{\adm,\BC}( \cX_K)\subset    
\wh Z^{1}_{\adm,\BC}(\wt\cX)$$ 
only depends on  $\phi$, but not on $K$.  
For $g\in G(\BA_F )$ and $a\in \BC$, define  
 $z_t(g, \phi)_{\fe}^{\cL,\aut}$ (resp. $z(g, \phi)_{\fe,a}^{\cL,\aut}$)  by the   formula  defining        $z_t(g, \phi)_{\fe}^{\nadm}$   (resp. $z(g, \phi)_{\fe,a}^{\nadm}$)  in Definition 
     \ref{smeway},   replacing $Z_t(\phi)^\nadm$  by $Z_t(\phi)^{\cL,\aut}$.   
By \Cref{gnm11} (2), \eqref  {conj1eq} is equivalent to %(we will not use this equivalence later) to   
\begin{equation*}  z(g, \phi)_{\fe,\fa+\frac{[F:\BQ]}{n}}^{\cL,\aut}\in  \cA_{\hol}(G,\fw)   \otimes\wh \Ch^{1}_{\adm,\BC} (\wt\cX).\end{equation*}

    Formally define the ``$t$-th Fourier coefficient" (and compare with \Cref{compwith} (1))
$$  z_t(g, \phi^\infty)_\fe^{\cL,\aut}=\frac{z_t(g, \phi)_\fe^{\cL,\aut}}{W^{\fw}_{\infty, t}(1)}.$$

    \subsubsection{Arithmetic mixed Siegel-Weil formula}
%Below in this paper,  except \ref{Proof of Theorem {main}},   $\phi$ is always specified to be $K$-invariant for some $K\in {\wKL}$. Then, 
%we always  understand $ [Z_t(\omega(g) \phi ) ^{\cL,\aut}]$ and $ z_t(g, \phi)_\fe^{\cL,\aut}$ as in $ \wh \Ch^{1}_{\adm,\BC}(\cX_K)$ in order to consider arithmetic intersections on $\cX_K$.  
Let     $\phi\in \ol\cS\lb \BV  \rb ^{{\wKL}}$. See \Cref
{asmp22}.
Assume that   $\phi$ is a pure tensor  for simplicity.
Define the automorphic form to appear in the formula as follows.
 For $K\in {\wKL}$  stabilizing $\phi$, 
     %    For $K\in {\wKL} $, $\phi\in \ol \cS\lb \BV  \rb^{K}$   a pure tensor  for simplicty,     and  $K'\in {\wKL} $ containing $K$,
and  $K'\in {\wKL}$  containing $K$,  let 
\begin{equation} \label{fh}\begin{split} 
 f_{\BW,K'}^K(g)= \sum _{k\in U(\BW)\bsl K' /K } \frac{d_{k^{-1}\BW,K}}{ d_{\BW,K'}} 
\bigg(&-\te'_{\chol}(0,g,\omega(k)\phi)-( 2{\fb}[F:\BQ] -{\fc})E(0,g,\omega(k)\phi)
\\
&+C(0,g,\omega(k)\phi)-  \sum_{v\in {\Ram}}  \te  \lb0,g,\omega(k)\phi^{v}\otimes\phi_v'\rb\bigg). 
\end{split}
\end{equation} 
Here, the index set and coefficients   are the ones in \eqref{PWK}.  We  choose a representative $k$ such that $k_v=1$ if $K_v=K'_v$.
In particular, $\omega(k)(\phi^{v}\otimes\phi_v')=\lb \omega(k)\phi^{v}\rb\otimes\phi_v'$ for $v$ nonsplit in $E$, which is  what we wrote $\omega(k)\phi^{v}\otimes\phi_v'$ for.
 Inside the bracket,  we have 
   4 automorphic forms in $\cA_{\hol}(G,\fw)$.     
Here we  use the orthogonal decomposition $\BV=\BW\oplus V^\sharp(\BA_E)$ to define  $\te'_{\chol}(\dots)$ and  $C(\dots)$,  and for a given $v\in \Ram$, we use the orthogonal decomposition $  V=W\oplus V^\sharp$  as in \eqref{VVUU} to define $  \te  \lb\dots \rb$. See \ref {Theta-Eisenstein seriesdefdef}, \ref{holsection}  and \ref{Derivative of Eisenstein series} for their definitions.    In particular, $f_{\BW,K'}^K\in \cA_{\hol}(G,\fw)$.  
Finally, the constant $\fb$ is as in \eqref
{Finalw},   $\fc$ is  as in \eqref{fcdef}  and the  term $2{\fb}[F:\BQ] -{\fc}$ appears in  both \eqref{Final} and \eqref{fadef1}.

\begin{rmk}
The 
 bracketed automorphic form indexed by $k\in U(\BW)\bsl K' /K$
on the right hand side of \eqref{fh} is 
 independent of the choice of $k$. 
%As we now explain the independence, the reader can also find the definition of them.
Indeed, the 4th
automorphic form $ \te  \lb0,g,\omega(k)\phi^{v}\otimes\phi_v'\rb
$ is independent of the choice of $k$, by the $K^v$-invariance of $\phi^v$ and the mixed   Siegel-Weil formula 
  \eqref{Ichthm1}. 
By \eqref{Final}, for $t\in F_{>0}$, 
  the $\psi_t$-Whittaker function of the rest 3 terms  becomes 
$$\te_{t,\qhol}'(0,g,\omega(k)\phi)-2\lb E'_{t,\rf}(0,g,\omega(k)\phi) +E_t(0,g,\omega(k)\phi) \log \Nm_{F/\BQ}t\rb.$$ Then the independence of the choice of $k$  follows from \eqref{E0tk},
\eqref{E'0tk} and 
\eqref{Eqhol'0tk}. Then  the  independence of the choice of $k$ follows from the automorphy. 
 \end{rmk}
%Recall the embedding  $ \BC\incl \wh\Ch^{1}_{\adm,\BC}(\wt\cX)$  in  \eqref{line}. 
% By         \Cref{gnm11} (2), we have 
%\begin{equation}\label{diffageq} 
%z(g, \phi)_{\fe,a+\frac{[F:\BQ]}{n}}^{\cL,\aut} -z(g, \phi)_{\fe,a }^{\nadm}  
%%\in   \cA_{\hol}(G,\fw) ,
%\end{equation}valued in the image of this embedding, and is independent of $K$. 

      \begin{thm}\label{nct}  Let     $\phi\in \ol\cS\lb \BV  \rb ^{{\wKL}}$ be a pure tensor such that $\phi_v=1_{\Lambda_v}$ 
      for every finite place $v$ of $F$ nonsplit in $E$.   Let  $K\in {\wKL}$  stabilize $\phi$.
For    $t\in F_{>0}$ and  $g\in P\lb\BA_{F,{\Ram} }\rb G\lb\BA_F^{{\Ram} \cup \infty }\rb $,   
    we have   \begin{align}\label{entry4}2   z_t(g, \phi^\infty)_\fe^{\cL,\aut}\cdot \pi_{_K,{_K}_{\Lambda}}^*\cP_{\BW,{K_\Lambda}} =  f_{\BW,{K_\Lambda},t}^{K,\infty}(g)   ,\end{align}   
  where the arithmetic intersection on the left hand side is taken  on $\cX_K$, and
  $ f_{\BW,{K_\Lambda},t}^{K,\infty}$ is the $t$-th Fourier coefficient of
 $f_{\BW,{K_\Lambda}}^K$. See \ref{Holomorphic  automorphic  forms}.

\end{thm} 

\Cref{nct}  will be proved in  (the end of) Section \ref{Intersections}.

% \begin{rmk}    By  \Cref{gnm11}, Theorem \ref{nct} implies Theorem \ref{masw}.And   as we have explained   below Theorem \ref{masw}, Theorem \ref{nct}  is an arithmetic analog of the mixed Siegel-Weil formula \eqref{Ichthm1}.The arithmetic   mixed Siegel-Weil  formula  has its origin  in the celebrated work of Gross and Zagier \cite{GZ}, as mentioned in \ref{Motivation}. It also plays an important role in the work of   Yuan, S.~Zhang and W.~Zhang \cite{YZZ} generalizing the Gross-Zagier formula to quaternionic Shimura curves over general totally real fields. Compared with these  works, our formula   is more general in the following sense: Gross and Zagier \cite{GZ} only compared coefficients  co-prime to their Iwahori level; Yuan, S.~Zhang and W.~Zhang   \cite{YZZ} only  dealt with certain ``regular test functions", so that they only need to compute proper intersections  (while the main difficulty  of proving Theorem \ref{masw}is to compute improper intersections as we will see). \end{rmk}
 \begin{rmk}\label{uselat}(1) By the projection formula, given $\phi$ and $K'$, the truth of \eqref{entry4} does not depend on the choice of $K$ (stabilizing the given $\phi$).

(2)  We use $\pi_{_K,_{K'}}^*\cP_{\BW,K'} $ with $K'=K_{\Lambda}$, instead of the more natural CM cycle $\cP_{\BW,K}$ (i.e., $K'=K$),  in order to apply  \Cref{pull1cor}
   to avoid computing  normalized admissible extensions. 
    (One may expect to reduce the whole \Cref{nct} to the    level  
    $K_\Lambda$ by the projection formula. However, this is not possible due to Remark \ref{notpush}.)
   %  Theorem \ref{nct} for a general $K'$ is clearly implied by Theorem \ref{nct} for $K'=K$. 

(3) We can also consider Theorem \ref{nct} for $K'=K$.
 Taking advantage of Theorem \ref{main}, 
the only  difficulty  in proving Theorem \ref{nct} for $K'=K$ is  computing   normalized admissible extensions at split place. 
By  considering admissible 1-cycles, the difficulty could  be solved as in 
\cite[8.5.1]{YZZ}.  Note that we do not need Assumption 5.4 in loc. cit.. Rather we  will get an extra Eisenstein series.

\end{rmk}

%The advantage of \eqref{entry4} compared to  \eqref{entryeq}  is that both sides of  \eqref{entry4}  are (expected to be) Fourier coefficients so that we can apply a suitable   modularity argument.

   \subsubsection{Proof of Theorem \ref{main}}\label{Proof of Theorem {main}}
  % qcl  Lemma \ref{diffag} and Lemma \ref{diffag1}    still holds with $G$ replaced by $ \BG=G(\BA_F^{\Ram})P(\BA_{F,\Ram})$. When referring to them,  we understand them with this replacement. 
 
 Assuming Theorem \ref{nct}, we prove Theorem \ref{main} as follows.
 Recall that as we have discussed immediately after  Theorem \ref{main},   that 
  \eqref{main1}  and  \eqref{main2} of Theorem \ref{main}  are equivalent.    We will prove \eqref{main1}  assuming that   $\phi$  is a  pure tensor such that         
$\phi_{v}=1_{\Lambda_v} $            for every finite place $v$ nonsplit in $E$  (so that   Theorem \ref{nct} applies to $\phi$).  
So   \eqref{main2}  holds. 
It follows from \Cref{Kmod} that    \eqref{main2} holds for  $ \omega(g)\phi$ with $g\in \BG^\infty $.   Thus   
Theorem \ref{main} holds by the definition of $
\ol \cS\lb \BV  \rb ^{{\wKL}}_\Ram$ above it. 

 So now we assume that   $\phi$  is a  pure tensor such that         
$\phi_{v}=1_{\Lambda_v} $            for every finite place $v$ nonsplit in $E$ and want to prove \eqref{main1}, that is 
\begin{equation}
 \label{AGGZ00}
  \lb\omega(g )\phi(0)  c_1(\ol \cL^\vee)+ \omega(g)\phi(0) \fa \rb+\sum_{t\in {F>0}}z_t(g ,\phi )_\fe^\nadm  \in  \cA_{\hol}(\BG,\fw)   \otimes\wh \Ch^{1}_{\adm,\BC}(\wt\cX)  .   
 \end{equation}

  Let
 $K\in \wt K_\Lambda$ such that $\phi$ is $K$-invariant.
Let  $ f_{\BW,K_\Lambda,0}^{K} $ be the $0$-th Whittaker coefficient of
$f_{\BW,K_\Lambda}^{K}$.  
Let $ A(\cdot, \phi)_K$  be the    $\BC$-valued function on $G(\BA)$,  understood as valued $\wh\Ch^{1}_{\adm,\BC}(\wt\cX)$ via  $ \BC\incl \wh\Ch^{1}_{\adm,\BC}(\wt\cX)$  as in  \eqref{line},  defined by 
     \begin{equation*}
    \label{AGGZ0}  
    2 \lb \omega(g )\phi(0)  \lb c_1(\ol \cL^\vee)+\frac{[F:\BQ]}{n}\rb + A(g,\phi) _K\rb \cdot \pi_{_K,{_K}_{\Lambda}}^*\cP_{\BW,K_\Lambda} = f_{\BW,K_\Lambda,0}^{K}(g),
    \end{equation*}
    i.e., 
      \begin{equation}    \label{AGGZ01}  
 A(g,\phi) _K 
= \frac{1}{2\deg  \pi_{_K,{_K}_{\Lambda}} }\lb f_{\BW,K_\Lambda,0}^{K}(g)-    2  \omega(g )\phi(0)  \lb c_1(\ol \cL^\vee)+\frac{[F:\BQ]}{n}\rb  \cdot \pi_{_K,{_K}_{\Lambda}}^*\cP_{\BW,K_\Lambda}  \rb .
    \end{equation}   
 By   Lemma \ref{modcri} (with $G$ and $G(\BA)$ replaced  by $\BG$)  and    Theorem \ref{nct}, 
   \begin{equation*}
    \lb\omega(g )\phi(0)   \lb c_1(\ol \cL^\vee)+\frac{[F:\BQ]}{n}\rb + A(g,\phi)_K \rb+\sum_{t\in {F>0}}z_t(g ,\phi )_\fe^{\cL,\aut}  \in  \cA_{\hol}(\BG,\fw)   \otimes\wh \Ch^{1}_{\adm,\BC}(\wt\cX)  .  
     \end{equation*} 
 Then   by \Cref{gnm11} (2),   we have   
 \begin{equation}
 \label{AGGZ}
  \lb\omega(g )\phi(0)  c_1(\ol \cL^\vee)+ A(g,\phi)_K \rb+\sum_{t\in {F>0}}z_t(g ,\phi )_\fe^\nadm  \in  \cA_{\hol}(\BG,\fw)   \otimes\wh \Ch^{1}_{\adm,\BC}(\wt\cX)  .   
 \end{equation}  
 Thus  \eqref{AGGZ00} is reduced to  the following lemma, whose proof requires some preparations.
 
 \begin{lem} \label{fadef} We have $    A(g, \phi)_K  = \omega(g)\phi(0) \fa.$ \end{lem}  

To determine $ A(g, \phi)_K$, one a priori needs to compute  $ f_{\BW,K_\Lambda,0}^{K}(g) $
which could be complicated. % due to the term $C(\dots)$  in  \eqref{fh}. 
Indeed,  
 by \eqref{Evdec1100} for $t=0$  and Lemma \ref{wr1} (1) (2),  we have 
\begin{equation}\label{fk0} \begin{split} f_{\BW,K_\Lambda,0}^{K}(g)
 & =
\deg  \pi_{_K,{_K}_{\Lambda}} 
\bigg(-( 2{\fb}[F:\BQ] -{\fc}) \omega ( g)\phi (0)  
\\
 &+\sum_{v }c(g_v,\phi_v)\omega ( g^v)\phi^v (0) - 2\sum_{v\in {\Ram}}   \omega(g) \lb \phi^{v}\otimes\phi_v'\rb (0)  \bigg). 
 \end{split}
\end{equation}
The terms in the second line cause the complicatedness.

We take a different approach. 
       We study the invariance properties of $ A(g, \phi)_K$.  
       We need a notation.
     Let $W_0(\BG)$ be the space of $\psi_0$-Whittaker functions on $\BG$, i.e.,   smooth $\BC$-valued  functions  $f$ such that
 $f(ng)=f(g)
 $ for $ n\in N(\BA).$ Then   
 \begin{equation}W_0(\BG)\cap\cA (\BG,\fw) =\{0\}.\label{constterm}\end{equation}
% Let $W_0^\dag(\BG)\subset W_0(\BG)$ be the subspace of functions
%$f $  such that $$f(m(a) g   )= \chi_{_\BV} (a)|\det a|_{ \BA_E}^{\dim \BV/2} f(g),\  a\in \BA_E^\times,$$
%i.e.,  $W_0^\dag(\BG)$ is the space of principal series where $\omega(g)\Phi(0)$ lies. 

   By \eqref{fk0}  and  \Cref{sameprin} (2), we have  $ f_{\BW,K_\Lambda,0}^{K} \in  W_0
  (\BG)$. Then by \eqref{AGGZ01}, we have 
%\begin{equation} \label{fadef0}
$A(g, \phi)_K\in  W_0(\BG).$
%\end{equation}   
 For $K'\in \wt K_\Lambda$ such that $\phi$ is $K'$-invariant, by \eqref{AGGZ},
  $A(g,\phi)_K-A(g,\phi)_{K'}\in \cA_{\hol}(\BG,\fw)$.
       Thus by \eqref{constterm}, $A(g,\phi)_K$ does not depend on $K$. 
Let us denote $A(g,\phi)_K$ by $ A(g, \phi)$.

\begin{lem} \label{AGPK} % (1) For $h\in U(\BV_\Xi)$, where $\Xi$ is as in  \Cref{hypo1}, $  A(g,\omega(h)\phi)=A(g,\phi)$;
 
(1)    For $g'\in  \BG^\infty$,  $  A(gg', \phi)=A(g ,\omega(g')\phi)$.

(2) For $h\in K_{\Lambda}$,   $  A(g,\omega(h)\phi)=A(g,\phi)$.

          \end{lem}
          \begin{proof}  % (1)   The     ``right translation by $h$" automorphism $\cX_{hKh^{-1}}\cong \cX_{K}$ in \Cref{hypo1}   induces an automorphism  on $\wh Z^{1}_{\adm,\BC}(\wt\cX),$ and $ \wh\Ch^{1}_{\adm,\BC}(\wt\cX) $ which  fixes $c_1(\ol\cL^\vee)$.
%   It is clear that  the automorphism sends $z_t(g ,\phi )_\fe^\nadm$ to 
%   $z_t(g ,\omega(h)\phi )_\fe^\nadm$, and induces the identity map on  $ \BC\subset \Ch^{1}_{\adm,\BC}(\wt\cX)$ as in  \eqref{line}.  
%  Applying the automorphism to \eqref{AGGZ} and taking difference, we have $  A(g,\omega(h)\phi)-A(g,\phi)\in W_0(\BG)\cap\cA_{\hol}(\BG,\fw).$
%Now  (1) follows from   \eqref{constterm}.

(1)  Since    $ A(g,\phi),  \omega(g )\phi(0)\in W_0(\BG)$,    $$B(g,\phi):= A(g,\phi) +  \omega(g )\phi(0) \big(-\log \delta_\infty(g_\infty)+ [F:\BQ]\lb \log \pi-(\log\Gamma)'(n+1)\rb \big)\in W_0(\BG). $$
Claim: 
          \begin{equation} \label{fadefB}
B(g,\phi)+\sum_{x\in \BV^\infty,\ q(x)\in F^\times}  \omega(g^\infty)\phi^\infty (   x)  [Z(x,g_\infty)_\ft^\nadm] W^{\fw}_{\infty, q(x)}(g_\infty) \in    \cA (\BG,\fw)   \otimes\wh \Ch^{1}_{\adm,\BC}(\wt\cX) .
\end{equation}
Indeed, by   Theorem \ref{LSthm}, the difference between the generating series in \eqref{AGGZ}  and 
 \eqref{fadefB} equals $$E'(0,g,\phi)-E(0,g,\phi) [F:\BQ]\lb \log \pi-(\log\Gamma)'(n+1)\rb.$$
 (This is the analog of  
 \eqref{zzE}.) 
Then the claim  follows from  \eqref{AGGZ}.

Now we prove (1).
Apply \eqref{fadefB} in two ways: first, replace $\phi $ by $\omega(g')\phi $ and call the resulted series $S_1$; second,  
replace $g$ by $gg'$, and call the resulted series $S_2$.  Then $S_1,S_2\in \cA (\BG,\fw)   \otimes\wh \Ch^{1}_{\adm,\BC}(\wt\cX) $. So  $$  B(g,\omega(g')\phi)-B(gg', \phi) =S_1-S_2\in \cA (\BG,\fw).$$
 Thus it must be 0 by \eqref{constterm}. This gives (1).

     (2) By  \Cref{rth2}, after the     ``right translation by $h$",  
  \eqref{AGGZ} becomes
 \begin{equation*}
  \lb\omega(g )  \phi(0)  c_1(\ol \cL^\vee)+ A(g,\omega(h)\phi) \rb+\sum_{t\in {F>0}}z_t(g ,\omega(h)\phi )_\fe^\nadm  \in  \cA_{\hol}(\BG,\fw)   \otimes\wh \Ch^{1}_{\adm,\BC}(\wt\cX)  .   
 \end{equation*}  
 (This is  similar to  \Cref{equiva}.)
 Since    \eqref{AGGZ} holds with $\phi$ replaced by $\omega(h) \phi$, taking the difference,  we have $  A(g,\omega(h)\phi)-A(g,\phi)\in \cA_{\hol}(\BG,\fw)$. It is 0 by \eqref{constterm}.   
            \end{proof}
        To prove \Cref{fadef},     we need a final ingredient, whose proof is easy and left to the reader.
            \begin{lem} \label{AGPK1}    For $v$ split in $E$,
identify   $\BW_v=E_v $ and $K_{\Lambda,v}\cap U(\BW_v)=\cO_{E_v}^\times $.  Then 
$\cS(\BW_v)^{K_{\Lambda,v}\cap U(\BW_v)}$  is spanned  by $f_a$'s, where $f_a(x):=  1_{\cO_{E_v}}( xa), a\in E_v^\times$.  
    \end{lem}

            \begin{proof}[Proof of \Cref{fadef}]
            
       We have some reduction steps  about  $\phi$.     
   By \Cref{AGPK} (2), we may assume that  $\phi$ is $K_\Lambda$-invariant.  In particular, 
   for $v$ split in $E$, 
   $\phi_v$ is $K_{\Lambda,v}\cap U(\BW_v)$-invariant.  Identify $\BW_v=E_v$ and  $q=\Nm$. 
     By \Cref{AGPK} (1), \Cref{AGPK1} and  the first {Weil representation} formula   in \ref{Weil representation},  we may further assume that    $\phi_v=\phi_{v,1}\otimes \phi_{v,2}$
where $\phi_{v,1}=1_{\cO_{E_v}}$ and $\phi_{v,2}\in \cS\lb V^\sharp(E_v)\rb$ (recall that $V^\sharp(E_v)$ is the orthogonal complement of $\BW_v$).

           Now we look at $f_{\BW,K_\Lambda,0}^{K}(g)$
            given by \eqref{fk0}. By the definition \eqref{cphiv} and  Lemma \ref{lSWlem1} (2) (3),    $c(g_v,\phi_v)=0$  unless $v\in \Ram$  or  $v$ split in $E$ (by the same argument as in Lemma \ref{sameprin}). 
              By \Cref{AGPK} (1), we may assume $g^\infty=1$.    
  Then 
by 
\Cref{sameprin1},  \Cref{YZ7.41},     \Cref{YZ7.41'}, 
                \begin{equation*}  f_{\BW,K_\Lambda,0}^{K}(g)
  =\deg  \pi_{_K,{_K}_{\Lambda}} 
\bigg(-( 2{\fb}[F:\BQ] -{\fc}) \omega ( g)\phi (0)  +2\log|\disc_F|  \omega ( g)\phi  (0)  \bigg). 
 \end{equation*}
         Since $c_1(\ol \cL_{K }^\vee)=\pi_{_K,{_K}_{\Lambda}}^*\cL_{K_\Lambda}^\vee$, by 
the projection formula,  $$c_1(\ol \cL_{K }^\vee)\cdot \pi_{_K,{_K}_{\Lambda}}^*\cP_{\BW,K_{\Lambda}} =\deg  \pi_{_K,{_K}_{\Lambda}} 
 c_1(\ol \cL_{K_\Lambda}^\vee)\cdot \cP_{\BW, K_\Lambda} .$$
 Now by the definition of    $    A(g, \phi)_K$ in \eqref{AGGZ01}  
and  the definition of $\fa$ in 
  the first line of \eqref{fadef1},   
the lemma  follows.
 \end{proof}

\subsection{Generalizations and applications}
\subsubsection{Higher codimensions}

Based on their modularity result for generating series of special divisors, Yuan, S.~Zhang and W.~Zhang \cite{YZZ1} proved the modularity  for higher-codimensional  special cycles inductively, assuming the convergence.  
One would like to mimic their proof in the arithmetic situation. 
%Ignore  the technicality in arithmetic  intersection theory and formally check their inductive scheme. 
Then  one  needs    a modularity theorem for divisors 
with
general level structures and test functions, even if the given test function is very good.  
Thus the generality of our  results is necessary toward     modularity  in the arithmetic situation in higher codimensions.

In the codimension $n$ case (i.e., for arithmetic 1-cycles),
S.~Zhang's theory of admissible cycles is unconditional modulo  vertical 1-cycles that are numerically trivial \cite{Zha20}. % (in arithmetic intersection theory). 
Then
  the method in
the current paper is still applicable to approach the   modularity in the arithmetic situation.

\subsubsection{Almost-self dual lattice} 
%In  the proof of Theorem \ref{main}, we use  the exotic smoothness of our integral models. See Lemma \ref{expectlem}). 
There is another  lattice level structure at a finite place considered in \cite{RSZ-arithmetic-diagonal-cycles}, defined by an almost-self dual lattice. 
The integral model is not smooth, but is explicitly described in \cite{Let}.  We hope to include this level structure in a future work.
In fact, if we also use admissible 1-cycles, 
 our approach combined with a recent result of Z. Zhang \cite[Theorem 14.6]{ZZY}
 is already  applicable to prove 
the analog of 
Theorem \ref{main}, after replacing normalized admissible extensions of special divisors by the ``admissible projections" of   the  Kudla-Rapoport divisors at these new places (provided that they can also be suitably defined on  our models).
However, the  difference  between two extensions is not clear so far.

\subsubsection{Faltings heights of Shimura varieties and Arithmetic Siegel-Weil formula}
\label{Faltings heights of Shimura varieties and arithmetic Siegel-Weil formula}
Following  Kudla  \cite{Kud02,Kud03}, we propose  the   arithmetic 
analog  of  the geometric Siegel Weil formula   \eqref{Proposition 4.1.5}.
\begin{problem}\label{theqSW}
 Match   $c_1(\ol \cL)^n\cdot  z(g, \phi)_{ \fa}^{\cL,\mathrm{Kud}}$ with a linear combination of 
$E(0,g, \phi)$ and $  E'(0,g, \phi)$ (possibly up to some error terms).  
\end{problem}
The  modularity of $z(g, \phi)_{ \fa}^{\cL,\mathrm{Kud}}$ helps to  attack this problem as follows.
The constant term of  $c_1(\ol \cL)^n\cdot  z(g, \phi)_{ \fa}^{\cL,\mathrm{Kud}}$  is indeed the Faltings height of $\cX_{K}$ itself, while the non-constant terms are given by Faltings heights of Shimura subvarieties with the numbers in   Definition  \ref{GK0}. (Here, by the 
Faltings height of     $\cX$, we mean   $   \deg c_1(\ol \cL)^{n+1}$.)
There is clearly an inductive scheme to compute the Faltings heights/attack arithmetic Siegel-Weil formula, by applying the modularity of the generating series. Moreover, 
we  only need to compute enough terms.  This might enable us to avoid   dealing with Shimura subvarieties of  general level structures from the inductive steps. 
We can also use   
$ z(g, \phi)_{\fk,\fa }^{\nadm}$  to attack \Cref{theqSW}, since  by a direct computation,
\begin{equation}\label{samefal} 
c_1(\ol \cL)^n\cdot z(g, \phi)_{ \fa}^{\cL,\mathrm{Kud}} =c_1(\ol \cL)^n \cdot z(g, \phi)_{\fk,\fa }^{\nadm}.
\end{equation}

For quaternionic Shimura curves, Faltings heights are computed in \cite{KRY2} \cite{Yuan}.
For unitary Shimura varieties, in the case $F=\BQ$, the Faltings height of $\cX_{K_\Lambda}$ (for a different lattice $\Lambda$)  was computed in \cite{BH} using   Borcherds'    theory.   See  \cite{BH} for other related results.

\subsubsection{Arithmetic theta lifting and Gross-Zagier type formula}
   Consider the Petersson inner product between   the modular generating series    of special divisors   on the generic fiber  and  a cusp form  $f$ of $G$ \cite{Kud02,Kud03}. When $n=1$,  it is  cohomological trivial and its  Beilinson-Bloch
height    was studied in \cite{Liu,Liu0}. However,  when $n>1$, the Picard group of $\Sh(\BV)_K$ is CM by \cite{KRa}, so that the inner product is 0 in most cases after cohomological trivialization. 
Thus it is necessary to consider   arithmetic   intersection pairing  on an integral model (without  cohomological trivialization). 

The  arithmetic   intersection pairing  with our CM 1-cycle as in \Cref {nct} is simply the Petersson  inner product 
between \eqref {fh} and  $f$. Such a pairing appeared  in the work of Gross and Zagier \cite{GZ}  and leads to their celebrated  formula. We hope to get a
  Gross-Zagier  type formula. 
   In the case $F=\BQ$ and $\phi^\infty=1_{\Lambda}$ (for a different $\Lambda$),  
  a  Gross-Zagier  type formula was obtained in
 \cite{BHY} and \cite{Bet2}. See also  \cite{BY}.  
 For   general $F$ and  test functions  as in our case, 
 a general  theory of Shimura type integrals  is to be developed.

%We also hope that  this idea  can be used to compute the Faltings heights of Shimura varieties. For certain unitary Shimura varieties over $\BQ$,  Bruinier and  Howard computed  the Faltings heights  using Borcherds' theory. 

%\subsubsection{Arithmetic fundamental lemma and trasnfer} A striking application of the modularity results of Bruinier et al. \cite{Bet}  is W.~Zhang's proof of the arithmetic fundamental lemma over $\BQ_p$ for the unramified  Rapoport-Zink space.  Presumably, our modularity result could be used to prove the arithmetic fundamental lemma over a general $p$-adic field following W.~Zhang's strateg, as well as the arithmetic transfer   conjecture for the exotic  smooth Rapoport-Zink space. However, the arithmetic fundamental lemma over a general $p$-adic field is already a theorem of Mihatsch and W.~Zhang \cite{MZ}, proved by modifying W.~Zhang's original approach and using cycles  of generic degree 0 as mentioned above. Moreover, their strategy  seems also applicable to the arithmetic transfer   conjecture for the exotic  smooth Rapoport-Zink space. 

%In the sequel     \bibitem{Bet2}  of \bibitem{Bet}, several applications of 
 
% Need unramified   
\section{Intersections}\label{Intersections}

In this section, we prove  the arithmetic mixed Siegel-Weil formula
(Theorem \ref{nct}). First, we prove \textit{local} analogs of the formula, under some regularity assumption which forces   improper   intersections to disappear.  Then,  to prove Theorem \ref{nct}, we use a \textit{global} argument involving 
admissibility of our arithmetic divisors and modularity on the generic fiber (more precisely, \Cref{fdimeadd}).
 
\subsection{Proper intersections}

 In order to state  the local analogs  of the Arithmetic mixed Siegel-Weil formula (\Cref{3=prop}), we need  some preliminaries.

  We  use the CM cycles  in \ref {CM cycles}, as well as the notations there, in particular, the orthogonal  decomposition     
$\BV=\BW\oplus V^\sharp(\BA_E). $
 \begin{defn}\label{regdef}For a finite place $v$ of $F$, a  Schwartz function  on $\BV(E_v)$ is $\BW_v$-regular if it   is supported outside  the orthogonal complement  $\BW_v^\perp=V^\sharp(E_{v})$.
\end{defn}
 Let $\phi\in \ol \cS\lb \BV  \rb $  be a pure tensor such that   $\phi_v=1_{\Lambda_v}$ 
      for every finite place $v$ nonsplit in $E$.
      Let $K\in \wt K_\Lambda$ such that $\phi$ is $K$-invariant.

We assume the following assumption throughout this  subsection.
 \begin{asmp}\label{regasmp}There  is a  nonempty set $R$ of  places of $F$ 
  such that $\phi_v$ is $\BW_v$-regular  for $v\in R$.
%(One may choose $R$ to be a singleton, and we 
\end{asmp}
 Since  $\phi_v=1_{\Lambda_v}$ 
      for every finite place $v$ nonsplit in $E$, $R$ necessarily contains only places split in $E$.

 As we have seen in the proof of \Cref{regmod}, for a finite place $v$ of $E$, there exists a finite unramified extension $N/E_v$ such that $\cX_{K ,\cO_{N}} $ is (part of) a PEL moduli space, and the supersingular locus is thus defined and obviously independent of the choice of $N$. Let $t\in F_{>0}$.

    \begin{lem}\label{221}  For $g\in P \lb\BA_{F,R  }\rb G\lb  \BA_F^{R }\rb$,    the supports of the Zariski closure of $ Z_t(\omega(g)\phi)^\zar $   and $ \cP_{\BW,K  }$ on $\cX_{K } $
only intersect on the supersingular loci at finite places of $E$ nonsplit over $F$.  
\end{lem}

\begin{proof}  The  regularity of  $\phi$ at $R$ is preserved by the action of $P \lb\BA_{F,R  }\rb$ on $\phi$. 
By the regularity     of   $\omega(g)\phi $,   the  lemma  follows from the same proof as \cite[Lemma 2.21]{KR14}. Or one may reduce the lemma to  (the version over a general CM field of) \cite[Lemma 2.21]{KR14}  as follows. Choose $n$ vectors $x_1,...,x_n$ spanning $V^\sharp $. Then $ \cP_{\BW,K ,E }\subset Z(x_1)\cap...\cap Z(x_n)$.  
Since $ Z_t(\omega(g)\phi) $ is a finite  sum of $Z(x)$'s with $x $ outside $V^\sharp(\BA_E) $, \cite[Lemma 2.21]{KR14} applies.
% (note that the Zariski closures are always contained in the moduli extensions on the integral model).  
\end{proof} 

 For $k\in U(\BV^{\infty\cup R})$, i.e., $k_v=1$ for $v\in R$. Then   $\BW_v=(k\BW)_v$ and
 $\phi_v$ is $(k\BW)_v$-regular.  By \Cref{221}, the intersection number $\lb Z_t(\omega(g) \phi) ^\zar \cdot \cP_{k^{-1}\BW,K} \rb_{\cX_{K,\cO_{E_v}}}$
  of 
the restrictions of  the cycles to $\cX_{K,\cO_{E_v}}$ is well defined  as in \eqref{YZCM}.
 
\begin{prop}\label{3=prop}
 Recall the set $\Ram$ of  finite places of $F$ nonsplit in $E$,  and ramified in $E$ or over $\BQ$.
Let %$t\in F_{>0}$ and 
$g\in P\lb\BA_{F,\Ram\cup R}\rb G\lb\BA_F^{\Ram\cup R}\rb $.
Let  $v$ be a place of $F$ nonsplit in $E$ and $k\in U(\BV^{\infty\cup R\cup\{v\}})$.
For $v\not\in \Ram\cup\infty$,  resp. $v\in \Ram $, resp.  $v\in \infty $,   we respectively have
\begin{align}  \label{hnspl}2\lb Z_t(\omega(g) \phi) ^\zar \cdot \cP_{k^{-1}\BW,K} \rb_{\cX_{K,\cO_{E_v}}}\log  q_{E_v}
=   \te'_t (0,g,\omega(k)\phi)(v) ,\end{align} 
\begin{align}  \label{hram}2\lb Z_t(\omega(g) \phi) ^\zar \cdot \cP_{k^{-1}\BW,K} \rb_{\cX_{K,\cO_{E_v}}}\log  q_{E_v}
=   \te'_t (0,g,\omega(k)\phi)(v) - \te  \lb0,g,\omega(k)\phi^{v}\otimes \phi_v'\rb ,\end{align}
\begin{align}  \label{hinf}2 \int_{\lb\cP_{k^{-1}\BW,K}\rb_{E_v}} \cG_{Z_t(\omega(g)\phi) _{E_v}}^\aut
=   \wt{\lim _{s\to 0}}\te'_{t,s} (0,g,\omega(k)\phi)(v). \end{align}

 \end{prop} 
%Note that for $k\in U(\BV^v)$,  i.e., $k_v=1$, we have $\BW_v=(k\BW)_v$. Thus 
We  will prove Proposition \ref{3=prop}  for $k=1$,
in Proposition \ref{yll} ($v\not\in \Ram\cup\infty$),  
Proposition \ref{yll+} and
Proposition \ref{yll1} ($v\in \Ram $),  
Proposition \ref{yllinf} ($v\in \infty $).
The proof for the general $k$ is the same, except one need to keep track of $k$.
For simplicity,  let $$\cP=\cP_{ \BW,K}.$$

%Note that in this case, $\pair{x,x}\in \Nm E$ for $x\in V^\sharp$. Indeed, we may simply take $W=E$ with $\pair{x,y}=\varpi x\ol y$ and $\Lambda=\cO_E.$  Also $V^\sharp=E$ with $\pair{x,y}=  x\ol y$.

We prepare more notations for later computations. 
For  a finite place   $v$ of $F$  nonsplit in $E$,   let   $E_v^\ur$ be the complete maximal unramified extension of   $E_v$.  
Let $\BE$ be  the unique formal $\cO_{F_v}$-module of relative height $2$ and dimension 1 over $\Spec E_v^\ur/\varpi_{E_v}$. 
The endomorphism ring of $\BE$ is  the maximal order
of the unique  division quaternion algebra $B$ over $F_v $.  Fixing an embedding $\iota:  E_v \incl B$  such that $\iota( \cO_{E_v})$ is in the maximal order
of $B$. Then 
$\BE$ becomes a formal $\cO_{E_v}$-module of relative height $1$ and dimension 1, which we still donote by $\BE$. Let $\ol\iota$ be  $\iota$ precomposed with the nontrivial $\Gal(E_v/{F_v})$-conjugation. 
It produces another $\cO_{E_v}$-module  $\ol\BE$.
Fix an $\cO_{F_v}$-linear principal polarization $\lambda_{\BE}$ on  $\BE$.
Let $\cE$ and $\ol\cE$ be the canonical liftings of $\BE$ and $\ol\BE$ respectively, as $\cO_{E_v}$-modules. They are isomorphic as formal $\cO_{F_v}$-modules, and equipped with a unique 
$\cO_{F_v}$-linear principal polarization $\lambda_\cE$ lifting $\lambda_{\BE}$.

\subsubsection{Finite places of $F$ inert in $E$}For such a $v$,  $\Lambda_v$ is self-dual.

Before we can compute  the intersection number, we need to uniformize the integral   model, CM cycle and special divisors using 
Rapoport-Zink spaces.  

For a non-negative integer $m$, let
$\cN_m $   be   the  unramified  relative unitary Rapoport-Zink space   of signature $(m,1)$ \cite{KR11}
\cite[2.1]{LZ}  over $\Spf \cO_{E_v^\ur}$. It is the deformation space of   the  polarized hermitian $\cO_{E_v}$-module $\BX_m:=\ol\BE\times \BE^m$ with the product polarization $\lambda_m$. It is  formally smooth of relative   dimension $m$.
The space $\Hom_{\cO_{E_v}}(\BE,\BX_m)_{\BQ}$ carries a natural hermitian pairing
\begin{equation}(x,y)\mapsto  \lambda_{\BE}^{-1}\circ x^\vee\circ\lambda_m \circ y\in \Hom_{\cO_{E_v}}(\BE, \BE)_{\BQ}\cong E_v.\label{hpair}\end{equation}

For $m=n$, we let $\cN=\cN_n$. 
By \cite[2.2]{LZ}, we have  $\Hom_{\cO_{E_v}}(\BE,\BX_n)_{\BQ}\cong V(E_v).$
And $U(V(E_v))$ is isomorphic to the group of $\cO_{E_v}$-linear  self-quasi-isogenies of $\BX_n$ preserving $\lambda_n$ \cite[(4.3)]{RSZ-AT}.  In particular $U(V(E_v))$ acts on $\cN$.
For every $x\in V(E_v)-\{0\} $, we have the  Kudla-Rapoport divisor $\cZ(x)$ on $ \cN $   \cite{KR11}
\cite[2.3]{LZ} that is the locus where $x$ lifts to a quasi-isogeny. 
It is a (possibly empty) relative Cartier divisor. See \cite[Proposition 3.5]{KR11}, which is only stated for $F_v=\BQ_p$ but holds in the general case.

Let $\wh{\cX_{K, \cO_{E_v^\ur}}^{ss}} $ be the formal completion of $\cX_{K, \cO_{E_v^\ur}}$ along the supersingular locus.
Then  we have the following formal uniformization \cite[13.1]{LZ}:
\begin{equation}\wh{\cX_{K, \cO_{E_v^\ur}}^{ss}} \cong  U(V)\bsl \lb \cN\times U\lb \BV^{\infty,v}\rb/ K^v\rb \label{ssun*}; \end{equation}
For $x\in V(E_v)-\{0\} $ and $h\in U\lb \BV^{\infty,v}\rb$, we have a relative cartier divisor $[\cZ(x),h]$ of $\wh{\cX_{K, \cO_{E_v^\ur}}^{ss}} $.

The Rapoport-Zink space
$\cN_0  $ is  naturally  a closed formal subscheme of $\cN$ by adding canonical liftings, i.e., the morphism $\cN_0\to \cN$ is given by $X\mapsto X\times \cE^n$. 
The subspace $$\Hom_{\cO_{E_v}} (\BE,\ol\BE )_{\BQ}\subset \Hom_{\cO_{E_v}}(\BE,\BX_n)_{\BQ}\cong V(E_v)$$ becomes the subspace $W(E_v)$ of $V(E_v)$, and the subgroup $U(W(E_v))$ stabilizes $\cN_0$.
We have \begin{align}\label{ascycle} \cP _{\cO_{E_v^\ur}} %|_{\wh{\cX_{K, \cO_{E_v^\ur}}^{ss}}} 
&=\frac{1}{d_{\BW,K} }  U(W) \bsl \lb \cN_{0}\times U \lb \BW^{\infty,v}\rb/ K_\BW^v\rb ,\end{align}
where $d_{\BW,K} $ is   the degree of the fundamental cycle of  $\Sh(\BW)_{K_\BW}$. See Definition \ref{dKPK}), and  the right hand side is defined using the formal uniformization \eqref{ssun*} of 
$\wh{\cX_{K, \cO_{E_v^\ur}}^{ss}}$ with $\cN_{0}$ understood as an 1-cycle on $\cN$.

Recall that   $V^t$ is the subset of $V(E)$ of elements of norm $t$.
\begin{prop}\label{czss}Under the formal uniformization \eqref{ssun*} of 
$\wh{\cX_{K, \cO_{E_v^\ur}}^{ss}}$, for $g\in G(\BA_F^{ v})$, we have        \begin{align} \label{ssuni0} {Z_t(\omega(g)\phi)}^\zar|_{\wh{\cX_{K, \cO_{E_v^\ur}}^{ss}}} 
=\sum_{x\in U(V)\bsl  V^t }\ \sum _{h\in U_x^v\bsl U\lb \BV^{\infty,v}\rb/K^v }   \omega(g)\phi^v(h^{-1}x)   [\cZ(x),h] .\end{align}

\end{prop}
%For  $x\in \BV^\infty_{\adm}$, let $U_x=U\lb (Ex)^\perp\rb\subset U(\BV^\infty),$ which is the stabilizer of $x$. 

\begin{proof} This follows from  \cite[(8.3)]{LL} and the flatness of $\cZ(x)$.
% For general $b$ in \eqref{ssuni0} follows from the definition of the Weil representation.   
\end{proof}
% \begin{rmk} The formal subscheme  $[\cZ(x),h]$  does not change replacing $h$ by $h_0h$ where $h_0\in U_x^v$, This can be seen from the moduli interpretation of $[\cZ(x),h]$, as stated in the proof of  \cite[Theorem 4.25]{Liu}.   \end{rmk}
% Then    $\Lambda_v$ is self-dual. 

Let $ \lb\cZ(x)\cdot \cN_0\rb_{\cN}$ be  the Euler-Poincar\'e characteristic of the derived tensor product $\cO_{\cZ(x) }\otimes^\BL \cO_{\cN_0}$.  
Since  $h\in U(W(E_v))$ stabilizes $\cN_0$,  $\lb\cZ(x)\cdot \cN_0\rb_{\cN}= \lb\cZ(hx)\cdot \cN_0\rb_{\cN}$.
%Note that $K_{\BW,v}=U(\BW_v).$
By  \Cref{221}, \eqref{ascycle},  Proposition \ref{czss} and a direct computation, for $g\in P (\BA_{F,R\cup\{v\}})G(\BA_F^{ R\cup\{v\}})$,
\begin{align} \label{cormn}\lb  {Z_t( \omega(g)\phi)}^\zar\cdot \cP   \rb_{\cX_{K,\cO_{E_v}} }  =\frac{1}{\Vol([U(W)])} \int_{ h\in [U(W)]}    \sum_{x\in    V^t-V^\sharp}    \lb\cZ(h_v^{-1}x) \cdot\cN_0\rb_{\cN} \omega(g) \phi^v(h^{-1} x)   dh \end{align}
%Here the inner summands for $x\in V^\sharp$ are understood as 0  under our regularity  assumption on    $\phi $.
%\begin{rmk}Note that  the maximal compact subgroup $K_{v}^{\max}$ of $G(F_v)$. See Definition \ref{Krmax}) acts on $\phi_v$ by a character explicitly determined by $\chi$, and $G(F_v)= M(F_v)  N(F_v) K_{v}^{\max}$ by the Iwasawa decomposition. Thus we may write an expression of  $L_v(g, \phi)$ using  Corollary \ref{cormn}.  \end{rmk}

\begin{prop} \label{yll}If $v$ is unramified over $\BQ$, for $g\in P (\BA_{F,R})G(\BA_F^{ R})$ and $k=1$, \eqref{hnspl} holds. \end{prop}
\begin{proof}   First, we compute $ \lb \cZ(x)\cdot \cN_0\rb_{\cN}$, i.e., we need
to compute the  length of the locus on $\cN_0$ to which $x$ lifts. Recall   that   $\Lambda_v^\sharp\subset  \BV(E_v)$ is
the orthogonal complement of $\cO_{E_v}e_v^{(0)}$ in $\Lambda_v$.
Under the isomorphism
$\Hom_{\cO_{E_v}}(\BE,\BX_n)_{\BQ}\cong V(E_v)$, the  image of $\Hom_{\cO_{E_v}}(\BE,  \BE^n)$ is  $\Lambda_v^\sharp$.
Then  for  $ x=(x_1,x_2)\in  W(E_v)\oplus V^\sharp (E_v) $  with $x_1\neq 0$,
if it lifts, then $x_2\in\Lambda _v^\sharp$. 
Moreover, by  Gross' result on  canonical lifting  \cite[Proposition3.3]{Gro}, we have \begin{equation}\label{ssuni3} \lb \cZ(x)\cdot \cN_0\rb_{\cN} = \frac{v({q(x_1)})+1}{2} 1_{\cO_{F_v}}({q(x_1)}) 1_{\Lambda _v^\sharp} (x_2)
\end{equation} 

Second,   if $g_v= 1$,   express the left hand side of    \eqref{hnspl} by    \eqref{cormn} and \eqref{ssuni3}. Compare it  with the expression of   the right hand side of    \eqref{hnspl}  given by    \eqref{Evdec11}.
By Lemma   \ref{YZZ6.8},
\eqref{hnspl} follows.  

Finally, we reduce  the general case   to the case $g_v=1$ in two  reduction steps.  (I)
We  claim:
replacing $g_v$ by $g_vn(b_v)$ for   $b_v\in F_v$  or by $g_vk_v  $ for   $k_v\in K_v^{\max}$ (see \ref{Group}),  both sides of  \eqref{hnspl}  are multiplied by the same constant.
Indeed, for   the left hand side of    \eqref{hnspl}, we directly use  the definition of the Weil representation. 
For the right hand side, besides the definition, we further need \eqref{ktrivial}  and Lemma \ref{YZZ6.6}.  (II),
By \Cref{mainv} and \Cref{trivial1},
we can replace $g$ by $m(a)g$
for some $a\in E^\times$. 
By   the Iwasawa decomposition, the fact that $E_v^\times=E^\times\cO_{E_v}^\times$, and the claim  in the first reduction step, 
we may assume that $g_v= 1$. 
The proposition is proved.   \end{proof}

 \begin{prop} \label{yll+}  If $v$ is ramified over $\BQ$,
for $g\in P (\BA_{F,R\cup\{v\}})G(\BA_F^{ R\cup\{v\}})$, \eqref{hram} holds. \end{prop}

\begin{proof}  The proof is    the same with   the proof  of Proposition \ref{yll} with the following exceptions:
\begin{itemize}\item Lemma   \ref{YZZ6.8} should be replaced by Lemma   \ref{YZ7.41};
\item
in the  claim in the  reduction step (I), remove ``or by $g_vk_v  $ for   $k_v\in K_v^{\max}$";
\item in the   reduction step (II), ``\Cref{mainv}" should be replaced by
``\Cref{mainv}  and  
that $ \te  \lb0,m(a)g,\phi^{v}\otimes \phi_v'\rb= \te  \lb0,g,\phi^{v}\otimes \phi_v'\rb $ for $a\in E^\times$". \end{itemize}
\end{proof} 

\subsubsection{Finite places of $F$ ramified in $E$}\label{ramplace}
For such a $v$,  $\Lambda_v$  is a $\pi_v$-modular or almost $\pi_v$-modular lattice.
This case is more complicated.

We still need formal uniformization using 
Rapoport-Zink spaces.  
For a non-negative integer $m$, let
$\cN_m$    be   the  exotic smooth  relative unitary Rapoport-Zink space  
of signature $(m,1)$ over $\Spf \cO_{E_v^\ur}$. See  \cite[Section 6,7]{RSZ-AT}, \cite[3.5]{RSZ17} and \cite[2.1]{LL2}.  It   will also be specified below. It
is formally smooth over $\Spf \cO_{E_v^\ur}$ of  relative dimension $m$.  
(Note that the case $m=0$ is not covered in either \cite{RSZ17} or  \cite{LL2}, but is specifically indicated in \cite[Example 7.2]{RSZ-AT}.)
Let   $\cN=\cN_n$.
%\begin{rmk} Note that the case $m=0$ is not covered in \cite{RSZ17}\cite{LL2}, but is specifically indicated in \cite[Example 7.2]{RSZ-AT}. %However, it is related to \cite[Section 6,7]{RSZ17} by the discussion in \cite[Example 12.2]{RSZ-AT}. We will use this connection is our ``Fifth" step below.
%\end{rmk}
We will use  $\cN $ for the formal uniformization of  $\cX_K$.    The analog of the formal uniformization \eqref{ascycle} of $\cP$ using 
$\cN_0$ is more subtle: 
we will   use 
$\cN_1$ to define  morphisms $\cN_0\to\cN_1\to \cN$ which will lead us to the   formal uniformization \eqref{ascycle1} of $\cP$. 
\begin{rmk}  The reason for this subtlety might be explained as follows.   
  Recall that  the construction of $\cP$  requires an additional rank 2  sub-lattice   $\Lambda_{v,1}$ as in \ref{The case n is odd*} at each finite place, which is a direct summand of $\Lambda_v$. And $E_v\Lambda_{v,1}$ contains a distinguished vector $e_v^{(0)}$ of unit norm. One might consider $\cO_{E_v}e_v^{(0)}\subset E_v\Lambda_{v,1}$ and $
\Lambda_{v,1}\subset \Lambda_v$  as being parallel to the morphisms $\cN_0\to\cN_1$ and $\cN_1\to \cN$. 
Note that $\cO_{E_v}e_v^{(0)}$ is not contained in $\Lambda_{v,1}$ ( or $\Lambda_v$). This makes it nontrivial to defined a
morphism $\cN_0\to\cN_1$ or $\cN_0\to \cN$. See also \cite[Remark 12.3]{RSZ-AT}.  The morphism $\cN_0\to\cN_1$  we use  is given by \cite[Section 12]{RSZ-AT}. 
\end{rmk}   

We have 6 steps before the main result \Cref{yll1} of this subsubsection \ref{ramplace}.

First, we specify $\cN_1$, $\cN$ and    $\cN_1\to \cN$. Assume that $\varpi_{E_v}^2=\varpi_{F_v}$.
The framing object $\BX_1$  for the deformation space  $\cN_1$ is the  Serre tensor $\cO_{E_v}\otimes_{\cO_{F_v}}\ol\BE$, which is the conjugate of the framing object  \cite[(3.5)]{RSZ17}, 
with the polarization  
conjugate to  the one   in \cite[(3.6)]{RSZ17}. 
In the case that $n$ is odd (the case of Lemma \ref{jac} (1)), the framing object for $\cN$ is $\BX_n:=\BX_1\times(\BE^2)^{(n-1)/2}$ with the product polarization $\lambda_n$, where the polarization on $\BE^2$ is given by  \begin{equation}\label{lambdadef} \lambda =   \begin{bmatrix}0&\lambda_{\BE}\iota(\varpi_{E_v})\\
-\lambda_{\BE}\iota(\varpi_{E_v})&0\end{bmatrix}.\end{equation} 
In the same way, we have  a polarization 
$\wt{\lambda} $  
on $\cE^2$ using $\lambda_\cE$.  This gives us  a morphism $\cN_1\to \cN$ by $X\mapsto X\times (\cE^2)^{(n-1)/2}$ with the polarization $\wt{\lambda} $ on each of $\cE^2$.
In the case that $n$ is even (the case of Lemma \ref{jac} (2)), the framing object for $\cN$ is  $\BX_n:=\BX_1\times(\BE^2)^{(n-2)/2} \times \BE$ where the polarization  $\lambda_{\BE}'$ on the last component is a multiple of $\lambda_{\BE}$ so that the induced hermitian pairing on $\Hom(\BE,\BE)_{\BQ}$  (defined as in \eqref{hpair})  has determinant  $q(e_v)$ as in Lemma \ref{jac} (2).
This gives us a morphism $\cN_1\to \cN$ by $X\mapsto X\times (\cE^2)^{(n-1)/2}\times \cE$ with the unique lifting of  $\lambda_{\BE}'$ on the last component $\cE$.
% It is clearly a closed embedding.

Second, the uniformizations  of $\wh{\cX_{K, \cO_{E_v^\ur}}^{ss}}$ and $Z_t(\phi)^\zar$ are as follows.
By  \cite[(3.10)]{RSZ17},  $V(E_v)\cong \Hom_{\cO_{E_v}}(\BE,\BX_n)_{\BQ} $
and  $U(V(E_v))$ is isomorphic to the group of $\cO_{E_v}$-linear  self-quasi-isogenies of $\BX_n$ preserving $\lambda_n$.  In particular $U(V)$ acts on $\cN$.   The analog of the formal uniformization \eqref{ssun*} of $\wh{\cX_{K, \cO_{E_v^\ur}}^{ss}}$ holds by \cite[(4.9)]{LL2}.
For every $x\in V(E_v)-\{0\} $, we have the  Kudla-Rapoport divisor $\cZ(x)$ on $ \cN $, which is a (possibly empty) relative Cartier divisor \cite[Lemma 2.41]{LL2}.
The analog of  Proposition \ref{czss} holds by 
\cite[Proposition 4.29]{LL2} combined with the argument in the proof of  Proposition \ref{czss}.
Though  \cite{LL2}   only uses even dimensional hermitian spaces,  the specific results that we cite hold in the general case by the same proof.

Third, we recall  the morphisms   $\cN_0\to \cN_1$  defined in \cite[Section 12]{RSZ-AT}. 
This   is rather complicated for general $\cN_{2m}\to \cN_{2m+1}$. Fortunately, in our case, we have the following convenient description. By \cite[Example 12.2]{RSZ-AT},  $ \cN_1$ is  isomorphic to the disjoint union of two copies of 
the Lubin-Tate deformation space for the formal $\cO_{F_v}$-module $\BE$. We write $\cN_1=\cN_1^+\coprod \cN_1^{-}$. 
Recall that $B$ is the      unique  division quaternion algebra over $F_v $, and its maximal order $\cO_B$ is the endomorphism ring of $\BE$. 
For $c\in  B^\times$, we   have two closed embeddings (moduli of the canonical lifting) $\cN_0\to \cN_1^{\pm}$ associated    to $c\iota c^{-1}:E_v\incl B$.  Let $\cN_0^{c,\pm}$ be the union of the images.

Fourth, we  need  to specify $c$ so that we can use $\cN_0^{c,\pm}$ to uniformize $\cP$. See \eqref{ascycle1} below.  
Let $e_v^{(1)}$ be as in   \ref{A useful description}. 
Then
$$B=\Hom_{\cO_{F_v}}(\BE,\ol\BE)_\BQ\cong \Hom_{\cO_{E_v}}(\BE,\BX_1)_\BQ\cong W(E_v)\oplus E_v e_v^{(1)},$$ 
where  the  middle is the adjunction isomorphism, and
the last isomorphism is compatible with   $\Hom_{\cO_{E_v}}(\BE,\BX_n)_{\BQ}\cong  V(E_v)$.
And the hermitian form on $W(E_v)\oplus E_v e_v^{(1)}$ corresponds to $-2\varpi_{F_v}\Nm_B $ (see the proof of \cite[Lemma 3.5]{RSZ17}), where $\Nm_B$ is the reduced norm on 
$B$.
Let $c$  correspond to $\varpi_{E_v} e_v^{(1)}$. Since $q\lb e_v^{(1)}\rb\in \cO_{F_v}^\times$ (see  \ref{Lattices at ramified places}), $c\in \cO_B^\times$ (note that $v\nmid 2$ here).

Fifth, we uniformize  $\cP$.  
By   \Cref{cpcp}, we have another description of
  $\cP$ via the diagram \eqref{5terms} of morphisms of  Shimura varieties. See also Remark \ref{multichain}. 
Comparing it with the moduli interpretation of $\cN_0\to \cN_1$ in \cite[Proposition 12.1]{RSZ-AT},
we have the following analog of \eqref{ascycle}: 
\begin{align}\label{ascycle1} \cP_{\cO_{E_v^\ur}} =\frac{1}{2d_{K_\BW^{(0)}}}  U(W) \bsl \lb \cN_{0,\cO_{E_v^\ur}}^{c,\pm}\times U \lb \BW^{\infty,v}\rb/ K_\BW^v\rb, \end{align} 
where $d_{K_\BW^{(0)}} $ is  
is   the degree of the   fundamental cycle of  $\Sh(\BW)_{K_\BW^{(0)}}$. See \eqref{5terms}.
Here the extra factor 2 comes from  Lemma \ref{index2}.

Finally,    we compute $ \lb \cZ(x)\cdot \cN_{0}^{c,\pm}\rb_{\cN}$. 
Since  $c\in \cO_B^\times$, by \cite[Lemma 6.5, Proposition 7.1]{RSZ17},  
\begin{equation}\label{ssuni411}\cN_0^{c,\pm}=\cY\lb \varpi_{E_v}e_v^{(1)}\rb.\end{equation}
Here $\cY\lb\varpi_{E_v}e_v^{(1)}\rb$ is the Kudla-Rapoport divisor on $\cN_1$, where $\varpi_{E_v}e_v^{(1)}$ lifts.
Let $\BX^\perp$ be the direct complement of $\BX_1$ in 
$\BX_n$, i.e., $\BX^\perp=(\BE^2)^{(n-1)/2}$
if $n$ is odd, and $\BX^\perp=(\BE^2)^{(n-2)/2} \times \BE$  is $n$ is even. 
Let    $  \Lambda_{v,1}^\perp $   be as in \ref{A useful description}.
By Lemma \ref{jac} and \eqref{lambdadef},
$$\Hom_{\cO_{E_v}}(\BE,\BX^\perp) \subset \Hom_{\cO_{E_v}}(\BE,\BX_n)_{\BQ}\cong  V(E_v)$$ 
corresponds to $\Lambda_v^\perp$.
Then (similar to the deduction of \eqref{ssuni3})
by \eqref{ssuni411}
and Gross' result on  canonical lifting  \cite[Proposition 3.3]{Gro}, we  have    \begin{equation}\label{ssuni4} \lb \cZ(x)\cdot \cN_0^{c,\pm}\rb_{\cN} = 2(v({q(x_1)})+1) 1_{\cO_{F_v}}({q(x_1)}) 1_{{\varpi_{E_v}
\cO_{E_v}e_v^{(1)}\oplus \Lambda_v^\perp}} (x_2)
\end{equation} for  $ x=(x_1,x_2)\in  W(E_v)\oplus V^\sharp (E_v) $  with $x_1\neq 0$.
Here    the extra factor 2 comes from  that $\cN_0^{c,\pm}$ has 2 components.

\begin{prop} \label{yll1}Assume that $g\in P (\BA_{F,R\cup\{v\}})G(\BA_F^{ R\cup\{v\}})$ and $k=1$. Then  \eqref{hram} holds. \end{prop}
\begin{proof} As an  analog of   \eqref{cormn}, we have  \begin{align*}\lb  {Z_t(\omega(g) \phi)}^\zar\cdot \cP   \rb_{\cX_{K,\cO_{E_v}} }  =\frac{1}{2\Vol([U(W)])} \int_{ h\in [U(W)]}    \sum_{x\in    V^t}    \lb\cZ(h_v^{-1}x) \cdot\cN_0\rb_{\cN}\omega(g) \phi^v(h^{-1} x)   dh.
\end{align*}
The rest of the proof is    the same as  the proof of Proposition \ref{yll+},
after replacing Lemma    \ref{YZ7.41}  by Lemma   \ref{YZ7.41'}.    \end{proof}

\subsubsection{Infinite places of $F$}  
Let $v\in \infty$.
  Under the complex uniformization \eqref{compl} of   $ \Sh(\BV)_{K,E_v}$,   $$\cP_{E_v}=  \frac{1}{d_{\BW,K} }    U(W)\bsl (\{o\} \times  U(\BW^\infty)/K_\BW),$$
where     $o:=[0,. . .,0]\in  \BB_n$, and $d_{\BW,K} $ is   the degree of the fundamental cycle of  $\Sh(\BW)_{K_\BW}$. %. See Definition \ref{dKPK}). 
For $g\in P (\BA_{F,R})G(\BA_F^{ R})$,
by    the definition of $\cG_{Z_t( \cdot)_{E_v},s}$ 
(above \eqref{GGZ}), 
a direct computation gives   \begin{align}\label{Hexp}  
\int_{\cP _{E_v}}\cG_{Z_t(\omega(g) \phi)_{E_v},s}  =\frac{ W^\fw_{v,t}(g_v) }{ \Vol([U(W)]) }  \int_{  [U(W)]}
\sum_{x\in    V^t-V^\sharp
}  G_{h_v^{-1}x,s}(o) \omega(g^v) \phi ^{ v}(  h^{v,-1} x)dh. \end{align}

Now we compare  the inner  sums  in \eqref {Evdechol} and  \eqref{Hexp}.  
Recall the involved functions $P_s$ and $Q_s$. See \eqref {Pst} and \eqref{Qs} respectively.  
From \eqref{Pst}, we have $$P_s(u)=\frac{1}{(s+n)u^{s+n}}  F\lb s+n, s+n, s+n+1, \frac{-1}{u}\rb,$$
where $F$ is the hypergeometric function.
In particular,
\begin{equation} \label{Pinf} P_s(u)=\frac{1}{(s+n)u^{s+n}}+O\lb \frac{1}{u^{s+n+1}}\rb , \ u\to \infty,\end{equation} 
where the constant for $O(\cdot)$ is uniform near $s=0$. 
We also have 
   \begin{equation} \label{P0} P_0(u)=\log(1+u)-\log u-\sum_{i=1}^{n-1}\frac{1}{i(1+u)^i},\ u>0.\end{equation} 
%(In fact,  \eqref{Pinf} and \eqref{P0} can be proved directly, without the hypergeometric expression of $P_s$.And they are the only properties that we will need later.)
From \eqref{Qs}, we have 
    \begin{equation} \label{Qinf} Q_s(1+u)=
\frac{\Gamma(s+n)\Gamma(s+1)}{\Gamma(2s+n+1)u^{s+n}}  +O\lb \frac{1}{u^{s+n+1}}\rb , \ u\to \infty\end{equation} 
where the constant for $O(\cdot)$ is uniform near $s=0$. We also have 
   \begin{equation} \label{Q0} Q_0(1+u)=\log(1+u)-\log u-\sum_{i=1}^{n-1}\frac{1}{i(1+u)^i},\ u>0.\end{equation}

\begin{lem} \label{difflem} For $s_0>-1$, on $\{s\in \BC,\Re s>s_0\}$,   the   sum
\begin{equation*} 
\sum_{x\in  V {\qclE} ^t-V^\sharp }\lb \lb\frac{\Gamma(s+n+1)\Gamma(s+1)}{\Gamma(2s+n+1) }\rb^{-1}G_{h_v^{-1}x,s}(o)- \lb \frac{\Gamma(s+n)}{\Gamma(n) (4\pi)^{s}}\rb^{-1} \wt\tw_{s} (h_v^{-1}x)\rb  \omega(g^v) \phi^v \lb h^{v,-1} x\rb   \end{equation*}
converges uniformly and absolutely.
And its value at $s=0$ is 0.
\end{lem}
\begin{proof}
We compare the sum in the lemma with the sum in \eqref{Hexp}, which is absolutely convergent  by Lemma \ref{gaodingle}.
By  \eqref{Pinf} and  \eqref{Qinf},  
%(note that both the $O(\cdot)$ in \eqref{Pinf}  and \eqref{Qinf} are uniform near 0), 
the  sum in the lemma for $s$ is 
dominated by a multiple of the sum in \eqref{Hexp} with $s$ replaced by $s+1$. 
The convergence  in the lemma follows.
By \eqref{P0}
and  \eqref{Q0}, the value    at $s=0$ is 0.
\end{proof}  
By Theorem \ref{OTthm},  the integration on the right hand side of \eqref{Hexp}  admits a meromorphic  continuation  to $s$ around 0 with a simple pole at $s=0$.
\begin{proof}[Proof of Lemma \ref{sconv}]
For this moment, we consider a general $ \phi  \in \ol\cS(\BV) $ (without any regularity assumption).  The above discussion still holds replacing $Z_t(\phi)$ by 
\begin{equation*}Z_t(\phi)-\sum_{x\in K\bsl K  V^\sharp, \ q(x)=t}\phi (   x) Z(x)  \end{equation*}
Then Lemma \ref{sconv} follows from \eqref{Hexp}   and  the first part of
Lemma \ref{difflem}. 
\end{proof}

%Recall that  $\wt{\lim\limits _{s\to 0}}$ denotes the constant term at $s=0$.
Recall the definition of  $\cG_{Z_t( \cdot) _{E_v}}^\aut$ in \eqref{GGZ}. Then  
\begin{align*} &2\int_{\cP _{E_v}}\cG_{Z_t( \omega(g)\phi) _{E_v}}^\aut = 2\wt{\lim\limits _{s\to 0}}  \lb\frac{\Gamma(s+n+1)\Gamma(s+1)}{\Gamma(2s+n+1) }\rb^{-1} \int_{\cP _{E_v}}\cG_{Z_t( \phi)_{E_v},s}\\
= & \wt{\lim _{s\to 0}}\lb \frac{\Gamma(s+n)}{\Gamma(n) (4\pi)^{s}}\rb^{-1}\te'_{t,s} (0,g,\phi)(v) =\wt{\lim _{s\to 0}}\te'_{t,s} (0,g,\phi)(v),\end{align*}
where     both multipliers  $(\cdot)^{-1}$ go to 1 as $s\to 0$ (this gives the first and third ``="), and the second ``=" follows from  Lemma \ref{difflem}. 
Thus we have proved the following proposition. 

\begin{prop} \label{yllinf}  For $g\in P (\BA_{F,R})G(\BA_F^{ R})$, \eqref{hinf} holds for $v$. \end{prop}

\subsection{Improper intersections } \label{Intersections with CM cycles}
 In this subsection, we prove the arithmetic mixed Siegel-Weil formula (Theorem \ref{nct}). 
 The proof starts in \ref{PWpull}. Before that, let us discuss  the strategy.
 \begin{lem}\label{Fourier}

Let $Y=X\oplus X'$ be  the orthogonal direct sum of two non-degenerate  quadratic  spaces   over a non-archimedean local field  of characteristic $\neq 2$. Let $\wh \cS(Y- X')_{\ol\BQ}$ be the space of the Fourier transforms of functions in  the  space  $\cS(Y- X')_{\ol\BQ}$ of $\ol\BQ$-valued Schwartz functions on $Y$ supported on $Y-X'$ (the Fourier transforms  are clearly also $\ol\BQ$-valued).
 Then %the  space $\cS(Y)_{\ol\BQ}$ of $\ol\BQ$-valued Schwartz functions on $Y$ satisfies
$$\cS(Y)_{\ol\BQ}=\cS(Y- X')_{\ol\BQ}+\wh \cS(Y- X')_{\ol\BQ},$$

\end{lem}
\begin{proof} Since Fourier transform respects  orthogonal direct sum, one may assume that $X'=\{0\}$.
Then the lemma is well-known and also easy to check directly. 
\end{proof}
 %For our application, we will let $X$ be $\BW_v$ and let  $X'$ be  $V^\sharp(E_v)$ for $v$ split in $E$.
%In this case, we can also consider an  ``enhanced version  of Lemma \ref{Fourier}", by  removing the orthogonal complements of several $k_v^{-1}\BW_v$'s in $\BV(E_v)$, rather than just $\BW_v$. Moreover, instead of Fourier transform, which is the action of an element of $  G(F_v)$ 
%under the Weil representation (up to the Weil index),   one can use the whole $G(F_v)$.   However, we do not see the truth/a direct proof of the ``enhanced version of Lemma \ref{Fourier}".

 We will use the following   notation. For a finite set $S $ of finite places of $F$, let  
 \begin{equation}
 \ol\BQ \log {{S}}:=\ol\BQ\{\log p: v|p\text{ for some }v\in S\}\subset \BC.\label{logS}
 \end{equation}

The main difficulty in proving   \Cref{nct}  is from improper-intersections. 
 However, if we choose   a  $\ol\BQ$-valued pure tensor $\phi$  that is $\BW_v$-regular at some  places $v$ in $S$, 
  we can    prove     a $\MOD \ol\BQ\log S$-version of   Theorem \ref{nct} for   $\phi$  ,   see \Cref{tech1}.   
 Note that both the CM cycle and the regularity assumption are   associated to $\BW$. 
The same result holds if we replace $\BW$ by some $\BW'$ (and 
use the corresponding CM cycle  and  regularity assumption). 
Accordingly, we make   \Cref{tab:my_label}.
 \begin{table}[thb] \centering
\caption{}
    \label{tab:my_label}
    
\begin{tabular}{cc|c|c|l}
\cline{3-4}
& & \multicolumn{2}{ c| }{CM cycle} \\ \cline{3-4}
& &$ \BW$ &$ \BW'$    \\ \cline{1-4}
\multicolumn{1}{ |c  }{\multirow{2}{*}{Regularity} } &
\multicolumn{1}{ |c| }{$ \BW$} & L &  &         \\ \cline{2-4}
\multicolumn{1}{ |c  }{}                        &
\multicolumn{1}{ |c| }{$ \BW'$} &  B  & L$'$ &        \\ \cline{2-4} 
\multicolumn{1}{ |c  }{}                        &
\multicolumn{1}{ |c| }{No} & C  &    &     \\ \cline{1-4} 

\end{tabular}
\\
\footnotesize{The table displays  the  $\MOD \ol\BQ\log S$-version of   Theorem \ref{nct} under different conditions. 
The second row indicates that we consider the CM cycle associated to  $\BW$ or $\BW'$. 
The second column   indicates that we impose the regularity assumption on $\phi$ associated to  $\BW$ or $\BW'$,  or no regularity assumption.
Then a cell   indexed by them indicates    the  $\MOD \ol\BQ\log S$-version of   Theorem \ref{nct}
for  the corresponding CM cycle under the corresponding regularity assumption on $\phi$. }
 \end{table}
\Cref{tech1} gives the cell L of  \Cref{tab:my_label}. Replacing $\BW$ by $\BW'$, we get L$'$.  
 
 We want to arrive at the cell C (proved in     \Cref{nctprop}),  the   $\MOD \ol\BQ\log S$-version of   Theorem \ref{nct} for the general $\phi$.  We use the cell  B  the bridge from L to C. The relation between  this cell B and L$'$ is the ``switch  CM cycles" indicated in the last   paragraph above
\ref {Non-holomorphic variant}. The relation between B and L could be considered as  ``switch  regularity assumptions".

Instead of considering this cell B directly,  we consider the generating series of arithmetic intersection numbers with  the difference of the two CM cycles, which   is  modular by \Cref{fdimeadd}.
Then using   Lemma \ref{Fourier},  we prove the $\MOD \ol\BQ\log S$-version of   Theorem \ref{nct} for the general $\phi$, after replacing the CM cycle by
 the difference. See \Cref{tech2}.  
 This gives P of \Cref{tab:my_label1}.   
 Then the combination of  L$'$ and P/$\BW'$     proves  the cell B. Here P/$\BW'$ is the special case of P under the extra regularity assumption associated to $\BW'$.
    \begin{table}[thb]
\centering
\caption{}
    \label{tab:my_label1}
\begin{tabular}{cc|c|c|c|l}
\cline{3-5}
& & \multicolumn{3}{ c| }{CM cycle} \\ \cline{3-5}
& &$ \BW$ &$ \BW'$ & Difference  \\ \cline{1-5}
\multicolumn{1}{ |c  }{\multirow{2}{*}{Regularity} } &
\multicolumn{1}{ |c| }{$ \BW$} & L &   &   &       \\ \cline{2-5}
\multicolumn{1}{ |c  }{}                        &
\multicolumn{1}{ |c| }{$ \BW'$} & L$'$+P/$\BW'\Rightarrow$ B & L$'$ & P/$\BW'$&      \\ \cline{2-5} 
\multicolumn{1}{ |c  }{}                        &
\multicolumn{1}{ |c| }{No} & C  &    & P &      \\ \cline{1-5} 

\end{tabular}
  \end{table}

 %$\MOD \ol\BQ\log S$-version of   Theorem \ref{nct} for  the CM cycle defined using $\BW$ under  the regularity assumption on $\phi$ defined using $\BW'$. See  \Cref{entry20}, and we will not use  the combination L+P/$\BW$.) We will use 
 % L$'$+P/$\BW'$  for various $\BW'$ to obtain  C (and we will not use the empty cell though it can be proved similarly).

 We need to remove ``$(\MOD \ol\BQ\log S)$".
We will use the following theorem. It is a  corollary of Baker's celebrated theorem on transcendence of logarithms of algebraic numbers
(see \cite[Theorem 1.1]{Waldsch}), and the fact that logarithms of prime numbers are $\BQ$-linearly independent. 
\begin{thm}\label{baker} Let $p_1,...,p_m$ be distinct prime numbers, then $\log p_1,...,\log p_m$ are  $  \ol\BQ$-linearly independent. 
\end{thm}
%\begin{rmk}By a detailed study of the algebraic structure of \eqref{fh}, one  might be able to avoid using this  big theorem. However, we choose  not to make our discussion more complicated.  \end{rmk}

\subsubsection{Set-up}\label{PWpull}

  We need the following convenient notation. For a set $T$ of   finite places of $F$, 
let \begin{equation*}\BG_T= P\lb\BA_{F,{\Ram}\cup T}\rb G\lb\BA_F^{{\Ram}\cup T\cup \infty }\rb.\label{logT}\end{equation*} 
For example, $ \BG_{ \emptyset}=G\lb\BA_F^{{\Ram} \cup \infty }\rb P\lb\BA_{F,{\Ram}}\rb $ is the group appearing in Theorem \ref{nct}.   
%We will take $T$ to be a subset of $S$.  In many cases below, we can  take $T$ to be $S$ with no harm to the proof of \Cref{nctprop}. However, we prefer to make the dependence of our reasoning  on this set $ T$ more transparent.

Below, let     $\phi\in \ol\cS\lb \BV  \rb ^{{\wKL}}$ be a pure tensor such that $\phi^\infty$ is $\ol\BQ$-valued  and $\phi_v=1_{\Lambda_v}$ 
      for every finite place $v$ of $F$ nonsplit in $E$.    
       Let  $K\in {\wKL}$  stabilize $\phi$.
 Let  $S$ be a set of finite places of $F$, and  $K'=K_S K_\Lambda^S$.
Let
\begin{equation*} \cP_\BW=\pi_{_K,_{K'}}^*\cP_{\BW,K'} .\end{equation*}
 Let $f_{\BW}=f_{\BW,K'}^K$ be  defined as in \eqref{fh}.
We   remind the reader  that we will    use  other incoherent hermitian spaces $\BW'$ over $\BA_E$ of dimension 1. These notations apply to $\BW'$ in the same way.

Let $t\in F_{>0}$.
Our goal is to prove that for  $g\in  \BG_{ \emptyset}  $ and a suitable set $S$ of finite places of $F$,  \begin{align}\label{entry}  2  z_t(g, \phi^\infty)_\fe^{\cL,\aut} \cdot \cP_\BW  \MOD\ol\BQ\log {{S}}= f_{\BW,t}^\infty(g) \MOD\ol\BQ\log {{S}} .\end{align}  

We also introduce an  equation  equivalent to  \eqref{entry}, and both will play roles in the proof of \Cref{nct}.
 The $\psi_t$-Whittaker function of the right hand side of \eqref{fh} (which is the definition of $f_{\BW}$),  coincides with the right hand side of \eqref{Final}
up to the last term. 
Comparing the definition of 
$z_t(g, \phi)_\fe^{\cL,\aut}$     with the left   hand side of \eqref{Final}, \eqref{entry} is equivalent to    the following equation:
\begin{equation}  \label{entryeqmod}
\begin{split}
%  &\lb  2  [Z_t(\omega(g) \phi^\infty ) ^{\cL,\aut}] \cdot  \cP _\BW \rb\MOD\ol\BQ\log {{S}} \lb= \frac{2  [Z_t(\omega(g) \phi ) ^{\cL,\aut}] \cdot  \cP _\BW}{W^{\fw}_{\infty, t}(1)}\MOD\ol\BQ\log {{S}}\rb
&    \frac{2  [Z_t(\omega(g) \phi ) ^{\cL,\aut}] \cdot  \cP _\BW}{W^{\fw}_{\infty, t}(1)}\MOD\ol\BQ\log {{S}}
= \\
 &\sum _{k\in U(\BW)\bsl K'/K} \left(- \frac{\te_{t,\qhol}'(0,g,\omega(k)\phi)}{W^{\fw}_{\infty, t}(1)}-\sum_{v\in {\Ram}}  \frac{ \te _t \lb0,g,\omega(k)\phi^{v}\otimes \phi_v'\rb}{W^{\fw}_{\infty, t}(1)}\right) \MOD\ol\BQ\log {{S}}.
\end{split} 
    \end{equation} 
        %   Note that   we must keep the denominator $W^{\fw}_{\infty, t}(1)$  due to the use of ``$\MOD\ol\BQ\log {{S}}$". 

\subsubsection{Regular test functions}\label{Arithmetic mixed Siegel-Weil formula for regular test functions}
We use Assumption \ref{regasmp} on regularity here. %, but not after the following proposition  unless otherwise specified. 

\begin{lem} \label{tech1}  Assume that   $S$  contains a  nonempty subset $R$ 
  such that $\phi_v$ is $\BW_v$-regular  for $v\in R$, i.e., 
   Assumption \ref{regasmp} holds. Then  for    $g\in \BG_{R}$,  
\eqref{entry} holds.
Equivalently, \eqref{entryeqmod} holds.
 \end{lem} 
 \begin{rmk} The statement in Lemma \ref{tech1}  becomes more transparent  if   $S=R$.  However, we need the flexibility
 to vary such $R$ in $S$ later.
 \end{rmk} 
\begin{proof}Recall  \eqref{PWK},
\begin{equation*} 
\pi_{_K,_{K'}}^*\cP_{\BW,K'}=\sum _{k\in (U(\BW)\cap K')\bsl K'/K} \frac{d_{k^{-1}\BW,K}}{ d_{\BW,K'}}\cP_{k^{-1}\BW,K}.
 \end{equation*}
Here we choose $k$ such that  $k_v=1$ for $v\in S$ or
  nonsplit in $E$.   (This is possible since $K'=K_S K_\Lambda^S$ and $K_v=K_{\Lambda,v}$ for $v$ nonsplit in $E$.)
  Then  Assumption \ref{regasmp} still   holds with $\BW$ replaced by  $k^{-1}\BW$. So  by  \Cref{221}, $Z_t(\omega(g) \phi^\infty)  $ and $\cP_{k^{-1}\BW,K}$
do not meet on the
generic fiber.   Then by \Cref{expectlem},     we can apply \Cref{pull1cor} at  finite places $v\not \in S$  (nothing happens if $K_v=K_{\Lambda,v}$). Then 
we have  
\begin{equation*}\label{entryeq1} \begin{split}
&  [Z_t(\omega(g) \phi^\infty) ^{\cL,\aut}] \cdot \cP_{\BW}  \MOD\ol{\BQ}\log S =\\
& \lb \sum_{\substack{v\not\in S\cup\infty,\\  \text{nonsplit in }E}}\lb Z_t(\omega(g) \phi^\infty) ^\zar \cdot \cP_{k^{-1}\BW,K} \rb_{\cX_{K,\cO_{E_v}}}\log  q_{E_v} 
+\sum_{v\in \infty } \int_{\lb \cP_{k^{-1}\BW,K}\rb_{E_v} }\cG_{Z_t(\omega(g)\phi^\infty) _{E_v}}^\aut \rb  \MOD\ol{\BQ}\log S.
\end{split}
\end{equation*}
Comparing this equation with \eqref{E'dec1hol},
\eqref{entryeqmod} is implied by  \Cref{3=prop}.
   \end{proof} %\begin{rmk} \label{feature}Let us address two features of  $\cP_\BW$. First, we use  the pullback   rather than $\cP_{\BW,K}$. The reason is explained in \Cref{uselat}.Second, the level does not change  over $S$. This will be explained in the next  remark.\end{rmk}
  
   %   Let $E_\Phi$ be as in \ref{RSZ-arithmetic-diagonal-cycles}. %, which is a finite unramified extension of the reflex field of the RSZ integral model. 

%Note that in this case, $\pair{x,x}\in \varpi \Nm E^\times$ for $x\in V^\sharp$.  Indeed, we may simply take $W=E$ with $\pair{x,y}=x\ol y$ and $\Lambda=\cO_E.$   Also $V^\sharp=E$ with  $\pair{x,y}=\varpi x\ol y$.
\subsubsection{CM cycles of  degree 0}\label{CM cycle of  degree 0}
  Let $\BW'$ be another  incoherent hermitian subspace of $\BW$ and $\cP_{\BW'}$   the {CM cycle}     defined accordingly as in  \ref{CM cycles}.     
Since the automorphic Green function is admissible and $ \cP_{\BW ,E}-\cP_{\BW',E} $ has degree 0, by Lemma \ref{fdimeadd},  
\begin{equation}\label
{zPA}   z(\cdot, \phi )_{\fe,a}^{\cL,\aut}\cdot \lb\cP_{\BW }-\cP_{\BW'}\rb \in \cA_{\hol}(G,\fw).
\end{equation}
Moreover,   \eqref{zPA}  is independent of the choice of $a$.  We   abbreviate $z(g, \phi)_{\fe,a}^{\nadm}$ to $z(g, \phi)_{\fe}^{\nadm}$. 
The 0-th Fourier coefficient   of \eqref{zPA}  is 0.  
Indeed, by \Cref{rth2}, the action of $K_\Lambda$  on 
$  \wh\Ch^{1}_{\adm,\BC}(\wt\cX)$    fixes $c_1(\ol\cL_K^\vee)$. 
The vanishing of  the 0-th Fourier coefficient  follows from \Cref{Kud2.2}.

%The general case follows from this case by the projection formula and \eqref{pushp}. 

\begin{prop} \label{tech2}  Assume that the cardinality of $S$ is at least 2. 
Given $\phi$ as in \ref{PWpull},  if $K_S$ is small enough (depending on $\phi_S$),  then
for all $g\in G(\BA_{F}^\infty)$, %we have
\begin{equation} \label{tech2eq}2   z_t(g, \phi^\infty)_\fe^{\cL,\aut} \cdot (\cP_{\BW } -\cP_{\BW'}) \MOD\ol\BQ\log {{S}}=\lb  f_{\BW ,t}^\infty (g)-f_{\BW',t}^\infty(g)\rb \MOD\ol\BQ\log {{S}} .\end{equation}

\end{prop}

\begin{proof}%[Proof of \Cref{tech2}] 

For $G^\der=SU(1,1)$, 
$G(F_v)=G^\der(F_v)   K^{\max}_v$.  By  Lemma \ref{wr},   it is enough to prove   \eqref  {tech2eq} for    $g\in G^\der(\BA_F) $.

We need a lemma whose statement requires some more notations.
Since $G^\der \cong \SL_{2,F}$, by
  the $q$-expansion principle for   $\SL_{2,F}$   \cite{Chai},   we have 
$$\cA_{\hol}(G^\der,\fw)=\cA_{\hol}(G^\der,\fw)_{\ol\BQ}\otimes _{\ol\BQ}  \BC.$$  
Here $\cA_{\hol}(G^\der,\fw)_{\ol\BQ}$ is as in   \ref{Holomorphic  automorphic  forms} with $G$ replaced by $G^\der$ and $\fw$ is  understood as the restriction of $\fw$ to $G^\der(F_\infty)\cap K^{\max}_v$ for $v\in \infty$. Thus, we have   Fourier coefficients as in  \ref{Holomorphic  automorphic  forms}.   
For $f\in \cA_{\hol}(G^\der,\fw)$, let $[f ]$ be its image in $ \cA_{\hol}(G^\der,\fw)_{\ol\BQ}\otimes _{\ol\BQ} \BC/ \ol\BQ\log {{S}} $. 
Then the  $ \BC/ \ol\BQ\log {{S}}$-valued locally constant function $f_t^\infty  \MOD\ol\BQ\log {{S}}$ on $G^\der(\BA_F^\infty)$  coincides with the $t$-th Fourier coefficient of $[f ]$.

\begin{lem} \label{teach21} Assume that $\phi_R$ is $\BW_R$-regular
where  $R\subset S$ consists of  a single element, and 
$\phi_{R'}$ is $\BW'_{R'}$-regular where
$R'  \subset S\bsl R$ consists of  a single element. Then we have the following equality in $ \cA_{\hol}(G^\der,\fw)_{\ol\BQ}\otimes _{\ol\BQ} \BC/ \ol\BQ\log {{S}} $ after restriction from $G$ to $G^\der$: 
\begin{align}\label{entry1} \left[ 2   z(\cdot , \phi)_{\fe}^{\cL,\aut}\cdot \lb\cP_{\BW }-\cP_{\BW'}\rb\right]= \left[ f_{\BW} - f_{\BW'}\right].  \end{align}    
%Here $f_{\BW^{(i,j)}} $ is   defined using \eqref{fh}, with $K_\Lambda$ replaced by $f$ with $\BW$ replaced by ${\BW^{(i,j)}}$. 

\end{lem}

\begin{proof}
Consider  the  difference  $$f=2   z(\cdot , \phi )_{\fe}^{\cL,\aut}\cdot \lb\cP_{\BW}-\cP_{\BW'}\rb- \lb f_{\BW} - f_{\BW'}\rb
\in \cA_{\hol}(G,\fw)  $$   of the two sides of \eqref{entry1}, before passing to $\cA_{\hol}(G^\der,\fw)_{\ol\BQ}\otimes _{\ol\BQ} \BC/ \ol\BQ\log {{S}} $. 
By the cuspidality of  \eqref{zPA} and Lemma \ref{wr1} (3),   the $0$-th Fourier coefficient    $ f_0^\infty(g)  =0$  for $g\in \BG_{\{v_1,v_2\}} .$ % For all 
Write  \begin{equation}\label{fmod1}[f|_{G^\der(\BA_F)}]=\sum_{i} f_i\otimes a_i\end{equation} as a finite sum, where  $f_i\in \cA_{\hol}(G^\der,\fw)_{\ol\BQ}$ and $a_i\in    \BC/ \ol\BQ\log {{S}}$ are $\ol{\BQ}$-linearly independent. Then 
for $t\in F_{>0}\cup\{0\}$,
the $t$-th Fourier coefficient of \eqref{fmod1} is 
%\begin{equation*} 
%f_t^\infty  \MOD\ol\BQ\log {{S}} =
$\sum_{i} f_{i,t}^\infty a_i $.
%\end{equation*}
%where the left hand side comes from the  discussion above the proposition, and the right hand side from definition of the Fourier coefficient. 
For   % $0$-th Fourier coefficient    $ f_0^\infty(g)  =0$  
 $g\in \BG_{R\cup R'} \cap G^\der(\BA_F)$,
$\sum_{i} f_{i,t}^\infty(g) a_i=0 $   for $t\in F_{>0}$  by    \Cref{tech1} (applied  to $\BW,\BW'$ respectively), and also for $t=0$ by the above discussion for the constant term. %   for $t\in F_{>0}$ and % $0$-th Fourier coefficient    $ f_0^\infty(g)  =0$  
%the above discussion for $t=0$ .$ 
 Thus $f_{i,t}^\infty(g)=0$ by the $\ol{\BQ}$-linear independence of $a_i$'s. So $f_i(g)=0$.
 %for $g\in \BG_{\{v_1,v_2\}} \cap G^\der(\BA_F)$.
  By  the density of $\BG_{R\cup R'} \cap G^\der(\BA_F) $ in $G^\der(F)\bsl G^\der(\BA_F)$, $f_i(g)=0$ for $G^\der(\BA_F)$.
So \eqref{entry1} holds. 
\end{proof}

\begin{rmk} The density argument can not be applied directly to  $[f|_{G^\der(\BA_F)}].$
\end{rmk}

 Now we  continue the proof of the proposition.  %We remove the assumption in \Cref{teach21} using Lemma \ref{Fourier}.
Recall that $w_{v} \in G^\der(F_{v})\subset G(F_{v})$   as in \ref{Group}   acts on $\cS(\BV_{v}) $ by  Fourier transform (multiplied by the Weil index) via the Weil representation $\omega$. See  \ref {Weil representation}.
   By Lemma \ref{Fourier}, for a finite place $v$ of $F$, there exists a $\BW_v$-regular  Schwartz function     $\Phi_{{{v}}} $  on $\BV(E_v)$  such that $\phi_{{{v}}}=  \Phi_{{{v}}}+\omega(w_{v})  \Phi_{{{v}}}.$
    Choose $K_R , K_{ R'}$ small enough to stabilize  $   \Phi_{{R}}, \Phi_{{R'}}$.
   By  \Cref{teach21}, \eqref  {tech2eq} with  $\phi_{{R}},\phi_{{R'}}$    replaced by $\Phi_{{R}},\Phi_{{R'}}$ holds for  $g\in G^\der(\BA_F) $  
 and thus it hold for $gw_{R},gw_{R'},gw_{R}w_{R'}\in G^\der(\BA_F) $  replacing $g$ as well.
    Then by  Lemma \ref{wr},  \eqref{tech2eq} with  one or both of  $\phi_{{R}},\phi_{{R'}}$    replaced by $ \omega(w_{R})  \Phi_{{R}},\omega(w_{R'})  \Phi_{{R'}}$ respectively holds for  $g\in G^\der(\BA_F) $.  Thus including   \eqref  {tech2eq}, we have  four equations in total. Taking their sum, 
we have  \eqref  {tech2eq} for  the original $\phi$  and $g\in G^\der(\BA_F) $. \end{proof}

  \subsubsection{Remove regularity and $\log S$}\label{removeS}

   \begin{lem}\label{entry20}  Assume that the cardinality of $S$ is at least 2. 
Assume that  $R\subset S$ consists of  a single element
and $\phi_R$ is $\BW'_R$-regular. If $K_S$ is small enough  (depending on $\phi_S$),   then for %$t\in F_{>0}$  and  
   $g\in \BG_{R} $,  \eqref{entry}  holds (literally, for $\BW$ rather than $\BW'$).
Equivalently, \eqref{entryeqmod} holds.

% \begin{align}\label{entry}2\lb  z_t(g, \phi)_\fe^{\cL,\aut}\cdot \cP \rb _{\cX_K[S^{-1}]}= f_{\BW,t}^\infty(g) \in \BC/ \ol\BQ\log \{{S}\}  .\end{align}    
\end{lem}
%\begin{rmk}    If $i=j=0$ (recall that $\BW^{(0,0)}=\BW$),   \Cref{entry20} (ignoring $K$) is  equivalent to \Cref{tech1}.

%In general, we switch  from $\BW_{u_k}^{(0,0)}$-regularity to $\BW_{u_k}^{(i,j)}$-regularity.

%   \end{rmk}

\begin{proof} %Let   $\cP_{\BW^{(i,j)}} $    be as in \ref{CM cycle of  degree 0}. 
By \Cref{tech1}  (with  $\BW$ replaced by ${\BW'}$), 
  for   %$t\in F_{>0}$ and   
    $g\in  \BG_R $,
we have 
 \begin{align}\label{entry2}2  z_t(g, \phi^\infty)_\fe^{\cL,\aut}\cdot \cP_{\BW' }  \MOD\ol\BQ\log S= f_{\BW',t}^\infty(g)   \MOD\ol\BQ\log  S \end{align}     
 Taking the difference between  \eqref  {tech2eq}  and \eqref{entry2},  \eqref{entry}  follows for % $t\in F_{>0}$  and   
  $g\in \BG_{ R} $.  % The equivalence between   \eqref{entry}  and \eqref{entryeqmod}     is explained in the paragraph of \eqref{entryeqmod}. 
\end{proof}  
  
  %Let $u_1,u_2$, ${\BW^{(i,j)}}$      be as in \ref{CM cycle of  degree 0}. 

% Let $\ol \cS\lb \BV  \rb_{\ol\BQ}$ be as in \ref {Functions}. %, and $  \ol \cS\lb \BV  \rb^{K}_{\ol\BQ}\subset \ol \cS\lb \BV  \rb_{\ol\BQ}$   the subspace of $K$-invariant functions as in \ref   {Weil representation}. 

\begin{cor}\label{nctprop} 
If $K$ is small enough,   then for %$t\in F_{>0}$  and  
   $g\in \BG_{\emptyset} $,  %\eqref{entryeqmod}      holds.
\eqref{entry} holds.
      
\end{cor}

\begin{proof}

We prove  \eqref{entryeqmod} which is equivalent to \eqref{entry}. Let   $R\subset S$ consist of  a single element. By Lemma \ref{wr} and   the Iwasawa decomposition,  it is enough to prove
\eqref{entryeqmod}    for   $g\in  \BG_{R}$.  Then by   \Cref{OOO0},  \Cref{OOO} and Lemma \ref{OOOO},   we may assume that $\phi_{R}(0)=0$. 
Such a $\phi_{R}$  can be written as a  sum of $\BW'_{R}$-regular functions  for finitely many $\BW'$'s (in fact, only depending on $\BW'_{R}$). 
 Since    \eqref{entryeqmod} is linear on $\phi_{R}$,  
   the corollary follows from   \Cref{entry20}  with   $\BW'^R=\BW^R$ and   $\BW'_{R}$ varying. 
\end{proof}

 \begin{proof}[Proof of Theorem \ref{nct}] We may assume that  $\phi^\infty$ is $\ol\BQ$-valued. 
It is enough to 
  prove  \eqref{entry4}
modulo $\ol\BQ\log {{S}}$.  
Indeed, choosing another set  $S'$ of 4 places split in $E$ modulo $\ol\BQ\log {{S'}}$ and  requiring $S$ and $S'$ to have   no same residue characteristics, then \eqref{entry4}
follows from  \Cref{baker}, i.e., the $\ol{\BQ}$-linear  independence of $\log p$'s.

Now we   to prove  \eqref{entry4}
modulo $\ol\BQ\log {{S}}$ by decomposing it into equations established  in 
 \Cref{nctprop}      for  $l^{-1}\BW$'s   where $l\in U(\BW)\bsl K_\Lambda/ K'$ (instead of a single $\BW$, and  the double coset   is  clarified    below \eqref{PWK}). 
  Note that by Remark \ref{uselat}, we may shrink  $K$ freely.  
For the left hand side of \eqref{entry4}, that is $2   z_t(g, \phi^\infty)_\fe^{\cL,\aut}\cdot \pi_{_K,{_K}_{\Lambda}}^*\cP_{\BW,{K_\Lambda}}$, by \eqref{PWK}, we have
\begin{equation}
\begin{split}\label{KKK}\pi_{_K,{_K}_{\Lambda}}^*\cP_{\BW,K_{\Lambda}} &=\pi_{_K,_{K'}}^* \pi_{_{K'},{_K}_{\Lambda}}^*\cP_{\BW,K_{\Lambda}} \\
&=\sum_{l\in U(\BW)\bsl K_\Lambda/ K'}\frac{d_{l^{-1}\BW,K'}}{ d_{\BW,K_{\Lambda}}}\pi_{_K,_{K'}}^*
\cP_{l^{-1}\BW,K'}.
\end{split}
\end{equation}
Now we consider the right hand side of \eqref{entry4}, that is $ f_{\BW,{K_\Lambda},t}^{K,\infty}(g) $.
We choose $l$ such that  $l_v=1$ for $v\not \in S$
(this is possible since $K'=K_S K_\Lambda^S$). In particular, $(l\BW)_v=\BW_v$ for $v\in \Ram$ so that  $\phi_v'$
defined in 
\eqref{Error functions} in terms of  $\BW_v$ is the same as that defined in 
\eqref{Error functions} in terms of  $(l\BW)_v$.
By the coset decomposition  
$$U(\BW)\bsl K_\Lambda/ K 
%  =\coprod_{l\in U(\BW)\bsl K_\Lambda/ K'}U(\BW)\bsl K'/ K  
=\coprod_{l\in U(\BW)\bsl K_\Lambda/ K'}
l\lb  U({l^{-1}\BW})\bsl K'\rb/ K  ,$$
and 
%and the following obvious product relation 
% between the  coefficients in \eqref{fh}:
   $$  \frac{d_{k^{-1}l^{-1}\BW,K}}{ d_{\BW,K_\Lambda}} %= \frac{d_{k^{-1}l^{-1}\BW,K}}{ d_{k^{-1}l^{-1}\BW,K'}}  \frac{d_{k^{-1}l^{-1}\BW,K'}}{ d_{k^{-1}l^{-1}\BW,K_\Lambda}}
   = \frac{d_{k^{-1}l^{-1}\BW,K}}{ d_{ l^{-1}\BW,K'}}
 \frac{d_{l^{-1}\BW,K'}}{ d_{\BW,K_\Lambda}},\ k\in U({l^{-1}\BW})\bsl K',$$
we deduce from the definition \eqref{fh} of $ f_{\BW,{K_\Lambda}}^{K} $ that
\begin{align}\label{KKK1}
f_{\BW,K_\Lambda}^K =\sum_{l\in U(\BW)\bsl K_\Lambda/ K'} \frac{d_{l^{-1}\BW,K'}}{ d_{\BW,K_{\Lambda}}}f_{l^{-1}\BW,K'}^K. \end{align}
By \Cref{nctprop}, \eqref{KKK} and \eqref{KKK1} imply  \eqref{entry4}
modulo $\ol\BQ\log {{S}}$.  
 \end{proof}

\addtocontents{toc}{\protect\setcounter{tocdepth}{1}}
\appendix
\section{Admissible  divisors}\label{htadm}

We recall  S.~Zhang's theory of  admissible  cycles on a polarized arithmetic variety \cite{Zha20}.  
They are cycles with ``harmonic curvatures".
We only   consider admissible  divisors \cite[2.5, Admissible cycles]{Zha20}.  
In particular, for a divisor on the generic fiber, 
we have its admissible extensions. With an extra local condition, 
we have  the normalized admissible extension.
The functoriality of (normalized) admissibile cycles  under flat morphisms is important for us.

It is worth mentioning that  while  the normalized admissible extension is defined purely locally and at the level of divisors,  
a  global   lifting of a divisors class on the generic fiber  is  defined  (and it is called   $\mathsf{L}$-lifting) in  \cite[Corollary 2.5.7]{Zha20}. 
It is an  admissible  extension  \cite[Corollary 2.5.7 (1)]{Zha20} with  vanishing  Faltings height  \cite[Corollary 2.5.7 (2)]{Zha20}. They will not be further discussed in this appendix, and are not needed in this paper.

%The formulation in this appendix is from a course given by S.~Zhang  in Fall 2018,  at Princeton.  and global version of   locally admissible  cycles (to be distinsguihed 

\subsection{Deligne-Mumford stacks over a Dedekind domain}\label{Deligne-Mumford stacks over a Dedekind domain}

Let $\cO$ be  a Dedekind domain.
Let $\cM$ be a connected regular  Deligne-Mumford stack  proper  flat over $\Spec \cO$ of relative dimension $n$.
Let $M$ be its generic fiber. %, and $\iota:\cM_s\to \cM$ the inclusion morphism. 
%We use the arithmetic intersection theory as in  \cite{Gil0}\cite{Vis}  on Deligne-Mumford stacks over $E$. 
Definitions of  cycles,   rational equivalence, proper pushforward and flat pullback for (Chow) cycles   are applicable to Deligne-Mumford stacks over $\Spec \cO$. See  \cite{Gil}.
Let $Z^*(M)$ (resp. $Z^*(\cM)$) be the graded  $\BQ$-vector space of  cycles  on $M$  (resp. $ \cM$)  with $\BQ$-coefficients.   
Let $\Ch^*(M)$ and $\Ch^*(\cM)$ be the  $\BQ$-vector spaces of Chow cycles.

%Arithmetic intersection theory    on  Deligne-Mumford stacks has been developed by Gillet \cite{Gil}. 
We shall only work under the following convenient assumption, which simplifies the local intersection theory.  For every   closed point $s\in \Spec \cO$, let $\cO_s$ be the completed local ring.

       \begin{asmp}  \label{A11asmp11}(1) There is a finite subset $S\subset \Spec \cO$,  
        a regular scheme  $\wt \cM$ proper  flat over $\Spec \cO-S$  and a finite \etale morphism 
$\pi:\wt \cM\to \cM|_{\Spec \cO-S}$ over $\Spec \cO-S$.

(2) For every   $s\in S$, there is a regular scheme  $\wt \cM$ proper  flat over $\Spec \cO_s$  and a finite \etale morphism 
$\pi:\wt \cM\to \cM_{\Spec \cO_s}$ over $\Spec \cO_s$.  
\end{asmp}
 In either case (1) or (2), we call 
$\wt \cM$   a covering of $ \cM|_{\Spec \cO-S}$ or $ \cM_{\Spec \cO_s}$.
For another covering $\wt \cM'$,   the fiber product 
is regular and proper  flat over $\Spec \cO-S$ or $\Spec \cO_s$, making a third covering. 

 A line bundle $\cL$  on $\cM$  is   ample
 if in both cases (1)  and (2), its pullback to some covering is ample. And
it is 
relatively positive   if 
$  \deg \cL|_ C> 0$ for every closed curve (1-dimensional  closed substack)  $C$  in every special fiber of $\cM$. % (Though relative positivity is weaker in general, relative positivity implies ampleness if $\cM$ is projective.)
 It is routine to check that the definition does not depend on the choice of  covering by making a third covering. 
 The following notions, which are used in the whole paper,  are also defined via   coverings: intersection number, Chern class, and  Zariski closure.

\subsection{Local cycles}\label{Local cycles}

Assume that $\cO$ is   a completed local ring (so a DVR). Let $s$ be the unique closed point of $\Spec \cO$.
We will use two intersection  pairings.
First,
for   $X\in Z ^i(\cM )$  and  $Y\in  Z^{n+1-i}(\cM )$
with disjoint supports on  $M$, define  \begin{equation}X\cdot Y =\frac{1}{\deg \pi} \pi^*(X)\cdot  \pi^*(Y)\in \BQ\label{YZCM},\end
{equation}   
where $\pi$ is a covering morphism
and  the latter intersection number is  usual one, defined either using  Serre's $\Tor$-formula  (equivalently rephrased as  the Euler-Poincar\'e characteristic of the derived tensor product $\cO_{\pi^*(X)}\bigotimes^\BL \cO_{\pi^*(Y)}$ ( \cite[4.3.8 (iv)]{GS}),  or as a cohomological pairing.
%The definition does not depend on the choice of $\wt\cM$. Indeed,    for another covering $\wt \cM'$, we further pullback to $\wt \cM\times_ \cM \wt \cM'$ to check the independence. Statements below involving a covering are independent of the choice of the covering, and can be  verified in the same way, and we will omit the verifications. 

Second, 
let $\cM_s$ be the special fiber of $\cM$ and  $Z_s^1(\cM )\subset Z^1(\cM)$ the 
subspace of divisors supported on $\cM_s$. We use the intersection pairing  between $  Z_s^1(\cM )$ and an $n$-tuple of $\BQ$-Cartier divisors as in \cite[Example 6.5.1]{Ful} (defined using  a covering   as in \eqref{YZCM}):
\begin{equation*}Z_s^1(\cM )\times Z^{1}(\cM )^{ n}\to \BQ. \end{equation*}
It only depends on the  rational equivalence classes of the Cartier divisors.
(In particular, by fixing $n-1$ rational equivalence classes of  Cartier divisors, we get a pairing between
$Z_s^1(\cM )$ and $ Z^{1}(\cM )$. This view point might be helpful.)
In this subsection, we will use this second intersection pairing until \Cref{pull1cor}.  

Let $\cL$ be a   line bundle on $\cM$.
 Let $B^1_\cL(\cM)\subset Z_s^{1}( {\cM }$ be the kernel of 
 the linear form  $Z_s^{1}( {\cM }) \to \BQ$ defined by intersection with  $ c_1(\cL)^{n}$.
Assume that the generic fiber of $\cL $ is   ample  
 and   $\cL$ is 
relatively positive.
The local index theorem \cite[Lemma 2.5.1]{Zha20} (see also \cite{YZ0}) implies the following lemma.
\begin{lem}   The   pairing $(X,Y)\mapsto X\cdot c_1(\cL)^{n-1} \cdot Y$ on $B_\cL^1(\cM)$ is negative definite. 
\end{lem}

Let $Z_\cL^1(\cM )$ be the orthogonal complement of $B_\cL^1(\cM)$ under the pairing $X\cdot c_1(\cL)^{n-1} \cdot Y$, i.e.
$$Z_\cL^1(\cM )=\{Y\in Z^1( {\cM }):  X \cdot c_1(\cL)^{n-1} \cdot Y =0\mbox{ for every }X\in B_\cL^1(\cM)\}.$$  
Then by definition, we have a decomposition  \begin{equation}\label{ZZB}   Z^1( {\cM })=Z_\cL^1(\cM )\oplus B^1_\cL(\cM),\end{equation}
 and  an exact sequence 
\begin{equation}\label{ZZBE}0\to  \BQ  \cM_s \to Z_\cL^1(\cM )\to Z^1(M)\to 0.\end{equation} 

For a prime cycle $X$ on $M$, let $X^\zar$ be its  Zariski closure on $\cM$. Extend the definition by linearity.  %scheme theoretic image via $M\to \cM$ (defined via a covering or as in \cite[Tag 0CMH]{stacks-project}). %Equivalently, $X^\zar=\pi\lb\pi^{-1}(X)^\zar\rb $  for a covering  where  $\pi$ is a covering morphism and $\pi^{-1}(X)^\zar$ is the usual Zariski closure. For a general cycle, we take linear combinations. 
\begin{defn}   \label{admextloc}  
(1) We call $Z_\cL^1(\cM )$ the space of admissible divisors  with respect to  $\cL$.

(2) For $X\in Z^1(M)$,  an admissible extension  with respect to  $\cL$ is an element in  its preimage  by   $ Z_\cL^1(\cM )\to  Z^1(M)$.   
Define the  normalized admissible extension $X^\cL$ of $X$  with respect to  $\cL$  to be the  projection of $X^\zar$ to $Z_\cL^1(\cM )$ in \eqref{ZZB}.

\end{defn} 
\begin{rmk}\label {Laplacian equation}
In terms of  \cite[Corollary 2.5.7 (1)]{Zha20},  $Z_\cL^1(\cM )\subset Z ^1(\cM )$ is the subspace of cycles $X$ with ``harmonic curvatures"  (compare with \Cref{rmk123} (3)), i.e.,   the  element  in $ \Hom_\BQ\lb Z_s^{1}(\cM ),\BQ\rb$ defined by intersection with 
$    X \cdot c_1(\cL)^{n-1} $  is a multiple of the one   defined by intersection with   $  c_1(\cL)^{n} $. %by definition, $\ker(?\cdot c_1(\cL)^{n} )= B_\cL^1(\cM)$ and the above ``$\sim$" relation follows from the definition of $Z_\cL^1(\cM )$.
\end{rmk}

Then by   definition, we  have the following lemma.
\begin{lem}\label{smoothadm}
Assume that $\cM$ is smooth over $\Spec \cO$. Then $B_\cL^1(\cM)=0$. In particular, 
for $X\in Z^1(M)$,  $X^\zar$ is the normalized admissible extension.

\end{lem}
  By the projection formula (and the commutativity of taking Zariski closure and closed pushforward/flat pullback), we easily deduce the following lemma.
\begin{lem} \label{pull1}
Let $\cM'$ be a regular  Deligne-Mumford stack  and 
$ f: \cM'\to \cM$   a  finite flat morphism. 
Let $\cL'$ be the pullback of $\cL$ to $\cM'$.
Consider $\cM'$ as a Deligne-Mumford stack over $\Spec \cO$ via $f$ and $\cM$.
Then $f^*\lb B_\cL^1(\cM)\rb\subset  B_{\cL'}^1(\cM') $,
$f_*\lb B_{\cL'}^1(\cM')\rb= B_{\cL}^1(\cM) ,$   $f^*\lb Z_\cL^1(\cM)\rb\subset  Z_{\cL'}^1(\cM') $ and 
$f_*\lb Z_{\cL'}^1(\cM')\rb= Z_{\cL}^1(\cM) .$ 
In particular,    the decomposition  \eqref{ZZB} and the formation of  normalized admissible extension
are preserved under pullback and pushforward by $f$.
\end{lem}
\begin{cor} \label{pull1cor}
In Lemma \ref{pull1} with $f$   finite, let $M'$ be the generic fiber of $\cM'$ and assume that $\cM\to \Spec \cO$ is smooth. Let  $X\in Z^1(M')$ and $Y\in Z^{n}(\cM)$ such that   $X$ and $f^{*}(Y)$ have disjoint supports.
Then we have 
$X^\cL\cdot f^*(Y)=X^\zar\cdot f^*(Y)$

\end{cor}
\begin{proof} By Lemma \ref{smoothadm} and \Cref{pull1}, we have  $f_*(X^\cL)=(f|_{M',*}X)^\cL=(f|_{M',*}X)^\zar=f_*(X^\zar)$. The corollary follows from the projection formula.
\end{proof}
\subsection{Admissible arithmetic Chow group of divisors}\label{until}
       Now let $\cO$ be the   ring of integers of a number field.             
                    In particular, $\cM_\BC:=\cM \otimes_\BZ \BC=M\otimes_\BQ\BC$ is a complex orbifold. 
                    Let $\ol\cL=(\cL,\|\cdot\|)$ be a hermitian   line bundle on $\cM$   such that
 the generic fiber of $\cL $ is ample,
  $\cL$ is   relatively positive,  and  the hermitian metric  $\|\cdot\|$ is invariant under the involution induced by complex conjugation. See \cite[3.1.2]{GS}. In particular,  endowing $\cM_\BC $ with the Chern curvature form $\curv\lb \ol\cL_\BC\rb$, it is a smooth K\"ahler orbifold.

\begin{defn}  \label{admext}  

(1) The group $\whz^1_{\adm,\BC}(\cM)$ of admissible (with respect to $\ol\cL$) arithmetic divisors on $\cM$ with  $\BC$-coefficients  is the $\BC$-vector space of pairs $(X,g)$ where:
\begin{itemize}\item $X\in Z^1(\cM) _\BC $
such that for every   closed point $s\in \Spec \cO$, the restriction of $X$ to $  \cM_{\cO_s}  $ is contained in 
        $Z^1_{\cL_{\cO_s} }\lb \cM_{\cO_s}\rb_\BC$, 
   
        \item $g$  is      a       Green  function for $ X_ \BC$  on $\cM_\BC$,  admissible with respect to $\ol\cL_\BC$, and   invariant under the involution induced by complex conjugation. Here admissibility means that the   curvature form        $ \delta_{X}+\frac{i }{2\pi }\partial \bar\partial g$ is harmonic.
 
                  \end{itemize} 
              
(2)  For $X\in Z^1(M)_\BC$, an admissible extension of $X$ with respect to  $\ol\cL$ is an element in the preimage of $X$ by the natural surjection $\whz^1_{\adm,\BC}(\cM) \to Z^1(M)_\BC$.

                  \end{defn}

                       \begin{rmk}   \label{rmk123}        (1)       A Green current on Deligne-Mumford stacks is defined in  \cite[Section 1]{Gil}. 
In our situation, a        Green function is simply an orbifold function whose pullback to 
                  the  finite \etale   cover  by  a smooth variety is a usual Green function.
           
            (2) Admissible Green functions  for $X_\BC$ always exists and  are the same modulo locally constant functions, and \eqref{ZZBE} is the non-archimedean analog of this fact.

        (3)  By \cite[2.2]{Zha20},  a closed $(1,1)$-form  $\alpha$ is harmonic if and and only if  
          on each connected component of $\cM_\BC$,       $ \alpha \wedge  \curv\lb \ol\cL_\BC\rb ^{n-1} $ is a constant multiple of $   \curv\lb \ol\cL_\BC\rb ^{n}.$
       % (3) By     \eqref{ZZBE} and the existence of admissible Green functions, the preimage of $X$ is nonempty 

\end{rmk}
                                 
\begin{defn}  \label{admextgreen}  
(1) An admissible Green  function     is   normalized   with respect to  $\ol\cL_\BC$  if it  has vanishing harmonic projection, i.e.,
 on each connected component of $\cM_\BC$, its integration  against $\curv(\ol \cL_\BC)^{n}$ is 0.

(2) An element in $\whz^1_{\adm,\BC}(\cM)$   is normalized with respect to  $\ol\cL$ if it is normalized at every  place.
  For $X\in Z^1(M)_{\BC}$, let $X^\nadm\in \whz^1_{\adm,\BC}(\cM) $ be its normalized admissible extension with respect to  $\ol\cL$. 
                  \end{defn} 
      % The following lemma is clear from definition.

                %. See also \cite[Definition 2.]{Fal}.)
%Definition \ref{admextloc} is the non-archimedean analog of the normalized admissible Green function.    
     Then the  normalized  admissible extension of  a divisor  on 
        $M$ exists  and is unique.

                   %\begin{rmk} \label{depecy}
% Note that the dependence of $\whz^1(\cM)_\adm$    and normalized admissible extensions     on the metric of $\ol\cL$ is determined by the curvature form of the metric.  

           %  \end{rmk}   
             
For every  nonzero rational function $f$ on $M$, $(\div(f),-\log |f|^2)$ is contained in $\whz^1_{\adm,\BC}(\cM)$.  
     
\begin{defn}\label{CCH}

(1) Let   $\wh \Ch^1_{\BC}(\cM) $ be  the quotient of the  space of       arithmetic divisors  with $\BC$-coefficients by  the $\BC$-span of   $(\div(f),-\log |f|^2)$'s for all nonzero rational functions.
 
(2) 
 Let  $\wh \Ch_{\adm,\BC}^1(\cM) $  
be the       quotient of  $\whz^1_{\adm,\BC}(\cM)$ by the $\BC$-span of $(\div(f),-\log |f|^2)$'s.  
\end{defn}

\begin{rmk}\label{cmc}
Let  $ \wh \Ch^{1}(\cM)$ be the     Chow group  of       arithmetic divisors  with $\BZ$-coefficients defined by   Gillet and Soul\'e \cite{GS} for schemes, which is
extended  to the stacky case   in
\cite{Gil}. Then $\wh \Ch^1_{\BC}(\cM) $  is the quotient of $\wh \Ch^1(\cM) _{\BC}$ by   the pullback of the kernel of the degree map
$   \wh \Ch^1\lb   \Spec \cO  \rb_{\BC}\to \BC$. In particular, we have an isomorphism $  \wh \Ch^1_{\BC}\lb   \Spec \cO  \rb\cong \BC$  by taking degrees.
    \end{rmk}   
       
       \begin{lem}\label{fdime}(1) The natural map $\wh\Ch^1_{\adm,\BC}(\cM) \to \Ch^1(M) _\BC$ is surjective.
Its kernel  is generated by  connected components of  special fibers of $\cM$ at all  finite places and    locally constant functions on $\cM_\BC$.

(2)        Assume that    $\cM$ is   connected.  Then  the kernel of  $\wh\Ch^1_{\adm,\BC}(\cM)\to \Ch^1(M)_\BC$ 
is 1-dimensional,  and is  the pullback of $\wh \Ch^1_{\BC}\lb \Spec \cO\rb$.
       \end{lem}
       \begin{proof} 
        By  \eqref{ZZBE} and \Cref{rmk123} (2), (1) holds.  If  $\cM$ is   connected, then $E_1:=\cO_M(M)$ is a  finite field extension of  
  the fraction field of $\cO$.  Then $\cM$ over $\Spec \cO_{E_1}$ has geometrically connected fibers by Stein factorization. 
        By (1), the kernel is  the pullback of $\wh \Ch^1_{\BC}\lb \Spec \cO_{E_1}\rb $, which is 1-dimensional by the finiteness of the  class number of $E_1$ and   Dirichlet's unit theorem. See \cite[3.4.3]{GS}.       
  And it equals   the pullback of $\wh \Ch^1_{\BC}\lb \Spec \cO\rb$.
                              \end{proof} 
                   
    \begin{eg} \label{admextinf}       %For  positive numbers   $c=(c_v)_{v\in \infty}$, let  $c\|\cdot\|$ be the scaling of $\|\cdot \|$ and let $\ol\cL_c$ be the same line bundle with the new metric. 
          The arithmetic first Chern  class $c_1(\ol\cL)$ of $\ol\cL$  is  the class of $ (\div(s),-\log  \|s\| ^2)$ for a nonzero rational section $s$.  By \Cref{Laplacian equation}  (or one may follow our definition), one immediately sees that $\div(s)$ has ``harmonic curvature" at every finite place.                   
                The curvature form of     $-\log  \|s\| ^2$ is by definition the K\"ahler form. So $c_1(\ol\cL)\in \wh\Ch^1_{\BC,\adm}(\cM)$.
            
\end{eg}
Now we consider the functoriality. 
By Lemma \ref{pull1}, we have the following proposition.
                    \begin{prop} \label{pull}Let $\cM'$ be a regular  Deligne-Mumford stack  and 
$ f: \cM'\to \cM$   a  finite flat morphism over $\Spec \cO$, such that the restriction of $f$ to the generic fibers  is finite \'etale.
Let $\cL'$ be the pullback of $\cL$ to $\cM'$.
Consider $\cM'$ as a Deligne-Mumford stack over $\Spec \cO$ via $f$ and $\cM$.
Then the formation of  $\whz^1_{\BC,\adm}(\cM)$,  $\wh \Ch^1_{\BC,\adm}(\cM) $ and 
normalized admissible extension with respect to  $\ol\cL$
is preserved under pullback and pushforward by $f$.
\end{prop}

\subsection{Arithmetic  intersection
    pairing}\label{arithmetic  intersection
    pairing}Let $Z_1(\cM)_\BC$ be the    group of 1-cycles on $\cM$. 
       We define  an arithmetic  intersection
    pairing following \cite[2.3.1]{BGS}:          \begin{align*}\wh\Ch^1(\cM)_{\BC} \times Z_1(\cM) _{\BC}\to \BC,\ 
(\wh x,Y)\mapsto \wh x\cdot Y .\end{align*}
 %There are  three equivalent definitions, and we follow the first two in loc. cit..
        
We   reduce the pairing to  the  arithmetic intersection pairing between   $\wh\Ch^1 (\cM) _{\BC}$ and  $\wh\Ch^n (\cM)  _{\BC}$, which is defined   in \cite{Gil} for general Deligne-Mumford stacks (without \Cref{A11asmp11}).
    Let $\omega({\wh x})$ be the curvature form of $\wh x$ which is    a smooth  $(1,1)$ form on the orbifold $M_{\BC}$     independent of the choice of a representative  
    of $\wh x$.  Choose  $\wh y=(Y,g_Y)\in\wh\Ch^n_{\BC}(\cM)$.  Then $ \wh x\cdot Y$ is
 the arithmetic intersection number  $\wh x\cdot \wh y $  minus   $\int_{\cM_\BC}\omega({\wh x}) g_Y$.

  Now assume \Cref{A11asmp11} and  that $(X,g_X)$  is a representative of $\wh x$ such that $X\cap Y$ is empty on the generic fiber of $\cM$. Then $ \wh x\cdot Y $ is the   sum of the intersection numbers of the restrictions of $X$ and $Y$ to  $\cM_{\cO_s}$ over all    closed points $s\in \Spec \cO$ defined in \eqref{YZCM}, and     $\int_{Y_\BC}g_X$.

The pullback of the kernel of the degree map
$   \wh \Ch^1\lb   \Spec \cO  \rb_{\BC}\to \BC$ to $\wh\Ch^1(\cM)_{\BC}$
   is annihilated by the arithmetic  intersection
    pairing with $Z_1(\cM)_\BC$.
    By Remark \ref   {cmc}, 
the above arithmetic  intersection pairing factors through an arithmetic  intersection pairing     \begin{align*}\wh\Ch^1_\BC(\cM)\times Z_1(\cM)_\BC\to \BC,\
(\wh z,Y)\mapsto \wh z\cdot Y .\end{align*}

  Similar to \Cref{pull1cor}, Proposition \ref{pull} implies the following result.
\begin{cor} \label{pullint}
In Proposition \ref{pull}, let $M'$ be the generic fiber of $\cM'$ and assume that $\cM\to \Spec \cO_E$ is smooth. Then for $X\in Z^1(M')_{\BC}$ and $Y\in Z^{n}_{\BC}(\cM)$, 
we have 
$$[X^\nadm]\cdot f^*(Y)=[(X^\zar,g_{X}^\nadm)]\cdot f^*(Y)$$
where $g_{X}^\nadm$ is the normalized admissible Green function for $X$. % and $[\dots]$ indicates the image in $\wh\Ch^1_{\BC}(\cM)$.

\end{cor}
%Here the Zariski closure is defined as the glueing of the ones over $S$ and outside $S$, and both are defined as  above \Cref{admextloc} using coverings.

\setcounter{tocdepth}{1}
\setcounter{secnumdepth}{4}

\newcommand{\N}{\mathbb{N}}
\newcommand{\Z}{\mathbb{Z}}
\newcommand{\R}{\mathbb{R}}
\newcommand{\C}{\mathbb{C}}
\newcommand{\Q}{\mathbb{Q}}
\newcommand{\F}{\mathbb{F}}
\newcommand{\A}{\mathbb{A}}
\newcommand{\mcA}{\mathcal{A}}
\newcommand{\mcB}{\mathcal{B}}

\newcommand{\Hh}{\mathbb{H}}
\newcommand{\mcH}{\mathcal{H}}

\newcommand{\Aa}{\mathfrak{a}}
\newcommand{\p}{\mathfrak{B}}
\newcommand{\m}{\mathfrak{m}}
\newcommand{\Oo}{\mathcal{O}}
\newcommand{\Id}{\operatorname{Id}}
\newcommand{\unr}{\operatorname{unr}}
\newcommand{\specialization}{\operatorname{sp}}
\newcommand{\Support}{\operatorname{Support}}
\newcommand{\Fil}{\operatorname{Fil}}
\newcommand{\iso}{\operatorname{iso}}
\newcommand{\Norm}{\operatorname{Nm}}
\newcommand{\Torsion}{\operatorname{Torsion}}
\newcommand{\fppf}{\operatorname{fppf}}
\newcommand{\conn}{\operatorname{conn}}
\newcommand{\Gp}{\operatorname{Gp}}
\newcommand{\GU}{\operatorname{GU}}
\renewcommand{\tilde}{\widetilde}
\renewcommand{\mod}{\operatorname{mod}}
\newcommand{\CH}{\operatorname{CH}}

\section{A comparison of the ``closure'' model with Rapoport--Smithling--Zhang model (appendix by Yujie Xu)}

\numberwithin{equation}{subsection}
\newtheorem{theorem}[equation]{Theorem}
\newtheorem{prop-appendix}[equation]{Proposition}
\newtheorem{lem-appendix}[equation]{Lemma}
\newtheorem{Coro-appendix}[equation]{Corollary}

\theoremstyle{definition}
\newtheorem{Defn}[equation]{Definition}
\newtheorem{remark}[equation]{Remark}
\newtheorem{numberedparagraph}[equation]{}

\maketitle

\subsection{Preliminaries} 

\begin{numberedparagraph} 
Let $F$ be a CM field and $F_0$ its maximal totally real subfield of index $2$. Let $a\mapsto \overline{a}$ be the nontrivial automorphism of $F/F_0$. We fix a presentation $F=F_0(\sqrt{\Delta})$ for some totally negative element $\Delta\in F_0$. Let $\Phi$ denote the CM type for $F$ determined by $\sqrt{\Delta}$, i.e. 
\begin{equation}\label{CM-type-Phi}
    \Phi:=\{\varphi:F\to\C|\varphi(\sqrt{\Delta})\in\R_{>0}\cdot\sqrt{-1}\}. 
\end{equation}

Let $W$ be a non-degenerate $F/F_0$-Hermitian space of dimension $n\geq 2$. 
Let 
\begin{equation}\label{defn-G-ResUW}
G:=\Res_{F_0/\Q}U(W).
\end{equation}
As in \cite[$\mathsection$2.1]{RSZ-arithmetic-diagonal-cycles}, we use the symbol $c$ to denote the similitude factor of a point on a unitary similitude group. We consider the following algebraic groups over $\Q$.
\begin{align}\label{ZQ-defn}
    Z^{\Q}&:=\{z\in \Res_{F/\Q}\G_m|\Norm_{F/F_0}(z)\in\G_m\}\\
    \label{GQ-defn}
    G^{\Q}&:=\{g\in\Res_{F_0/\Q}\GU(W)|c(g)\in\G_m\}\\
    \label{Gtilde-defn}
    \widetilde{G}&:=Z^{\Q}\times_{\G_m}G^{\Q}=\{(z,g)\in Z^{\Q}\times G^{\Q}|\Norm_{F/F_0}(z)=c(g)\}
\end{align}
Note that $Z^{\Q}$ is naturally a central subgroup of $G^{\Q}$ and we have the following product decompositions
\begin{align}\label{Gtiles-to-ZQ}
    \widetilde{G}\xlongrightarrow{\sim}Z^{\Q}\times G,\
    (z,g)\longmapsto (z,z^{-1}g).
\end{align}
\end{numberedparagraph}

\begin{numberedparagraph}
From now on we assume moreover that the Hermitian space $W$ has the following signatures at the archimedean places of $F_0$: for a distinguished element $\varphi_0\in\Phi$, the signature of $W_{\varphi_0}$ is $(1,n-1)$; and for all other $\varphi\in\Phi$, the signature of $W_{\varphi}$ is $(0,n)$. In order to define a Shimura datum $(G^{\Q},\{h_{G^{\Q}}\})$, by the canonical inclusions $G_{\R}^{\Q}\subset \prod\limits_{\varphi\in\Phi}\GU(W_{\varphi})$, it suffices to define the components $h_{G^{\Q},\varphi}$ of $h_{G^{\Q}}$. Consider the matrices
\begin{equation}
    J_{\varphi}:=\begin{cases}\mathrm{diag}(1,(-1)^{(n-1)}),&\varphi=\varphi_0,\\
    \mathrm{diag}(-1,-1,\cdots,-1), &\varphi\in\Phi\setminus\{\varphi
    _0\}.\end{cases}
\end{equation}
We also choose bases $W_{\varphi}\simeq \C^n$ such that the Hermitian form on $W_{\varphi}$ is given by  $J_{\varphi}$. 
%\begin{equation}
 %   (x,y)_{\varphi}:={}^txJ_{\varphi}\overline{y}. \qcl this is not precise...
%\end{equation}
Consider the $\R$-algebra homomorphisms
\begin{align}
    \C\longrightarrow\End(W_{\varphi}),\
    \sqrt{-1}\longmapsto \sqrt{-1}J_{\varphi},
\end{align}
which induce our desired component maps $h_{G^{\Q},\varphi}:\C^{\times}\to \GU(W_{\varphi})(\R)$. This gives us our desired Shimura datum $(G^{\Q},\{h_{G^{\Q}}\})$. 
\end{numberedparagraph}

\begin{numberedparagraph}\label{ZQ-Gtilde-Shimura-data-section}
For the group $Z^{\Q}$ defined in \ref{ZQ-defn}, the CM type $\Phi$ induces an identification
\begin{equation}\label{ZQR-isom}
    Z^{\Q}(\R)\cong\Big\{(z_{\varphi})\in(\C^{\times})^{\Phi}\Big| |z_{\varphi}|=|z_{\varphi'}|\text{ for all }\varphi,\varphi'\in\Phi\Big\},
\end{equation}
which allows us to define $h_{Z^{\Q}}:\C^{\times}\to Z^{\Q}(\R)$ as the diagonal embedding (via the identification \ref{ZQR-isom}) precomposed with complex conjugation. This gives us a Shimura datum $(Z^{\Q},\{h_{Z^{\Q}}\})$ with reflex field 
\begin{equation}\label{E-Phi-reflex-field}
E(Z^{\Q},\{h_{Z^{\Q}}\})=E_{\Phi},
\end{equation}
which is the reflex field for the CM type $\Phi$. Recall that this can be computed as the fixed field in $\C$ of the group $\{\sigma\in \Aut(\C)=|\sigma\circ\Phi=\Phi\}$.

For the group $\widetilde{G}$ defined in \ref{Gtilde-defn}, we consider the map 
\begin{equation}\label{h-tildeG-defn}
    h_{\widetilde{G}}: \C^{\times}\xrightarrow{(h_{Z^{\Q}},h_{G^{\Q}})}\widetilde{G}(\R),  
\end{equation}
which gives us a Shimura datum $(\widetilde{G},\{h_{\widetilde{G}}\})$. Let 
\begin{equation}\label{reflex-field-E}
    E:=E(\widetilde{G},\{h_{\widetilde{G}}\})
\end{equation} 
be the reflex field for $(\widetilde{G},\{h_{\widetilde{G}}\})$, by definition it is computed via
\begin{equation}
    \Aut(\C/E)=\{\sigma\in \Aut(\C)|\sigma\circ\Phi=\Phi, \sigma\circ\varphi_0=\varphi_0\}.
\end{equation}
Note that $E$ is the compositum of $E_{\Phi}$ (as in \ref{E-Phi-reflex-field}) and $F$. 
By \cite{Deligne-travaxu-de-Shimura}, we have canonical models $\Sh_{K_{\widetilde{G}}}(\widetilde{G},\{h_{\widetilde{G}}\})$ over $E$, for compact open subgroups $K_{\widetilde{G}}\subset\widetilde{G}(\A_f)$. 
By \cite{Kisin-integral-model} (resp.~\cite{Kisin-Pappas} depending on the level structure), we have integral models $\mathscr{S}_{K_{\widetilde{G}}}(\widetilde{G},\{h_{\widetilde{G}}\})$ over $\Oo_{E,(v)}$. 
\end{numberedparagraph}
\begin{numberedparagraph}\label{section-defn-Shimura-datum-G}
%Note that the map in \ref{Gtiles-to-ZQ} gives rise to a morphism of Shimura data 
%\begin{equation}\label{morShimura-datum}
 %   (\widetilde{G},\{h_{\widetilde{G}}\})\to (Z^{\Q},\{h_{Z^{\Q}}\})
%\end{equation}
%induced by the natural projection onto $Z^{\Q}$. 

%where $K_G=K_G^p\times K_{G,p}\subset G(\A_{F_0,f})$ is an open compact subgroup, and $K_{Z^{\Q}}\subset Z^{\Q}(\A_f)$ is the unique maximal compact subgroup defined as
%\begin{equation} K_{Z^{\Q}}:=Z^{\Q}(\widehat{\Z})=\{z\in (\Oo_F\otimes \widehat{\Z})^{\times}|\Norm_{F/F_0}(z)\in\widehat{\Z}^{\times}\}.\end{equation}

%Likewise, we have a morphism of Shimura data
%\begin{equation}
 %   (\widetilde{G},\{h_{\widetilde{G}}\})\to(G^{\Q},\{h_{G^{\Q}}\})
%\end{equation}
%induced by the natural projection, which then induces a morphism of Shimura varieties
%\begin{equation}
  %  \Sh_{K_{\widetilde{G}}}(\widetilde{G},\{h_{\widetilde{G}}\})\to \Sh_{K_{G^{\Q}}}(G^{\Q},\{h_{G^{\Q}}\})
%\end{equation}

We introduce a Shimura datum $(G,\{h_G\})$. Let $h_G$ be the map 
\begin{equation}
    h_G: \C^{\times}\xrightarrow{h_{\widetilde{G}}}\widetilde{G}(\R)\xrightarrow{\ref{Gtiles-to-ZQ}}G(\R)
\end{equation}
defined by composing $h_{\widetilde{G}}$ (from \ref{h-tildeG-defn}) with the projection onto the second factor in the map \ref{Gtiles-to-ZQ}. The reflex field for $(G,\{h_G\})$ is $F$, embedded into $\C$ via $\varphi_0$. 

Moreover, the decomposition in \ref{Gtiles-to-ZQ} induces a decomposition of Shimura data
\begin{equation}\label{decomposition-Gtilde-Shimura-data}
    (\widetilde{G},\{h_{\widetilde{G}}\})=(Z^{\Q},\{h_{Z^{\Q}}\})\times (G,\{h_G\}).
\end{equation}
Let $K_{\widetilde{G}}$ be  decomposed via \ref{Gtiles-to-ZQ} into
\begin{equation}\label{level-KGtilde-decomp}
    K_{\widetilde{G}}=K_{Z^{\Q}}\times K_G,
\end{equation}
The natural projections in \ref{decomposition-Gtilde-Shimura-data} then induce morphisms of Shimura varieties
\begin{equation}\label{ProjwGZQ}
    \Sh_{K_{\widetilde{G}}}(\widetilde{G},\{h_{\widetilde{G}}\})\to \Sh_{K_{Z^{\Q}}}(Z^{\Q},\{h_{Z^{\Q}}\})_E,
\end{equation}
\begin{equation}
    \Sh_{K_{\widetilde{G}}}(\widetilde{G},\{h_{\widetilde{G}}\})\to \Sh_{K_G}(G,\{h_G\})_E.
\end{equation}

Note that the Shimura variety $\Sh_{K_G}(G,\{h_G\})$, which originally appeared in \cite{GGP}, is not of PEL type. However, it is of abelian type, and we have an integral model $\mathscr{S}_{K_G}(G,\{h_G\})$ defined over $\Oo_{F,(\nu)}$ by \cite{Kisin-integral-model} (resp. \cite{Kisin-Pappas} depending on the level structure\footnote{Note that the construction of \cite{Kisin-Pappas} assumes that $p>2$, and that $G$ splits over a tamely ramified extension of $\Q_p$, and that $p$ does not divide the order $|\pi_1(G^{\der})|$ of the algebraic fundamental group of the derived group $G^{\der}$ over $\overline{\Q}_p$. We expect that the condition ``$G$ splits over a tamely ramified extension of $\Q_p$'' can certainly be relaxed using \cite{Kisin-Zhou}.}). 

\end{numberedparagraph}

\begin{numberedparagraph}\label{mathcal-M0-defn}
Let $\nu|p$ be a place of $E$. Let $\mathcal{M}_0$ be the moduli functor which associates to each locally Noetherian $\Oo_{E,\nu}$-scheme $S$ the groupoid of tuples $\mathcal{M}_0(S):=(A_0,\lambda_0,\iota_0
)$ where
\begin{itemize}
    \item[(i)] $A_0$ is an abelian scheme over $S$;
    \item[(ii)] $\iota_0: \Oo_F\to \End_S(A_0)$ is an $\Oo_F$-endomorphism structure on $A_0$ satisfying the Kottwitz condition of signature $((0,1)_{\varphi\in \Phi})$, i.e. 
    \[\Char(\iota_0(a)|\Lie A_0)=\prod\limits_{\varphi\in\Phi}(T-\overline{\varphi}(a))\quad\text{ for all }a\in \Oo_F;\]
    \item[(iii)] $\lambda_0$ is a principal polarization of $A_0$ such that the associated Rosati involution induces the non-trivial Galois automorphism of $F/F_0$ on $\Oo_F$ via $\iota_0$. 
\end{itemize}
Then $\mathcal{M}_0$ is representable by a Deligne-Mumford stack $\mathcal{M}_0$ finite  \'etale over $\Spec\Oo_{E,\nu}$.  Moreover, we assume  $K^p$ is small enough, so that $\mathcal{M}_0$ is nonempty. We shall assume throughout the rest of this appendix that $\mathcal{M}_0$ is nonempty. 
 
\begin{lem-appendix}\cite[Lemma 3.4]{RSZ-arithmetic-diagonal-cycles} The stack $\mathcal{M}_0$ admits the following decompositon into open and closed substacks\footnote{Here the index set $\mathcal{L}_{\Phi}/\sim$ need not be specified for our purposes, for more details see \cite{RSZ-arithmetic-diagonal-cycles}}:
\begin{equation}\label{mathcal-M0-decomposition-xi}
    \mathcal{M}_0=\bigsqcup\limits_{\xi\in\mathcal{L}_{\Phi}/\sim}\mathcal{M}_0^{\xi}
\end{equation}
such that the generic fiber of $\mathcal{M}_0^{\xi}$ is canonically isomorphic to  $ \Sh_{K_{Z^{\Q}}}(Z^{\Q},\{h_{Z^{\Q}}\})_E$.
\end{lem-appendix}
 
\end{numberedparagraph}

\begin{numberedparagraph}\label{defn-mathcal-M-Gtilde} 
Let $F_{0,v}$ be the $v$-adic completion of $F_0$, and we set $F_v:=F\otimes_{F_0}F_{0,v}$. 
Suppose for now that the place $v_0$ of $F_0$ is unramified over $p$, and that $v_0$ either splits in $F$ or is inert in $F$. Suppose moreover that the Hermitian space $W_{v_0}$ is split. If there exists a prime $v$ of $F_0$ above $p$ that is non-split in $F$, we assume additionally that $p\neq 2$. We choose a vertex lattice $\Lambda_v$ in the $F_v/F_{0,v}$-Hermitian space $W_v$. 
For now we assume that $\Lambda_{v_0}$ is self-dual. 
We recall that an $\Oo_{F,v}$-lattice $\Lambda$ in an $F_v/F_{0,v}$-Hermitian space is called a \textit{vertex lattice of type $r$} if $\Lambda\subset^r\Lambda^*\subset\pi_v^{-1}\Lambda$.\footnote{Here the notation $\Lambda\subset^r\Lambda^*$ means that $\Lambda$ is an $R$-submodule of $\Lambda^*$ of finite colength $r$} An $\Oo_{F,v}$-lattice $\Lambda$ in an $F_v/F_{0,v}$-Hermitian space is called a \textit{vertex lattice} if it is a vertex lattice of type $r$ for some $r$. Here $\pi_v$ is a uniformizer in $F_v:=F\otimes_FF_{0,v}$, where $F_{0,v}$ is the $v$-adic completion of $F_0$ for a place $v$ of $F_0$. In particular, a self-dual lattice is simply a vertex lattice of type $0$. 
Assume that  $K_{G,v}=\mathrm{Stab}(\Lambda_v)\subset G(F_{0,v})$. 
\end{numberedparagraph}

\begin{numberedparagraph}
Let $\mathcal{M}_{K_{\widetilde{G}}}(\widetilde{G})$ be the moduli functor which associates to each locally Noetherian $\Oo_{E,(\nu)}$-scheme $S$ the groupoid of triples $\mathcal{M}_{K_{\widetilde{G}}}(\widetilde{G})(S):=(A_0,\iota_0,\lambda_0, A,\iota,\lambda,\overline{\eta}^p)$ where 
\begin{itemize}
    \item $(A_0,\iota_0,\lambda_0)\in \mathcal{M}_0^{\xi}(S)$ as is defined in \ref{mathcal-M0-decomposition-xi};
    \item $(A,\iota)$ is an abelian scheme over $S$, equipped with an $\Oo_F\otimes\Z_{(p)}$-endomorphism structure $\iota$ satisfying the Kottwitz condition of signature $((1,n-1)_{\varphi_0},(0,n)_{\varphi\in \Phi\setminus\{\varphi_0\}})$, 
    i.e. 
    \[\Char(\iota(a)|\Lie A)=(T-\varphi_0(a))(T-\varphi_0(\overline{a}))^{n-1}\prod\limits_{\varphi\in\Phi\setminus\{\varphi_0\}}(T-\varphi(\overline{a}))^n\quad\text{ for all }a\in F;\]
    \item $\lambda$ is a polarization of $A$ such that the associated Rosati involution induces the non-trivial Galois automorphism of $F/F_0$ on $\Oo_F\otimes\Z_{(p)}$ via $\iota$, and such that the following additional assumption in \cite[(4.2)]{RSZ-arithmetic-diagonal-cycles} is also satisfied: the action of $\Oo_{F_0}\otimes\Z_p\cong\prod\limits_{v|p}\Oo_{F_0,v}$ on the $p$-divisible group $A[p^{\infty}]$ induces a decomposition 
   % \begin{equation}\label{decomposition-Apinfty-Avinfty}
     $   A[p^{\infty}]=\prod\limits_{v|p}A[v^{\infty}],$
   % \end{equation}
    where $v$ ranges over the places of $F_0$ above $p$; the polarization $\lambda$ then induces a polarization \begin{equation}
        \lambda_v: A[v^{\infty}]\to A^{\vee}[v^{\infty}]\cong A[v^{\vee}]^{\vee}
    \end{equation}
    for each $v$; we require $\ker\lambda_v$ to be contained in $A[\iota(\pi_v)]$ of rank $\#(\Lambda_v^*/\Lambda_v)$ for each place $v$ of $F_0$ above $p$; 
    \item $\overline{\eta}^p$ is a $K_G^p$-orbit of the $\A^p_{F,f}$-linear isometry
    \begin{equation}
       \Hom_F(\widehat{V}^p(A_0),\widehat{V}^p(A))\simeq -W\otimes_F \A^p_{F,f},\label{leveloutp}
    \end{equation} 
   where the Hermitian form on the left hand side is $(x,y)\mapsto \lambda_0^{-1}\circ y^{\vee}\circ\lambda\circ x.$
    \item For each $v\neq v_0$ over $p$, we impose the \textit{sign condition} and \textit{Eisenstein condition} at $v$ \cite[(4.4), (4.10)]{RSZ-arithmetic-diagonal-cycles}.
\end{itemize}

By \cite[Theorem 4.1]{RSZ-arithmetic-diagonal-cycles}, the forgetful map $(A_0,\iota_0,\lambda_0,A,\iota,\lambda,\overline{\eta}^p)\mapsto (A_0,\iota_0,\lambda_0)$ is representable and induces a morphism of $\Oo_{E,(\nu)}$-schemes
\begin{equation}\label{map-mathcalM-Gtilde-to-mathcal-M0}
\mathcal{M}_{K_{\widetilde{G}}}(\widetilde{G},\{h_{\widetilde{G}}\})\to \mathcal{M}_0^{\xi}:=\mathcal{M}_{K_{Z^{\Q}}}(Z^{\Q})\cong \mathscr{S}_{K_{Z^{\Q}}}(Z^{\Q},\{h_{Z^{\Q}}\}).
\end{equation}
On the level of generic fibres, \eqref{map-mathcalM-Gtilde-to-mathcal-M0} recovers the map \eqref{ProjwGZQ}.

\end{numberedparagraph}

\subsection{Comparison of integral models}
\begin{numberedparagraph}\label{construction-closure-models-ab-type}
Let $\mathscr{S}_{K_{\widetilde{G}}}(\widetilde{G},\{h_{\widetilde{G}}\})$ be the 
integral model defined over $\Oo_{E,(\nu)}$ for $\Sh_{K_{\widetilde{G}}}(\widetilde{G},\{h_{\widetilde{G}}\})$ as constructed in 
\cite[$\mathsection$4.6]{Kisin-Pappas}. 
Recall that the abelian type integral model $\mathscr{S}_{K_{\widetilde{G}}}(\widetilde{G},\{h_{\widetilde{G}}\})$ is built out of the integral model $\mathscr{S}(G^{\Q},\{h_{G^{\Q}}\})$ for a corresponding Hodge type Shimura variety associated to the abelian type $\Sh_{K_{\widetilde{G}}}(\widetilde{G},\{h_{\widetilde{G}}\})$, which, as in 
\cite[4.6.21]{Kisin-Pappas} (see also \cite[3.4.13]{Kisin-integral-model} and \cite{Deligne-Shimura2} for more details), is $\Sh_{K_{G^{\Q}}}(G^{\Q},\{h_{G^{\Q}}\})$. 
Recall that $\mathscr{S}(G^{\Q},\{h_{G^{\Q}}\})$ is the $\Oo_{F,(\nu)}$-scheme constructed by taking the flat closure $\mathscr{S}^-_{}(G^{\Q},\{h_{G^{\Q}}\})$ of the generic fibre $\Sh_{}(G^{\Q},\{h_{G^{\Q}}\})$ inside some suitable Siegel integral model $\mathscr{S}_{K'}(\GSp,S^{\pm})_{\Oo_{F,(\nu)}}$ for some prime $\nu$ of $F$ above a fixed prime $p$. For convenience of expositions, we shall fix a symplectic embedding $i:(G^{\Q},\{h_{G^{\Q}}\})\hookrightarrow(\GSp(V,\psi),S^{\pm})$.\footnote{Note that the independence on symplectic embeddings of Hodge type integral models was first proven in \cite[Theorem 2.3.8]{Kisin-integral-model}. For the parahoric integral models constructed in \cite{Kisin-Pappas}, it was later proven in \cite[Theorem 8.1.6]{pappas2021integral} that the Kisin-Pappas models (constructed under the assumption that the group $G$ in $(G,X)$ splits over a tamely ramified extension of $\Q_p$) are independent of the choice of a symplectic embedding. In our case, 
we use the same symplectic space $(V,\psi)$ to construct the right-hand side of Lemmas \ref{Gtilde-models-comparison} and \ref{G-model-comparison} as the one used on the left-hand-side compatible with $W$ from \eqref{defn-G-ResUW}.} 
By \cite{xu-normalization,PEL-embedding}\footnote{which reference to use depends on the specific level structure at $p$}, this flat closure
$\mathscr{S}^-_{}(G^{\Q},\{h_{G^{\Q}}\})\cong \mathscr{S}_{}(G^{\Q},\{h_{G^{\Q}}\})$ is the desired  
integral model\footnote{and it is moreover normal when in the setting of \cite{xu-normalization}.}. 
We shall use the model for $\Sh(G^{\Q},\{h_{G^{\Q}}\})$ as building blocks for integral models for $\Sh_{K_{\widetilde{G}}}(\widetilde{G},\{h_{\widetilde{G}}\})$ and $\Sh_{K_G}(G,\{h_G\})$.

Fix a connected component $\{h_{G^{\Q}}\}^+\subset \{h_{G^{\Q}}\}$, and let $\Sh(G^{\Q},\{h_{G^{\Q}}\})^+\subset \Sh(G^{\Q},\{h_{G^{\Q}}\})$ be the geometrically connected component which is the image of $\{h_{G^{\Q}}\}^+\times 1$. Let $F^p\subset\overline{F}$ be the maximal extension of $F$ that is unramified at primes dividing $p$. By \cite[Theorem 2.6.3]{Deligne-Shimura2}, the action of $\Gal(\overline{F}/F)$ on $\Sh_{K_{G^{\Q},p}}(G^{\Q},\{h_{G^{\Q}}\})^+$ factors through $\Gal(\overline{F}^p/F)$. We abuse the notation and still denote  $\Sh_{K_{G^{\Q},p}}(G^{\Q},\{h_{G^{\Q}}\})^+$ as the $F^p$-scheme obtained via descent. 
Let $\mathscr{S}_{K_{G^{\Q},p}}(G^{\Q})^+$ be the closure of $\Sh_{K_{G^{\Q},p}}(G^{\Q},\{h_{G^{\Q}}\})^+$ in $\mathscr{S}_{K_{G^{\Q},p}}^-(G^{\Q},\{h_{G^{\Q}}\})\otimes_{\Oo_{F,(\nu)}}\Oo_{F^p,(\nu)}$. Here the notation $\Oo_{F^p,(\nu)}$ denotes the ring of integers of $F^p$ localized at $(p)$.\footnote{Here we are abusing the notation $\nu$ to always denote the prime above $p$ in the relevant fields.} 

Let $\mathscr{A}(\widetilde{G}_{\Z_{(p)}})$ (resp.~$\mathscr{A}(G^{\Q}_{\Z_{(p)}})^{\circ}$) 
be the group defined in \cite[4.6.8]{Kisin-Pappas} 
for $\widetilde{G}$ (resp. $G^{\Q}$), which was originally defined in \cite{Deligne-Shimura2}. We recall that 
\begin{equation}\label{mathscrA-Gtilde-Zp}
\mathscr{A}(\widetilde{G}_{\Z_{(p)}}):=\widetilde{G}(\A_f^p)/Z_{\widetilde{G}}(\Z_{(p)})^-*_{\widetilde{G}^{\circ}(\Z_{(p)})_+/Z_{\widetilde{G}}(\Z_{(p)})}(G^{\Q})^{\ad\circ}(\Z_{(p)})^+,
\end{equation}
and $\mathscr{A}(G^{\Q}_{\Z_{(p)}})^{\circ}:=(G^{\Q})^{\circ}(\Z_{(p)})_+^-/Z(\Z_{(p)})^-*_{(G^{\Q})^{\circ}(\Z_{(p)})_+/Z(\Z_{(p)})}(G^{\Q})^{\ad\circ}(\Z_{(p)})^+$, where $(G^{\Q})^{\circ}(\Z_{(p)})_+^-$ is the closure of $(G^{\Q})^{\circ}(\Z_{(p)})_+$ in $(G^{\Q})(\A_f^p)$. By \cite[Lemma 4.6.10]{Kisin-Pappas}, we have an inclusion 
\begin{equation}\label{inclusion-cosets-mathscrA-groups}
    \mathscr{A}(G^{\Q}_{\Z_{(p)}})^{\circ}\backslash \mathscr{A}(\widetilde{G}_{\Z_{(p)}})\hookrightarrow \mathscr{A}(G^{\Q})^{\circ}\backslash \mathscr{A}(\widetilde{G})/K_{\widetilde{G},p}. 
\end{equation}
Here
$\mathscr{A}(\widetilde{G}):=\widetilde{G}(\A_f)/\widetilde{Z}(\Q)^-*_{\widetilde{G}(\Q)_+/\widetilde{Z}(\Q)}\widetilde{G}^{\ad}(\Q)^+$
where $\widetilde{Z}(\Q)^-$ denotes the closure of $Z_{\widetilde{G}}(\Q)$ in $\widetilde{G}(\A_f)$, and 
\[\mathscr{A}(G^{\Q})^{\circ}:=G^{\Q}(\Q)_+^-/Z(\Q)^-*_{G^{\Q}_+/Z(\Q)}(G^{\Q})^{\ad}(\Q)^+,\]
where $G^{\Q}(\Q)_+^-$ denotes the closure of $G^{\Q}(\Q)_+$ in $G^{\Q}(\A_f)$. 
Let $\widetilde{J}\subset \widetilde{G}(\Q_p)$ denote a set which maps bijectively to a set of coset representatives for the image of $\mathscr{A}(\widetilde{G}_{\Z_{(p)}})$ in $\mathscr{A}(G^{\Q})^{\circ}\backslash \mathscr{A}(\widetilde{G})/K_{\widetilde{G},p}$ under \eqref{inclusion-cosets-mathscrA-groups}. 
Recall from 
\cite[4.6.15]{Kisin-Pappas}, we have 
\begin{equation}\label{defn-mathscrS-Gtilde}
    \mathscr{S}_{K_{\widetilde{G},p}}(\widetilde{G},\{h_{\widetilde{G}}\})=\Big[[\mathscr{A}(\widetilde{G}_{\Z_{(p)}})\times 
    \mathscr{S}_{K_{G^{\Q},p}}(G^{\Q})^+]/\mathscr{A}(G^{\Q}_{\Z_{(p)}})^{\circ}\Big]^{|\widetilde{J}|}.
\end{equation}
Note that by analogous arguments as 
\textit{loc.cit.}
, the right-hand side of \eqref{defn-mathscrS-Gtilde} has a natural structure of a $\Oo_{E,(\nu)}:=\Oo_E\otimes_{\Oo_F}\Oo_{F,(\nu)}$-scheme with $\widetilde{G}(\A_f^p)$-action and is a model for $\Sh_{K_{\widetilde{G},p}}(\widetilde{G},\{h_{\widetilde{G}}\})$. Moreover, for sufficiently small $K_{\widetilde{G}}^p$, the quotient $\mathscr{S}_{K_{\widetilde{G},p}}(\widetilde{G},\{h_{\widetilde{G}}\})/K_{\widetilde{G}}^p:=\mathscr{S}_{K_{\widetilde{G}}}(\widetilde{G},\{h_{\widetilde{G}}\})$ is a finite type $\Oo_{E,(\nu)}$-scheme extending $\Sh_{K_{\widetilde{G}}}(\widetilde{G},\{h_{\widetilde{G}}\})$. 
\end{numberedparagraph}

\begin{lem-appendix}\label{Gtilde-models-comparison}
$\mathcal{M}_{K_{\widetilde{G}}}(\widetilde{G},\{h_{\widetilde{G}}\})\cong\mathscr{S}_{K_{\widetilde{G}}}(\widetilde{G},\{h_{\widetilde{G}}\})$ as $\Spec\Oo_{E,(\nu)}$-schemes.
\end{lem-appendix}
\begin{proof}

The moduli description for $\mathcal{M}_{K_{\widetilde{G}}}(\widetilde{G},\{h_{\widetilde{G}}\})$ in \ref{defn-mathcal-M-Gtilde}  
induces a natural map 
\begin{align}\label{forgetting-A0-map}
    \mathcal{M}_{K_{\widetilde{G},p}}(\widetilde{G},\{h_{\widetilde{G}}\})&\to \mathscr{S}_{K_{G^{\Q},p}}(G^{\Q})\\
    (A_0,\iota_0,\lambda_0,A,\iota,\lambda,\overline{\eta}^p)&\mapsto (A,\iota,\lambda)
\end{align}
by simply forgetting the component $(A_0,\iota_0,\lambda_0,\overline{\eta}^p)$ in the tuple. 
Note that the map \eqref{forgetting-A0-map} is proper, in particular closed. 
We fix an arbitrary $(A^{\star},\iota^{\star},\lambda^{\star})\in \mathscr{S}_{K_{G^{\Q},p}}(G^{\Q})^+$, and suppose  
\[(A_0^{\star},\iota_0^{\star},\lambda_0^{\star},A^{\star},\iota^{\star},\lambda^{\star},\eta^{\star})\mapsto (A^{\star},\iota^{\star},\lambda^{\star})\]
under the map \eqref{forgetting-A0-map}. Take any $(h,\gamma^{-1})\in\mathscr{A}(\widetilde{G}_{\Z_{(p)}})$. As in \cite[4.5.3]{Kisin-Pappas}, let $\widetilde{\mathcal{P}}_{\gamma}\subset\widetilde{G}$ be the torsor given by the fibre over $\gamma\in \widetilde{G}^{\ad}(\Z_{(p)})$. First we check that 
\[(h,\gamma^{-1})\cdot (A_0^{\star},\lambda_0^{\star}, \iota_0^{\star},\eta^{\star})=((A_0^{\star})^{\widetilde{\mathcal{P}}_{\gamma}}, (\lambda_0^{\star})^{\widetilde{\mathcal{P}}_{\gamma}}, (\iota_0^{\star})^{\widetilde{\mathcal{P}}_{\gamma}}, (\eta^{\star})^{\widetilde{\mathcal{P}}_{\gamma}})\]
gives another point in the fibre over $(A^{\star},\iota^{\star},\lambda^{\star})$ under the map \eqref{forgetting-A0-map}. This is clear as we only need to check that $(\eta^{\star})^{\widetilde{\mathcal{P}}_{\gamma}}$ are $\A_{F,f}$-linear isometries
\begin{equation}
    (\eta^{\star})^{\widetilde{\mathcal{P}}_{\gamma}}: \widehat{V}\Big((A_0^{\star})^{\widetilde{\mathcal{P}}_{\gamma}},A^{\star}\Big)\simeq -W\otimes_F\A_{F,f}, 
\end{equation}
but this is simply given by the composite $\widetilde{\gamma}^{-1}\circ\eta^{\star}\circ\iota_{\widetilde{\gamma}}^{-1}$, where $\iota_{\widetilde{\gamma}}$ is as defined in \cite[4.5.3]{Kisin-Pappas}.

It then remains to check that $\ker(\mathscr{A}(G^{\Q}_{\Z_{(p)}})^{\circ}\to \mathscr{A}(\widetilde{G}_{\Z_{(p)}}))$ acts freely on $\mathscr{S}_{K_{G^{\Q},p}}(G^{\Q})^+$, and this follows from \cite[4.6.17]{Kisin-Pappas}
and the fact that $\ker(\mathscr{A}(G^{\Q}_{\Z_{(p)}})^{\circ}\to \mathscr{A}(\widetilde{G}_{\Z_{(p)}}))$ is a subgroup of $\Delta(G^{\Q},(G^{\Q})^{\ad}):=\ker(\mathscr{A}(G^{\Q}_{\Z_{(p)}})\to \mathscr{A}(G^{\Q\ad}_{\Z_{(p)}}))$. 
Thus 
\begin{equation}\label{Gtilde-isom-SGtilde}
\mathcal{M}_{K_{\widetilde{G},p}}(\widetilde{G},\{h_{\widetilde{G}}\})
\cong\Big[[\mathscr{A}(\widetilde{G}_{\Z_{(p)}})\times \mathscr{S}_{K_{G^{\Q},p}}(G^{\Q})^+]/\mathscr{A}(G^{\Q}_{\Z_{(p)}})^{\circ}\Big]^{|\widetilde{J}|}\cong \mathscr{S}_{K_{\widetilde{G},p}}(\widetilde{G},\{h_{\widetilde{G}}\}). 
\end{equation}
In particular, $\mathcal{M}_{K_{\widetilde{G}}}(\widetilde{G},\{h_{\widetilde{G}}\})\cong\mathscr{S}_{K_{\widetilde{G}}}(\widetilde{G},\{h_{\widetilde{G}}\})$. 
\end{proof}

\begin{numberedparagraph}
For an arbitrary extension $L/E$, taking the fibre in \eqref{map-mathcalM-Gtilde-to-mathcal-M0} over a fixed $\Oo_{L,(\nu)}$-point $(A_0^{\star},\iota_0^{\star},\lambda_0^{\star})$ of $\mathcal{M}_0^{\Oo_F,\xi}$ gives a flat integral model $\mathcal{M}^{\star}_{K_G}(G,\{h_G\})$ over $\Oo_L$. Here we use the upper script $\star$ to emphasize that the model $\mathcal{M}^{\star}_{K_G}(G,\{h_G\})$ thus obtained a priori depends on the choice of a base point $(A_0^{\star},\iota_0^{\star},\lambda_0^{\star})$. On the other hand, recall from \ref{section-defn-Shimura-datum-G} that the reflex field for $(G,\{h_G\})$ is $F$, by 
\cite{Kisin-Pappas} we also have a normal integral model $\mathscr{S}_{K_G}(G,\{h_G\})$ over $\Spec \Oo_{F,(v)}$, which is given by
\begin{equation}\label{defn-mathscrS-G}
    \mathscr{S}_{K_{G,p}}(G,\{h_G\}):=\Big[[\mathscr{A}(G_{\Z_{(p)}})\times \mathscr{S}_{K_{G^{\Q},p}}(G^{\Q}    )^+_{\Oo_{L,(\nu)}}]/\mathscr{A}(G^{\Q}_{\Z_{(p)}})^{\circ}\Big]^{|J|}
\end{equation}
Here $\mathscr{A}(G_{\Z_{(p)}})$ is the analogous group for $G$ as defined in \eqref{mathscrA-Gtilde-Zp}, and $J\subset G(\Q_p)$ denote the set analogous to $\widetilde{J}$ defined above \eqref{defn-mathscrS-Gtilde}, using the analogous map for $G$ as in \eqref{inclusion-cosets-mathscrA-groups}.

\begin{lem-appendix}\label{G-model-comparison}
$\mathcal{M}_{K_G}(G,\{h_G\})\cong\mathscr{S}_{K_G}(G,\{h_G\})_{\Oo_{L,(\nu)}}$ as $\Spec\Oo_{L,(\nu)}$-schemes.
\end{lem-appendix}
\begin{proof}
Consider the map 
\begin{align}\label{forgetting-A0-map-for-G}
    \mathcal{M}^{\star}_{K_{G,p}}(G,\{h_{G}\})&\to \mathscr{S}_{K_{G^{\Q},p}}(G^{\Q})_{\Oo_{L,(\nu)}}
    \\
    (A_0^{\star},\iota_0^{\star},\lambda_0^{\star},A,\iota,\lambda,\overline{\eta}^p)&\mapsto (A,\iota,\lambda)
\end{align}
given by forgetting the component $(A_0^{\star},\iota_0^{\star},\lambda_0^{\star})$ in the tuple. 
Let $\mathscr{S}_{K_{G^{\Q},p}}^{\star}(G^{\Q})
_{\Oo_{L,(\nu)}}$ denote the image of the map \eqref{forgetting-A0-map-for-G}. 
We take an arbitrary $(A,\iota,\lambda)\in \mathscr{S}_{K_{G^{\Q},p}}^{\star}(G^{\Q})^+_{\Oo_{L,(\nu)}}$, and thus by construction of $\mathcal{M}^{\star}_{K_{G}}(G,\{h_{G}\})$ we clearly have
\begin{equation}
    (A_0^{\star},\iota_0^{\star},\lambda_0^{\star},A,\iota,\lambda,\overline{\eta}^p)\in \mathcal{M}_{K_{\widetilde{G},p}}(\widetilde{G}).
\end{equation}
Take any $(h,\gamma^{-1})\in\mathscr{A}(G_{\Z_{(p)}})$. In particular, $\gamma\in G^{\ad}\cong (G^{\Q})^{\ad}$. Again as in \cite[4.5.3]{Kisin-Pappas}, let $\mathcal{P}_{\gamma}\subset G^{\Q}$ be the torsor given by the fibre over $\gamma\in (G^{\Q})^{\ad}(\Z_{(p)})$. By the same reasoning as in the proof of Lemma \ref{Gtilde-models-comparison}, we also have  
\[(h,\gamma^{-1})\cdot (A_0^{\star},\iota_0^{\star},\lambda_0^{\star},A,\iota,\lambda,\overline{\eta}^p)=(A_0^{\star},\iota_0^{\star},\lambda_0^{\star},A^{\mathcal{P}_{\gamma}}, \lambda^{\mathcal{P}_{\gamma}}, \iota^{\mathcal{P}_{\gamma}}, \eta^{\mathcal{P}_{\gamma}})\in \mathcal{M}_{K_{\widetilde{G},p}}(\widetilde{G})\]
gives another point in the fibre over $(A_0^{\star},\iota_0^{\star},\lambda_0^{\star})$ under the map \eqref{map-mathcalM-Gtilde-to-mathcal-M0}. The rest of the argument proceeds similarly as in the proof of Lemma \ref{Gtilde-models-comparison}, i.e.~the kernel 
$\ker(\mathscr{A}(G^{\Q}_{\Z_{(p)}})^{\circ}\to \mathscr{A}(G_{\Z_{(p)}}))$ acts freely on $\mathscr{S}_{K_{G^{\Q},p}}(G^{\Q})^+$. 
In particular, we have
\begin{equation}
\mathcal{M}^{\star}_{K_{G,p}}(G)\cong \Big[[\mathscr{A}(G_{\Z_{(p)}})\times \mathscr{S}_{K_{G^{\Q},p}}^{\star}(G^{\Q})_{\Oo_{L,(\nu)}}^+]/\mathscr{A}(G^{\Q}_{\Z_{(p)}})^{\circ}\Big]^{|J|},
\end{equation}
and thus 
$\mathcal{M}^{\star}_{K_G}(G)\cong\mathscr{S}_{K_G}(G)_{\Oo_{L,(\nu)}}$. \\
(Since the choice of base point $\star$ does not affect the proof, we may drop the upper script $\star$ from our notations.) 
\end{proof}

\begin{numberedparagraph}
\label{Drinfeld level structure}
We consider the Drinfeld level structure integral models analogous to those in \cite[$\mathsection$ 4.3]{RSZ-arithmetic-diagonal-cycles}. 
Consider the embedding $\widetilde{\nu}: \overline{\Q}\hookrightarrow\overline{\Q}_p$, which identifies 
\begin{equation}\label{nu-tilde-Hom-id}
    \Hom_{\Q}(F,\overline{\Q})\simeq \Hom_{\Q}(F,\overline{\Q}_p). 
\end{equation}
The above identification \eqref{nu-tilde-Hom-id} then gives an identification
\begin{equation}\label{restricting-to-Fw}
    \{\varphi\in\Hom_{\Q}(F,\overline{\Q})|w_{\varphi}=w\}\simeq \Hom_{\Q_p}(F_w,\overline{\Q}_p),
\end{equation}
where $w_{\varphi}$ denote the $p$-adic place in $F$ induced by $\widetilde{\nu}\circ\varphi$.

We fix a place $v_0$ of $F$ over $p$ that is split in $F$ (and possibly ramified over $p$) into $w_0$ and another place $\overline{w}_0$ in $F$. We require moreover that the CM type $\Phi$ considered in \eqref{CM-type-Phi} and the chosen place $\nu$ of $E$ above $p$ satisfy the following \textit{matching condition}:
\begin{equation}\label{matching-condition}
    \{\varphi\in \Hom(F,\overline{\Q})|w_{\varphi}=w_0\}\subset\Phi.
\end{equation}
This condition \eqref{matching-condition} only depends on the place $\nu$ of $E$ induced by $\widetilde{\nu}$.

Now we introduce a \textit{Drinfeld level structure} at $v_0$. Recall the level structure subgroup $K_G$ from \eqref{level-KGtilde-decomp}. We define a variant compact open subgroup $K_G^m\subset G(\A_{F_0,f})$ in exactly the same way as $K_G$, except that, in the $v_0$-factor, we require $K_{G,v_0}^m\subset G(F_{0,v_0})$ to be the principal congruence subgroup modulo $\mathfrak{p}_{v_0}^m$ inside $K_{G,v_0}$. Clearly $K_G=K_G^{m=0}$. As in \eqref{level-KGtilde-decomp}, we define $K_{\widetilde{G}}^m=K_{Z^{\Q}}\times K_G^m$. 

Let $\Lambda_{v_0}=\Lambda_{w_0}\oplus\Lambda_{\overline{w}_0}$ denote the natural decomposition of the lattice $\Lambda_{v_0}$ attached to the split place $v_0$. For a point $(A_0,\iota_0,\lambda_0,A,\iota,\lambda,\overline{\eta}^p)\in \mathcal{M}_{K_{\widetilde{G}}}(\widetilde{G})(S)$, we have a decomposition of $p$-divisible groups
\begin{equation}\label{first-decomp-pdivgp}
    A[p^{\infty}]=\prod\limits_{w|p}A[w^{\infty}], 
\end{equation}
where $w$ ranges over the places of $F$ lying over $p$. Moreover, we further decompose the $v_0$-term in \eqref{first-decomp-pdivgp} and consider
\begin{equation}\label{second-decomp-pdivgp-A}
    A[v_0^{\infty}]=A[w_0^{\infty}]\times A[\overline{w}_0^{\infty}],
\end{equation}
where, when $p$ is locally nilpotent on $S$, the $p$-divisible group $A[w_0^{\infty}]$ satisfies the Kottwitz condition of type $r|_{w_0}$ for the action of $\Oo_{F,w_0}$ on its Lie algebra, in the sense of \cite[$\mathsection$8]{Rapoport-Zink-Drin}. Here $r|_{w_0}$ denotes the restriction of the function $r$ on $\Hom_{\Q}(F,\overline{\Q})$ to $\Hom_{\Q_p}(F_{w_0},\overline{\Q}_p)$ under  \eqref{restricting-to-Fw}.

Likewise, we have the same decomposition as \eqref{second-decomp-pdivgp-A} for $A_0$, i.e. we have
\begin{equation}
    A_0[v_0^{\infty}]=A_0[w_0^{\infty}]\times A_0[\overline{w}_0^{\infty}].
\end{equation}

Let $\pi_{w_0}$ be a uniformizer of $F_{0,w_0}$. In addition to the moduli functor $\mathcal{M}_{K_{\widetilde{G}}}(\widetilde{G})$ which classifies tuples $(A_0,\iota_0,\lambda_0,A,\iota,\lambda,\overline{\eta}^p)$, we impose the following additional \textit{Drinfeld level structure} as in \cite[$\mathsection$ II.2]{Harris-Taylor}, i.e. 
\begin{itemize}\label{Drinfeld-level-structure}
    \item an $\Oo_{F,w_0}$-linear homomorphism of  finite flat group schemes
    \begin{equation}
    \eta:\pi_{w_0}^{-m}\Lambda_{w_0}/\Lambda_{w_0}\to \underline{\Hom}_{\Oo_{F,w_0}}(A_0[w_0^m],A[w_0^m]). 
    \end{equation}
\end{itemize}
We denote the resulting moduli problem by $\mathcal{M}_{K_{\widetilde{G}}^m}(\widetilde{G})$, which is relatively representable by a finite flat morphism to $\mathcal{M}_{K_{\widetilde{G}}}(\widetilde{G})$. In fact, $\mathcal{M}_{K_{\widetilde{G}}^m}(\widetilde{G})$ is regular and flat over $\Spec\Oo_{E,(\nu)}$ by \cite[Lemma III.4.1]{Harris-Taylor}.
\end{numberedparagraph}

\begin{numberedparagraph}\label{defn-mathscrS-G2}
Recall the integral model $\mathscr{S}_{K_{\widetilde{G}}}(\widetilde{G},\{h_{\widetilde{G}}\})$ (resp.~$\mathscr{S}_{K_G}(G,\{h_G\})$) defined in \eqref{defn-mathscrS-Gtilde} (resp.~\eqref{defn-mathscrS-G}). We define $\mathscr{S}_{K_{\widetilde{G}}^m}(\widetilde{G},\{h_{\widetilde{G}}\})$ (resp.~$\mathscr{S}_{K_G^m}(G,\{h_G\})_{\Oo_{L,(\nu)}}$) as the normalization of $\mathscr{S}_{K_{\widetilde{G}}}(\widetilde{G},\{h_{\widetilde{G}}\})$ (resp.~\\
$\mathscr{S}_{K_G}(G,\{h_G\})_{\Oo_{L, (\nu)}}$) inside $\Sh_{K_{\widetilde{G}}^m}(\widetilde{G},\{h_{\widetilde{G}}\})\cong M_{K_{\widetilde{G}}^m}(\widetilde{G})$ (resp.~$\Sh_{{K_G^m}}(G,\{h_G\})_L\cong M_{{K_G^m}}(G)_L$). 
\begin{Coro-appendix}\label{Corollary 2.30}
$\mathscr{S}_{K_{\widetilde{G}}^m}(\widetilde{G},\{h_{\widetilde{G}}\})\cong \mathcal{M}_{K_{\widetilde{G}}^m}(\widetilde{G})$ as $\Spec\Oo_{E,(\nu)}$-schemes, and 
\begin{equation}\mathcal{M}_{K_G^m}(G,\{h_G\})\cong\mathscr{S}_{K_G^m}(G,\{h_G\})_{\Oo_{L,(\nu)}}
\end{equation} 
as $\Spec\Oo_{L,(\nu)}$-schemes.
\end{Coro-appendix}
\begin{proof}
By Lemma \ref{Gtilde-models-comparison} (resp.~\ref{G-model-comparison}),  $\mathscr{S}_{K_{\widetilde{G}}^m}(\widetilde{G},\{h_{\widetilde{G}}\})$ (resp.~$\mathscr{S}_{K_G^m}(G,\{h_G\})_{\Oo_L}$) is the normalization of $\mathscr{S}_{K_{\widetilde{G}}}(\widetilde{G},\{h_{\widetilde{G}}\})\cong \mathcal{M}_{K_{\widetilde{G}}}(\widetilde{G},\{h_{\widetilde{G}}\})$ (resp.~$\mathscr{S}_{K_G}(G,\{h_G\})_{\Oo_L}\cong \mathcal{M}_{K_G}(G,\{h_G\})$) inside $\Sh_{K_{\widetilde{G}}^m}(\widetilde{G},\{h_{\widetilde{G}}\})\cong M_{K_{\widetilde{G}}^m}(\widetilde{G})$ (resp.~$\Sh_{{K_G^m}}(G)_L\cong M_{{K_G^m}}(G)_L$). 
Since $\mathcal{M}_{K_{\widetilde{G}}^m}(\widetilde{G})$ is regular and flat, in particular it is normal. Thus by \cite[IV-2, 6.14.1]{EGA-IV}, 
$\mathcal{M}_{K_G^m}(G)$ is normal (even though it may not necessarily be regular). By \cite[035I]{stacks-project} applied to the 
scheme $\mathcal{M}_{K_{\widetilde{G}}^m}(\widetilde{G})$ (resp.~$\mathcal{M}_{K_G^m}(G)$), there exists a unique morphism
$\mathscr{S}_{K_{\widetilde{G}}^m}(\widetilde{G},\{h_{\widetilde{G}}\})\to \mathcal{M}_{K_{\widetilde{G}}^m}(\widetilde{G})$ (resp.~$\mathscr{S}_{K_{G}^m}(G,\{h_{G}\})\to \mathcal{M}_{K_{G}^m}(G)$), 
which is the normalization of $\mathcal{M}_{K_{\widetilde{G}}^m}(\widetilde{G})$ (resp.~$\mathcal{M}_{K_G^m}(G)$) in $\Sh_{K_{\widetilde{G}}^m}(\widetilde{G},\{h_{\widetilde{G}}\})$ (resp.~$\Sh_{K_G^m}(G,\{h_G\})$). Since $\mathcal{M}_{K_{\widetilde{G}}^m}(\widetilde{G})$ (resp.~$\mathcal{M}_{K_G^m}(G)$) is already normal, we have an isomorphism $\mathscr{S}_{K_{\widetilde{G}}^m}(\widetilde{G},\{h_{\widetilde{G}}\})\cong \mathcal{M}_{K_{\widetilde{G}}^m}(\widetilde{G})$ (resp.~$\mathcal{M}_{K_G^m}(G,\{h_G\})\cong\mathscr{S}_{K_G^m}(G,\{h_G\})_{\Oo_{L,(\nu)}}$). 
\end{proof}

\end{numberedparagraph}

\end{numberedparagraph}

\begin{numberedparagraph}\label{AT-parahoric-level-section}
In this last section, we recall  
the construction of semi-global integral models with \textit{AT parahoric level} as in \cite[$\mathsection$ 4.4]{RSZ-arithmetic-diagonal-cycles}. Recall the notion of \textit{vertex lattice} from $\mathsection$\ref{defn-mathcal-M-Gtilde}. We say that a vertex lattice $\Lambda$ is \textit{almost self-dual} if it is a vertex lattice of type $1$. We say that a vertex lattice $\Lambda$ is \textit{$\pi_v$-modular} (resp.~\textit{almost $\pi_v$-modular}) if $\Lambda^*=\pi_v^{-1}\Lambda$ (resp.~$\Lambda\subset \Lambda^*\subset^1 \pi_v^{-1}\Lambda$).

Suppose $p\neq 2$ and $v_0$ is unramified over $p$. As in $\mathsection$\ref{defn-mathcal-M-Gtilde}, we take a vertex lattice $\Lambda_v\subset W_v$ for each prime $v$ of $F_0$ above $p$. Unlike in $\mathsection$\ref{defn-mathcal-M-Gtilde}, let $(v_0,\Lambda_{v_0})$ be of one of the following types:
\begin{enumerate}
    \item $v_0$ is inert in $F$ and $\Lambda_{v_0}$ is almost self-dual as an $\Oo_{F,v_0}$-lattice;
    \item $n$ is even, $v_0$ ramifies in $F$ and $\Lambda_{v_0}$ is $\pi_{v_0}$-modular;
    \item $n$ is odd, $v_0$ ramifies in $F$ and $\Lambda_{v_0}$ is almost $\pi_{v_0}$-modular;
    \item $n=2$, $v_0$ ramifies in $F$ and $\Lambda_{v_0}$ is self-dual.
\end{enumerate}
To the moduli functor $\mathcal{M}_{K_{\widetilde{G}}}(\widetilde{G})$ which classifies tuples $(A_0,\iota_0,\lambda_0,A,\iota,\lambda,\overline{\eta}^p)$ as in $\mathsection$\ref{defn-mathcal-M-Gtilde} (except that the condition on $(v_0,\Lambda_{v_0})$ is different), we impose the following additional condition:
\begin{itemize}
    \item When the pair $(v_0,\Lambda_{v_0})$ is of AT type (2), (3) or (4), we impose the \textit{Eisenstein condition} on the summand $\Lie_{\psi}A[v_0^{\infty}]$ \cite[4.10]{RSZ-arithmetic-diagonal-cycles};
    \item When the pair $(v_0,\Lambda_{v_0})$ is of AT type (2), we impose additionally the \textit{wedge condition} \cite[4.27]{RSZ-arithmetic-diagonal-cycles}
    and the \textit{spin condition} \cite[4.28]{RSZ-arithmetic-diagonal-cycles}
  \item When the pair $(v_0,\Lambda_{v_0})$ is of AT type (3), we impose additionally the \textit{refined spin condition} \cite[(7.9)]{RSZ-AT} on $\Lie_{\psi_0}A[v_0^{\infty}]$. 
\end{itemize}
By \cite[Theorem 4.7]{RSZ-arithmetic-diagonal-cycles}, the moduli functor above is representable by a Deligne-Mumford stack flat over $\Spec\Oo_{E,(\nu)}$ and relatively representable over $\mathcal{M}_0^{\Oo_F,\xi}$, i.e.~\eqref{map-mathcalM-Gtilde-to-mathcal-M0} still holds in this case. 
To see that $\mathcal{M}_{K_{\widetilde{G}}}(\widetilde{G},\{h_{\widetilde{G}}\})\cong\mathscr{S}_{K_{\widetilde{G}}}(\widetilde{G},\{h_{\widetilde{G}}\})$ as $\Spec\Oo_{E,(\nu)}$-schemes, one simply proceeds as in Lemma \ref{Gtilde-models-comparison}.

\end{numberedparagraph}

\bibliographystyle{amsalpha}
\bibliography{bibfile}

\newcommand{\etalchar}[1]{$^{#1}$}
\providecommand{\bysame}{\leavevmode\hbox to3em{\hrulefill}\thinspace}
\providecommand{\MR}{\relax\ifhmode\unskip\space\fi MR }
\providecommand{\MRhref}[2]{%
  \href{http://www.ams.org/mathscinet-getitem?mr=#1}{#2}
}
\providecommand{\href}[2]{#2}
\begin{thebibliography}{BHK{\etalchar{+}}20b}

\bibitem[Ara74]{Ara}
S.~Ju. Arakelov, \emph{An intersection theory for divisors on an arithmetic
  surface}, Izv. Akad. Nauk SSSR Ser. Mat. \textbf{38} (1974), 1179--1192.
  \MR{0472815}

\bibitem[BBGK07]{BBGK}
Jan~H. Bruinier, Jos\'{e}~I. Burgos~Gil, and Ulf K\"{u}hn, \emph{Borcherds
  products and arithmetic intersection theory on {H}ilbert modular surfaces},
  Duke Math. J. \textbf{139} (2007), no.~1, 1--88. \MR{2322676}

\bibitem[BGS94]{BGS}
J.-B. Bost, H.~Gillet, and C.~Soul\'{e}, \emph{Heights of projective varieties
  and positive {G}reen forms}, J. Amer. Math. Soc. \textbf{7} (1994), no.~4,
  903--1027. \MR{1260106}

\bibitem[BH21]{BH}
Jan~Hendrik Bruinier and Benjamin Howard, \emph{Arithmetic volumes of unitary
  shimura varieties}, arXiv preprint arXiv:2105.11274 (2021).

\bibitem[BHK{\etalchar{+}}20a]{Bet}
Jan~H. Bruinier, Benjamin Howard, Stephen~S. Kudla, Michael Rapoport, and
  Tonghai Yang, \emph{Modularity of generating series of divisors on unitary
  {S}himura varieties}, Ast\'{e}risque (2020), no.~421, Diviseurs
  arithm\'{e}tiques sur les vari\'{e}t\'{e}s orthogonales et unitaires de
  Shimura, 7--125. \MR{4183376}

\bibitem[BHK{\etalchar{+}}20b]{Bet2}
\bysame, \emph{Modularity of generating series of divisors on unitary {S}himura
  varieties {II}: {A}rithmetic applications}, Ast\'{e}risque (2020), no.~421,
  Diviseurs arithm\'{e}tiques sur les vari\'{e}t\'{e}s orthogonales et
  unitaires de Shimura, 127--186. \MR{4183377}

\bibitem[BHY15]{BHY}
Jan~Hendrik Bruinier, Benjamin Howard, and Tonghai Yang, \emph{Heights of
  {K}udla-{R}apoport divisors and derivatives of {$L$}-functions}, Invent.
  Math. \textbf{201} (2015), no.~1, 1--95. \MR{3359049}

\bibitem[BLR90]{BLR}
Siegfried Bosch, Werner L\"{u}tkebohmert, and Michel Raynaud, \emph{N\'{e}ron
  models}, Ergebnisse der Mathematik und ihrer Grenzgebiete (3) [Results in
  Mathematics and Related Areas (3)], vol.~21, Springer-Verlag, Berlin, 1990.
  \MR{1045822}

\bibitem[Bor99]{Bor}
Richard~E. Borcherds, \emph{The {G}ross-{K}ohnen-{Z}agier theorem in higher
  dimensions}, Duke Math. J. \textbf{97} (1999), no.~2, 219--233. \MR{1682249}

\bibitem[Bru02]{Bru0}
Jan~H. Bruinier, \emph{Borcherds products on {O}(2, {$l$}) and {C}hern classes
  of {H}eegner divisors}, Lecture Notes in Mathematics, vol. 1780,
  Springer-Verlag, Berlin, 2002. \MR{1903920}

\bibitem[Bru12]{Bru}
Jan~Hendrik Bruinier, \emph{Regularized theta lifts for orthogonal groups over
  totally real fields}, J. Reine Angew. Math. \textbf{672} (2012), 177--222.
  \MR{2995436}

\bibitem[BY09]{BY}
Jan~Hendrik Bruinier and Tonghai Yang, \emph{Faltings heights of {CM} cycles
  and derivatives of {$L$}-functions}, Invent. Math. \textbf{177} (2009),
  no.~3, 631--681. \MR{2534103}

\bibitem[Cha90]{Chai}
C.-L. Chai, \emph{Arithmetic minimal compactification of the
  {H}ilbert-{B}lumenthal moduli spaces}, Ann. of Math. (2) \textbf{131} (1990),
  no.~3, 541--554. \MR{1053489}

\bibitem[Del71]{Deligne-travaxu-de-Shimura}
Pierre Deligne, \emph{Travaux de {S}himura}, S\'{e}minaire {B}ourbaki, 23\`eme
  ann\'{e}e (1970/71), {E}xp. {N}o. 389, 1971, pp.~123--165. Lecture Notes in
  Math., Vol. 244. \MR{0498581}

\bibitem[Del79]{Deligne-Shimura2}
\bysame, \emph{Vari\'{e}t\'{e}s de {S}himura: interpr\'{e}tation modulaire, et
  techniques de construction de mod\`eles canoniques}, Automorphic forms,
  representations and {$L$}-functions ({P}roc. {S}ympos. {P}ure {M}ath.,
  {O}regon {S}tate {U}niv., {C}orvallis, {O}re., 1977), {P}art 2, Proc. Sympos.
  Pure Math., XXXIII, Amer. Math. Soc., Providence, R.I., 1979, pp.~247--289.
  \MR{546620}

\bibitem[ES18]{ES18}
Stephan Ehlen and Siddarth Sankaran, \emph{On two arithmetic theta lifts},
  Compos. Math. \textbf{154} (2018), no.~10, 2090--2149. \MR{3867297}

\bibitem[Ful84]{Ful}
William Fulton, \emph{Intersection theory}, Ergebnisse der Mathematik und ihrer
  Grenzgebiete (3) [Results in Mathematics and Related Areas (3)], vol.~2,
  Springer-Verlag, Berlin, 1984. \MR{732620}

\bibitem[GGP12]{GGP}
Wee~Teck Gan, Benedict~H. Gross, and Dipendra Prasad, \emph{Symplectic local
  root numbers, central critical {$L$} values, and restriction problems in the
  representation theory of classical groups}, no. 346, 2012, Sur les
  conjectures de Gross et Prasad. I, pp.~1--109. \MR{3202556}

\bibitem[GI14]{GI}
Wee~Teck Gan and Atsushi Ichino, \emph{Formal degrees and local theta
  correspondence}, Invent. Math. \textbf{195} (2014), no.~3, 509--672.
  \MR{3166215}

\bibitem[Gil09]{Gil}
Henri Gillet, \emph{Arithmetic intersection theory on {D}eligne-{M}umford
  stacks}, Motives and algebraic cycles, Fields Inst. Commun., vol.~56, Amer.
  Math. Soc., Providence, RI, 2009, pp.~93--109. \MR{2562454}

\bibitem[Gro67]{EGA-IV}
A.~Grothendieck, \emph{\'{E}l\'{e}ments de g\'{e}om\'{e}trie alg\'{e}brique.
  {IV}. \'{E}tude locale des sch\'{e}mas et des morphismes de sch\'{e}mas
  {IV}}, Inst. Hautes \'{E}tudes Sci. Publ. Math. (1967), no.~32, 361.
  \MR{238860}

\bibitem[Gro86]{Gro}
Benedict~H. Gross, \emph{On canonical and quasicanonical liftings}, Invent.
  Math. \textbf{84} (1986), no.~2, 321--326. \MR{833193}

\bibitem[GS90]{GS}
Henri Gillet and Christophe Soul\'{e}, \emph{Arithmetic intersection theory},
  Inst. Hautes \'{E}tudes Sci. Publ. Math. (1990), no.~72, 93--174 (1991).
  \MR{1087394}

\bibitem[GS19]{GS19}
Luis~E. Garcia and Siddarth Sankaran, \emph{Green forms and the arithmetic
  {S}iegel-{W}eil formula}, Invent. Math. \textbf{215} (2019), no.~3, 863--975.
  \MR{3935034}

\bibitem[GZ86]{GZ}
Benedict~H. Gross and Don~B. Zagier, \emph{Heegner points and derivatives of
  {$L$}-series}, Invent. Math. \textbf{84} (1986), no.~2, 225--320. \MR{833192}

\bibitem[HMP20]{HMP}
Benjamin Howard and Keerthi Madapusi~Pera, \emph{Arithmetic of {B}orcherds
  products}, Ast\'{e}risque (2020), no.~421, Diviseurs arithm\'{e}tiques sur
  les vari\'{e}t\'{e}s orthogonales et unitaires de Shimura, 187--297.
  \MR{4183378}

\bibitem[How79]{Ho}
R.~Howe, \emph{{$\theta $}-series and invariant theory}, Automorphic forms,
  representations and {$L$}-functions ({P}roc. {S}ympos. {P}ure {M}ath.,
  {O}regon {S}tate {U}niv., {C}orvallis, {O}re., 1977), {P}art 1, Proc. Sympos.
  Pure Math., XXXIII, Amer. Math. Soc., Providence, R.I., 1979, pp.~275--285.
  \MR{546602}

\bibitem[HT01]{Harris-Taylor}
Michael Harris and Richard Taylor, \emph{The geometry and cohomology of some
  simple {S}himura varieties}, Annals of Mathematics Studies, vol. 151,
  Princeton University Press, Princeton, NJ, 2001, With an appendix by Vladimir
  G. Berkovich. \MR{1876802}

\bibitem[Ich04]{Ich}
Atsushi Ichino, \emph{A regularized {S}iegel-{W}eil formula for unitary
  groups}, Math. Z. \textbf{247} (2004), no.~2, 241--277. \MR{2064052}

\bibitem[Jac62]{Jaco}
Ronald Jacobowitz, \emph{Hermitian forms over local fields}, Amer. J. Math.
  \textbf{84} (1962), 441--465. \MR{150128}

\bibitem[K\"01]{Kuh}
Ulf K\"{u}hn, \emph{Generalized arithmetic intersection numbers}, J. Reine
  Angew. Math. \textbf{534} (2001), 209--236. \MR{1831639}

\bibitem[Kis10]{Kisin-integral-model}
Mark Kisin, \emph{Integral models for {S}himura varieties of abelian type}, J.
  Amer. Math. Soc. \textbf{23} (2010), no.~4, 967--1012.

\bibitem[KP18]{Kisin-Pappas}
M.~Kisin and G.~Pappas, \emph{Integral models of {S}himura varieties with
  parahoric level structure}, Publ. Math. Inst. Hautes \'{E}tudes Sci.
  \textbf{128} (2018), 121--218. \MR{3905466}

\bibitem[KR11]{KR11}
Stephen Kudla and Michael Rapoport, \emph{Special cycles on unitary {S}himura
  varieties {I}. {U}nramified local theory}, Invent. Math. \textbf{184} (2011),
  no.~3, 629--682. \MR{2800697}

\bibitem[KR14]{KR14}
\bysame, \emph{Special cycles on unitary {S}himura varieties {II}: {G}lobal
  theory}, J. Reine Angew. Math. \textbf{697} (2014), 91--157. \MR{3281653}

\bibitem[KRY06]{KRY2}
Stephen~S. Kudla, Michael Rapoport, and Tonghai Yang, \emph{Modular forms and
  special cycles on {S}himura curves}, Annals of Mathematics Studies, vol. 161,
  Princeton University Press, Princeton, NJ, 2006. \MR{2220359}

\bibitem[Kud97a]{Kud97}
Stephen~S. Kudla, \emph{Algebraic cycles on {S}himura varieties of orthogonal
  type}, Duke Math. J. \textbf{86} (1997), no.~1, 39--78. \MR{1427845}

\bibitem[Kud97b]{Kud1}
\bysame, \emph{Central derivatives of {E}isenstein series and height pairings},
  Ann. of Math. (2) \textbf{146} (1997), no.~3, 545--646. \MR{1491448}

\bibitem[Kud02]{MR1886765}
\bysame, \emph{Derivatives of {E}isenstein series and generating functions for
  arithmetic cycles}, no. 276, 2002, S\'{e}minaire Bourbaki, Vol. 1999/2000,
  pp.~341--368. \MR{1886765}

\bibitem[Kud03]{Kud02}
\bysame, \emph{Modular forms and arithmetic geometry}, Current developments in
  mathematics, 2002, Int. Press, Somerville, MA, 2003, pp.~135--179.
  \MR{2062318}

\bibitem[Kud04]{Kud03}
\bysame, \emph{Special cycles and derivatives of {E}isenstein series}, Heegner
  points and {R}ankin {$L$}-series, Math. Sci. Res. Inst. Publ., vol.~49,
  Cambridge Univ. Press, Cambridge, 2004, pp.~243--270. \MR{2083214}

\bibitem[KZ21]{Kisin-Zhou}
Mark Kisin and Rong Zhou, \emph{Independence of $\ell$ for frobenius conjugacy
  classes attached to abelian varieties}, 2021.

\bibitem[Liu11a]{Liu}
Yifeng Liu, \emph{Arithmetic theta lifting and {$L$}-derivatives for unitary
  groups, {I}}, Algebra Number Theory \textbf{5} (2011), no.~7, 849--921.
  \MR{2928563}

\bibitem[Liu11b]{Liu0}
\bysame, \emph{Arithmetic theta lifting and {$L$}-derivatives for unitary
  groups, {II}}, Algebra Number Theory \textbf{5} (2011), no.~7, 923--1000.
  \MR{2928564}

\bibitem[LL21]{LL}
Chao Li and Yifeng Liu, \emph{Chow groups and {$L$}-derivatives of automorphic
  motives for unitary groups}, Ann. of Math. (2) \textbf{194} (2021), no.~3,
  817--901. \MR{4334978}

\bibitem[LL22]{LL2}
\bysame, \emph{Chow groups and {$L$}-derivatives of automorphic motives for
  unitary groups, {II}}, Forum Math. Pi \textbf{10} (2022), Paper No. e5, 71.
  \MR{4390300}

\bibitem[LTX{\etalchar{+}}22]{Let}
Yifeng Liu, Yichao Tian, Liang Xiao, Wei Zhang, and Xinwen Zhu, \emph{On the
  {B}eilinson-{B}loch-{K}ato conjecture for {R}ankin-{S}elberg motives},
  Invent. Math. \textbf{228} (2022), no.~1, 107--375. \MR{4392458}

\bibitem[LZ21]{LZ}
Chao Li and Wei Zhang, \emph{Kudla--rapoport cycles and derivatives of local
  densities}, Journal of the American Mathematical Society (2021).

\bibitem[MR92]{KRa}
V.~Kumar Murty and Dinakar Ramakrishnan, \emph{The {A}lbanese of unitary
  {S}himura varieties}, The zeta functions of {P}icard modular surfaces, Univ.
  Montr\'{e}al, Montreal, QC, 1992, pp.~445--464. \MR{1155237}

\bibitem[MZ21]{MZ}
Andreas Mihatsch and Wei Zhang, \emph{On the arithmetic fundamental lemma
  conjecture over a general $ p $-adic field}, arXiv preprint arXiv:2104.02779
  (2021).

\bibitem[OT03]{OT}
Takayuki Oda and Masao Tsuzuki, \emph{Automorphic {G}reen functions associated
  with the secondary spherical functions}, Publ. Res. Inst. Math. Sci.
  \textbf{39} (2003), no.~3, 451--533. \MR{2001185}

\bibitem[Pap21]{pappas2021integral}
G.~Pappas, \emph{On integral models of shimura varieties}, 2021.

\bibitem[Qiu21]{Qiu21}
Congling Qiu, \emph{Modularity and heights of cm cycles on kuga-sato
  varieties}, arXiv preprint arXiv:2105.12561 (2021).

\bibitem[Ral82]{Ral}
Stephen Rallis, \emph{Langlands' functoriality and the {W}eil representation},
  Amer. J. Math. \textbf{104} (1982), no.~3, 469--515. \MR{658543}

\bibitem[RSZ17]{RSZ17}
M.~Rapoport, B.~Smithling, and W.~Zhang, \emph{On the arithmetic transfer
  conjecture for exotic smooth formal moduli spaces}, Duke Math. J.
  \textbf{166} (2017), no.~12, 2183--2336. \MR{3694568}

\bibitem[RSZ18]{RSZ-AT}
\bysame, \emph{Regular formal moduli spaces and arithmetic transfer
  conjectures}, Math. Ann. \textbf{370} (2018), no.~3-4, 1079--1175.
  \MR{3770164}

\bibitem[RSZ20]{RSZ-arithmetic-diagonal-cycles}
\bysame, \emph{Arithmetic diagonal cycles on unitary {S}himura varieties},
  Compos. Math. \textbf{156} (2020), no.~9, 1745--1824. \MR{4167594}

\bibitem[RZ17]{Rapoport-Zink-Drin}
Michael Rapoport and Thomas Zink, \emph{On the {D}rinfeld moduli problem of
  {$p$}-divisible groups}, Camb. J. Math. \textbf{5} (2017), no.~2, 229--279.
  \MR{3653061}

\bibitem[{Sta}18]{stacks-project}
The {Stacks Project Authors}, \emph{\textit{Stacks Project}},
  \url{https://stacks.math.columbia.edu}, 2018.

\bibitem[Tan99]{Tan}
Victor Tan, \emph{Poles of {S}iegel {E}isenstein series on {${\rm U}(n,n)$}},
  Canad. J. Math. \textbf{51} (1999), no.~1, 164--175. \MR{1692899}

\bibitem[Tr{\'e}67]{Tre}
Fran\c{c}ois Tr{\'e}ves, \emph{Topological vector spaces, distributions and
  kernels}, Academic Press, New York-London, 1967. \MR{0225131}

\bibitem[Vis04]{Vis1}
Angelo Vistoli, \emph{Notes on grothendieck topologies, fibered categories and
  descent theory}, arXiv preprint math/0412512 (2004).

\bibitem[Wal03]{Waldsch}
Michel Waldschmidt, \emph{Linear independence measures for logarithms of
  algebraic numbers}, Diophantine approximation ({C}etraro, 2000), Lecture
  Notes in Math., vol. 1819, Springer, Berlin, 2003, pp.~250--344. \MR{2009832}

\bibitem[Xu21]{xu-normalization}
Yujie Xu, \emph{Normalization in integral models of shimura varieties of hodge
  type}, 2021.

\bibitem[Xu25]{PEL-embedding}
Yujie Xu, \emph{On the {H}odge embedding for integral models of {S}himura
  varieties}, Boll. Unione Mat. Ital. \textbf{18} (2025), no.~2, 561--575.

\bibitem[Yua22]{Yuan}
Xinyi Yuan, \emph{Modular heights of quaternionic shimura curves}, arXiv
  preprint arXiv:2205.13995 (2022).

\bibitem[YZ17]{YZ0}
Xinyi Yuan and Shou-Wu Zhang, \emph{The arithmetic {H}odge index theorem for
  adelic line bundles}, Math. Ann. \textbf{367} (2017), no.~3-4, 1123--1171.
  \MR{3623221}

\bibitem[YZ18]{YZ}
\bysame, \emph{On the averaged {C}olmez conjecture}, Ann. of Math. (2)
  \textbf{187} (2018), no.~2, 533--638. \MR{3744857}

\bibitem[YZZ09]{YZZ1}
Xinyi Yuan, Shou-Wu Zhang, and Wei Zhang, \emph{The {G}ross-{K}ohnen-{Z}agier
  theorem over totally real fields}, Compos. Math. \textbf{145} (2009), no.~5,
  1147--1162. \MR{2551992}

\bibitem[YZZ13]{YZZ}
\bysame, \emph{The {G}ross-{Z}agier formula on {S}himura curves}, Annals of
  Mathematics Studies, vol. 184, Princeton University Press, Princeton, NJ,
  2013. \MR{3237437}

\bibitem[Zha20]{Zha20}
Shou-Wu Zhang, \emph{Standard conjectures and height pairings}, arXiv preprint
  arXiv:2009.07089 (2020).

\bibitem[Zha21a]{Zha19}
W.~Zhang, \emph{Weil representation and arithmetic fundamental lemma}, Ann. of
  Math. (2) \textbf{193} (2021), no.~3, 863--978. \MR{4250392}

\bibitem[Zha21b]{ZZY}
Zhiyu Zhang, \emph{Maximal parahoric arithmetic transfers, resolutions and
  modularity}, arXiv preprint arXiv:2112.11994 (2021).

\end{thebibliography}

\end{document}